\documentclass[oneside,english,reqno]{amsart}
\usepackage[T1]{fontenc}
\usepackage[latin9]{inputenc}
\usepackage{geometry}
\usepackage{verbatim}
\usepackage{amsthm}
\usepackage{amstext}
\usepackage{amssymb}
\usepackage{graphicx}
\usepackage{esint}
\usepackage{url}
\usepackage{yhmath}
\usepackage[all]{xy}
\usepackage{color}

\usepackage{shuffle}

\usepackage{enumitem}
\usepackage{scalerel}

\usepackage{thm-restate}

\makeatletter

\numberwithin{equation}{section}
\numberwithin{figure}{section}
\theoremstyle{plain}
\newtheorem*{thm*}{\protect\theoremname}
\theoremstyle{plain}
\newtheorem{thm}{\protect\theoremname}[section]
\theoremstyle{plain}
\newtheorem{lem}[thm]{\protect\lemmaname}
\theoremstyle{plain}
\newtheorem{rem}[thm]{\protect\remarkname}
\theoremstyle{plain}

\theoremstyle{plain}

\theoremstyle{plain}
\newtheorem*{prop*}{\protect\propositionname}
\theoremstyle{plain}
\newtheorem{prop}[thm]{\protect\propositionname}
\theoremstyle{plain}
\newtheorem*{cor*}{\protect\corollaryname}
\theoremstyle{plain}
\newtheorem{cor}[thm]{\protect\corollaryname}
\theoremstyle{plain}

\theoremstyle{plain}
\newtheorem{defn}[thm]{Definition}
\theoremstyle{plain}

\usepackage{mathrsfs}
\usepackage{setspace,wrapfig}

\AtBeginDocument{

}

\makeatother

\usepackage{babel}
\providecommand{\corollaryname}{Corollary}
\providecommand{\lemmaname}{Lemma}
\providecommand{\propositionname}{Proposition}
\providecommand{\remarkname}{Remark}
\providecommand{\theoremname}{Theorem}
\providecommand{\conjecturename}{Conjecture}
\providecommand{\notename}{Note}
\providecommand{\examplename}{Example}

\begin{document}


\global\long\def\sF{\mathcal{F}}
\global\long\def\sZ{\mathcal{Z}}
\global\long\def\sD{\mathcal{D}}
\global\long\def\sC{\mathcal{C}}
\global\long\def\sL{\mathcal{L}}
\global\long\def\sA{\mathcal{A}}
\global\long\def\sR{\mathcal{R}}
\global\long\def\sS{\mathcal{S}}
\global\long\def\sP{\mathcal{P}}
\global\long\def\sM{\mathcal{M}}

\global\long\def\bR{\mathbb{R}}
\global\long\def\bRpos{\bR_{> 0}}
\global\long\def\bRnn{\bR_{\geq 0}}
\global\long\def\bZ{\mathbb{Z}}
\global\long\def\bN{\mathbb{N}}
\global\long\def\bZpos{\mathbb{Z}_{> 0}}
\global\long\def\bZnn{\mathbb{Z}_{\geq 0}}
\global\long\def\bQ{\mathbb{Q}}
\global\long\def\bC{\mathbb{C}}

\global\long\def\PR{\mathbb{P}}
\global\long\def\EX{\mathbb{E}}

\global\long\def\one{\scalebox{1.1}{\textnormal{1}} \hspace*{-.75mm} \scalebox{0.7}{\raisebox{.3em}{\bf |}} }

\global\long\def\bD{\mathbb{D}}
\global\long\def\bH{\mathbb{H}}
\global\long\def\re{\Re\mathfrak{e}}
\global\long\def\im{\Im\mathfrak{m}}
\global\long\def\arg{\mathrm{arg}}
\global\long\def\ii{\mathfrak{i}}
\global\long\def\domain{D}
\global\long\def\bdry{\partial}
\global\long\def\cl#1{\overline{#1}}
\global\long\def\confmap{\varphi}
\global\long\def\zbar{\bar{z}}

\newcommand{\invbreve}[1]{\overset{\rotatebox{180}{$\breve{}\,$}}{#1}}

\global\long\def\diam{\mathrm{diam}}
\global\long\def\dist{\mathrm{dist}}

\global\long\def\OO{\mathcal{O}}
\global\long\def\oo{\mathit{o}}

\global\long\def\ud{\mathrm{d}}
\global\long\def\der#1{\frac{\ud}{\ud#1}}
\global\long\def\pder#1{\frac{\partial}{\partial#1}}
\global\long\def\pdder#1{\frac{\partial^{2}}{\partial#1^{2}}}
\global\long\def\pddder#1{\frac{\partial^{3}}{\partial#1^{3}}}

\global\long\def\SLE{\mathrm{SLE}}
\global\long\def\SLEk{\mathrm{SLE}_{\kappa}}

\global\long\def\caglad{\smash\gamma^\flat}
\global\long\def\cadlag{\smash\gamma^\sharp}
\global\long\def\newmathringcadlag{\mathring{\gamma}^\sharp}

\newcommand{\wt}[1]{\smash{\scalebox{.5}{$\widetilde{\scalebox{2}{$#1$}}$}}}

\global\long\def\Poisson{N}
\global\long\def\PoissonComp{\smash{\wt{\Poisson}}}

\global\long\def\PreGirsB{B}
\global\long\def\GirsB{\breve{B}}
\global\long\def\PreGirsN{N}
\global\long\def\GirsN{\breve{N}}
\global\long\def\GirsPR{\breve{\PR}}
\global\long\def\GirsEX{\breve{\EX}}
\global\long\def\RN{R}

\global\long\def\mgle{\smash{\check{M}}}

\global\long\def\microDriver{\smash{\wt{W}}}
\global\long\def\macroDriver{W}
\global\long\def\jump{v}

\global\long\def\growing{\xi^{\text{out}}}
\global\long\def\grown{\xi^{\text{in}}}
\global\long\def\growingpt{z^{\text{out}}}
\global\long\def\grownpt{z^{\text{in}}}

\global\long\def\maxjumpH#1{\smash{\lambda_{#1}^{\textnormal{h\"ol}}}}
\global\long\def\maxjumpT#1{\smash{\lambda_{#1}^{\textnormal{tr}}}}

\newcommand{\ThetaHmax}[2]{\smash{\theta_{#1,#2}^{\textnormal{h\"ol}}}}
\newcommand{\ThetaTmax}[2]{\smash{\theta_{#1,#2}^{\textnormal{tr}}}}
\newcommand{\ThetaTmaxbeta}[2]{\smash{\vartheta_{#1,#2}^{\textnormal{tr}}}}
\newcommand{\maxbetaT}[2]{\smash{\alpha_{#1,#2}^{\textnormal{tr}}}}

\global\long\def\rparamH{r}
\global\long\def\rparamT{r}

\global\long\def\Grid{\mathcal{G}}


\allowdisplaybreaks



\author{E.~Peltola and A.~Schreuder}

\

\vspace{2.5cm}

\begin{center}
\LARGE \bf \scshape{
Loewner traces driven by L\'evy processes
}
\end{center}

\vspace{0.75cm}

\begin{center}
{\large \scshape Eveliina Peltola}\\
{\footnotesize{\tt eveliina.peltola@hcm.uni-bonn.de}}\\
{\small{Department of Mathematics and Systems Analysis,}}\\
{\small{P.O. Box 11100, FI-00076, Aalto University, Finland}}\\
{\small{and}}\\
{\small{Institute for Applied Mathematics, University of Bonn}}\\
{\small{Endenicher Allee 60, D-53115 Bonn, Germany}}\\
\bigskip{} \bigskip{} 
{\large \scshape Anne Schreuder}\\
{\footnotesize{\tt as3068@cam.ac.uk}}\\
{\small{Department of Pure Mathematics and Mathematical Statistics, University of Cambridge}}\\
{\small{Wilberforce Road, Cambridge CB3 0WA, United Kingdom}}
\end{center}

\vspace{0.75cm}

\begin{center}
\begin{minipage}{0.85\textwidth} \footnotesize
{\scshape Abstract.}
Loewner chains with L\'evy drivers have been proposed 
as models for random dendritic growth in two dimensions, and 
as candidates for finding extremal multifractal spectra in problems in classical function theory. 
These processes are not scale-invariant in general, but they do enjoy a natural domain Markov property thanks to the stationary independent increments of L\'evy processes. 
The associated Loewner hulls feature remarkably intricate topological properties, 
of which very little is known rigorously.

\qquad
We prove that a chordal Loewner chain driven by a L\'evy process $W$ satisfying mild regularity conditions (including stable processes) is a.s.~generated by a c\`adl\`ag curve. 
Specifically, if the diffusivity parameter of the driving process $W$ is $\kappa \in [0,8)$, then the jump measure of $W$ is required to be locally (upper) Ahlfors regular near the origin, while if $\kappa > 8$, no constraints are imposed.  
In particular, we show that the associated Loewner hulls are a.s.~locally connected and path-connected. 
We also show that, the complements of the hulls are a.s.~H\"older domains when $\kappa \neq 4$ (which is not expected to hold when $\kappa=4$), without any regularity assumptions. 
The proofs of these results mainly rely on careful derivative estimates for both the forward and backward Loewner maps obtained using delicate but robust enough supermartingale domination arguments. 
As one cannot control the jump accumulation of general Levy processes, we must circumvent all reasoning that would use continuity. To prove the local connectedness, we use an extension of part of the Hahn-Mazurkiewicz theorem: hulls generated by c\`adl\`ag curves are locally connected even when jumps would occur at infinite intensity.
\end{minipage}
\end{center}

\newpage

\setcounter{tocdepth}{2}
\tableofcontents

\newpage

\bigskip{}
\section{\label{sec: intro}Introduction}
In this work, we consider the geometry of growing Loewner hulls whose driving function is a general L\'evy process.
When the driving function is a standard Brownian motion with diffusivity parameter $\kappa \geq 0$, the hulls are generated by 
a random continuous fractal curve (Loewner trace)~\cite{LSW:Conformal_invariance_of_planar_LERW_and_UST, Rohde-Schramm:Basic_properties_of_SLE}:
the celebrated Schramm-Loewner evolution ($\SLE_\kappa$),
which has turned out to be a universal and remarkably useful object in probability theory and mathematical physics\footnote{For instance, $\SLE_\kappa$ processes have played a key role in establishing rigorous results for scaling limits of many critical lattice models, e.g., 
in~\cite{Schramm:Scaling_limits_of_LERW_and_UST, Smirnov:Critical_percolation_in_the_plane, LSW:Conformal_invariance_of_planar_LERW_and_UST, Schramm:ICM, Schramm-Sheffield:Contour_lines_of_2D_discrete_GFF,  CDHKS:Convergence_of_Ising_interfaces_to_SLE},
and for important questions in probability theory and conformal geometry: 
Brownian intersection exponents 
\cite{Duplantier-Kwon:Conformal_invariance_and_intersections_of_random_walks, LSW:Brownian_intersection_exponents1, LSW:Brownian_intersection_exponents2, LSW:Brownian_intersection_exponents3, Werner:Girsanovs_transformation_for_SLE_kappa_rho_intersection_exponents_and_hiding_exponents} 
and Hausdorff dimension of the Brownian frontier 
\cite{LSW:The_dimension_of_the_planar_Brownian_frontier_is_four_thirds},
constructions of conformal restriction measures 
\cite{LSW:Conformal_restriction_the_chordal_case},
couplings with the Gaussian free field~\cite{Dubedat:SLE_and_free_field, Sheffield-Miller:Imaginary_geometry1, Sheffield-Miller:Imaginary_geometry2, Sheffield-Miller:Imaginary_geometry3, Sheffield-Miller:Imaginary_geometry4},
constructions of random metric or measure  spaces 
\cite{DMS:Liouville_quantum_gravity_as_mating_of_trees, BGS:Permutons_meanders_and_SLE-decorated_Liouville_quantum_gravity} (see also references therein),
and recent results concerning the relationship of fractal objects in random geometry (such as $\SLEk$ type paths) with conformal field theory, see~\cite{Peltola:Towards_CFT_for_SLEs, ARS:FZZ_formula_of_boundary_Liouville_CFT_via_conformal_welding} and references therein.}.
Its applications are manifestly conformally invariant, and hence rather special, while the purpose of the present work is to relax this constraint slightly.

Namely, the wide applicability of random Loewner evolutions motivates one to consider more general driving functions than the ubiquitous Brownian motion. 
For example, Loewner evolutions driven by certain Bessel processes were used in~\cite{Sheffield:Exploration_trees_and_CLEs}
to construct a random continuum exploration tree, pertaining to the scaling limit of an exploration tree in critical percolation encoding a loop ensemble of its interfaces. 
One can also approximate Bessel processes by ones with random jumps, as in~\cite[Section~3.2]{Sheffield:Exploration_trees_and_CLEs}.
This process is still conformally invariant, as Bessel processes satisfy the Brownian scaling property. 
A~related conformally invariant excursion process was considered 
in~\cite{Pete-Wu:Conformally_invariant_growth_process_of_SLE_excursions}.
In contrast, 
in~\cite{ROKG:Stochastic_Loewner_evolution_driven_by_Levy_processes}
a generalization of $\SLEk$ involving a stable L\'evy process was suggested to be useful for describing models related to diffusion-limited aggregation~\cite{Hastings-Levitov:Laplacian_growth_as_one-dimensional_turbulence, Carleson-Makarov:Aggregation_in_the_plane_and_Loewners_equation}
(although we have not been able to identify the inferred separate publication by the authors addressing this application), 
as well as to general branching, dendritic, possibly off-critical (non-scale invariant) models.
Such a model was also investigated in~\cite{Johansson-Sola:Rescaled_LLE_hulls_and_random_growth}, where it was related to a growth process similar to a Hastings-Levitov model.
See also~\cite{Sheffield-Miller:QLE,NST:Scaling_limits_for_planar_aggregation_with_subcritical_fluctuations} for other examples of general growth models.
As further generalizations, let us mention that 
Loewner hulls whose driving function is composed with a random time change were considered in~\cite{KLS:Effect_of_random_time_changes_on_Loewner_hulls}, 
where it was shown that a time-changed Brownian motion process does not always generate a simple curve; 
while in~\cite{MSY:On_Loewner_chains_driven_by_semimartingales_and_complex_Bessel-type_SDEs} it was shown that Loewner hulls with a certain class of regular enough continuous semimartingale drivers are generated by continuous curves.
There has also been some interest in generalizations of Loewner evolutions to complex drivers~\cite{Tran:Loewner_equation_driven_by_complex-valued_functions, 
Lind-Utley:Phase_transition_for_a_family_of_complex-driven_Loewner_hulls, Gwynne_Pfeffer:Loewner_evolution_driven_by_complex_Brownian_motion}.

A crucial property of SLE curves, which allows one to relate them to many statistical physics models, is 
that they can be generated by adding infinitesimal independent and stationary increments together in a scale-invariant manner. 
More precisely, requiring this property results in a characterization that the growth of an SLE curve must be governed by a scalar multiple of Brownian motion (possibly with a drift, which one can exclude by reflection symmetry). 
Relaxing the scale-invariance but keeping the requirement of independent, stationary increments, one is naturally led to consider growth processes arising from L\'evy drivers. 
(In some special cases, such as for stable L\'evy processes, the hulls have a specific scaling property that also changes the spatial scale of the process, see~\cite[Section~3.2]{Chen-Rohde:SLE_driven_by_symmetric_stable_processes}.) 
Nonetheless, the independent and stationary increments of the driving function translate into a domain Markov property of the corresponding Loewner evolutions, similar to that for 
$\SLE_\kappa$. 
In turn, jumps in the driving function induce branching behavior in the Loewner hulls. 
As a result, L\'evy Loewner evolutions are highly fractal trees that come with a natural embedding into the complex plane, and which satisfy the domain Markov property. 
As such, they are natural candidates for scaling limits of planar statistical mechanics models displaying treelike behavior in the limit.

In general, replacing Brownian motion by a L\'evy process 
yields hulls which can have a very complicated structure: 
they can be tree-like, forest-like, or looptree-like, say,
with several or countably infinitely many components --- 
and it is not clear at all that these would even be path-connected (for an illustration, see the simulations in~\cite[Section~3]{ROKG:Stochastic_Loewner_evolution_driven_by_Levy_processes}). 
These growth processes thus provide descriptions of much more general random models than SLEs. 
The interest in them among mathematicians was also originally motivated by the belief that they could produce objects with large multifractal spectra~\cite{Beliaev-Smirnov:Harmonic_measure_and_SLE, Chen-Rohde:SLE_driven_by_symmetric_stable_processes}  
(possibly towards problems around Brennan's conjecture~\cite{Carleson-Jones:On_coefficient_problems_for_univalent_functions_and_conformal_dimension,
Pommerenke:Boundary_behaviour_of_conformal_maps,
Bertilsson:On_Brennans_conjecture_in_conformal_mapping}, as explained in~\cite[Page~2]{Chen-Rohde:SLE_driven_by_symmetric_stable_processes}),
and for their relationship with the problem of Bieberbach coefficients of conformal mappings~\cite{Loutsenko:SLE_correlation_functions_in_the_coefficient_problem, Loutsenko-Yermolayeva:New_exact_results_in_spectra_of_stochastic_Loewner_evolution, DNNZ:The_coefficient_problem_and_multifractality_of_whole-plane_SLE_and_LLE}. 
Some results towards understanding phase transitions of special cases of these processes 
have also been obtained in the literature~\cite{ROKG:Stochastic_Loewner_evolution_driven_by_Levy_processes,
ORGK:Global_properties_of_SLE_driven_by_Levy_processes, Guan-Winkel:SLE_and_aSLE_driven_by_Levy_processes}.

The goal of the present work is to extend the family of random planar fractal curves by showing that also for the case of a general L\'evy driver (under a mild regularity assumption, see Definition~\ref{def: Ahlfors regular intro}), the Loewner hulls are in fact generated by a c\`adl\`ag function with path-connected but branching growth profile (when $\kappa \neq 8$). 
This proves part of~\cite[Conjecture~1]{Guan-Winkel:SLE_and_aSLE_driven_by_Levy_processes}. 
Moreover, the associated hulls are locally connected and, when $\kappa \neq 4$, they bound H\"older domains. 
While our proof for the existence of the trace builds on the technique of deriving derivative estimates for moments of the associated conformal (Loewner) maps, 
in order to establish these estimates, we cannot follow the usual ``pointwise-in-time'' approach that relies on interpolation from dyadic (countably many) times (applicable for Brownian motion, which admits a strong modulus of continuity~\cite[Section~3.2]{Rohde-Schramm:Basic_properties_of_SLE}). 
Indeed, the needed estimates are rather subtle in general, due to the occurrence of accumulating jumps in the driving process, causing significant technical problems dealt with throughout this article. 
In addition, the results appearing in the literature~\cite{LSW:Conformal_invariance_of_planar_LERW_and_UST,
Rohde-Schramm:Basic_properties_of_SLE,
Guan:Cadlag_curves_of_SLE_driven_by_Levy_processes, Guan-Winkel:SLE_and_aSLE_driven_by_Levy_processes, 
Chen-Rohde:SLE_driven_by_symmetric_stable_processes} 
thus far rely on very specific martingales, that are unavailable for the general case. 
We therefore must content ourselves with supermartingale domination arguments to carry out the present work.
Thus, the estimates that we obtain are not expected or attempted to be sharp.
Nevertheless, our results include all of the prior known cases: the continuous $\SLE_\kappa$ curves~\cite{Rohde-Schramm:Basic_properties_of_SLE} (with $\kappa \neq 8$) as well as Loewner evolutions driven by symmetric $\alpha$-stable processes~\cite{Guan:Cadlag_curves_of_SLE_driven_by_Levy_processes, Guan-Winkel:SLE_and_aSLE_driven_by_Levy_processes, Chen-Rohde:SLE_driven_by_symmetric_stable_processes}\footnote{The article~\cite{Guan:Cadlag_curves_of_SLE_driven_by_Levy_processes}  announces results which follow as special cases of our main results.  
Unfortunately,~\cite{Guan:Cadlag_curves_of_SLE_driven_by_Levy_processes} has never been published, and we have not been able to verify the arguments presented there.}.

\noindent 
\begin{figure}
\includegraphics[width=.8\textwidth]{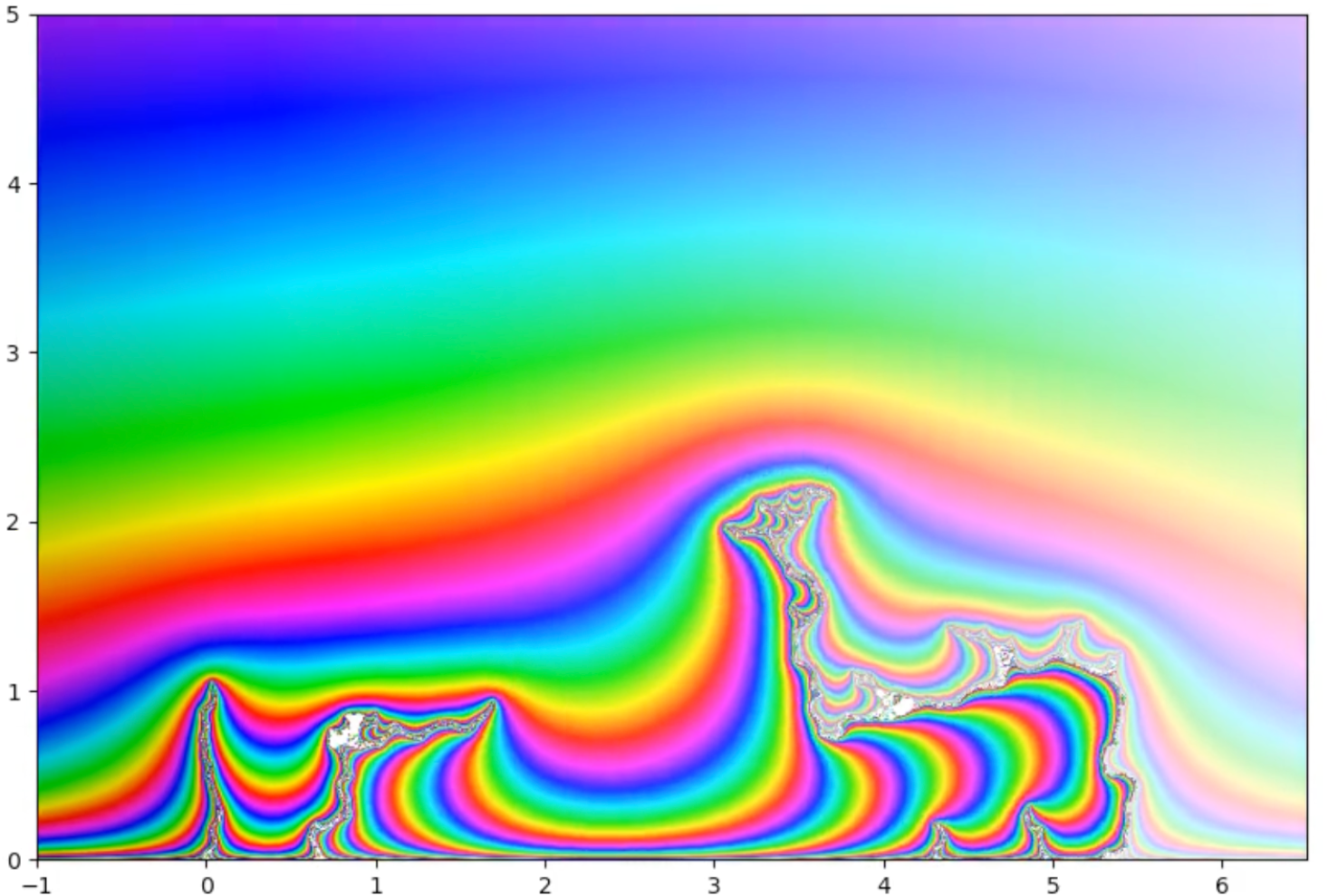}
\caption{\label{fig: illustration_LLE}
Illustration of the Loewner hull driven by a c\`adl\`ag function, 
made by Toby Cathcart Burn from his GitHub {\protect\url{https://github.com/penteract/sle}}.}
\end{figure}

\medskip{}
{\bf Statements of main results.}
Classical theory of Charles Loewner can be used to construct families of compact hulls in the plane growing in time.
More precisely, in the present work we are interested in \emph{L\'evy-Loewner hulls} 
whose growth is governed by random c\`adl\`ag driving functions of the form
\begin{align} \label{eq: general DF}
\macroDriver(t) = a t + \sqrt{\kappa} B(t) + \int_{|\jump| \leq 1} \jump \, \PoissonComp(t, \ud \jump) + \int_{|\jump| > 1} \jump \, \Poisson(t, \ud \jump) , \qquad
a \in \bR, \; \kappa \geq 0 ,
\end{align}
where $B$ is a standard one-dimensional Brownian motion,
$\Poisson = \smash{\Poisson_\nu}$ is an independent Poisson point process with L\'evy intensity measure $\nu$ (see below), and 
$\smash{\PoissonComp}(t, \ud \jump) := \Poisson(t, \ud \jump) - t \nu(\ud \jump)$ is the related compensated Poisson point process. 
The first term in~\eqref{eq: general DF} is a linear drift,
the second term is the diffusion component, 
and the last two terms represent the microscopic (small) and macroscopic (large) jumps, respectively. 
On any compact time interval, the process $W$ has only finitely many jumps of size greater than one, so the fourth term in~\eqref{eq: general DF} is a compound Poisson process (that is, a random finite sum of jumps).

Let us recall that a L\'evy measure is a non-negative Borel measure $\nu$ on $\bR$ such that $\nu(\{0\}) = 0$ and
\begin{align*}
\int_\bR (1 \wedge \jump^2) \, \nu(\ud \jump) < \infty.
\end{align*}
Note that $\nu$ may have infinite total mass (activity), in which case the process $W$ can have infinitely many small jumps (of size smaller than one) on compact time intervals.
As the small jumps might not be summable, one adds the compensator term $t \nu$ in $\smash{\PoissonComp}$ to ensure that~\eqref{eq: general DF} is well-defined and, moreover, 
the thus obtained compensated sum of microscopic jumps, i.e., the third term in~\eqref{eq: general DF}, is an $L^2$-martingale\footnote{Throughout, we work with a filtered probability space satisfying the usual conditions
(i.e., the filtration is right-continuous and the probability space is completed). 
All (in)equalities should be read as ``up to indistinguishability''.}.

Equation~\eqref{eq: general DF} is the \emph{L\'evy-It\^o decomposition} of a L\'evy process
(see~\cite[Theorem~2.4.16]{Applebaum:Levy_processes_and_stochastic_calculus}
or~\cite[Theorem~13.5.9]{Cohen-Elliott:Stochastic_calculus_and_applications}):
$\macroDriver$ is a c\`adl\`ag (right-continuous with left limits) stochastic process with independent and stationary increments.  
This property makes martingale arguments amenable for the analysis of Loewner chains driven by L\'evy processes. In particular, the Loewner chains satisfy a domain Markov property (see Figure~\ref{fig: Loewner maps}). 
However, the Brownian scaling property (and thus, conformal invariance) is lost when $\nu$ is non-trivial. The main technical difficulty compared to the case of Brownian motion is that the sample paths of L\'evy processes are almost surely discontinuous and can in particular have a dense set of jumps. 
Also, many computations in SLE theory crucially rely on the scale-invariance, 
which cannot be extended naturally to the present case. 
These difficulties explain the substantial technical work needed to carry out the proofs of the basic properties of L\'evy-Loewner hulls of the present article.

Given a driving process $W$ as in~\eqref{eq: general DF}, the Loewner equation (involving the right derivative $\partial_t^{+}$)
\begin{align*}
\partial_t^{+} g_t(z) = \frac{2}{g_t(z) - W(t)} \qquad \textnormal{with initial condition} \qquad g_0(z) = z 
\end{align*}
is an ordinary differential equation in time, which has a unique absolutely continuous solution $t \mapsto g_t(z)$ for each fixed point $z \in \bH$ in the upper half-plane $\bH = \{ z \in \bC \; | \; \im(z) > 0 \}$.
The solution is not in general defined for all times -- it exists up to the first time when the denominator of the Loewner equation is zero (see Section~\ref{sec: preli} for more details).
Importantly, for each time $t$ the map $g_t \colon \bH \setminus K_t \to \bH$ is a conformal bijection defined on a simply connected subset of $\bH$, 
which is the complement of a compact hull $K_t \subset \overline{\bH}$. 
The hulls $(K_t)_{t \geq 0}$ define a growth process in the (closure of the) upper half-plane.

In the special case of $\SLE_\kappa$ processes, i.e., $W = \sqrt{\kappa} B$, a crucial feature for many applications is that the associated growing hulls are generated by a \emph{continuous random curve}, the $\SLE_\kappa$ trace (in the sense detailed below). 
This is also --- almost --- the case for L\'evy-Loewner chains, as we shall prove in the present work: even though discontinuous, the trace will have left and right limits at all times, which makes it amenable to analysis.
The discontinuities in the driving function correspond to branching behavior in the Loewner trace. 
Moreover, the left-right-continuity ensures good topological properties of the hulls: path-connectedness and local path-connectedness 
(phrased in Proposition~\ref{prop: locally connected} and Corollary~\ref{cor: locally path-connected}).

\begin{defn} 
We say that the Loewner chain $(g_t)_{t \geq 0}$ with associated hulls 
$(K_t)_{t \geq 0}$
is \emph{generated} by a function $\eta \colon [0,\infty) \to \overline{\bH}$ if, for each $t \geq 0$, the set $\bH \setminus K_t$ is 
the unbounded connected component of $\bH \setminus \eta[0,t]$. 
\textnormal{(}In the literature, it is often assumed that $\eta$ is continuous, which we will not assume here.\textnormal{)}

If $\eta$ is c\`agl\`ad \textnormal{(}left-continuous with right limits\textnormal{)}, then we write $\eta = \caglad$ and say that 
the Loewner chain is \emph{generated by the c\`agl\`ad curve}~$\caglad$. 
In this case, we also use the term 
\emph{``generated by the c\`adl\`ag curve''}~$\cadlag$ \textnormal{(}right-continuous with left limits\textnormal{)}, that is the counterpart of $\caglad$ in the sense that
\begin{align*} 
\cadlag \colon [0,\infty) \to \overline{\bH} ,
\qquad \cadlag(t) := \lim_{s \to t+} \caglad(s) ,
\qquad \textnormal{and} \qquad 
\caglad(t) = \lim_{s \to t-} \cadlag(s) .
\end{align*}
Note that
$\cadlag[0,t] \cup \caglad[0,t] = \overline{\cadlag[0,t]} = \overline{\caglad[0,t]}$. 
We call either $\cadlag$ or $\caglad$ the \emph{Loewner trace}. 
\end{defn}

\begin{restatable}{prop}{LocConnProp} 
\label{prop: locally connected}
Let $(g_t)_{t \geq 0}$ be a Loewner chain with associated hulls $(K_t)_{t \geq 0}$ generated by a c\`adl\`ag curve $\cadlag$ \textnormal{(}viz.~a c\`agl\`ad curve $\caglad$\textnormal{)}. 
Then, for each $t \geq 0$, the set $\bdry (\bH \setminus K_t)$ is locally connected.
\end{restatable}

Proposition~\ref{prop: locally connected} could be regarded as an extension of part of the Hahn-Mazurkiewicz theorem. 
We prove it in Section~\ref{subsec: connectedness}.
It also implies that $K_t \cup \bR$ is path-connected for any hull generated by a c\`adl\`ag curve. 

\medskip

{\textbf{Setup.}}
Next, our main results are summarized in Theorems~\ref{thm: LLE curve if kappa not 8} and~\ref{thm: LLE Holder if kappa not 4}, both of which are proven in Section~\ref{subsec: conclusion}.
We work under one of the following assumptions:
\begin{enumerate}[label=\textnormal{\bf{Ass.~\arabic*.}}, ref=Ass.~\arabic*.]
\item \label{item: ass1 intro}
either the diffusivity parameter $\kappa > 8$,

\medskip

\item \label{item: ass2 intro}
or the diffusivity parameter $\kappa \in [0,8)$, and 
the variance measure of the L\'evy measure $\nu$ is locally (upper) Ahlfors regular near the origin in the sense of Definition~\ref{def: Ahlfors regular intro} below.
\end{enumerate}

For each L\'evy measure $\nu$, 
define a Borel measure $\mu_\nu$ by $\mu_\nu(A) := \int_A \jump^2 \, \nu(d\jump)$ for all Borel sets $A \subset \bR$. 
We call $\mu_\nu$ the \emph{variance measure} of the L\'evy measure $\nu$.

\begin{defn} \label{def: Ahlfors regular intro}
We say that the variance measure $\mu_\nu$ of the L\'evy measure $\nu$ is \emph{locally (upper) Ahlfors regular near the origin} if the following holds.
There exists $\epsilon_\nu \in (0,1/2)$ 
only depending on $\nu$ such that  
the restriction of $\mu_\nu$ to $[-\epsilon_\nu, \epsilon_\nu]$ is upper Ahlfors regular:
there exist constants 
$\alpha_\nu, c_\nu \in (0,\infty)$ 
and $\rho_\nu \in (0,1)$ 
only depending on $\nu$ such that 
for any $x \in [-\epsilon_\nu, \epsilon_\nu]$ and 
for any $\rho < \rho_\nu$, we have
\begin{align} \label{eq: Ahlfors regularity intro}
\mu_\nu ((x-\rho,x+\rho) \cap [-\epsilon_\nu, \epsilon_\nu]) 
\; = \;  \int_{x-\rho}^{x+\rho} \jump^2 \, \one_{[-\epsilon_\nu, \epsilon_\nu]}(\jump) \, \nu(\ud \jump)
\; \leq \; c_\nu \, \rho^{\alpha_\nu} .
\end{align}
Note that this implies that $\mu_\nu$ is dominated by the Lebesgue measure near the origin 
\textnormal{(}in particular, $\nu$ does not have atoms accumulating to the origin, but it may have atoms elsewhere\textnormal{)}.
\end{defn}

\begin{restatable}{thm}{MainCurveThm} 
\label{thm: LLE curve if kappa not 8}
Fix $T > 0$, $\kappa \in [0,\infty) \setminus\{8\}$, $a \in \bR$, and a L\'evy measure $\nu$. Suppose that either~\ref{item: ass1 intro} or~\ref{item: ass2 intro} holds. 
Then, the following hold almost surely for the Loewner chain driven by $\macroDriver$ on $[0,T]$.
\begin{enumerate}[label=\textnormal{(\alph*):}, ref=(\alph*)]
\item \label{item: LLE curve if kappa not 8}
The Loewner chain is generated by a c\`adl\`ag curve on $[0, T]$.

\medskip

\item \label{item: LLE locally conn if kappa not 8}
For each $t \in [0, T]$, the map $z \mapsto g_t^{-1}(z)$ extends to a continuous map from $\overline{\bH}$ onto $\overline{\bH \setminus K_t}$. 
\end{enumerate}
\end{restatable}

In short, the proof of this result relies on 
a decomposition of the L\'evy process into parts with large (easy to manage) and small (hard to manage) jumps, 
careful estimates for moments of derivatives of the \emph{forward} Loewner flow 
(comprising Section~\ref{sec: forward bounds})
and local connectedness\footnote{Note that here, is it crucial that the boundaries $\bdry(\bH \setminus K_t)$ are locally connected simultaneously for all times $t$, which cannot be established from pointwise in time almost sure arguments.} 
of the associated hulls (Proposition~\ref{prop: locally connected}). 
The needed arguments are carried out in several carefully set-up steps (Sections~\ref{sec: forward bounds}--\ref{sec: main results}).

A common approach to prove the existence of the trace uses the \emph{backward} Loewner flow (discussed in Section~\ref{sec: backward bounds}), which gives estimates for the derivative of the inverse map 
$g_t^{-1}$ \emph{pointwise in time}. However, for Loewner chains driven by general L\'evy processes, such pointwise-in-time estimates do not seem sufficient because general L\'evy processes do not seem to admit a c\`adl\`ag modulus of continuity which would be equally strong as that of Brownian motion (see~\cite{Durand:Singularity_sets_of_Levy_processes, Jaffard:The_multifractal_nature_of_Levy_processes, Balanca:Fine_regularity_of_Levy_processes_and_linear_fractional_stable_motion}).
We therefore derive estimates \emph{uniformly} in time, which can only be achieved by considering the \emph{forward} flow directly. 
To this end, we utilize a discrete grid approximation of the forward flow, similar to the recent~\cite{MSY:On_Loewner_chains_driven_by_semimartingales_and_complex_Bessel-type_SDEs, Yuan:Refined_regularity_of_SLE}. 
Combining it with extremely careful tuning of the various parameters appearing in the estimates in Section~\ref{sec: forward bounds}, we are able to obtain the  necessary control under the above assumptions 
(\ref{item: ass1 intro} or~\ref{item: ass2 intro}). 

To explain the usage of these assumptions, 
let us note that the essential problem in the forward flow estimates seems to occur when the hulls swallow small regions (``bubbles'') by a jumping mechanism. 
More precisely, the possibility of 
closing bubbles by accumulated jumps rather than by a continuous trace makes a term in the stochastic differential of our observable process of Section~\ref{subsec: Supermartingale bounds} difficult to control.
Our local upper Ahlfors regularity assumption kills such behavior sufficiently well.
In the case where the diffusivity parameter $\kappa > 8$ (the anticipated space-filling regime), this assumption is not needed. 
Also, in the case where the diffusivity parameter $\kappa \in [0,4)$ (the anticipated regime where the trace should be simple), it seems possible to argue the existence of the trace by utilizing the property that the hulls have an empty interior, similarly as in~\cite[Theorem~7.1]{Chen-Rohde:SLE_driven_by_symmetric_stable_processes} and~\cite[Theorem~1.3(i)]{Guan-Winkel:SLE_and_aSLE_driven_by_Levy_processes}. 
However, because this argument cannot hold in the cases where $\kappa \in (4,8)$ anyway, while our estimates are quite robust, we shall only present the case $\kappa < 8$ under the local upper Ahlfors regularity assumption~(\ref{item: ass2 intro}).

A special case of Theorem~\ref{thm: LLE curve if kappa not 8} where $W$ is a symmetric stable pure jump process was established in~\cite[Theorem~7.1]{Chen-Rohde:SLE_driven_by_symmetric_stable_processes},
where Chen \& Rohde first prove the local connectedness of the hulls for all times, and then use this to conclude that the trace exists and is c\`adl\`ag.
Their proof of the local connectedness relies on the fact that the hulls have an empty interior, proven by Guan \& Winkel~\cite[Theorem~1.3(i)]{Guan-Winkel:SLE_and_aSLE_driven_by_Levy_processes},
and a result from complex analysis due to Warschawski from the 1950s (concerning the modulus of continuity of conformal maps of the disc). 
This argument is not available for the general case considered in the present work. 
Our proof proceeds, in a sense, in converse order: 
we first show the existence of a c\`adl\`ag generating curve by careful derivative estimates for the associated Loewner maps, 
and we then prove, using the generating curve, that the hulls are indeed locally (path-)connected (Proposition~\ref{prop: locally connected}).

Let us finally remark that the estimates leading to Theorem~\ref{thm: LLE curve if kappa not 8}~\ref{item: LLE curve if kappa not 8}  
are not strong enough to show that 
the Loewner chain is generated by a c\`adl\`ag curve when $\kappa = 8$, 
although we believe that this is the case. In the case where no jumps are allowed, 
it is already known from~\cite{LSW:Conformal_invariance_of_planar_LERW_and_UST} that $\SLE_8$ driven by $\sqrt{8} B$ is the scaling limit of an uniform spanning tree (UST) Peano curve, which shows a posteriori that it is indeed generated by a continuous curve almost surely. 
(Recently, analytical proofs for this latter fact were given 
using couplings of SLE with the Gaussian free field~\cite{KMS:Regularity_of_the_SLE4_uniformizing_map_and_the_SLE8_trace, 
Ambrosio-Miller:Continuous_proof_of_the_existence_of_SLE8_curve}.)  
Unfortunately, this approach seems not useful for the more general case of drivers with jumps, for instance because no such coupling is known --- nor expected --- to exist. Note also that the modulus of continuity for the $\SLE_8$ curve (in the capacity parameterization) is logarithmic~\cite{Alvisio-Lawler:Modulus_of_continuity_of_SLE8, KMS:Regularity_of_the_SLE4_uniformizing_map_and_the_SLE8_trace}, and we expect that adding jumps will not improve the regularity. 

\medskip

Next, our second main result, which is the pointwise-in-time H\"older regularity of the inverse map $g_t^{-1}$, 
can be obtained from estimates for moments of derivatives for the \emph{backward} Loewner flow (comprising Section~\ref{sec: backward bounds}) 
combined with 
a decomposition of the L\'evy process into parts with large and small jumps as before, 
and the fact that compositions of H\"older continuous maps are still H\"older continuous (though with unknown H\"older exponent; thus the constants in Theorem~\ref{thm: LLE Holder if kappa not 4} are random). 
The proof of the following Theorem~\ref{thm: LLE Holder if kappa not 4} is completed in Section~\ref{sec: main results} using the estimates from Section~\ref{sec: backward bounds}.

\begin{restatable}{thm}{MainRegThm} 
\label{thm: LLE Holder if kappa not 4}
Fix $t > 0$, $\kappa \in [0,\infty) \setminus\{4\}$, $a \in \bR$, and a L\'evy measure $\nu$.
Then, the following hold almost surely for the Loewner chain driven by $\macroDriver$.
\begin{enumerate}[label=\textnormal{(\alph*):}, ref=(\alph*)]
\item \label{item: LLE Holder if kappa not 4}
$\bH \setminus K_t$ is a H\"older domain, 
meaning that there exist random constants $\theta(\kappa, \nu ,t) \in (0,1]$ and $H(\theta,t) \in (0,\infty)$ such~that
\begin{align*} 
| g_t^{-1}(z) - g_t^{-1}(w) | \leq H(\theta,t) \, \max \{ | z-w |^{\theta} , | z-w | \} \qquad \textnormal{for all } z,w \in \bH .
\end{align*}
In particular, the map $z \mapsto g_t^{-1}(z)$ extends to a continuous map from $\overline{\bH}$ onto $\overline{\bH \setminus K_t}$. 

\medskip

\item \label{item: LLE H-dim if kappa not 4}
The Hausdorff dimension of $\bdry K_t$ satisfies $\mathrm{dim} (\bdry K_t) < 2$, and we have $\mathrm{area} (\bdry K_t) = 0$. 
\end{enumerate}
\end{restatable}

Let us cautiously observe that item~\ref{item: LLE Holder if kappa not 4} of Theorem~\ref{thm: LLE Holder if kappa not 4} only implies that $g_t^{-1}$ extends to the real line for each \emph{fixed} $t$, whereas item~\ref{item: LLE locally conn if kappa not 8} of Theorem~\ref{thm: LLE curve if kappa not 8} holds \emph{simultaneously} for all $t \in [0,T]$. The latter result uses the assumptions \ref{item: ass1 intro} or~\ref{item: ass2 intro}, while the former needs not.

It is known~\cite{GMS-Almost_sure_multifractal_spectrum_of_SLE} that for $\SLE_4$ driven by $\sqrt{4} B$, the complements $\bH \setminus K_t$ of the hulls are not H\"older domains (see also Remark~\ref{rem: inverse map param holder optimal}). 
We do not expect item~\ref{item: LLE Holder if kappa not 4} of Theorem~\ref{thm: LLE Holder if kappa not 4} to hold with $\kappa = 4$, while item~\ref{item: LLE H-dim if kappa not 4} of Theorem~\ref{thm: LLE Holder if kappa not 4} should hold with $\kappa = 4$ as well. (We shall not, however, pursue this here.)

Let us finally recall that the boundary of any H\"older domain is conformally removable~\cite[Corollary~2]{Jones-Smirnov:Removability_theorems_for_Sobolev_functions_and_quasiconformal_maps},
but this is not at all clear for other kinds of fractals. 
For $\SLEk$ with $\kappa = 4$, it was proven only very recently~\cite{KMS:Conformal_removability_of_SLE4} 
that the $\SLEk$ curve is indeed conformally removable, using couplings of SLE with the Gaussian free field (GFF).
We do not foresee that those techniques could be adapted as such to the present setup because there are no couplings with the GFF available. 
In the general case, the jumps in $W$ might or might not make a difference: the particular property of $\SLE_4$ curves that they come arbitrarily close to themselves while still being simple curves makes their analysis harder and is at the heart of the failure of the H\"older property for the complementary domains. 
Introducing jumps could, in principle, prevent the curve managing to approach arbitrarily close to itself, so that one could expect its trace to be conformally removable. 
Conversely, jumps could also introduce additional ``dust'' behavior in the hulls, rendering the conformal removability impossible. However, 
knowing the fact that our results imply that 
such dust does not affect the local connectedness of L\'evy-Loewner hulls (at least under the assumptions \ref{item: ass1 intro} or~\ref{item: ass2 intro}), 
to us it appears plausible that conformal removability would also hold for general L\'evy-Loewner hulls with diffusivity parameter $\kappa=4$.

\medskip{}
{\bf Acknowledgments.}
We thank Zhen-Qing Chen, Steffen Rohde, and Jeff Steif for useful correspondences, 
and especially Yizheng Yuan for discussions concerning~\cite{Yuan:Refined_regularity_of_SLE}.
We thank the anonymous referee for very useful comments. 

A.S. is funded by the Cantab Capital Institute for the Mathematics of Information.

While carrying out this project, 
E.P.~has been supported 
by the Deutsche Forschungsgemeinschaft (DFG, German Research Foundation) under Germany's Excellence Strategy EXC-2047/1-390685813, 
the DFG collaborative research centre ``The mathematics of emerging effects'' CRC-1060/211504053,
as well as by 
the Academy of Finland Centre of Excellence Programme grant number 346315 ``Finnish centre of excellence in Randomness and STructures (FiRST)'',  
the Academy of Finland grant number 340461 ``Conformal invariance in planar random geometry'', 
and the European Research Council (ERC) under the European Union's Horizon 2020 research and innovation programme (101042460): 
ERC Starting grant ``Interplay of structures in conformal and universal random geometry'' (ISCoURaGe).

Some of this work was carried out while E.P. participated in a program hosted by the Mathematical Sciences Research Institute (MSRI) in Berkeley, California, during Spring 2022, 
supported by the National Science Foundation
under grant number DMS-1928930.

\bigskip{}
\section{\label{sec: preli}Loewner chains: basic properties and topology of Loewner hulls}
We will consider simply connected domains $\domain$ regarded as 
subsets of the Riemann sphere $\smash{\hat{\bC}} = \bC \cup \{\infty\}$, 
which we endow with either the Euclidean or the spherical metric depending on the context. 
Throughout, we denote the upper half-plane by $\bH = \{ z \in \bC \; | \; \im(z) > 0 \}$. 
By a conformal map we always refer to a biholomorphic function between two domains in the complex plane. 
For basic notions, readers may consult, e.g., the textbooks~\cite{Pommerenke:Boundary_behaviour_of_conformal_maps,
Lawler:Conformally_invariant_processes_in_the_plane, Kemppainen:SLE_book, Beliaev:Conformal_maps_and_geometry}.

The purpose of this section is to gather some terminology and  results for Loewner chains. 
\begin{itemize}[leftmargin=*]
\item
Section~\ref{subsec: conformal maps preli} concerns basic properties of 
conformal maps on the upper half-plane --- 
we, e.g., gather standard distortion estimates (Lemma~\ref{lem: Koebe}) 
and recall a well-known condition for H\"older continuity in terms of a derivative estimate
(Lemma~\ref{lem: global Holder continuity}). 

\smallskip

\item 
In Section~\ref{subsec: LC preli}, we collect basic notions of Loewner chains driven by c\`adl\`ag functions.

\smallskip

\item 
In Section~\ref{subsec: grid preli}, we derive estimates for such a Loewner flow in terms of a discrete approximating grid. 
While the content Sections~\ref{subsec: conformal maps preli},~\ref{subsec: LC preli}, and the subsequent Section~\ref{subsec: trace preli} are quite standard, 
to our knowledge only the recent~\cite{MSY:On_Loewner_chains_driven_by_semimartingales_and_complex_Bessel-type_SDEs} and \cite{Yuan:Refined_regularity_of_SLE} 
really develop a systematic usage of the grid approximation in Section~\ref{subsec: grid preli} to prove existence of Loewner traces with continuous driving functions.

\smallskip

\item 
Lastly, in Section~\ref{subsec: trace preli} we give a slight generalization of a well-known criterion for the Loewner chain to be generated by a trace --- in our case, a c\`adl\`ag curve (Proposition~\ref{prop: sufficient condition curve}).

\smallskip

\item 
In the final Section~\ref{subsec: connectedness}, we consider important topological properties of Loewner hulls: we prove in particular that hulls generated by a c\`adl\`ag curve are path-connected and locally (path-)connected. To us, the general results in Proposition~\ref{prop: locally connected} and Corollary~\ref{cor: locally path-connected} appear to be new. 

\end{itemize}

\subsection{Hulls and conformal maps}
\label{subsec: conformal maps preli}

We call a closed subset $K \subset \overline{\bH}$ a \emph{hull} 
if $K$ is bounded for the Euclidean metric and 
$\bH \setminus K$ is simply connected.
We write $\bdry K \subset \overline{\bH}$ and $\mathrm{int} (K) \subset \overline{\bH}$ 
respectively for the boundary and interior of the hull in the relative topology,
and $\bdry_{\mathrm{in}} K := \bdry K \cap \bH \subset \bH$. 
Riemann's mapping theorem implies that for each hull $K$, 
there exists a unique conformal map 
\begin{align} \label{eq: mof Laurent exp}
g_K \colon \bH \setminus K \to \bH ,
\qquad \qquad
g_K(z) = z + \sum_{n = 1}^{\infty} a_n(K) \, z^{-n} , \qquad |z| \to \infty ,
\end{align}
with real coefficients $a_n(K)$. We call $g_K$ the \emph{mapping-out function} of $K$ (normalized at $\infty$). 
The first coefficient $\mathrm{hcap}(K) := a_1(K) \geq 0$ in~\eqref{eq: mof Laurent exp} is 
always non-negative, and we call it the \emph{half-plane capacity} 
of the hull $K$. 
Intuitively, the half-plane capacity describes the size of $K$ as seen from $\infty$,
and it is an increasing function in the sense that 
$\mathrm{hcap}(K) \leq \mathrm{hcap}(K')$ for $K \subset K'$.

\subsubsection{Carath\'eodory convergence of hulls}
The following notion of convergence plays well with Loewner theory 
(see, e.g.,~\cite[Section~3.3]{Beliaev:Conformal_maps_and_geometry} and~\cite{Pommerenke:Boundary_behaviour_of_conformal_maps} for details). 
Let $(K_n)_{n \in \bZnn}$ be a sequence of hulls, which are uniformly bounded, i.e., there exists $R \in (0,\infty)$ such that $K_n \subset B(0,R)$ for all~$n$.
We say that $K_n$ converge in the \emph{Carath\'eodory sense} to a hull $K$ if
$\smash{g_{K_n}^{-1}}$ converge uniformly on all compact subsets of $\bH$ to $\smash{g_K^{-1}}$.
(Note that here, we assume that the limit is not trivial, i.e., $K$ is a hull, which is the only setup we will need later.)
Geometrically (see, e.g.,~\cite[Theorem~1.8]{Pommerenke:Boundary_behaviour_of_conformal_maps}),
this is equivalent to the \emph{Carath\'eodory kernel convergence}
of the complementary domains $\bH \setminus K_n$ to $\bH \setminus K$ 
with respect to any interior point $w_0 \in \bH \setminus B(0,R)$ in the following sense:
\begin{itemize}[leftmargin=2em]
\item $w_0 \in \bH \setminus K$, 

\smallskip

\item some neighborhood of every point $z \in \bH \setminus K$ belongs to $\bH \setminus K_n$ for sufficiently large $n$, and 

\smallskip

\item for each point $z \in \bdry (\bH \setminus K)$, there exists a sequence $z_n \in \bdry (\bH \setminus K_n)$ such that $z_n \to z$ as $n \to \infty$. 
\end{itemize}

\subsubsection{Distortion estimates}
To begin, we gather some standard estimates for distortion of conformal maps needed later. 
For two non-negative quantities $a,b$, we use the shorthand notation
\begin{align*}
& a \asymp b \qquad \Longleftrightarrow \qquad 
C^{-1} \, a \leq b \leq C \, a ,  && \textnormal{with $C \in (0,\infty)$ a universal constant}, \\
& a \lesssim b \qquad \Longleftrightarrow \qquad 
a \leq C \, b , &&\textnormal{with $C \in (0,\infty)$ a universal constant}.
\end{align*}

Recall that \emph{Koebe distortion and $1/4$ theorems} 
(see, e.g.,~\cite[Theorem~1.3~\&~Corollary~1.4]{Pommerenke:Boundary_behaviour_of_conformal_maps})
show that, for any conformal map $\phi$ 
on the unit disc $\bD = \{ z \in \bC \; | \; |z| = 1 \}$, we have 
\begin{align} 
\label{eq: Koebe quarter theorem}
\tfrac{1}{4} (1 - |z|)^2 \, |\phi'(z)| \, & \leq \, \dist( \phi(z), \bdry \phi(\bD)) \, \leq \, (1 - |z|)^2 \, |\phi'(z)| , 
\\
\label{eq: Koebe distortion theorem distance}
|\phi'(0)| \, \frac{|z|}{(1 + |z|)^2}
\, & \leq \, |\phi(z) - \phi(0)| \, \leq \,
|\phi'(0)| \, \frac{|z|}{(1 - |z|)^2} , 
\\
\label{eq: Koebe distortion theorem}
|\phi'(0)| \, \frac{1 - |z|}{(1 + |z|)^3}
\, & \leq \, |\phi'(z)| \, \leq \,
|\phi'(0)| \, \frac{1 + |z|}{(1 - |z|)^3} , \qquad z \in \bD .
\end{align}
We will use the following consequences of Koebe theorems for conformal maps on $\bH$.

\newpage

\begin{lem}[Koebe distortion in $\mathbb{H}$] \label{lem: Koebe}
For any conformal map $\confmap$ on $\bH$, the following hold.
\begin{enumerate}[label=\textnormal{(\alph*):}, ref=(\alph*)]
\item \label{item: Koebe1}
$|\confmap'(\ii y)| \asymp |\confmap'(\ii a y)|$ for all $y > 0$ and $a \in [1/2, 2]$.

\medskip

\item \label{item: Koebe2}
$|\confmap'(y(x + \ii))| \lesssim (1+x^2)^{3} |\confmap'(\ii y)|$ for all $y > 0$ and $x \in \bR$.

\medskip

\item \label{item: Koebe3}
For each fixed $w_0 = x_0 + \ii y_0 \in \bH$, the inverse map satisfies
\begin{align*}
| \confmap^{-1}(z) - w_0 | \, \leq \, \tfrac{1}{2} \, y_0
\qquad \textnormal{and} \qquad
\frac{48}{125}
\, \leq \, \frac{|(\confmap^{-1})'(z)| }{|(\confmap^{-1})'(\confmap(w_0))|} 
\, \leq \,
\frac{80}{27}
\end{align*}
for all $z \in B \big( \confmap(w_0), \tfrac{1}{8} \, y_0 \, | \confmap'(w_0) | \big)$.
\end{enumerate}
\end{lem}

\begin{proof}
Items~\ref{item: Koebe1} and~\ref{item: Koebe2} are standard applications of~\eqref{eq: Koebe distortion theorem}. 
Item~\ref{item: Koebe3} can be derived in a straightforward manner from the left inequality in~\eqref{eq: Koebe quarter theorem}, 
the right inequality in~\eqref{eq: Koebe distortion theorem distance}, and the estimate~\eqref{eq: Koebe distortion theorem}.   
\end{proof}

\subsubsection{H\"older continuity}

It is well known that controlling the derivative of a conformal map gives an estimate for its local H\"older continuity modulus. We will use this result in the following form.

\begin{lem}
\label{lem: global Holder continuity}
Fix\footnote{We use the convention that for $R = +\infty$, the domain is the upper half-plane $(-R,R) \times \ii (0,\infty) = \bH$.} $R \in [1,+\infty]$. 
Let $\confmap$ be a conformal map on $(-R,R) \times \ii (0,\infty)$,  
and let $\theta \in (0, 1]$.
The following are qualitatively equivalent, meaning that
$H(\theta,R)$ only depends on $\theta, R$, and $C(\theta,R)$, and $C(\theta,R)$ only depends on $\theta,R$ and $H(\theta,R)$.
\begin{enumerate}[label=\textnormal{(\alph*):}, ref=(\alph*)]

\item \label{item: global Holder continuity}
$\confmap$ is H\"older continuous with exponent $\theta$:  
there exists a constant 
$H(\theta,R) \in (0,\infty)$ such that 
\begin{align} \label{eq: global Holder continuity} 
| \confmap(z) - \confmap(w) | \leq H(\theta,R) \, 
\big( \, | z-w |^{\theta} \, \vee \, | z-w | \, \big) 
\qquad \textnormal{for all } z,w \in (-R,R) \times \ii (0,\infty) .
\end{align}

\medskip

\item \label{item: global derivative estimate}
There exists a constant $C(\theta,R) \in (0,\infty)$ such that
\begin{align} \label{eq: needed bound for derivative of confmap}
|\confmap'(z)| \leq C(\theta,R) \, 
\big( \, (\im(z))^{\theta - 1}\, \vee \, 1 \, \big) 
\qquad
\textnormal{for all } z \in (-R,R) \times \ii (0,\infty) .
\end{align}
\end{enumerate}
In particular, if either property holds, the map $\confmap$  extends to a continuous function on $\smash{\overline{(-R,R) \times \ii (0,\infty)}}$. 
\end{lem} 

\begin{proof}
The equivalence of~\ref{item: global Holder continuity} and~\ref{item: global derivative estimate}
is a standard application of Koebe distortion theorem, as detailed, e.g., in~\cite[Lemma~2.7]{Kinneberg-Loewner_chains_and_Holder_geometry}.
The implication
\ref{item: global Holder continuity}~$\Rightarrow$~\ref{item: global derivative estimate}
is almost immediate, while for the implication
\ref{item: global derivative estimate}~$\Rightarrow$~\ref{item: global Holder continuity}, 
 the idea is to integrate the bound~\eqref{eq: needed bound for derivative of confmap} along hyperbolic geodesics to obtain~\eqref{eq: global Holder continuity}.

To show that  $\confmap$ extends continuously to the boundary, 
first fix $0 < y_1 < y_2 \leq y \leq 1$ and integrate~\eqref{eq: needed bound for derivative of confmap} to obtain
\begin{align*}
|\confmap(x + \ii y_2) - \confmap(x + \ii y_1)| 
\; \leq \; \int_{y_1}^{y_2} |\confmap'(x + \ii u)| \, \ud u
\; \leq \; C(\theta,R) \int_{0}^{y} u^{\theta-1} \, \ud u 
\; \leq \; \frac{C(\theta,R)}{\theta} \, y^{\theta} .
\end{align*}
Taking $y \to 0+$, we see that the radial limit $\smash{\confmap(x) := \underset{y \to 0+}{\lim} \confmap(x + \ii y)}$ exists for all $x \in (-R,R)$.  Next, fix 
\begin{align*}
n \in \bZpos \textnormal{ and } x_1, x_2 \in (-R,R) \textnormal{ such that } x_1 \leq x_2 \textnormal{ and  } |x_2 - x_1 | \leq 2^{-n} , \textnormal{ and  }  y_1, y_2 \in [0,2^{-n}] . 
\end{align*}
Then, using the bound~\eqref{eq: needed bound for derivative of confmap} similarly as above, we have
\begin{align*}
& \; \; |\confmap(x_2 + \ii y_2) - \confmap(x_1 + \ii y_1)| \\
\leq & \; \;  |\confmap(x_2 + \ii y_2) - \confmap(x_2 + \ii 2^{-n})| 
\; + \; |\confmap(x_2 + \ii 2^{-n}) - \confmap(x_1 + \ii 2^{-n})| 
\; + \; |\confmap(x_1 + \ii 2^{-n}) - \confmap(x_1 + \ii y_1)| \\
\leq & \; \; \frac{2 \, C(\theta,R)}{\theta} \, 2^{-n \theta} 
\; + \; \int_{x_1}^{x_2} |\confmap'(u + \ii 2^{-n})| \, \ud u
\; + \; \frac{C(\theta,R)}{\theta} \, 2^{-n \theta} 
\; \leq \; \Big( \frac{2}{\theta} + 1 \Big) C(\theta,R) \, 2^{-n \theta}  ,
\end{align*}
which implies that $\confmap$ extends to a continuous function on $\smash{\overline{(-R,R) \times \ii (0,\infty)}}$. 
\end{proof}

\subsubsection{Boundary behavior}
Next we briefly discuss boundary behavior of conformal maps $\confmap \colon \bH \to \domain$
onto a simply connected domain $\domain \subsetneq \smash{\hat{\bC}}$. 
For extensive literature on this rather delicate subject, see the textbook~\cite[Chapter~2]{Pommerenke:Boundary_behaviour_of_conformal_maps}
and the more recent~\cite[Chapter~2]{Beliaev:Conformal_maps_and_geometry}. 
Let us recall a few basic notions:
\begin{itemize}[leftmargin=*]
\item
A \emph{crosscut} in $\domain$ is an open Jordan arc $S \subset \domain$ which touches the boundary
at its endpoints $a,b \in \bdry \domain$ (which may coincide): $\overline{S} = S \cup \{ a,b \} \subset \overline{\domain}$. 

\smallskip

\item
A \emph{null-chain} $(S_n)_{n \in \bZnn}$ is a sequence of nested crosscuts such that for all $n$, we have 
$S_n \cap S_{n+1} = \emptyset$,
the crosscut $S_n$ separates $S_0$ and $S_{n+1}$, and $\diam(S_n) \to 0$ as $n \to \infty$.

\smallskip

\item 
Two null-chains $(S_n)_{n \in \bZnn}$ and $(S_n')_{n \in \bZnn}$ are 
equivalent if and only if
for each sufficiently large $m$, the crosscut $S_m$ (resp.~$S_m'$) separates all but finitely many $S_n'$ from $S_{m-1}$ (resp.~$S_n$ from $S_{m-1}'$).

\smallskip

\item
A \emph{prime end} $\xi$ of $\domain$ is an equivalence class of null-chains.

\smallskip

\item The \emph{impression} of a prime end $\xi$ of $\domain$ is defined as 
\begin{align*}
I(\xi) := \bigcap_{n \in \bZnn} \overline{\mathrm{int}_{\mathrm{in}} (S_n)} ,
\end{align*}
where $\mathrm{int}_{\mathrm{in}} (S_n)$ is 
the interior of the connected component of $\domain \setminus S_n$ not containing $S_0$.
Note that $I(\xi)$ is a non-empty compact connected set, whence it is either a single point or a continuum. If $I(\xi)$ is a single point, then it is a boundary point of $\domain$ and we say that the prime end $\xi$ is \emph{degenerate}.

\smallskip

\item A set $A \subset \bC$ is (uniformly) \emph{locally connected} if for every $\varepsilon > 0$ there exists $\delta > 0$ such that, for any pair of points $z,w \in A$ such that $|z-w| < \delta$, there exists a closed connected set $S$ such that $z,w \in S \subset A$ and $\diam(S) < \varepsilon$.
By~\cite[Lemma~6.7]{Sagan:Space_filling_curves}, a sufficient condition for this 
is that $A$ is compact, connected, and locally connected at every point $z \in A$, that is, 
for every $z \in A$ and $\varepsilon > 0$, there exists a radius $r_{z,\varepsilon} > 0$ such that for every $w \in A \cap B(z,r_{z,\varepsilon})$, there exists a closed connected set $S$ such that $z,w \in S \subset A \cap B(z,\varepsilon)$.

\end{itemize}

Carath\'eodory's theorem (see~\cite[Chapter~2]{Pommerenke:Boundary_behaviour_of_conformal_maps})
implies that a conformal map $\confmap \colon \bH \to \domain$ 
extends to a homeomorphism $\overline{\bH} \to \overline{\domain}$ if and only if $\bdry \domain$ is a Jordan curve. 
Also, $\confmap$ has a continuous extension to $\overline{\bH}$
if and only if $\bdry \domain$ is locally connected, which
is also equivalent to $\bdry \domain$ being a continuous curve, 
but perhaps not an injection (in which case $\confmap$ has no inverse on $\bdry \domain$).
In any case, the conformal map $\confmap$ always induces a one-to-one correspondence between 
the boundary points of $\bH$ (also including $\infty \in \bdry \bH \subset \smash{\hat{\bC}}$) 
and the prime ends $\xi$ of $\domain$
(cf.~\cite[Theorem~2.15]{Pommerenke:Boundary_behaviour_of_conformal_maps}). 
We write $\xi = \smash{\invbreve{\confmap}}(x) \in \smash{\invbreve{\bdry}} \domain$ for the prime end $\xi$ corresponding to the boundary point $x \in \bdry \bH$, and $\smash{\invbreve{\bdry}} \domain = \smash{\invbreve{\confmap}}(\bdry \bH)$ for the boundary of $\domain$ comprising its prime ends. 
In particular, for any null-chain $\smash{(S_n)_{n \in \bZnn}}$ representing the prime end $\xi$ in $\domain$, its inverse image $\smash{(\confmap^{-1}(S_n))_{n \in \bZnn}}$ is a null-chain in $\bH$ that shrinks to $x = \smash{\invbreve{\confmap}}\smash{{}^{-1}}(\xi)$:
\begin{align*}
\{x\} = \bigcap_{n \in \bZnn} \overline{\mathrm{int}_{\mathrm{in}} (\confmap^{-1}(S_n))} .
\end{align*}

We say that a prime end $\xi$ is \emph{accessible} if,
for any interior point $w \in \domain$, there exists a Jordan arc $J$ in $\overline{\domain}$ starting at $w$ which lies entirely in $\domain$ except at its endpoint in $I(\xi) \cap \bdry \domain$. 
In this case, we say that $J$ accesses the prime end $\xi$, and the endpoint of $J$ is an accessible point.  
By~\cite[Proposition~2.14]{Pommerenke:Boundary_behaviour_of_conformal_maps}, $\confmap^{-1} (J)$ is then a curve in $\overline{\bH}$ which lies entirely in $\bH$ except at its endpoint in $\bdry \bH$.
Furthermore, if $J_1$ and $J_2$ are two Jordan arcs accessing two distinct prime ends of $\domain$,
then the curves $\confmap^{-1} (J_1)$ and $\confmap^{-1} (J_2)$ also have distinct endpoints in $\bdry \bH$.
(Here, it is crucial that the image domain of $\confmap^{-1}$ is nice, e.g., $\bH$.)

For any boundary point $x \in \bdry \bH$,
by~\cite[Corollary~2.17 and Exercise~2.5.5]{Pommerenke:Boundary_behaviour_of_conformal_maps},
if the (unrestricted) limit of $\confmap$ at $x$,
\begin{align} \label{eq: unrestricted limit}
\confmap(x) := \lim_{z \to x} \confmap(z) \, \in \, \bdry \domain \qquad \textnormal{along} \quad z \in \bH ,
\end{align}
exists, then the prime end $\xi = \invbreve{\confmap}(x)$ is degenerate and accessible, and we have $I(\xi) = \{ \confmap(x) \}$.

Conversely, if $J \colon [0,1) \to \domain$ is a Jordan arc accessing a prime end $\xi$ of $\domain$, then the limit of $\confmap$ exists along the curve $L := \confmap^{-1} \circ J \colon [0,1) \to \bH$ by~\cite[Corollary~2.17 and Exercise~5]{Pommerenke:Boundary_behaviour_of_conformal_maps}:
\begin{align} \label{eq: radial limit}
J(1) = \lim_{s \to 1-} \confmap(L(s)) \, \in \, \bdry \domain \qquad \textnormal{along} \quad L[0,1) \subset \bH ,
\end{align}
which is also equivalent to the existence of a \emph{radial limit} of $\confmap$ at $\xi$~\cite[Corollary~2.17(i)]{Pommerenke:Boundary_behaviour_of_conformal_maps}. 
(However, this does not guarantee the existence of the unrestricted limit~\eqref{eq: unrestricted limit}.)

\subsection{Loewner chains}
\label{subsec: LC preli}
 
Let $W \colon [0, \infty) \to \bR$ be a c\`adl\`ag function 
(i.e., right-continuous with left limits). 
A (chordal) \emph{Loewner chain} driven by $W$ 
(or, with driving function $W$)
is a family $(g_t)_{t \geq 0}$ of mapping-out functions 
which solve, for each $z \in \overline{\bH}$, the Loewner differential equation\footnote{Note that via Schwarz reflection, each mapping-out function $g_t = g_{K_t}$ extends to a conformal map on $\smash{\hat{\bC}} \setminus (K_t \cup K_t^*)$, where $K_t^*$ is the complex conjugate of $K_t$. Thus, $g_t = g_{K_t}$ is well-defined on $\bR \setminus K_t$.}
\begin{align} \label{eq: LE}
\tag{LE}
\partial_t^{+} g_t(z) = \frac{2}{g_t(z) - W(t)} \qquad \textnormal{with initial condition} \qquad g_0(z) = z ,
\end{align}
where $\partial_t^{+}$ denotes the right derivative 
and $t \mapsto g_t(z)$ is the unique absolutely continuous solution to~\eqref{eq: LE} defined up to the blow-up time
\begin{align} \label{eq: blow-up time}
\tau(z) := \sup \big\{ s \geq 0 \; | \; \inf_{u \in [0, s]} | g_u(z) - W(u) | > 0 \big\} \; \in \; [0,\infty] .
\end{align}

\begin{rem}
\textnormal{(}See also~\cite[Chapter~6]{Pommerenke:Univalent_functions}.\textnormal{)}
\label{rem: existence and uniqueness ODE}
The existence and uniqueness of an absolutely continuous solution $t \mapsto g_t(z)$ to~\eqref{eq: LE} follows from general ODE theory~\cite[Chapter~I.5.,~Theorems~5.1--5.3]{Hale:Ordinary_differential_equations}.
Indeed, the existence follows from checking the Carath\'eodory conditions:
for $(t, w) \in [0,\infty) \times \bH$ such that $|w - W(t)| > 0$, 
the map $\smash{t \mapsto \frac{2}{w - W(t)}}$ is measurable (namely, c\`adl\`ag),
the map $\smash{w \mapsto \frac{2}{w - W(t)}}$ is continuous, and 
$\smash{\frac{2}{|w - W(t)|}}$ is bounded on compacts.
The uniqueness follows since $\smash{w \mapsto \frac{2}{w - W(t)}}$ is locally Lipschitz.
Furthermore, the map $(t, z) \mapsto g_t(z)$ is also jointly continuous on
$\{ (t, z) \in  [0, \infty) \times \bH \; | \; t < \tau(z) \}$.
We gather some further properties of the mapping-out functions in Appendix~\ref{app: inverse Loewner chain}.
\end{rem}

One can show that the growing hulls associated to the Loewner chain have the form 
\begin{align*}
K_t = \{ z \in \overline{\bH} \, | \,  \tau(z) \leq t \} , \qquad
g_t = g_{K_t} \colon \bH \setminus K_t \to \bH ,
\end{align*}
and in particular,~\eqref{eq: LE} implies that for each $t \geq 0$, the conformal map $g_t$ is the associated mapping-out function normalized at $\infty$.
By the choice of the constant ``$2$'' in~\eqref{eq: LE},
the Loewner chain is parameterized by capacity, i.e., 
we have $\mathrm{hcap}(K_t) = 2 t$ for all $t \geq 0$.
For each fixed $z \in \overline{\bH}$, 
the blow-up time $\tau(z)$ of~\eqref{eq: LE} is the first time when the given point $z$ satisfies one of the following mutually exclusive properties:~it~is 
\begin{itemize}
\item either \emph{swallowed} by the growing hulls at time $\tau(z)$, i.e., we have
\begin{align*}
z \in \mathrm{int}(K_{\tau(z)}) \setminus \bigcup_{s < \tau(z)} K_s ,
\end{align*}
in which case we necessarily have
$\smash{\underset{t \to \tau(z)-}{\liminf}} \; |g_t(z) - W(t)| = 0$; 

\medskip

\item or \emph{hit} by the growing hulls at time $\tau(z)$, i.e., we have 
\begin{align*}
z \in (\bdry K_{\tau(z)}) \setminus \underset{s < \tau(z)}{\bigcup} K_s 
\qquad \textnormal{and} \qquad 
\liminf_{t \to \tau(z)-} \; |g_t(z) - W(t)| = 0 ;
\end{align*}

\medskip

\item or a \emph{branch point} at time $\tau(z)$, i.e., we have 
\begin{align*}
z \in \bdry K_{\tau(z)} \cup (\bR \setminus K_{\tau(z)}) 
\qquad \textnormal{but} \qquad 
\liminf_{t \to \tau(z)-} \; |g_t(z) - W(t)| > 0 ,
\end{align*}
in which case $W$ has a jump at time $\tau(z)$ and $g_{\tau(z)}(z) = W(\tau(z)+)$.
\end{itemize}
Note that swallowed points are never accessible from $\bH \setminus K_t$, while hit and branch points can be accessible or inaccessible from $\bH \setminus K_t$.

\begin{lem} \label{lem: some markov property}
\textnormal{(}See Figure~\ref{fig: Loewner maps}.\textnormal{)}
Fix $\sigma \geq 0$ and define for each $t \geq 0$ the sets
\begin{align*}
\big( \mathring{K}^\sigma_{t} \big)_{t \geq 0}
:= & \; \big( \overline{g_\sigma(K_{\sigma + t} \setminus K_\sigma) - W(\sigma)} \big)_{t \geq 0} .
\end{align*}
Then, $\smash{\mathring{K}^\sigma_{t}}$ are hulls parameterized by capacity driven by $\smash{\mathring{W}^\sigma}(t) := W(\sigma+t) - W(\sigma)$, and the associated mapping-out functions $\smash{\mathring{g}^\sigma_{t}}(z) 
= ( g_{\sigma+t} \circ g_\sigma^{-1} )(z + W(\sigma)) - W(\sigma)$ solve~\eqref{eq: LE} with driving function $\smash{\mathring{W}^\sigma}$. 
\end{lem}

\begin{proof}
The formula 
$\smash{\mathring{g}^\sigma_{t}}(z) 
= \smash{g_{\mathring{K}^\sigma_{t}}} (z)
= ( g_{\sigma+t} \circ g_\sigma^{-1} )(z + W(\sigma)) - W(\sigma)$
and the Loewner equation~\eqref{eq: LE} 
for the mapping-out functions follow from a computation
and the uniqueness of the expansion~\eqref{eq: mof Laurent exp}.
\end{proof}

We see from Lemma~\ref{lem: some markov property} also that the inverse maps $\smash{\mathring{f}^\sigma_{t}} :=(\smash{\mathring{g}^\sigma_{t}})^{-1} = (\smash{g_{\mathring{K}^\sigma_{t}}})^{-1}$ 
and $f_s := g_s^{-1}$ satisfy
\begin{align} \label{eq: inverse markov property}
f_{\sigma+t}(z) = f_\sigma \big( \mathring{f}^\sigma_{t}(z - W(\sigma)) + W(\sigma) \big) , \qquad t \geq 0 .
\end{align}

\noindent 
\begin{figure}
\includegraphics[width=\textwidth]{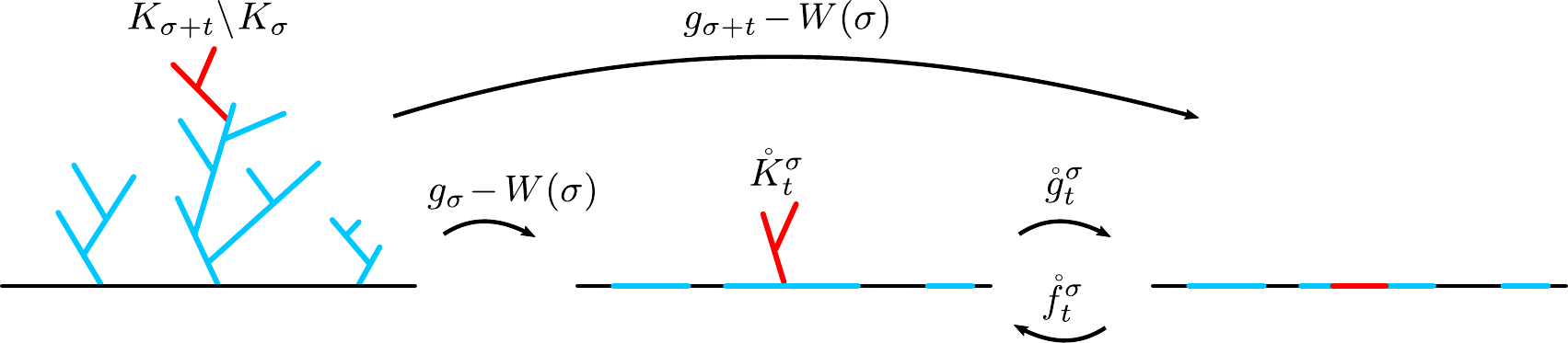}
\caption{\label{fig: Loewner maps}
Illustration of the maps and hulls in Lemma~\ref{lem: some markov property}.
}
\end{figure}

\subsubsection{Local growth}
The Loewner hulls are ``bilaterally'' locally growing in the sense that 
at each time $t \geq 0$, the following properties hold\footnote{In the literature, 
e.g.,~\cite[Theorem~2.6]{LSW:Brownian_intersection_exponents1} and~\cite[Chapter~4]{Kemppainen:SLE_book}, 
one usually considers~\eqref{eq: LE} with continuous driving functions, in 
which case the local growth property reads as follows:  
for all $\varepsilon > 0$ there exists $\delta = \delta(\varepsilon) > 0$ such that for each $t$, 
there exists a 
crosscut $S_{\delta} \subset \bH \setminus K_t$
with $\diam(S_{\delta}) < \varepsilon$ separating $K_{t + \delta} \setminus K_t$ from $\infty$ in $\bH \setminus K_t$. 
In particular, $\delta$ is uniform over $t$.
However, such uniformity fails for discontinuous driving functions, for instance when $K_t = \gamma[0,t]$ for a continuous curve $\gamma$ that crosses itself. 
In this case, the conditions involving $S_{\delta}^{\textnormal{out}}$ and $S_{\delta}^{\textnormal{in}}$ still hold.}:
for all $\varepsilon > 0$ there exist 
$\delta = \delta(\varepsilon, t) > 0$ and two crosscuts 
$S_{\delta}^{\textnormal{out}} \subset \bH \setminus K_t$ and 
$S_{\delta}^{\textnormal{in}} \subset \bH \setminus K_{t-\delta}$ 
with $ \diam(S_{\delta}^{\textnormal{out}}) , \diam(S_{\delta}^{\textnormal{in}}) < \varepsilon$
such that $S_{\delta}^{\textnormal{out}}$ separates $K_{t + \delta} \setminus K_t$ from $\infty$ in 
$\bH \setminus K_t$, and $S_{\delta}^{\textnormal{in}}$ separates $K_{t} \setminus K_{t-\delta}$ from $\infty$ in $\bH \setminus K_{t-\delta}$.
This can be proven analogously to the standard proof for the case of~\eqref{eq: LE} with continuous driving functions, see, e.g.,~\cite[Chapter~4.2.2]{Kemppainen:SLE_book}. 
The shrinking crosscuts
$S_{\delta}^{\textnormal{out}}$ and $S_{\delta}^{\textnormal{in}}$ correspond respectively 
to the right and left limits of the driving function $W$ at time $t$, which might be distinct 
(for $W$ is only assumed to be c\`adl\`ag; see also~\cite{PS:In_prep}):
\begin{align} \label{eq: shrink to points}
\bigcap_{\delta > 0} \overline{g_t( \mathrm{int}_{\mathrm{in}} (S_{\delta}^{\textnormal{out}}) )} 
= \{W(t)\} 
\qquad \textnormal{and} \qquad
\bigcap_{\delta > 0} \overline{g_t( \mathrm{int}_{\mathrm{in}} (S_{\delta}^{\textnormal{in}}) )} 
= \{W(t-)\} .
\end{align}
More precisely, at each fixed time $t$, the local growth gives rise to two null-chains\footnote{Note that if
$\smash{S_{\delta}^{\textnormal{in}} \subset \bH \setminus K_{t-\delta}}$ separates $K_{t} \setminus K_{t-\delta}$ from $\infty$ in $\bH \setminus K_{t-\delta}$, then $\smash{S_{\delta}^{\textnormal{in}} \subset \bH \setminus K_{t}}$ is also a crosscut in $\bH \setminus K_{t}$.} 
$\smash{(S_{\delta_n}^{\textnormal{out}})_{n \in \bZnn}}$ 
and $\smash{(S_{\delta_n}^{\textnormal{in}})_{n \in \bZnn}}$ in $\bH \setminus K_{t}$, with $\delta_n = \delta_n(t) \to 0+$ as $n \to \infty$, which 
represent two unique prime ends in $\bH \setminus K_t$~\cite[Theorem~2.15]{Pommerenke:Boundary_behaviour_of_conformal_maps}:
\begin{align*}
\smash{\invbreve{f}_t}(W(t-)) = \grown_t \in \smash{\invbreve{\bdry}} (\bH \setminus K_t) 
\qquad \textnormal{and} \qquad
\smash{\invbreve{f}_t}(W(t)) = \growing_t  \in \smash{\invbreve{\bdry}} (\bH \setminus K_t) .
\end{align*}
We call $\growing_t$ the \emph{growing end} for the Loewner chain at time $t$, 
and $\grown_t$ the \emph{grown end} at time $t$. 
Also, by the term \emph{growing point} at time $t$ we refer to points in the impression $I(\growing_t)$, and by the term \emph{grown point} at time $t$ we refer to points 
that are swallowed at time $t$ or belong to the impression $I(\grown_t)$. 
Note that the hulls might be generated by a 
self-crossing or self-touching curve $\gamma$, in which case a grown point $z$ might also be a double-point of the curve, i.e., $z = \gamma(s) = \gamma(t) \in K_s \cap K_t$ for some $s < t$.

\begin{rem} \label{rem: shrinking hulls}
In the setup of Lemma~\ref{lem: some markov property}
\textnormal{(}here, we consider the fixed time $\sigma$ equaling $t$ or $t-\delta$ and denote by $\delta$ the small time parameter\textnormal{)},
by the local growth and Wolff's lemma 
\textnormal{(}e.g.,~\cite[Lemma~4.6]{Kemppainen:SLE_book}\textnormal{)}, 
we have $\smash{\diam(\mathring{K}^t_\delta)} \to 0$ and $\smash{\diam(\mathring{K}^{t-\delta}_\delta)} \to 0$ as $\delta \to 0+$. 
Using this and~\cite[Lemma~4.5]{Kemppainen:SLE_book}, we see that
\begin{align} 
\label{eq: shrinking hulls out}
\Big(
\sup_{z \in \bH \setminus \mathring{K}^t_\delta}
|\mathring{g}^t_\delta(z) - z|
\Big) \, \vee \, \Big( \sup_{w \in \bH}
|\mathring{f}^t_\delta(w) - w| \Big)
\lesssim \; & \diam(\mathring{K}^t_\delta) 
\quad \quad \overset{\delta \to 0+}{\longrightarrow} \quad
0 , \\
\label{eq: shrinking hulls in}
\Big(
\sup_{z \in \bH \setminus \mathring{K}^{t-\delta}_\delta}
|\mathring{g}^{t-\delta}_\delta(z) - z|
\Big) \, \vee \, \Big( \sup_{w \in \bH}
|\mathring{f}^{t-\delta}_\delta(w) - w| \Big)
\lesssim \; & \diam(\mathring{K}^{t-\delta}_\delta) 
\quad \overset{\delta \to 0+}{\longrightarrow} \quad
0 .
\end{align}
\end{rem}

\begin{rem}
In fact, one can also show conversely that any locally growing family of hulls parameterized by capacity gives rise to a c\`adl\`ag driving function via~\eqref{eq: shrink to points} such that the associated mapping-out functions solve~\eqref{eq: LE}.
It also follows from the local growth and distortion estimates 
\textnormal{(}see Lemma~\ref{lem: FTC for f} in Appendix~\ref{app: inverse Loewner chain} for an analogous argument\textnormal{)} 
that, for each $z \in \bH$, the mapping-out function $t \mapsto g_t(z)$ is continuous \textnormal{(}up to the blow-up time $\tau(z)$\textnormal{)}. 
Hence, for each $z \in \bH$, since the right-hand side of~\eqref{eq: LE} as a function of $t$ is Lebesgue-integrable on any compact sub-interval of $[0,\tau(z))$, the map $t \mapsto g_t(z)$ is absolutely continuous and
\begin{align*}
g_t(z) \; = \; z + \int_0^t \partial_s^{+} g_s(z) \, \ud s 
\; = \; z + \int_0^t \frac{2 \, \ud s }{g_{s}(z) - W(s-)} , \qquad t \in [0,\tau(z)) .
\end{align*}
A short computation gives
\begin{align} \label{eq: CRg}
\frac{\im(g_t(z))}{|g_t'(z)|}
\; = \; \; & \im(z) \, \exp \bigg( - \int_0^t \frac{4 \, |\im(g_{s}(z))|^2 \ud s}{|g_{s}(z) - W(s-)|^2} \bigg) , \qquad t \in [0,\tau(z)) ,
\end{align}
which implies in particular that $|g_t'(z)| \geq \frac{\im(g_t(z))}{\im(z)}$ for all $t \in [0,\tau(z))$. 
\end{rem}

\subsection{Grid of points of interest for forward flow estimates}
\label{subsec: grid preli}

Let $(g_t)_{t \geq 0}$ be a Loewner chain driven by a c\`adl\`ag function $W$. 
Also, let $f_t:= g_t^{-1}$ be the inverse Loewner chain, and set
\begin{align*} 
\tilde{f}_t(w) := f_t(w + W(t)) , \qquad w \in \bH .
\end{align*}
Note that for each fixed time instant $t$, this map is just $f_t$ pre-composed with a translation.

We next consider a useful approximating grid 
for the (forward) Loewner flow $(t,z) \mapsto g_t(z) - W(t)$ (Lemma~\ref{lem: grid}). 
Such an approach was developed systematically very recently~\cite{MSY:On_Loewner_chains_driven_by_semimartingales_and_complex_Bessel-type_SDEs} \&~\cite{Yuan:Refined_regularity_of_SLE}. 
In particular, from Koebe distortion one can derive estimates for the derivative $|g_t'(z)|$ \emph{uniformly in time}, which will be important for verifying the existence of the Loewner trace. 
Perhaps the most common approach to prove the existence of the trace uses the \emph{backward} Loewner flow (discussed in Section~\ref{sec: backward bounds}), which gives estimates for the derivative of the inverse map $f_t:= g_t^{-1}$ \emph{pointwise in time}. However, for Loewner chains driven by general L\'evy processes, such pointwise in time estimates do not seem sufficient.

\begin{defn} \label{def: grid}
For $a, R, T > 0$, define a grid of mesh size $a/8$ as 
\begin{align*}
\Grid (a, T, R)
:= \Big\{ z \in \bH \;\; | \;\; 
\re(z) = & \; \tfrac{a}{8} \, \ell \, \in \, [-R, R] , \; 
\ell \in \bZ , 
\textnormal{ and }
\\ 
\im(z) = & \; \tfrac{a}{8} \, (k + 8) \, \in \, [a , \sqrt{1+4T}] , \; 
k \in \bZnn 
\Big\} 
\quad \subset \quad [-R, R] \times \ii [a , \sqrt{1+4T} ]  .
\end{align*}
\end{defn} 

We will need to estimate the size of the grid $\Grid (a, T, R)$ (put parameters $q=r=0$ in Lemma~\ref{lem: sum over grid}), 
and to compute the sum over the grid of the initial value $\im(z_0)^{q - 2r} |z_0|^{2r}$ of a certain process 
in Section~\ref{sec: forward bounds}.

\begin{lem}
\label{lem: sum over grid}
Fix $a \in (0,1]$, $R, T > 0$, and $r,q \in \bR$. 
There exists a constant 
$c_{\rm grid}(q, r, T, R) \in (0,\infty)$, 
that depends polynomially on $T$ and $R$, 
such that
\begin{align} 
\nonumber 
\sum_{z_0 \in \Grid (a, T, R) } \im(z_0)^{q-2r} \, |z_0|^{2r} 
\leq \; & c_{\rm grid}(q, r, T, R) \, \chi_{q, r}(a) , \\[1em]
\label{eq: grid chi}
\textnormal{where} \qquad 
\chi_{q, r}(a) =  \; &
\begin{cases}
a^{q} , & r < -1/2 , \; q + 2 < 0 , \\[.5em] 
a^{-2} ( \log (1/a) \vee 1 ) , & r < -1/2 , \; q + 2 = 0 , \\[.5em]
a^{-2} , & r < -1/2 , \; q + 2 > 0 , \\[.5em]
a^{q} (\log (1/a) \vee 1) , & r = -1/2 , \; q + 2 < 0 , \\[.5em] 
a^{-2} (\log (1/a) \vee 1)^2 , & r = -1/2 , \; q + 2 = 0 , \\[.5em]
a^{-2} (\log (1/a) \vee 1) , & r = -1/2 , \; q + 2 > 0 , \\[.5em]
a^{q - 2r - 1} , & r > -1/2 , \; q - 2r + 1 < 0 , \\[.5em] 
a^{-2} (\log (1/a) \vee 1) , & r > -1/2 , \; q - 2r + 1 = 0 , \\[.5em]
a^{-2} , & r > -1/2 , \; q - 2r + 1 > 0 .
\end{cases}
\end{align}
\end{lem}

\begin{proof}
This is a relatively straightforward computation --- 
see~\cite[Lemma~2.6]{Yuan:Refined_regularity_of_SLE}.
\end{proof}

We also obtain useful estimates for the derivative of the Loewner chain (Lemma~\ref{lem: grid}). Note that $|g_t'(z)|$ large implies $|f_t'(z)|$ small. 

\begin{lem} \label{lem: grid}
Fix $T > 0$, $u > 0$, and $\delta \in (0,1)$, and write 
\begin{align*}
R(T) := \sup_{t \in [0,T]} | W(t) | .
\end{align*}
If $\smash{|\tilde{f}_t'(\ii \, \delta)|} \geq u$ 
for some $t \in [0,T]$, then there exists a grid point $z_0 \in \Grid (u \, \delta, T, R(T)) \setminus K_t$ 
such that 
\begin{align*}
| g_t(z_0) - W(t) - \ii \, \delta | \, \leq \, \frac{\delta}{2} 
\qquad \textnormal{and} \qquad 
| g_t'(z_0) | \, \leq \, \frac{80}{27} \, \frac{1}{u} .
\end{align*}
\end{lem}

Note that the width of the grid $\Grid (u \, \delta, T, R(T))$
depends on the Loewner driving function $\smash{(W(t))_{t \in [0,T]}}$. 

\begin{proof}
Fix $t \in [0,T]$ such that $\smash{|\tilde{f}_t'(\ii \, \delta)|} \geq u$. 
Using basic properties of Loewner flows from Appendix~\ref{app: inverse Loewner chain}, we see that, on the one hand  
\begin{align} \label{eq: f derivative Schwarz lemma bound} 
|\tilde{f}_t'(\ii \, \delta)| 
\leq \tfrac{1}{\delta} \, \im(\tilde{f}_t(\ii \, \delta)) 
\leq \tfrac{1}{\delta} \, \sqrt{\delta^2 + 4t} 
\leq \tfrac{1}{\delta} \, \sqrt{\delta^2 + 4T} 
\end{align} 
--- the first inequality follows from Schwarz lemma (cf.~\cite[Chapter~3.4, Theorem~13~\&~Exercise~2]{Ahlfors:Complex_analysis}), the second from Equation~\eqref{eq: imf sqrt bound}, and the third since $t \leq T$. 
On the other hand, we also have
\begin{align*}
|\re(\tilde{f}_t(\ii \, \delta))| 
\, \leq \, \underset{s \in [0,t]}{\sup} | W(s) | 
\, \leq \, R(T) ,
\end{align*}
by Equation~\eqref{eq: ref bound} from Appendix~\ref{app: inverse Loewner chain}. 
Thus, we conclude that
\begin{align*}
\tilde{f}_t(\ii \, \delta) := f_t(\ii \, \delta + W(t)) \; \in \; [-R(T), R(T)] \times [\ii \, u \, \delta , \ii \, \sqrt{1+4T}] .
\end{align*}
In particular, by the choice of the grid, there exists a point 
$z_0 \in \Grid (u \, \delta, T, R(T))$ 
in it such that
\begin{align*}
|z_0 - \tilde{f}_t(\ii \, \delta) | \leq \tfrac{1}{8} \, \delta \, u ,
\end{align*}
which especially implies that
\begin{align*}
z_0 \in B \big( \tilde{f}_t(\ii \, \delta), \tfrac{1}{8} \, \delta \, u \big)
\subset B \big( \tilde{f}_t(\ii \, \delta), \tfrac{1}{8} \, \delta \, | \tilde{f}_t'(\ii \, \delta) | \big) 
\subset B \big( \tilde{f}_t(\ii \, \delta), \tfrac{1}{4} \, \delta \, | \tilde{f}_t'(\ii \, \delta) | \big) \subset \bH \setminus K_t 
\end{align*}
by Koebe $1/4$ theorem (left-hand side of~\eqref{eq: Koebe quarter theorem}). 
Thus, we may conclude using Koebe distortion: 
indeed, using item~\ref{item: Koebe3} of Lemma~\ref{lem: Koebe} with $\confmap^{-1} = g_t$, and $w_0 = W(t) + \ii \, \delta$, and $w = z_0$, we see that
\begin{align*}
| g_t(z_0) - W(t) - \ii \, \delta | \, \leq \, \frac{\delta}{2} ,
\end{align*}
and
\begin{align*}
u \, | g_t'(z_0) | 
\, \leq \, |\tilde{f}_t'(\ii \, \delta)| \, | g_t'(z_0) |
\, = \, \frac{| g_t'(z_0) |}{ | g_t'(\tilde{f}_t(\ii \, \delta)) | } 
\, \leq \, \frac{80}{27} ,
\end{align*}
which is what we sought to prove. 
\end{proof}

\subsection{Sufficient condition for the existence of a Loewner trace}
\label{subsec: trace preli}

Let $(g_t)_{t \geq 0}$ be a Loewner chain driven by a c\`adl\`ag function $W$, and let $f_t:= g_t^{-1}$ be the inverse Loewner chain.
We now consider a basic question: the existence of a trace for the Loewner chain. We allow the trace to have discontinuities but require continuity from both sides.

\begin{defn} \label{def: Loewner trace}
We say that the Loewner chain $(g_t)_{t \geq 0}$ with associated hulls 
$(K_t)_{t \geq 0}$
is \emph{generated} by a function $\eta \colon [0,\infty) \to \overline{\bH}$ if, for each $t \geq 0$, the set $\bH \setminus K_t$ is 
the unbounded connected component of $\bH \setminus \eta[0,t]$. 
In the literature, it is often assumed that $\eta$ is continuous, which we will not assume here.

If $\eta$ is c\`agl\`ad \textnormal{(}left-continuous with right limits\textnormal{)}, then we write $\eta = \caglad$ and say that 
the Loewner chain is \emph{generated by the c\`agl\`ad curve}~$\caglad$. 
In this case, we also use the term 
\emph{``generated by the c\`adl\`ag curve''}~$\cadlag$ \textnormal{(}right-continuous with left limits\textnormal{)}, that is the counterpart of $\caglad$ in the sense that
\begin{align} \label{eq: caglad_cadlag}
\cadlag \colon [0,\infty) \to \overline{\bH} ,
\qquad \cadlag(t) := \lim_{s \to t+} \caglad(s) ,
\qquad \textnormal{and} \qquad 
\caglad(t) = \lim_{s \to t-} \cadlag(s) .
\end{align}
Note that
$\cadlag[0,t] \cup \caglad[0,t] = \overline{\cadlag[0,t]} = \overline{\caglad[0,t]}$. 
We call either $\cadlag$ or $\caglad$ the \emph{Loewner trace}.
\end{defn}

We will give a sufficient condition for the existence of the trace for the Loewner chain, in terms of an estimate for the derivative of the inverse map $f_t$ near the driving point $W(t)$ uniformly in time. It is similar in spirit to the one used in the literature for proving the existence of the $\SLE_\kappa$ trace, in the case where $W = B$ is a standard Brownian motion (in particular, continuous), see~\cite[Theorem~6.4]{Kemppainen:SLE_book} and~\cite[Theorem~4.1]{Rohde-Schramm:Basic_properties_of_SLE}. 
The crucial difference is that the required estimate~\eqref{eq: needed uniform bound for f'} for the derivative of $f_t$ is stronger than what one needs for Brownian motion. Namely, the modulus of continuity for Brownian motion guarantees that for the existence of the $\SLE_\kappa$ trace, it is sufficient to derive the derivative estimate~\eqref{eq: needed uniform bound for f'} at dyadic times. In the present case where $W$ is allowed to be a L\'evy process, however, it appears that $W$ (barely) fails an analogous c\`adl\`ag modulus of continuity property, rendering the usage of~\eqref{eq: needed uniform bound for f'} only at dyadic times insufficient.

\begin{prop} \label{prop: sufficient condition curve}
Fix $T > 0$. 
Suppose that there exists a constant $\theta \in (0, 1)$ such that
\begin{align} \label{eq: needed uniform bound for f'}
|f_t'(W(t) + \ii 2^{-n})| \lesssim 2^{n (1-\theta)}  \qquad 
\textnormal{for all } n \in \bZpos \textnormal{ and } t \in [0,T] .
\end{align}
Then, the following hold for the Loewner chain $(g_t)_{t \geq 0}$ and the inverse chain $(f_t)_{t \geq 0}$ on $[0,T]$. 
\begin{enumerate}[label=\textnormal{(\alph*):}, ref=(\alph*)]
\item \label{item: curve}
The following limit defines a c\`adl\`ag curve $\cadlag \colon [0,T] \to \overline{\bH}$\textnormal{:}
\begin{align} \label{eq: limit curve cadlag}
\cadlag(t) := \lim_{y \to 0+} f_t(W(t) + \ii y) \qquad \textnormal{for all } t \in [0, T] .
\end{align}
\end{enumerate}

\begin{enumerate}[label=\textnormal{(a'):}, ref=(a')] 
\item  \label{item: curve caglad}  
The following limit defines a c\`agl\`ad curve $\caglad \colon [0,T] \to \overline{\bH}$\textnormal{:}
\begin{align} \label{eq: limit curve caglad}
\caglad(t) := \lim_{y \to 0+} f_t(W(t-) + \ii y) \qquad \textnormal{for all } t \in [0, T] .
\end{align}
The curves $\cadlag$ and $\caglad$ are each others' c\`adl\`ag-c\`agl\`ad counterparts as in~\eqref{eq: caglad_cadlag}. 
\end{enumerate}

\begin{enumerate}[label=\textnormal{(\alph*):}, ref=(\alph*)]
 \setcounter{enumi}{1} 
\item \label{item: generated caglad}
Either curve $\cadlag$ or $\caglad$ generates the Loewner chain on $[0,T]$.

\medskip

\item \label{item: locally connected}
For each $t \in [0,T]$, the set $\bdry (\bH \setminus K_t)$ is locally connected.
\end{enumerate}
\end{prop}

We will need uniqueness of radial limits in the following sense (cf.~\cite[Proposition~2.14]{Pommerenke:Boundary_behaviour_of_conformal_maps} and~\cite[Lemma~6.5]{Kemppainen:SLE_book}).
Even though the idea is standard, for completeness we present a proof for the needed result in the case of possibly discontinuous driving functions, relying on further results outlined in Appendix~\ref{app: inverse Loewner chain}. 

\begin{lem} \label{lem: radial limit exists}
Fix $z_0 \in \bH$ and  
$0 < r_0 < \mathrm{dist}(z_0 , \bR)$, 
and write $B := B(z_0,r_0) \subset \bH$.
Suppose that $\overline{B}$ is hit but not swallowed by $K_t$, that is, 
$K_t \cap \overline{B} \neq \emptyset$, and 
$\overline{B} \setminus K_t \neq \emptyset$, and
$K_s \cap \overline{B} = \emptyset$ for all $s < t$.  
Then, $\bdry_{\mathrm{in}} K_t \cap \overline{B} = \{ \eta(t) \}$
is given by the radial limit\footnote{Recall that $\bdry_{\mathrm{in}} K_t := \bdry K_t \cap \bH$.}
\begin{align} \label{eq: limit map radial caglad}
\eta(t) := \lim_{y \to 0+} f_t(W(t-) + \ii y) \; \in \; \bdry B .
\end{align}
\end{lem}

Note that the limit~\eqref{eq: limit map radial caglad} is the unique point 
in
$\smash{\underset{s<t}{\bigcap} \overline{K_t \setminus K_s}}$ 
accessible from $\bH \setminus K_t$,
a grown point at time $t$ for the Loewner chain. 

\begin{proof}
Note that $\bdry_{\mathrm{in}} K_t \cap \overline{B} \neq \emptyset$ by the assumptions of the lemma.
By Lemma~\ref{lem: f is continuous} (from Appendix~\ref{app: inverse Loewner chain}), 
if $s \to t-$, 
then $f_s \to f_t$ uniformly on all compact subsets of $\bH$. Hence, 
the Carath\'eodory kernel convergence theorem 
(see, e.g.,~\cite[Theorem~1.8]{Pommerenke:Boundary_behaviour_of_conformal_maps})
shows that,
for each $z \in \bdry_{\mathrm{in}} K_t \cap \overline{B}$, we can find a sequence $z_n \in \bdry_{\mathrm{in}} K_{t_n}$ such that $z_n \to z$ as $n \to \infty$, where $t_n \to t-$. Since $z_n \notin \overline{B}$ for all $n$, we see that $z \in \bdry B$, which shows that $\bdry_{\mathrm{in}} K_t \cap \overline{B} = \bdry_{\mathrm{in}} K_t \cap \bdry B$. 
Now, each point in this set is a grown point accessible from $\bH \setminus K_t$ by an arc from $z_0$.
We conclude from~\eqref{eq: shrink to points} that the radial limit of $g_t$ exists: 
\begin{align} \label{eq: radial limit proof of lemma}
\lim_{s \to 1-} g_t(J(s)) = W(t-) .
\end{align}
By~\cite[Proposition~2.14]{Pommerenke:Boundary_behaviour_of_conformal_maps}, if $J_1$ and $J_2$ are two Jordan arcs accessing two distinct boundary points in $\bH \setminus K_t$,
then the curves $g_t (J_1)$ and $g_t (J_2)$ also have distinct endpoints in $\bdry \bH$, which shows by~\eqref{eq: radial limit proof of lemma} that the set $\bdry_{\mathrm{in}} K_t \cap \overline{B} = \{ w \}$ must be a singleton.
Finally, we see from~(\ref{eq: radial limit},~\ref{eq: radial limit proof of lemma}) that also the radial limit of $f_t$ at $W(t-)$ exists and equals $w = \eta(t)$.
This proves the lemma. 
\end{proof}

\begin{rem}
Note that Lemma~\ref{lem: radial limit exists} does not imply that $\eta$ is a c\`agl\`ad curve. 
This needs a separate argument addressing the behavior of~\eqref{eq: limit map radial caglad} when the time $t$ is varied. 
For instance, an estimate of type~\eqref{eq: needed uniform bound for f'} assumed in Proposition~\ref{prop: sufficient condition curve} 
gives \emph{uniform} convergence that guarantees that $\eta$ is c\`agl\`ad.
\end{rem}

\begin{proof}[Proof of Proposition~\ref{prop: sufficient condition curve}.]
By Lemma~\ref{lem: f is continuous} (from Appendix~\ref{app: inverse Loewner chain}), the map $t \mapsto f_t(W(t) + \ii y)$
is c\`adl\`ag for each fixed $y > 0$. Hence, for item~\ref{item: curve} it suffices to show that 
the assumed bound~\eqref{eq: needed uniform bound for f'} implies that 
the limit~\eqref{eq: limit curve cadlag}  
exists and is approached uniformly on $[0, T] \ni t$.

By Koebe distortion (Lemma~\ref{lem: Koebe}) and the assumed bound~\eqref{eq: needed uniform bound for f'}, we have
\begin{align*}
\sup_{t \in [0, T]}
|f_t'(W(t) + \ii y)| \lesssim y^{\theta - 1} \qquad \textnormal{for all } y \in (0, 1) ,
\end{align*}
To get the limit~\eqref{eq: limit curve cadlag}, we fix $0 < y_1 < y_2 \leq y \leq 1$ and integrate this bound:
\begin{align*}
|f_t(W(t) + \ii y_2) - f_t(W(t) + \ii y_1)| 
\;\leq \;\; & \int_{y_1}^{y_2} |f_t'(W(t) + \ii u)| \, \ud u
\; \lesssim \; \int_{0}^{y} u^{\theta - 1} \, \ud u 
\; = \; \frac{y^{\theta}}{\theta}
\quad \overset{y \to 0+}{\longrightarrow} \quad
0 ,
\end{align*}
which implies that the limit~\eqref{eq: limit curve cadlag}  
exists and is approached uniformly on $[0, T] \ni t$, thus proving~\ref{item: curve}.

Item~\ref{item: curve caglad} follows from item~\ref{item: curve}  by noticing that $\caglad$ is the c\`agl\`ad counterpart of $\cadlag$ as in~\eqref{eq: caglad_cadlag}:
\begin{align*}
\lim_{s \to t-} \cadlag(s) 
= \lim_{s \to t-} \lim_{y \to 0+} f_s(W(s) + \ii y)
= \lim_{y \to 0+} f_t(W(t-) + \ii y) 
=: \caglad(t) \qquad \textnormal{for all } t \in [0, T] ,
\end{align*}
by Lemma~\ref{lem: f is continuous} and the fact that 
$y \to 0+$ is approached uniformly on $[0, T] \ni t$ by the proof of item~\ref{item: curve}.

To prove item~\ref{item: generated caglad}, 
we will show that the set $\bH \setminus K_t$ is 
the unbounded component of $\smash{\bH \setminus \caglad[0,t]}$.
Since $\smash{\caglad[0,t] \subset K_t \cup \bR}$, by topological considerations it suffices to verify that $\smash{\bdry_{\mathrm{in}} K_t := \bdry K_t \cap \bH \subset \caglad[0,t]}$. 
To this end, we consider a point $z \in \bdry_{\mathrm{in}} K_t$ and show that $\smash{z \in \caglad[0,t]}$. 
We let $0 < r < \mathrm{dist}(z , \bR)$, write $B_r := B(z,r) \subset \bH$ as in Lemma~\ref{lem: radial limit exists}, and set
$\sigma_r := \inf \{ s \geq 0 \; | \; K_s \cap \overline{B}_r \neq \emptyset \} \leq t$.
Since $z \in \bdry_{\mathrm{in}} K_t \subset \bH$, 
all assumptions of Lemma~\ref{lem: radial limit exists} are satisfied at time $\sigma_r$: indeed, by compactness and local growth, we have 
$K_{\sigma_r} \cap \overline{B}_r \neq \emptyset$, and we clearly have 
$\overline{B}_r \setminus K_{\sigma_r} \supset \overline{B}_r \setminus K_t \neq \emptyset$, and
$K_s \cap \overline{B}_r = \emptyset$ for all $s < {\sigma_r}$.
Hence, Lemma~\ref{lem: radial limit exists} 
and the left-continuity of $\caglad$ from item~\ref{item: curve} together imply 
that $z = \smash{\underset{r \to 0}{\lim} \; \caglad(\sigma_r) \, \in \, \caglad[0,t]}$.
Thus, we conclude that $\smash{\bdry_{\mathrm{in}} K_t \subset \caglad[0,t]}$, which implies item~\ref{item: generated caglad}.

Item~\ref{item: locally connected} follows from item~\ref{item: generated caglad} combined with Proposition~\ref{prop: locally connected}, which we prove in the next section. 
\end{proof}

\subsection{Topological properties of c\`adl\`ag Loewner hulls}
\label{subsec: connectedness}

In this section, we analyze the topology of growing Loewner hulls.
In particular, the local connectedness of c\`adl\`ag hulls (Proposition~\ref{prop: locally connected})
will be needed for the proof of the existence of the Loewner trace for drivers which include macroscopic jumps in Section~\ref{sec: main results} (Proposition~\ref{prop: generated by a curve still containing epsilon}). 
This fact, albeit perhaps intuitive, is quite subtle --- as boundary behavior of planar fractals in general\footnote{Local connectedness of fractals is a very subtle topic --- e.g., it is believed that the Mandelbrot set is locally connected, 
but despite of several breakthrough results, there is no proof as of today (this problem is known as the MLC conjecture).}.
The key is that the c\`adl\`ag hulls are generated by a function that possesses both left and right limits.
To us, Proposition~\ref{prop: locally connected} appears to be new, and therefore we shall include a detailed discussion here. It also implies path-connectedness of the hulls (Corollary~\ref{cor: locally path-connected}).

As a warning example, consider the comb space 
\textnormal{(}see also Figure~\ref{fig: comb}\textnormal{)} 
\begin{align} \label{eq: comb space}
K := \{ \ii y \; | \; 0 \leq y \leq 1\} \cup \{ 2^{-n} + \ii y \; | \;  0 \leq y \leq 1, \, n \in \bZnn \} \; \subset \; \overline{\bH} .
\end{align}
The union $K \cup \bR$ of this comb with the real line is a path-connected set which is not locally connected.
It can be constructed in various ways, e.g., by using a continuous driving function, or using a c\`adl\`ag driving function. It coincides with  
the graph of the function $\eta \colon [0,3) \to \overline{\bH}$ 
\textnormal{(}a curve having no left limit at time $t=3$, which is however c\`adl\`ag elsewhere\textnormal{)}, 
\begin{align*}
\eta(t) :=
\begin{cases}
\ii t , & 0 \leq t < 1 , \\ 
1 + \ii (t-1) , & 1 \leq t < 2 , \\ 
2^{-n} + \ii \big( 2 + 2^n (t-3) \big) , 
& 3 - 2^{1-n} \leq t < 3 - 2^{-n} , \, n \in \bZpos . 
\end{cases}
\end{align*}

\begin{wrapfigure}{R}{5cm}
\centering
\includegraphics[scale=0.4]{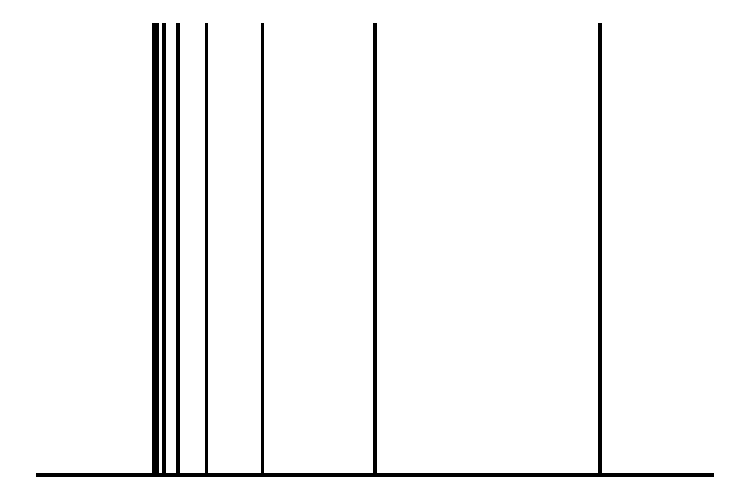}
\caption{\label{fig: comb}
Illustration of a path connected but not locally connected comb.
}
\end{wrapfigure}

The comb space~\eqref{eq: comb space} shows that, first of all,
for Loewner chains with c\`adl\`ag (or even continuous) driving functions, local connectedness may fail, 
and second of all, 
not all Loewner chains with c\`adl\`ag driving functions are generated by c\`adl\`ag curves (by Proposition~\ref{prop: locally connected}). 
Indeed, it is not hard to check that that the hulls

\begin{minipage}{5cm}
\begin{align*} 
K_t :=
\begin{cases}
\overline{\eta[0,t]} , & 0 \leq t < 3 , \\ 
K , & t = 3 , 
\end{cases}
\end{align*}
\end{minipage}

are locally growing on $[0,3] \ni t$, which shows that their driving function $W$ is actually c\`adl\`ag --- in particular it
has a unique left limit as $t \to 3-$, 
the image of the point $\ii$ under the conformal map $g_3 \colon \bH \setminus K \to \bH$.
\textnormal{(}This is in contrast to the fact that the curve $\eta$ itself has no left limit as $t \to 3-$.\textnormal{)} 
However, the boundary $\bdry (\bH \setminus K_3)$ of the complement of the hull $K_3 = K$ is not locally connected.
The impression of the prime end of $\bH \setminus K_3$ containing $\ii$ is the right side of the segment $\{ \ii y \; | \; 0 \leq y \leq 1\}$. 
This prime end may be both grown and growing at time $t=3$, while the only accessible point from $\bH \setminus K_3$ in its impression is $\ii$.

\medskip

Marshall \&~Rohde~\cite{Marshall-Rohde:The_loewner_differential_equation_and_slit_mappings} constructed a logarithmic spiral 
spinning around the unit circle as an example of a Loewner chain which has a (H\"older-$1/2$) continuous driving function, but which is not generated by a continuous (or even c\`adl\`ag) curve and
for which local connectedness fails.
See also~\cite[Figure~5.9]{Beliaev:Conformal_maps_and_geometry} for another example, and \cite{KNK:Exact_solutions_for_Loewner_evolutions, Lind:Sharp_condition_for_Loewner_equation_to_generate_slits}  for related results.

\subsubsection{Local connectedness}

By the Hahn-Mazurkiewicz theorem (cf.~\cite[Theorem~2, page~256]{Kuratowski:Topology}), 
if a Loewner chain is generated by a continuous  curve $\gamma$, then the boundary of the corresponding domain $\bH \setminus K_t$ is locally connected for each time $t$. 
Note also that local connectedness of the image of the generating curve is a priori a different property than local connectedness of the boundaries of the associated domains $\bH \setminus K_t$: 
\cite[Figure~5.10]{Beliaev:Conformal_maps_and_geometry} gives an example of a right-continuous curve generating a Loewner chain, 
whose driving function is continuous, and
for which the curve itself is not locally connected but the boundaries of the associated domains $\bH \setminus K_t$ are still locally connected.

Motivated by the Hahn-Mazurkiewicz theorem, 
one could ask whether the property that the Loewner chain is generated by a \emph{c\`adl\`ag} (or c\`agl\`ad) curve would also imply local connectedness. We will answer this questions affirmatively in this section. 
This rather natural result appears to us to be new. 

\LocConnProp* 

This result was proven for the special case of a Loewner chain driven by a symmetric stable pure jump process in~\cite[Proposition~7.2]{Chen-Rohde:SLE_driven_by_symmetric_stable_processes}. 
In that case, one could use the property that the hulls have empty interiors~\cite[Theorem~1.3(i)]{Guan-Winkel:SLE_and_aSLE_driven_by_Levy_processes} in combination with a result from complex analysis due to Warschawski from the 1950s (concerning the modulus of continuity of conformal maps of the disc).
The empty interior property ensures that one can \emph{bootstrap} the local connectedness at a fixed time $T$ (obtained from the backward Loewner chain) to local connectedness at all times $t \in [0,T]$, since the hulls have \emph{empty interior}.
In the general case, however, the hulls can be much more complicated and this proof does not apply. We provide a direct proof below, by considering crossings of annuli 
\begin{align*}
\overline{\mathbb{A}}(z_0,r_0,R_0) = \{ w \in \bC \; | \; r_0 \leq |w - z_0| \leq R_0 \} . 
\end{align*}
The key to guarantee local connectedness is that the hulls are generated by a \emph{c\`adl\`ag curve}, that in particular possesses both left and right limits, which implies that annulus crossings are controlled. 

\begin{proof}[Proof of Proposition~\ref{prop: locally connected}]
We prove the claim by contradiction. 
Suppose that for some $t \geq 0$, the set $A_t := \bdry (\bH \setminus K_t)$ is not locally connected. 
Then, there exists a point $z \in A_t$, radius $r > 0$, and points $z_n \to z$ as $n \to \infty$ such that all of $z_n$ and $z$ lie in different connected components of $U(z,r) := A_t \cap B(z,r)$. 
We may furthermore assume that all of the points $z_n$ are inside $B(z,\frac{r}{10})$. 
Since $A_t$ is a connected subset of $\overline{\cadlag[0, t)} \cup \bR = \caglad[0, t] \cup \bR$, 
the points $z_n \in A_t \cap B(z,\frac{r}{10})$ must all be connected together in $A_t$ outside of $U(z,r)$. 
In particular, the set $A_t$ makes infinitely many distinct crossings across the annulus $\smash{\overline{\mathbb{A}}(z,\tfrac{r}{2},r)}$,
that is, the set $A_t \cap \smash{\overline{\mathbb{A}}(z,\tfrac{r}{2},r)}$ has infinitely many distinct connected components (note that we do not yet know that $A_t$ is path-connected).  
We prove that this is impossible by the right-continuity of $\cadlag$ and the left-continuity of $\caglad$.

To this end, fix a point $z_0 \in \overline{\bH}$ and two radii $0 < r_0 < R_0$. Denote
$\smash{\overline{\mathbb{A}}_0 = \overline{\mathbb{A}}(z_0,r_0,R_0)}$.
Consider the time 
\begin{align*}
T_0 = T(z_0,r_0,R_0) 
:= \; & \inf \{ s \geq 0 \; | \; \textnormal{$A_t$ makes infinitely many distinct crossings across $\smash{\overline{\mathbb{A}}_0}$} \} \\
= \; & \inf \{ s \geq 0 \; | \; \textnormal{$A_t \cap \smash{\overline{\mathbb{A}}_0}$ has infinitely many connected components} \\
& \qquad\qquad\quad
\textnormal{touching both $\bdry B(z_0,r_0)$ and $\bdry B(z_0,R_0)$} \} \\
= \; & \inf \{ s \geq 0 \; | \; \textnormal{$A_t \cap \smash{\overline{\mathbb{A}}_0}$ has infinitely many connected components $(S_j)_{j \in J}$ such that} \\
& \qquad\qquad\quad
\textnormal{for all $j \in J$, we have $S_j \cap \bdry B(z_0,r_0) \neq \emptyset$ and $S_j \cap \bdry B(z_0,R_0) \neq \emptyset$} \} .
\end{align*}
Suppose that $T_0 < \infty$. 
Then, the following mutually contradictory properties~\ref{item: contradictory 1}~\&~\ref{item: contradictory 2} hold.
\begin{enumerate}[label=\textnormal{(\alph*):}, ref=(\alph*)]
\item \label{item: contradictory 1}
First, the set $A_{T_0}$ cannot make infinitely many distinct crossings across $\smash{\overline{\mathbb{A}}_0}$.
Indeed, if this would be the case, then for any strictly smaller time $s < T_0$, the set 
\begin{align*}
A_{T_0} \setminus A_s 
\; = \; \bdry (\bH \setminus K_{T_0}) \setminus (K_s \cup \bR) 
\; \subset \; \caglad[0, T_0] \setminus K_s 
\; \subset \; \caglad(s, T_0]
\end{align*}
would make infinitely many crossings across the annulus $\smash{\overline{\mathbb{A}}_0}$.
This violates the left-continuity of $\caglad$, since there exists $\delta = \delta(r_0, R_0) > 0$ such that, taking $s = T_0 - \delta$, we arrive at a contradiction:
\begin{align*}
\sup_{ u,v \in ( T_0 - \delta , T_0 ] } | \caglad(u) - \caglad(v) | \; < \; \tfrac{1}{2}(R_0 - r_0) .
\end{align*}

\item \label{item: contradictory 2}
Second, consider a sequence $t_n \to T_0+$ as $n \to \infty$ such that each set $A_{t_n}$ makes infinitely many distinct crossings across $\smash{\overline{\mathbb{A}}_0}$.
If the set $A_{T_0}$ only makes finitely many distinct crossings across $\smash{\overline{\mathbb{A}}_0}$, then the set 
\begin{align*}
A_{t_n} \setminus A_{T_0} 
\; = \; \bdry (\bH \setminus K_{t_n}) \setminus (K_{T_0} \cup \bR) 
\; \subset \; \overline{\cadlag[0, t_n)} \setminus K_{T_0} 
\; \subset \; \overline{\cadlag(T_0, t_n]}
\end{align*}
makes infinitely many crossings across the annulus $\smash{\overline{\mathbb{A}}_0}$.
However, this violates the right-continuity of $\cadlag$, since there exists $\delta = \delta(r_0, R_0) > 0$ such that, for $t_n \leq T_0 + \delta$, we arrive at a contradiction:
\begin{align*}
\sup_{ u,v \in [ T_0, T_0  + \delta ) } | \cadlag(u) - \cadlag(v) | \; < \; \tfrac{1}{2}(R_0 - r_0) .
\end{align*}
Hence, the set $A_{T_0}$ must make infinitely many distinct crossings across $\smash{\overline{\mathbb{A}}_0}$.
\end{enumerate}
In summary, as $z_0 \in \overline{\bH}$ and $0 < r_0 < R_0$ were arbitrary, 
we have $T(z_0,r_0,R_0) = \infty$, which shows in particular that 
$A_t$ makes only finitely many distinct crossings across $\smash{\overline{\mathbb{A}}(z,\tfrac{r}{2},r)}$ --- which contradicts our earlier observation (first paragraph). 
Hence, the set $A_t$ is locally connected for all $t \geq 0$, as claimed.  
\end{proof}

\subsubsection{Path-connectedness}

A set $A \subset \bC$ is said to be (uniformly) \emph{locally path-connected} 
if for every $\varepsilon > 0$ there exists $\delta > 0$ such that, for any pair of points $z,w \in A$ such that $|z-w| < \delta$, there exists a continuous path $\gamma$ connecting $z$ and $w$ in $A \cap B(z,\varepsilon) \cap B(w,\varepsilon)$.

\begin{cor} \label{cor: locally path-connected}
Let $(g_t)_{t \geq 0}$ be a Loewner chain with associated hulls $(K_t)_{t \geq 0}$ generated by a c\`adl\`ag curve $\cadlag$ \textnormal{(}viz.~a c\`agl\`ad curve $\caglad$\textnormal{)}. 
Then, for each $t \geq 0$, the set $K_t \cup \bR$ is uniformly locally path-connected.
In particular, $K_t \cup \bR$ is path-connected.
\end{cor}

\begin{proof}
Note that any connected and locally path-connected set is path-connected, so it suffices to show the former properties. 
The set $K_t \cup [-R,R]$ is compact and connected for any $R > 0$ large enough such that $K_t \subset B(0,R)$. 
Proposition~\ref{prop: locally connected} implies that this set is also locally connected.
Therefore,~\cite[Theorem~6.7.2]{Sagan:Space_filling_curves} implies that $K_t \cup [-R,R]$ is uniformly locally path-connected\footnote{According to~\cite{Sagan:Space_filling_curves}, this was proven by Hahn;
see also the Mazurkiewicz-Moore-Menger theorem~\cite[page~254]{Kuratowski:Topology}.}. 
Taking $R \to \infty$ clearly does not chance this property, so we see that $K_t \cup \bR$ is uniformly locally path-connected as well.
\end{proof}

\bigskip{}
\section{\label{sec: backward bounds}Estimates for mirror backward Loewner flow with martingale L\'evy drivers}
This section is concerned with estimates needed to prove Theorem~\ref{thm: LLE Holder if kappa not 4}: local H\"older continuity of the inverse Loewner maps driven by a L\'evy process with diffusion parameter $\kappa \in [0,\infty) \setminus\{4\}$.
The main results of this section are Propositions~\ref{prop: BLE micro jumps Holder}~\&~\ref{prop: BLE micro jumps Holder with drift kappa = 0}, which give the local H\"older continuity of solutions to backward Loewner flows when the jumps of the L\'evy process are small enough. 
To this end, it suffices to derive derivative estimates to fulfill the property~\eqref{eq: needed bound for derivative of confmap} in item~\ref{item: global derivative estimate} of Lemma~\ref{lem: global Holder continuity}.
Our estimates are quite rough, though, and do not seem to give optimal H\"older exponents (cf.~Remark~\ref{rem: inverse map param holder optimal}).

\subsection{Mirror backward Loewner flow}

One of our main tools to investigate the geometry of the Loewner hulls is 
the (mirror\footnote{The phrasing ``mirror'' just refers to the choice of the sign of the driving process, which is different from the ``usual'' backward Loewner equation, and much more convenient for computations.}) 
backward Loewner equation driven by a c\`adl\`ag function $W$: 
\begin{align} \label{eq: mBLE}
\tag{mBLE}
\partial_t^{+} h_t(z) = \frac{-2}{h_t(z) + W(t)} \qquad \textnormal{with initial condition} \qquad h_0(z) = z  ,
\end{align}
which, similarly to~\eqref{eq: LE}, has a unique absolutely continuous solution $t \mapsto h_t(z)$ for each $z \in \bH$.
In addition, this solution exists for all times $t \geq 0$, since the imaginary part of the solution to~\eqref{eq: mBLE} is increasing in $t$ (see~\eqref{eq: imh}). 
One can also show (see Lemma~\ref{lem: FTC for f} in Appendix~\ref{app: inverse Loewner chain} for an analogous argument) that, for each $z \in \bH$, the function $t \mapsto h_t(z)$ is continuous. 
Hence, for each $z \in \bH$, since the right-hand side of~\eqref{eq: mBLE} as a function of $t$ is Lebesgue-integrable on any compact sub-interval of $[0,\infty)$, the map $t \mapsto h_t(z)$ is absolutely continuous, and
\begin{align*}
h_t(z) \; = \; z + \int_0^t \partial_s^{+} h_s(z) \, \ud s 
\; = \; z - \int_0^t \frac{2 \, \ud s }{h_{s}(z) + W(s-)} , \qquad t \geq 0 .
\end{align*}
The differential equation~\eqref{eq: mBLE} also implies that
\begin{align*}
h_t'(z) = \exp \bigg( \int_0^t \frac{2 \, \ud s}{(h_{s}(z) + W(s-))^2} \bigg) , \qquad t \geq 0 \; \textnormal{ and } \; z \in \bH .
\end{align*}
The imaginary part of~\eqref{eq: mBLE} gives
\begin{align} \label{eq: imh}
\im(h_t(z))
\; = \; \; & \im(z) \, \exp \bigg( \int_0^t \frac{2 \, \ud s}{|h_{s}(z) + W(s-)|^2} \bigg) ,
\end{align}
which implies in particular that
$\im(h_t(z)) \geq \im(z)$ for all $t \geq 0$ and $z \in \bH$.
Hence, we see that 
\begin{align} 
\nonumber
|h_t'(z)| 
\; \leq \; & \exp \bigg( \int_0^t \frac{2 \, \ud s}{|h_{s}(z) + W(s-)|^2} \bigg) 
\; \leq \; \exp \bigg( \int_0^t \frac{2 \, \ud s }{ | \im(h_{s}(z)) |^2} \bigg) \\
\label{eq: derh bound}
\; \leq \; & \exp \bigg( \frac{2 \, t }{| \im(z) |^2} \bigg)  , \qquad t \geq 0 \; \textnormal{ and } \; z \in \bH .
\end{align}

The following domain Markov property will be needed in Section~\ref{sec: main results} (proof of Proposition~\ref{prop: BLE micro jumps Holder still containing epsilon}).

\begin{lem} \label{lem: some markov property for h}
Let $(h_t)_{t \geq 0}$ be the solution to~\eqref{eq: mBLE} with c\`adl\`ag driving function $W$. 
Fix $\sigma \geq 0$. Then, 
$\smash{\mathring{h}^\sigma_{t}}(z) 
:= ( h_{\sigma + t} \circ h_{\sigma}^{-1} )(z - W(\sigma)) + W(\sigma)$ solve~\eqref{eq: mBLE} with driving function $\smash{\mathring{W}^\sigma}(t) := W(\sigma+t) - W(\sigma)$.
\end{lem}

\begin{proof}
This is a straightforward computation, analogous to Lemma~\ref{lem: some markov property}. 
\end{proof}

The main reason why solutions of~\eqref{eq: mBLE} are useful is the observation that derivative estimates for the
inverse Loewner chain $f_t := g_t^{-1}$ can be obtained from corresponding estimates for $h_t$ (see Lemma~\ref{lem: f' = h' in distribution}).
The latter, in turn, can be derived by using suitable local martingales bounded by supermartingales (see Proposition~\ref{prop: M supermgle} in Section~\ref{subsec: derivative local martingale}).
An important caveat here is that the estimates for $f_t$ thus obtained only hold \emph{pointwise} in time, due to this restriction for the equality in distribution in the following Lemma~\ref{lem: f' = h' in distribution}. 
To simplify the statement, we will write
\begin{align*} 
\tilde{f}_t(z) := f_t(z + W(t)) , \qquad z \in \bH .
\end{align*}
Note that for each fixed $t$, this map is just $f_t$ pre-composed with a translation.

\begin{lem} \label{lem: f' = h' in distribution}
Let $W$ be a L\'evy process. 
Fix $t \geq 0$. Then, the following hold.
\begin{enumerate}[label=\textnormal{(\alph*):}, ref=(\alph*)]
\item \label{item: f = h in distribution}
The map 
$z \mapsto \smash{\tilde{f}_t(z)} - W(t)$ 
has the same distribution as $z \mapsto h_t(z)$. 

\medskip

\item \label{item: f' = h' in distribution}
The map 
$z \mapsto \smash{\tilde{f}_t'(z)}$ 
has the same distribution as $z \mapsto h_t'(z)$.
\end{enumerate}
\end{lem}

Lemma~\ref{lem: f' = h' in distribution} is a generalization of~\cite[Lemma~3.1]{Rohde-Schramm:Basic_properties_of_SLE}, where the result was shown when $W$ is a scalar multiple of Brownian motion. 
The proof only relies on the property that the driving function $W$ has stationary and identically distributed increments, which holds for any L\'evy process.

\begin{proof}
First, we show that the map $z \mapsto f_t(z)$ is equal in distribution with $z \mapsto k_t(z)$, 
the (absolutely continuous) solution to the backward Loewner equation
(see, e.g.,~\cite[Lemma~4.10]{Kemppainen:SLE_book})
\begin{align*}
\partial_s^{+} k_s(z) = \frac{-2}{k_s(z) - W(t-s)} \qquad \textnormal{with initial condition} \qquad k_0(z) = z  .
\end{align*}
Indeed, the function $s \mapsto k_{t-s}(z)$ solves~\eqref{eq: LE} on $[0, t]$ with initial condition $k_t(z)$. Hence, we have $z = k_0(z) = g_{t}(k_{t}(z))$, 
and after applying $f_t$ to both sides of this equation, we find that $f_t(z) = k_t(z)$.

By the Markov property of the L\'evy process, $W(t) - W(t - s)$ equals $W(s)$ in distribution for each $s \leq t$. Therefore, a short computation shows that, for each $z \in \bH$ and $s \in [0,t]$, the map 
\begin{align*}
s \mapsto k_s(z + W(t)) - W(t)
\end{align*}
solves~\eqref{eq: mBLE}, 
so it equals $s \mapsto h_s(z)$ in distribution. 
This implies both items~\ref{item: f = h in distribution} and~\ref{item: f' = h' in distribution}.
\end{proof}

\subsection{Derivative supermartingale}
\label{subsec: derivative local martingale}

For the estimates needed to analyze the growing hulls,
the key players in the driving function~\eqref{eq: general DF} will be the diffusion part (i.e., Brownian motion) 
and the compensated sum of microscopic jumps.
Hence, we shall work with a cutoff $\varepsilon > 0$ throughout this section. 
We fix a L\'evy measure $\nu$ 
and consider martingale driving functions of the form
\begin{align} \label{eq: BM with micro}
\microDriver_\varepsilon^{\kappa}(t) := \sqrt{\kappa} B(t) + \int_{|\jump| \leq \varepsilon} \jump \, \PoissonComp(t, \ud \jump) , \qquad
\kappa \geq 0 , \, \varepsilon > 0 ,
\end{align}
where $B$ is a standard Brownian motion and 
$\smash{\PoissonComp}(t, \ud \jump) := \Poisson(t, \ud \jump) - t \nu(\ud \jump)$ is the compensated Poisson point process
of a Poisson point process $\Poisson$ independent of $B$ with L\'evy intensity measure $\nu$.
We denote the variance of the jumps of the random driving function~\eqref{eq: BM with micro} by
\begin{align} \label{eq: jump variance}
\lambda_\varepsilon  = \lambda_\varepsilon(\nu) 
:= \int_{|\jump| \leq \varepsilon} \jump^2 \, \nu(\ud \jump) \; \geq 0 .
\end{align}

\begin{rem} \label{rem: tune epsilon}
Using continuity of measures,  
we see\footnote{Define a Borel measure $\mu$ on $[-1, 1]$ by $\mu(A) := \int_A \jump^2 \nu(d\jump)$ for all Borel sets $A \subset [-1, 1]$. Since $\nu$ is a L\'evy measure, $\mu$ is a finite measure. 
Thus, by~\cite[Theorem 3.2]{Bauer-Measure_and_integration_theory} we have $\underset{\varepsilon \to 0+}{\lim} \mu([-\varepsilon, \varepsilon]) = \mu(\{0\}) = 0$.} that,
for each L\'evy measure $\nu$ and for each constant $\lambda > 0$, we can find a cutoff $\varepsilon_\lambda = \varepsilon_\lambda(\nu) > 0$ such that the variance~\eqref{eq: jump variance} satisfies
\begin{align} \label{eq: jump variance bound}
\varepsilon < \varepsilon_\lambda
\qquad \Longrightarrow \qquad
\lambda_\varepsilon 
< \lambda .
\end{align}
We will use this fact repeatedly.
For example, if $\alpha \in (0,2)$ 
and $\nu(\ud \jump) = |\jump|^{-1-\alpha} \, \ud \jump$ 
is the jump measure of a symmetric $\alpha$-stable L\'evy process, 
then we have
\begin{align*}
\lambda_\varepsilon 
= \frac{2 \, \varepsilon^{2-\alpha}}{2-\alpha} 
< \lambda
\qquad \Longleftrightarrow \qquad
\varepsilon < \varepsilon_\lambda
= \Big(\frac{2 - \alpha}{2} \Big)^{\frac{1}{2-\alpha}}  \lambda^{\frac{1}{2-\alpha}} .
\end{align*}
\end{rem}

Let $(h_t)_{t \geq 0}$ be the solution to~\eqref{eq: mBLE} driven by $W = \smash{\microDriver_\varepsilon^{\kappa}}$. 
Fix a starting point $z_0 = x_0 + \ii y_0 \in \bH$ implicitly throughout, and define the processes 
\begin{align}
\nonumber
Z(t) = Z_\varepsilon^{\kappa}(t,z_0) := \; & h_t(z_0) + \microDriver_\varepsilon^{\kappa}(t) =: X(t) + \ii Y(t) , \\
\nonumber
X(t) = X_\varepsilon^{\kappa}(t,z_0) := \; & \re \big( Z_\varepsilon^{\kappa}(t,z_0) \big) , \\
\nonumber
Y(t) = Y_\varepsilon^{\kappa}(t,z_0) := \; & \im \big( Z_\varepsilon^{\kappa}(t,z_0) \big) , \\
\label{eq: general martingale candidate}
M(t) = M_{p,q,r}(t,z_0) := \; & |h_t'(z_0)|^p \,Y(t)^{q} \, ( \sin \arg Z(t) )^{-2r} , \qquad 
p, q \in \bR,\textnormal{ and } r > 0 .
\end{align}
Note that $M(0) = y_0^{q - 2r} |z_0|^{2r}$. 
For general L\'evy driving processes, $M$ is not a martingale, but 
we will show in Proposition~\ref{prop: M supermgle} that with judicious choices of the parameters, $M$ is dominated by a supermartingale. 
Hence, we can use it to derive pointwise-in-time H\"older continuity results\footnote{Alternatively, the forward process discussed in Section~\ref{sec: forward bounds} can give H\"older continuity results for some range of $\kappa$, but not, for instance, for the range $\kappa \in [4,8]$.  
Hence, we use the backward flow instead.}.
Indeed, the supermartingale allows us to bound expectations of the form $\EX \big[ |h_t'(z_0)|^p \big]$ for specific values of $p$. 
By Lemma~\ref{lem: f' = h' in distribution}, we obtain the same bounds for derivatives of the inverse Loewner map at fixed time instants.
In particular, we may deduce that for all fixed $t \geq 0$, the map $z \mapsto f_t(z + W(t))$ is H\"older continuous.

\begin{rem}
If $W$ is a Brownian motion with variance $\kappa > 0$, then~\eqref{eq: general martingale candidate} reduces to the process used in~\cite{Rohde-Schramm:Basic_properties_of_SLE} to prove the existence of the trace of Schramm's $\SLEk$:   
if $p,q,r$, and $\kappa$ satisfy the relations
\begin{align*} 
q = p - \tfrac{1}{2} r \kappa \qquad \textnormal{and}  \qquad
p = \tfrac{1}{2} r ( \kappa + 4 - \kappa r ) 
\end{align*}
and $W = \sqrt{\kappa} B$, 
then $M$ is a local martingale for every $r, \kappa > 0$, with
\begin{align*}
\ud M(t) = 2 r \sqrt{\kappa} \, \frac{X(t)}{X(t)^2 + Y(t)^2} \, M(t) \, \ud B(t) .
\end{align*}
Furthermore, the time-changed process $\hat{M}(s) := M(\sigma(s))$, with
\begin{align} \label{eq: time change}
\sigma(s) = S^{-1}(s) , \qquad 
S(t) = \int_0^t \frac{\ud u}{X(u-)^2 + Y(u-)^2}
= \int_0^t \frac{\ud u}{|Z(u-)|^2} ,
\end{align}
is a true martingale \textnormal{(}see, e.g.,~\cite[Theorem~5.5]{Kemppainen:SLE_book}\textnormal{)}.
Rohde \&~Schramm 
used this martingale to derive bounds for the derivative of the inverse Loewner chain pointwise in time and then to deduce the existence of the $\SLEk$ trace using the modulus of continuity of the driving Brownian motion~\cite{Rohde-Schramm:Basic_properties_of_SLE}.

Let us note, however, that the process~\eqref{eq: general martingale candidate} seems not sufficient to derive the existence of general Loewner traces driven by L\'evy processes, 
because unlike for Brownian motion, whose modulus of continuity is well-understood, one cannot interpolate a L\'evy process between two consecutive dyadic points due to its possibly  uncontrolled number of small jumps (after some investigations, we concluded that even very precise tail estimates for L\'evy processes do not seem sufficient for this). 
For this reason, new techniques are needed in the case of present interest. In Section~\ref{sec: forward bounds}, we use another process derived from the \emph{forward} Loewner flow to obtain estimates uniform in time and sufficient to conclude the existence of the trace.
\end{rem}

\begin{lem} \label{lem: SDE M}
Fix $\kappa \geq 0$, a L\'evy measure $\nu$, and $\varepsilon > 0$. 
Let $(h_t)_{t \geq 0}$ be the solution to~\eqref{eq: mBLE} driven by 
$\smash{\microDriver_\varepsilon^{\kappa}}$~\eqref{eq: BM with micro},
fix $z_0 \in \bH$, 
and consider the process $M$ defined in~\eqref{eq: general martingale candidate}.
Then, for each $t \geq 0$, we have
\begin{align*}
M(t) = \; & \mgle(t) 
\; + \; \int_0^t M(s-) \; \bigg( D_{r}(s)
+ 2 p \, \frac{X(s-)^2 - Y(s-)^2}{|Z(s-)|^4} 
+ \frac{2 q}{|Z(s-)|^2} \bigg) \, \ud s  , \\[1em]
\textnormal{where} \qquad 
D_{r}(s) = \; & 
\frac{ r (\kappa (2 r - 1) - 8) X(s-)^2 + r \kappa Y(s-)^2}{|Z(s-)|^4} \\ 
\; & 
+ \; \int_{|\jump| \leq \varepsilon}   \bigg( \bigg| \frac{Z(s-) + \jump}{Z(s-)} \bigg|^{2r} - 1 - \frac{2r \jump X(s-)}{|Z(s-)|^2} \bigg) \nu(\ud \jump) , \qquad 
p, q \in \bR, \textnormal{ and } r > 0 ,
\end{align*}
and where 
\begin{align} 
\begin{split} 
\label{eq: supermgle N}
\mgle(t) := \; & 
M(0) \; + \;
2 r \sqrt{\kappa} \int_0^t M(s-) \; \frac{X(s-)}{|Z(s-)|^2} \, \ud B(s) \\
\; & + \; \int_0^t M(s-) \int_{|\jump| \leq \varepsilon} \bigg( \bigg| \frac{Z(s-) + \jump}{Z(s-)} \bigg|^{2r} - 1 \bigg) \, \PoissonComp(\ud s, \ud \jump)  .
\end{split} 
\end{align}
\end{lem}

\begin{proof}
By~\eqref{eq: mBLE} and a straightforward application of It\^o's formula, we have
\begin{align*}
|h_t'(z_0)|^p 
\; = \; \; & 1 \; + \; 2 p \int_0^t |h_{s}'(z_0)|^{p} \; \frac{X(s-)^2 - Y(s-)^2}{|Z(s-)|^4} \, \ud s , \\
Y(t)^q
\; = \; \; & y_0 \; + \; 2 q \int_0^t  \frac{Y(s-)^q}{|Z(s-)|^2} \, \ud s .
\end{align*}
A more involved application of It\^o's formula (see Lemma~\ref{lem: SDE sin(arg(Z)) power} in Appendix~\ref{app: Ito calculus} with $a=0$) gives
\begin{align*}
( \sin \arg Z(t) )^{-2r} 
\; = \; \; & ( \sin \arg z_0 )^{-2r} 
\; + \; 2 r \sqrt{\kappa} \int_0^t ( \sin \arg Z(s-) )^{-2r} \; \frac{X(s-)}{|Z(s-)|^2} \, \ud B(s) \\
\; \; & 
+ \; \int_0^t \int_{|\jump| \leq \varepsilon} ( \sin \arg Z(s-)  )^{-2r} \; \bigg( \bigg| \frac{Z(s-) + \jump}{Z(s-)} \bigg|^{2r} - 1 \bigg) \, \PoissonComp(\ud s, \ud \jump) \\
\; \; & 
+ \; \int_0^t( \sin \arg Z(s-) )^{-2r}
\; D_{r}(s) \, \ud s .
\end{align*} 
Combining these, we obtain the asserted identity for $M$. 
\end{proof}

Recall that we denote the variance of the jumps of the random driving function~\eqref{eq: BM with micro} by
\begin{align*} 
\lambda_\varepsilon := \int_{|\jump| \leq \varepsilon} \jump^2 \, \nu(\ud \jump) \; \geq 0 .
\end{align*}

\begin{prop} \label{prop: M supermgle}
Fix $\kappa \geq 0$, a L\'evy measure $\nu$, and $\varepsilon > 0$. 
Fix $r \in (0, 1]$ and set
\begin{align} \label{eq: p and q alternative}
p = p(\kappa, r ) 
:= \tfrac{1}{2} r ( \kappa + 4 - \kappa r \big) 
\qquad \textnormal{and}  \qquad
q = q(\kappa, \lambda_\varepsilon, r ) 
:= p(\kappa, r ) - \tfrac{1}{2} r (\kappa + \lambda_\varepsilon) .
\end{align}
Let $(h_t)_{t \geq 0}$ be the solution to~\eqref{eq: mBLE} driven by 
$\smash{\microDriver_\varepsilon^{\kappa}}$~\eqref{eq: BM with micro},
fix $z_0 \in \bH$, and consider the processes 
$M$ defined in~\eqref{eq: general martingale candidate}
and $\mgle$ defined in~\eqref{eq: supermgle N}. 
Then, we have $M(0) = \mgle(0)$ and $M(t) \leq \mgle(t)$ for all $t \geq 0$. 
\end{prop}

With parameters~\eqref{eq: p and q alternative}, 
the process $\mgle$ in~\eqref{eq: supermgle N} 
is a non-negative local martingale, thus a supermartingale~\cite[Lemma~5.6.8]{Cohen-Elliott:Stochastic_calculus_and_applications}.
Hence, Proposition~\ref{prop: M supermgle} shows by the supermartingale property that
\begin{align*}
\EX[M(t)] \; \leq \; \EX[\mgle(t)] \; \leq \; \mgle(0) \; = \; M(0) \quad \textnormal{ for all } t \geq 0 . 
\end{align*}

\begin{proof}
Using the fact that $(1+x)^r \leq 1 + rx$ for all $x \geq -1$ and $r \in (0,1]$, we see that
\begin{align*}
\bigg| \frac{Z(s-) + \jump}{Z(s-)} \bigg|^{2r} - 1 - \frac{2r \jump X(s-)}{|Z(s-)|^2} \;
= \;\; & \bigg( 1 + \frac{2 \jump X(s-) + \jump^2}{|Z(s-)|^2} \bigg)^r - 1 - \frac{2r \jump X(s-)}{|Z(s-)|^2}  \\
\leq \; & \frac{r \jump^2 }{|Z(s-)|^2} .
\end{align*}
Therefore, using Lemma~\ref{lem: SDE M} and the definition~\eqref{eq: jump variance} of the variance $\lambda_\varepsilon$, we obtain
\begin{align*}
M(t) 
\; \leq \; \; & \mgle(t)
\; + \; \int_0^t M(s-) \; \bigg( A_{r}(s)
+ \frac{2 p X(s-)^2 - 2 p Y(s-)^2}{|Z(s-)|^4}
+ \frac{2 q}{|Z(s-)|^2} \bigg) \, \ud s ,
\end{align*}
where 
\begin{align*}
A_{r}(s) \; = \; \; & 
\frac{ r (\kappa (2 r - 1) - 8) X(s-)^2 + r \kappa Y(s-)^2}{|Z(s-)|^4} 
\; + \; \frac{r \lambda_\varepsilon}{|Z(s-)|^2} .
\end{align*}
Plugging in the assumed identities~\eqref{eq: p and q alternative}, we see that the drift equals zero, so we have $M(t) \leq \mgle(t)$.
\end{proof}

\subsection{Derivative tail estimate}
\label{subsec: main estimate}

The goal of this section is to use the processes from Proposition~\ref{prop: M supermgle} 
to derive bounds for the probabilities 
$\smash{\PR \big[ |h_t'(z_0)| \geq \zeta \big]}$ for $\zeta > 0$, where
$(h_t)_{t \geq 0}$ is the solution to~\eqref{eq: mBLE} 
driven by $\smash{\microDriver_\varepsilon^{\kappa}}$ defined in~\eqref{eq: BM with micro}, 
and $z_0 = x_0 + \ii y_0 \in \bH$ is fixed. 
As before, we write 
\begin{align*}
Z(t) = Z_\varepsilon^{\kappa}(t,z_0) := \; & h_t(z_0) + \microDriver_\varepsilon^{\kappa}(t) =: X(t) + \ii Y(t) .
\end{align*}  
We will also use the time-changed processes 
\begin{align*}
\hat{X}(s) := X(\sigma(s))  , \qquad 
\hat{Y}(s) := Y(\sigma(s))  , \qquad 
\hat{Z}(s) := Z(\sigma(s))  , \qquad
\hat{h}_s := h_{\sigma(s)} , \qquad 
\hat{M}(s) := M(\sigma(s)) ,
\end{align*}
where $\sigma$ is given by~\eqref{eq: time change}.
The imaginary part of~\eqref{eq: mBLE} gives $Y(0) = y_0$ and 
\begin{align} \label{eq: Y}
Y(t) 
\; = \; \; & y_0 \, \exp \bigg( \int_0^t \frac{2}{|Z(s-)|^2} \, \ud s \bigg) 
\; = \; y_0 \; + \; 2 \int_0^t \frac{Y(s-)}{|Z(s-)|^2} \, \ud s ,
\end{align}
which implies in particular that
\begin{align*} 
0 < y_0 \leq Y(t) \leq \sqrt{y_0^2 + 4t}  \qquad \textnormal{and} \qquad
\hat{Y}(t) = y_0 \, e^{2t}
\qquad \textnormal{for all } t \geq 0 .
\end{align*} 
Also,~\eqref{eq: time change} can be written as
$\smash{S(t) = \frac{1}{2} \log \big( \frac{Y(t)}{y_0} \big)}$, 
so we see that $\underset{t \to \infty}{\lim} \sigma(t) = \infty$ almost surely.

The following fundamental bound holds uniformly in $t$ in compact sub-intervals of $[0,\infty)$.
It is a (slightly modified) generalization of~\cite[Corollary~3.5]{Rohde-Schramm:Basic_properties_of_SLE}, where the result was shown when $W$ is a scalar multiple of Brownian motion. 
(See also~\cite[Corollary~5.1]{Kemppainen:SLE_book}.)

\begin{lem} \label{lem: bound for derivative of h}
Fix $T > 0$, $\kappa \geq 0$, a L\'evy measure $\nu$, and $\varepsilon > 0$. 
Let $(h_t)_{t \geq 0}$ be the solution to~\eqref{eq: mBLE} driven by 
$\smash{\microDriver_\varepsilon^{\kappa}}$~\eqref{eq: BM with micro}. 
Fix $r \in (0, 1]$ and set $p$ and $q$ as in~\eqref{eq: p and q alternative}. 
Then, for any $x_0 \in \bR$, we have
\begin{align*}
\PR \big[ |h_t'(x_0 + \ii y_0)| \geq \zeta \big]
\leq \; &
c_0 \, \Big(\frac{|x_0 + \ii y_0|}{y_0}\Big)^{2r} \zeta^{-p} \, \chi_r(y_0, \zeta) \qquad
\textnormal{for all } t \in [0,T] , \;  y_0 \in (0,1] ,
\textnormal{ and } 
\zeta \in (0, 1/y_0] , \\[1em]
\textnormal{where} \qquad 
\chi_r(y_0, \zeta) =  \; &
\begin{cases}
\zeta^{\frac{r}{2} (\kappa + \lambda_\varepsilon) - p} , & \kappa r < 4 - \lambda_\varepsilon , \\[.5em]
1 - \log (\zeta y_0)  , & \kappa r = 4 - \lambda_\varepsilon , \\[.5em]
y_0^{p - \frac{r}{2} (\kappa + \lambda_\varepsilon)}  , & \kappa r > 4 - \lambda_\varepsilon ,
\end{cases}
\qquad \textnormal{and} \qquad 
p = \tfrac{1}{2} r (\kappa + 4 - \kappa r) ,
\end{align*}
and where $c_0 = c_0(\kappa, \lambda_\varepsilon, r, T) \in (0,\infty)$ 
is a constant. 
\end{lem}

\begin{proof}
Fix $t \in [0,T]$, $z_0 = x_0 + \ii y_0 \in \bH$ with $y_0 \in (0,1]$, and $\zeta \in (0, 1/y_0] $. 
We first work with the time-changed processes.
Note that $r \in (0, 1]$ implies $p > 0$ by~\eqref{eq: p and q alternative}.
Hence, by Markov's inequality, 
\begin{align} \label{eq: use Chebyshev}
\PR \big[ |\hat{h}_s'(z_0)| \geq \zeta \big] 
\leq \zeta^{-p} \; \EX \big[ |\hat{h}_s'(z_0)|^p \big] , \qquad s \geq 0 .
\end{align} 
From Proposition~\ref{prop: M supermgle}, we see that
$\EX [ \hat{M}(s) ] \leq \hat{M}(0)$, so 
\begin{align}
\nonumber
\EX \big[ |\hat{h}_s'(z_0)|^p \big] 
\; = \; \; & \EX \big[ \hat{Y}(s)^{-q} \, \big( \sin \arg \hat{Z}(s) \big)^{2r} \hat{M}(s) \big] \\
\nonumber
\; \leq \; \; & \EX \big[ \hat{Y}(s)^{-q} \, \hat{M}(s) \big] 
\, = \; y_0^{-q} \, e^{-2sq} \; \EX \big[ \hat{M}(s) \big] \\
\label{eq: aux moment bound}
\; \leq \; \; & y_0^{-q} \, e^{-2sq} \; \hat{M}(0)
\; = \; y_0^{-q} \, e^{-2sq} \; y_0^{q} \, \bigg(\frac{|z_0|}{y_0} \bigg)^{2r} 
\; = \; e^{-2sp} e^{s r (\kappa + \lambda_\varepsilon)} \bigg(\frac{|z_0|}{y_0} \bigg)^{2r} ,
\end{align}
also using the bound $(\sin \arg \hat{Z}(s))^{2r} \leq 1$ for $r \geq 0$ 
and the relation~\eqref{eq: p and q alternative} to write $q = p - \frac{r}{2} (\kappa + \lambda_\varepsilon)$.

Next, since $Y(t) \leq \sqrt{y_0^2 + 4t} \leq \sqrt{1 + 4T}$, writing $L := \frac{1}{2} \log \big(\sqrt{1 + 4T}/y_0 \big)$
and using the bound
\begin{align} \label{eq: aux bound for h}
|\hat{h}_{s+u}'(z_0)| \leq e^{2u} |\hat{h}_{s}'(z_0)| \qquad \textnormal{for all } u \geq 0 ,
\end{align}
which follows from~\eqref{eq: mBLE}, we obtain
\begin{align*}
\PR \big[ |h_t'(z_0)| \geq \zeta \big]
\leq \; &  
\PR \Big[ \sup_{0 \leq s \leq L} |\hat{h}_s'(z_0)| \geq \zeta \Big] \\
\leq \; & \sum_{j = 0}^{\lfloor L \rfloor} 
\PR \Big[ |\hat{h}_j'(z_0)| \geq e^{-2} \zeta  \Big]	
 \; \leq  \; \sum_{j = 0 \vee \lceil \log \sqrt{\zeta} - 1 \rceil}^{\lfloor L \rfloor}
\PR \Big[ |\hat{h}_j'(z_0)| \geq e^{-2} \zeta  \Big]
\qquad && \textnormal{[using~\eqref{eq: aux bound for h}]} \\
\leq \; & e^{2p} \zeta^{-p} \bigg(\frac{|z_0|}{y_0} \bigg)^{2r} 
\sum_{j = 0 \vee \lceil \log \sqrt{\zeta} - 1 \rceil}^{\lfloor L \rfloor}
e^{-2jp} e^{j r (\kappa + \lambda_\varepsilon)} 
\qquad && \textnormal{[by~(\ref{eq: use Chebyshev},~\ref{eq: aux moment bound})]} \\[.5em]
\leq \; & 
c_0(\kappa, \lambda_\varepsilon, r, T) \, \zeta^{-p} \Big(\frac{|z_0|}{y_0}\Big)^{2r}  
\times
\begin{cases}
\zeta^{\frac{r}{2} (\kappa + \lambda_\varepsilon) - p} , & 2p > r (\kappa + \lambda_\varepsilon) , \\[.5em]
1 - \log (\zeta y_0)  , & 2p = r (\kappa + \lambda_\varepsilon) , \\[.5em]
y_0^{p - \frac{r}{2} (\kappa + \lambda_\varepsilon)}  , & 2p < r (\kappa + \lambda_\varepsilon) .
\end{cases}
\end{align*}
Finally, using the relation~\eqref{eq: p and q alternative} to write $2 p = (\kappa + 4)r - \kappa r^2 $, we obtain the asserted estimate.
\end{proof}

\subsection{Uniform H\"older continuity}
\label{subsec: Holder continuity}

In this section, we prove that the Loewner chain $(h_t)_{t \geq 0}$ solving~\eqref{eq: mBLE} driven by $\smash{\microDriver_\varepsilon^{\kappa}}$ is uniformly H\"older continuous for $t$ in compact time intervals
(Proposition~\ref{prop: BLE micro jumps Holder}), unless $\kappa = 4$. For this purpose, we derive a boundary estimate for the derivative of $h_t$.

Making use of Remark~\ref{rem: tune epsilon}, we may choose $\lambda_\varepsilon$ arbitrarily small by picking a small enough cutoff $\varepsilon > 0$.  
For deriving the boundary estimate, we shall make a Borel-Cantelli argument (see Corollary~\ref{cor: uniform spatial tail estimate for h}), 
where for summability of certain probabilities, it is necessary for our argument that 
\begin{align*}
p(\kappa, r ) + q(\kappa, \lambda_\varepsilon, r )
= \tfrac{1}{2} r (\kappa - \lambda_\varepsilon + 8) - \kappa r^2 
> 1 + 2r 
\end{align*}
with suitably chosen parameter $r \in (0,1]$ and cutoff $\varepsilon > 0$. 
It turns out that this is possible whenever $\kappa \neq 4$ (see Lemma~\ref{lem: alphaprime and thetaprime}).
Note that the function $r \mapsto p(\kappa, r ) + q(\kappa, 0, r ) - 2r$ 
has a unique maximum at $r = \frac{1}{4} + \frac{1}{\kappa}$.
This motivates the following choices (which are not optimal, but sufficiently convenient)\footnote{Note also that the constant $1/6$ obtained in~\cite[Theorem~5.2]{Chen-Rohde:SLE_driven_by_symmetric_stable_processes} coincides with our $\ThetaHmax{0}{0}$ with $\kappa \to 0$ and $\lambda_\varepsilon \to 0$.}.

\begin{lem} \label{lem: alphaprime and thetaprime}
Fix $\kappa \geq 0$, a L\'evy measure $\nu$, and $\varepsilon > 0$. 
Assume that identities~\eqref{eq: p and q alternative} hold with
\begin{align} \label{eq: choice of r}
r = \rparamH(\kappa) := \big( \tfrac{1}{4} + \tfrac{1}{\kappa} \big) \wedge 1 \; \in \; (0,1] ,
\end{align}
with the convention that $\rparamH(0) = 1$, and define
\begin{align} \label{eq: lambdatildemax}
\maxjumpH{\kappa} 
:= (2 - \kappa) \vee \frac{(\kappa - 4)^2}{2(\kappa + 4)} \; \geq \; 0 .
\end{align}
Then, we have
$\maxjumpH{\kappa} = 0$ if and only if $\kappa = 4$. 
Define also
\begin{align} \label{eq: Thetatilde}
\ThetaHmax{\kappa}{\lambda_\varepsilon} 
:= \; &
\frac{p(\kappa, \rparamH(\kappa) ) + q(\kappa, \lambda_\varepsilon, \rparamH(\kappa) ) - 2 \rparamH(\kappa) - 1}{2 p(\kappa, \rparamH(\kappa) ) + q(\kappa, \lambda_\varepsilon, \rparamH(\kappa) )} 
\; = \; 
\begin{cases}
1 + \frac{10}{\kappa +\lambda_\varepsilon - 12} , &
0 < \kappa \leq 4/3 ,  \\[.3em]
\frac{2 (\kappa - 4 )^2 - 4 (\kappa + 4) \lambda_\varepsilon }{(\kappa + 4) (5 \kappa - 4 \lambda_\varepsilon + 36)}  , &
4/3 < \kappa . 
\end{cases}
\end{align}
Then, if $\kappa = 4$, then we have $\ThetaHmax{\kappa}{\lambda_\varepsilon} \leq 0$, and
otherwise, we have
\begin{align*}
\kappa \neq 4
\qquad \textnormal{and} \qquad
0 \leq \lambda_\varepsilon < \maxjumpH{\kappa}
\qquad \Longrightarrow \qquad
\begin{cases}
\ThetaHmax{\kappa}{\lambda_\varepsilon} \in (0,2/5) 
,  \\[.3em]
p(\kappa, \rparamH(\kappa) ) + q(\kappa, \lambda_\varepsilon, \rparamH(\kappa) ) 
> 1 + 2 \rparamH(\kappa) .
\end{cases}
\end{align*}
In particular, $0 \leq \lambda_\varepsilon < \maxjumpH{\kappa}$ implies that, 
for all $0 < \theta < \ThetaHmax{\kappa}{\lambda_\varepsilon}$, 
we have
\begin{align*}
(1 - 2\theta) \, p(\kappa, \rparamH(\kappa) ) 
+ (1-\theta) \, q(\kappa, \lambda_\varepsilon, \rparamH(\kappa) ) 
> 1 + 2 \rparamH(\kappa) .
\end{align*}
\end{lem}

\begin{proof}
The map $\lambda_\varepsilon \mapsto \ThetaHmax{\kappa}{\lambda_\varepsilon}$~\eqref{eq: Thetatilde} 
is decreasing when $\lambda_\varepsilon \geq 0$ and $\kappa \geq 0$. 
Moreover, we have 
\begin{align*}
\ThetaHmax{\kappa}{\lambda_\varepsilon} = 0 \qquad \Longleftrightarrow \qquad  
\lambda_\varepsilon = \maxjumpH{\kappa} .
\end{align*}
With $\lambda_\varepsilon = 0$, the map 
$\kappa \mapsto \ThetaHmax{\kappa}{0}$ 
is decreasing when $\kappa < 4$
and increasing when $\kappa > 4$, and 
\begin{align*}
\lim_{\kappa \to 0} \ThetaHmax{\kappa}{0} 
= 1/6 
, \qquad 
\lim_{\kappa \to 4} \ThetaHmax{\kappa}{0} = 0 ,
\qquad \textnormal{and} \qquad 
\lim_{\kappa \to \infty} \ThetaHmax{\kappa}{0}
= 2/5 . 
\end{align*}
Hence, for fixed $\kappa \neq 4$, 
the parameter~\eqref{eq: Thetatilde} satisfies $\ThetaHmax{\kappa}{\lambda_\varepsilon} \in (0,2/5)$ when $0 \leq \lambda_\varepsilon < \maxjumpH{\kappa}$,
See also Figures~\ref{fig: Alphaprime and Thetaprime 1}~\&~\ref{fig: Alphaprime and Thetaprime 2}. 
The other claims follow directly from the definition~\eqref{eq: Thetatilde} of $\ThetaHmax{\kappa}{\lambda_\varepsilon}$.
\end{proof}

\noindent 
\begin{figure}[ht!]
\includegraphics[width=.4\textwidth]{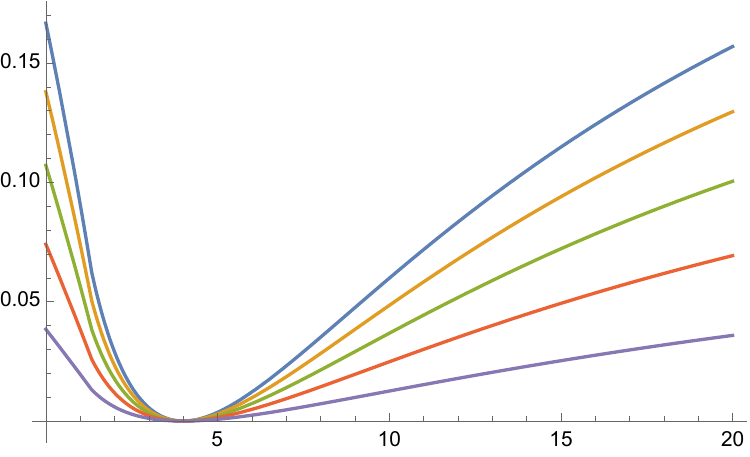}
\qquad
\includegraphics[width=.4\textwidth]{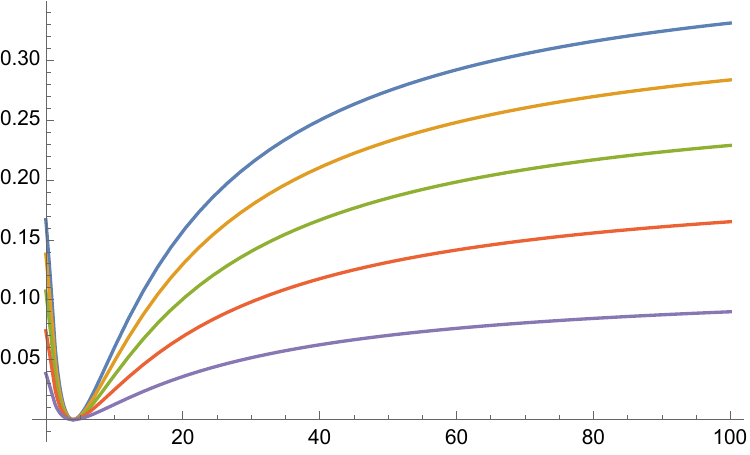}
\caption{\label{fig: Alphaprime and Thetaprime 1}
Illustrating quantities in Lemma~\ref{lem: alphaprime and thetaprime}: 
Plots of $\kappa \mapsto \ThetaHmax{\kappa}{\lambda_\varepsilon}$
with discrete values $\lambda_\varepsilon = c \, \maxjumpH{\kappa}$ for $c \in \{ 0, \frac{1}{5}, \frac{2}{5} ,\frac{3}{5} , \frac{4}{5} \}$. The largest plot (blue) has $\lambda_\varepsilon = 0$.
}
\end{figure}

\noindent 
\begin{figure}[ht!]
\includegraphics[width=.4\textwidth]{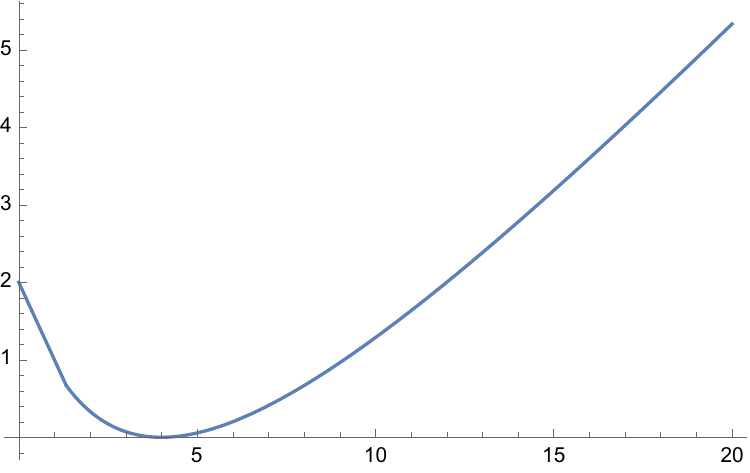}
\qquad
\includegraphics[width=.4\textwidth]{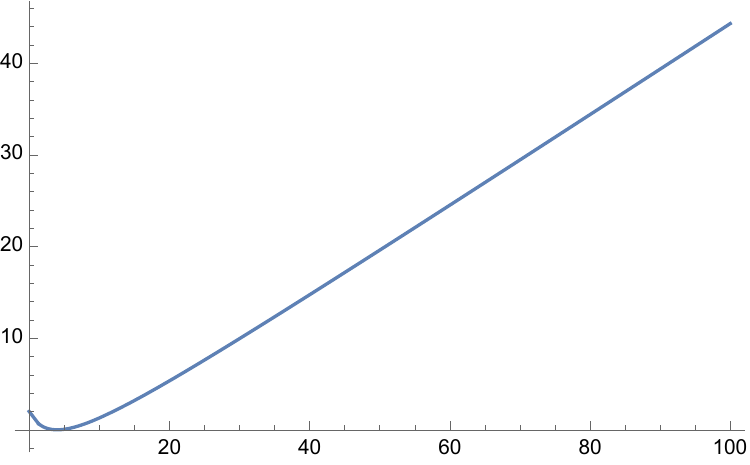}
\caption{\label{fig: Alphaprime and Thetaprime 2}
Illustrating quantities in Lemma~\ref{lem: alphaprime and thetaprime}: 
Plots of $\kappa \mapsto \maxjumpH{\kappa}$.
}
\end{figure}

The following lemma is a variant of~\cite[Lemma~5.1]{Chen-Rohde:SLE_driven_by_symmetric_stable_processes}.
Importantly, it gives a bound for the modulus of the derivative of $h_t$ uniformly in time, which will be needed later in the proof of Proposition~\ref{prop: BLE micro jumps Holder}.

\begin{lem} \label{lem: uniform spatial tail estimate for h}
Fix $T > 0$, $\kappa \in [0,\infty) \setminus\{4\}$, a L\'evy measure $\nu$, and $\varepsilon > 0$ such that 
$\lambda_\varepsilon < \maxjumpH{\kappa}$. 
Let $(h_t)_{t \geq 0}$ be the solution to~\eqref{eq: mBLE} driven by 
$\smash{\microDriver_\varepsilon^{\kappa}}$~\eqref{eq: BM with micro}. 
Fix $r = \rparamH(\kappa)$ as in~\eqref{eq: choice of r}. 
Then, for any $R > 0$,
for any $0 < \theta < \ThetaHmax{\kappa}{\lambda_\varepsilon}$ as in~\eqref{eq: Thetatilde}, 
and for any
\begin{align*}
x = \ell R 2^{-n} \in \sR_{n}^R := \{ \ell R 2^{-n} \; | \; \ell = 0,\pm 1,\pm 2,\ldots, \pm 2^{n} \} , \qquad 
\textnormal{ with }
n \in \bZnn ,
\end{align*}
we have
\begin{align} \label{eq: uniform spatial tail estimate for h}
\PR \Big[ \max_{t \in [0,T]}  |h_t'(\ell R 2^{-n} + \ii 2^{-n})| \geq  2^{n (1-\theta)} \Big] 
\leq \; & c_0 \, (\ell^2 R^2 + 1)^{r} \, 2^{-n \beta} ,
\end{align}
where $c_0 = c_0(\kappa, T) \in (0,\infty)$ is a constant, and
\begin{align} \label{eq: beta}
\beta = (1-2\theta) p + (1-\theta) q > 1 + 2 r , 
\end{align}
with $r = \rparamH(\kappa)$, and
$p = p( \kappa, \rparamH(\kappa))$, and
$q = q( \kappa, \lambda_\varepsilon, \rparamH(\kappa) )$ as in~\eqref{eq: p and q alternative}.
\end{lem}

\begin{proof}
Using~(\ref{eq: mBLE},~\ref{eq: Y}) and the time-change~\eqref{eq: time change}, we find
\begin{align*}
|h_t'(z_0)| 
= \; & \exp \bigg( 2 \int_0^t  \frac{X(s-)^2 - Y(s-)^2}{|Z(s-)|^4} \, \ud s \bigg)
= \exp \bigg( 2 \int_0^{S(t)}  \hat{U}(u) \, \ud u \bigg) , 
\end{align*}
where $\smash{S(t) = \frac{1}{2} \log \big( \frac{Y(t)}{y_0} \big)}$ and
\begin{align} \label{eq: less than one}
\hat{U}(u) := \frac{\hat{X}(u-)^2 - \hat{Y}(u-)^2}{\hat{X}(u-)^2 + \hat{Y}(u-)^2} 
\qquad \textnormal{satisfies} \qquad
|\hat{U}(u)|  \leq 1 .
\end{align} 
\begin{itemize}[leftmargin=*]
\item If $Y(t) < y_0^{\theta}$, then $|h_t'(z_0)| \leq e^{2 S(t)} = \frac{Y(t)}{y_0} < y_0^{\theta-1}$.

\medskip

\item If $y_0^{\theta} \leq Y(t) \leq 1$, then 
\begin{align*}
|h_t'(z_0)| 
= \; & \exp \bigg( 2 \int_0^{\frac{\theta-1}{2} \log y_0} \hat{U}(u) \, \ud u
\; + \; 2 \int_{\frac{\theta-1}{2} \log y_0}^{S(t)} \hat{U}(u) \, \ud u \bigg) \\
= \; &
|\hat{h}_{\frac{\theta-1}{2} \log y_0}'(z_0)| 
\, \exp \bigg( 2 \int_{\frac{\theta-1}{2} \log y_0}^{S(t)} \hat{U}(u) \, \ud u \bigg) 
&& \textnormal{[since $\hat{h}_s = h_{\sigma(s)}$ and $\sigma(s) = S^{-1}(s)$]} \\
\leq \; & |\hat{h}_{\frac{\theta-1}{2} \log y_0}'(z_0)| \, Y(t) \, y_0^{- \theta} 
 && \textnormal{[by~\eqref{eq: less than one}]} \\
\leq \; & |\hat{h}_{\frac{\theta-1}{2} \log y_0}'(z_0)|  \, y_0^{- \theta} .
 && \textnormal{[since $Y(t) \leq 1$]} 
\end{align*}

\medskip

\item If $Y(t) \geq 1$, then 
\begin{align*}
|h_t'(z_0)| 
= \; & \exp \bigg( 2 \int_0^{- \frac{1}{2} \log y_0} \hat{U}(u) \, \ud u
\; + \; 2 \int_{- \frac{1}{2} \log y_0}^{S(t)} \hat{U}(u) \, \ud u \bigg) \\
= \; &
|\hat{h}_{- \frac{1}{2} \log y_0}'(z_0)| 
\, \exp \bigg( 2 \int_{- \frac{1}{2} \log y_0}^{S(t)} \hat{U}(u) \, \ud u \bigg) 
&& \textnormal{[since $\hat{h}_s = h_{\sigma(s)}$ and $\sigma(s) = S^{-1}(s)$]} \\
\leq \; & |\hat{h}_{- \frac{1}{2} \log y_0}'(z_0)| \, Y(t)
 && \textnormal{[by~\eqref{eq: less than one}]} \\
\leq \; & |\hat{h}_{- \frac{1}{2} \log y_0}'(z_0)| \, \sqrt{y_0^2 + 4t} .
&& \textnormal{[since $Y(t) \leq \sqrt{y_0^2 + 4t}$]} 
\end{align*}
\end{itemize}
Using these bounds, we can estimate
\begin{align*} 
\PR \Big[ \max_{t \in [0,T]}  |h_t'(z_0)| \geq y_0^{\theta-1} \Big] 
\; \leq \;  \; & 
\PR \Big[  |\hat{h}_{\frac{\theta-1}{2} \log y_0}'(z_0)| \, y_0^{-\theta} \geq y_0^{\theta-1} \Big] 
\, + \, \PR \Big[ |\hat{h}_{- \frac{1}{2} \log y_0}'(z_0)| \, \sqrt{y_0^2 + 4T} \geq y_0^{\theta-1} \Big] .
\end{align*}
Now, using~(\ref{eq: use Chebyshev},~\ref{eq: aux moment bound}) with $y_0 \in (0,1]$ and parameters chosen as 
$r = \rparamH(\kappa)$, and $p = p( \kappa, \rparamH(\kappa) )$ as in~\eqref{eq: p and q alternative}, 
we find that the first term on the right-hand side 
can be bounded as
\begin{align} \label{eq: estimate got hhatprime 1}
\PR \Big[  |\hat{h}_{\frac{\theta-1}{2} \log y_0}'(z_0)| \geq y_0^{2\theta-1} \Big] 
\; \leq \; \; & y_0^{p (1-2\theta)} \, \EX \Big[ |\hat{h}_{\frac{\theta-1}{2} \log y_0}'(z_0)|^{p} \Big]  
\; \leq \; 
y_0^{\beta} \bigg(\frac{|z_0|}{y_0} \bigg)^{2r} ,
\end{align}
where $\beta = (1-2\theta) p + (1-\theta) q$ as in~\eqref{eq: beta}. 
Similarly, the second term can be bounded as
\begin{align} 
\nonumber
\PR \Big[ |\hat{h}_{- \frac{1}{2} \log y_0}'(z_0)| \, \sqrt{y_0^2 + 4T} \geq y_0^{\theta-1} \Big] 
\; \leq \; \; &  y_0^{p (1-\theta)} \, (y_0^2 + 4T)^{p/2} \, \EX \Big[  |\hat{h}_{- \frac{1}{2} \log y_0}'(z_0)|^{p} \Big] \\
\label{eq: estimate got hhatprime 2}
\; \leq \; \; &  
(y_0^2 + 4T)^{p/2} \, 
y_0^{\beta'} \bigg(\frac{|z_0|}{y_0} \bigg)^{2r} ,
\end{align}
where $\beta' = (1-\theta) p + q$. 
By the choice of $\theta < \ThetaHmax{\kappa}{\lambda_\varepsilon}$, 
Lemma~\ref{lem: alphaprime and thetaprime} shows that
$\beta' > \beta > 1 + 2 r$.

To finish, taking $z_0 = \ell R 2^{-n} + \ii 2^{-n}$, where $y_0 = 2^{-n} \leq 1$,
we obtain from~(\ref{eq: estimate got hhatprime 1},~\ref{eq: estimate got hhatprime 2}) the estimate 
\begin{align*}
\PR \Big[ \max_{t \in [0,T]}  |h_t'(\ell R 2^{-n} + \ii 2^{-n})| \geq 2^{n(1-\theta)} \Big] 
\leq \; & 
c_0(\kappa, r, T)  \, 
(\ell^2 R^2 + 1)^{r} \, 2^{-n \beta} ,
\end{align*}
with $r = \rparamH(\kappa)$ and $\beta = (1-2\theta) p + (1-\theta) q$ as in~\eqref{eq: beta}. 
This proves the asserted estimate~\eqref{eq: uniform spatial tail estimate for h}. 
\end{proof}

\begin{cor} \label{cor: uniform spatial tail estimate for h}
Fix $T > 0$, $\kappa \in [0,\infty) \setminus\{4\}$, a L\'evy measure $\nu$, 
and $\varepsilon > 0$ such that 
$\lambda_\varepsilon < \maxjumpH{\kappa}$.
Let $(h_t)_{t \geq 0}$ be the solution to~\eqref{eq: mBLE} driven by 
$\smash{\microDriver_\varepsilon^{\kappa}}$~\eqref{eq: BM with micro}. 
Then, for any $R > 0$ and for any $0 < \theta < \ThetaHmax{\kappa}{\lambda_\varepsilon}$ as in~\eqref{eq: Thetatilde},
there exist almost surely finite random constants 
$\smash{C_{\lambda_\varepsilon}^{\kappa}}(\theta, T, R)$ such that
\begin{align} \label{eq: needed bound for h' large radius}
\max_{t \in [0,T]} |h_t'(x + \ii 2^{-n})| \leq 
C_{\lambda_\varepsilon}^{\kappa}(\theta, T, R) \, 2^{n (1-\theta)}   \qquad 
\textnormal{for all } n \in \bZnn \textnormal{ and } x \in \sR_{n}^R  .
\end{align}
\end{cor}

\begin{proof}
We use Lemma~\ref{lem: uniform spatial tail estimate for h} with $r = \rparamH(\kappa)$ and $\beta > 1 + 2r$ as in~\eqref{eq: beta} to obtain
\begin{align*}
\sum_{n = 0}^\infty \sum_{x \in \sR_{n}^R } 
\PR \Big[ \max_{t \in [0,T]}  |h_t'(x + \ii 2^{-n})| \geq 2^{n (1-\theta)} \Big] 
\leq \; & 
c_0(\kappa, T) 
\sum_{n = 0}^\infty
\sum_{x \in \sR_{n}^R } (x^2 2^{2n} + 1)^{r} \,
2^{-n \beta} \\
\leq \; & 
c_1(\kappa, T, R) 
\sum_{n = 0}^\infty
\sum_{x \in \sR_{n}^R } 2^{-n ( \beta - 2 r )} \\
\leq \; & 
c_2(\kappa, T, R) 
\sum_{n = 0}^\infty
2^{-n ( \beta - 2 r - 1 )} 
\; < \; \infty . 
\end{align*} 
The Borel-Cantelli lemma now implies
that almost surely, we have $\smash{\underset{t \in [0,T]}{\max} |h_t'(x + \ii 2^{-n})| \leq 2^{n (1-\theta)}}$ 
except for possibly finitely many 
$n \in \bZnn$ and $x \in \sR_{n}^R$, so the formula 
\begin{align*}
C_{\lambda_\varepsilon}^{\kappa}(\theta, T, R) 
:= \sup_{n \in \bZnn , \; x \in \sR_{n}^R}
2^{-n (1 - \theta)} \max_{t \in [0,T]} |h_t'(x + \ii 2^{-n})| 
\end{align*}
gives the sought almost surely finite random constant.
\end{proof}

Using these estimates, we now conclude with an analogue of~\cite[Theorem~5.2]{Rohde-Schramm:Basic_properties_of_SLE}
(and~\cite[Theorem~5.2]{Chen-Rohde:SLE_driven_by_symmetric_stable_processes}): 
H\"older continuity of the mirror backward Loewner chain uniformly on compact time intervals.

\begin{prop} \label{prop: BLE micro jumps Holder}
Fix $T > 0$, $\kappa \in [0,\infty) \setminus\{4\}$, a L\'evy measure $\nu$, and $\varepsilon > 0$ such that 
$\lambda_\varepsilon < \maxjumpH{\kappa}$.
Let $(h_t)_{t \geq 0}$ be the solution to~\eqref{eq: mBLE} driven by 
$\smash{\microDriver_\varepsilon^{\kappa}}$~\eqref{eq: BM with micro}. 
Then, for any $R > 0$ and for any $0 < \theta < \ThetaHmax{\kappa}{\lambda_\varepsilon}$ as in~\eqref{eq: Thetatilde},
there exist almost surely finite random constants 
$\smash{H_{\lambda_\varepsilon}^{\kappa}}(\theta, T, R)$ 
such that 
\begin{align*} 
| h_t(z) - h_t(w) | \leq 
H_{\lambda_\varepsilon}^{\kappa}(\theta, T, R)  \, 
\big( \, | z-w |^{\theta} \, \vee \, | z-w | \, \big) 
\qquad \textnormal{for all } t \in [0, T]
\textnormal{ and } z,w \in (-R,R) \times \ii (0,\infty)  .
\end{align*}
In particular, almost surely, each $h_t$ extends to a continuous function on $\smash{\overline{(-R,R) \times \ii (0,\infty)}}$. 
\end{prop}

\begin{proof}
By Lemma~\ref{lem: global Holder continuity}, 
it suffices to verify the bound~\eqref{eq: needed bound for derivative of confmap} uniformly for all $h_t'$ with $t \in [0,T]$.
\begin{itemize}[leftmargin=2em]
\item 
On the one hand, 
for each point $z = x + \ii y$ with 
$(x,y) \in (-R,R) \times \ii (0, 1)$, take $n \in \bZpos$ and $x_0 \in \sR_{n}^R \cap (-R,R)$ 
such that $2^{-n} \leq y < 2^{-n+1}$ and $x_0 \leq x < x_0 + 2^{-n}$.
Then, by Koebe distortion (Lemma~\ref{lem: Koebe}) 
and the estimate~\eqref{eq: needed bound for h' large radius} from Corollary~\ref{cor: uniform spatial tail estimate for h}, 
we obtain  
\begin{align*}
\max_{t \in [0,T]} |h_t'(x + \ii y)| 
\lesssim \max_{t \in [0,T]}  |h_t'(x_0 + \ii 2^{-n})|
\leq C_{\lambda_\varepsilon}^{\kappa}(\theta, T, R) \, y^{\theta-1} \qquad \textnormal{for all } (x,y) \in (-R,R) \times \ii (0, 1) .
\end{align*}

\item 
On the other hand, 
the estimate~\eqref{eq: derh bound} shows that
\begin{align*}
\max_{t \in [0,T]} |h_t'(x + \ii y)|  \leq e^{2T} \qquad \textnormal{for all } (x,y) \in (-R,R) \times [1, \infty) .
\end{align*}
\end{itemize}
This implies~\eqref{eq: needed bound for derivative of confmap} with almost surely finite constant $C_{\lambda_\varepsilon}^{\kappa}(\theta, T, R) \vee e^{2T}$. 
\end{proof}

\subsection{Drivers involving microscopic jumps and a linear drift}
\label{subsec: kappa is zero}

In this section, we consider driving functions with no diffusion part $(\kappa=0)$ but allowing microscopic jumps and a linear drift:
\begin{align} \label{eq: BM with micro and drift kappa is zero}
\microDriver_\varepsilon^{0, a}(t) = a t + \int_{|\jump| \leq \varepsilon} \jump \, \PoissonComp(t, \ud \jump) , \qquad
a \in \bR, \, \varepsilon > 0 ,
\end{align}
where $\smash{\PoissonComp}(t, \ud \jump) := \Poisson(t, \ud \jump) - t \nu(\ud \jump)$ is the compensated Poisson point process
of a Poisson point process $\Poisson$ with L\'evy intensity measure $\nu$. 
We derive estimates analogous to those obtained in Sections~\ref{subsec: main estimate}--\ref{subsec: Holder continuity}.

\subsubsection{Supermartingale in the case of microscopic jumps with linear drift}

Let $(h_t)_{t \geq 0}$ be the solution to~\eqref{eq: mBLE} with driving function 
$\smash{\microDriver_\varepsilon^{0, a}}$~\eqref{eq: BM with micro and drift kappa is zero}. 
Fix $z_0 = x_0 + \ii y_0 \in \bH$ and define the process 
\begin{align*}
Z(t) = Z_\varepsilon^{0, a}(t,z_0) := h_t(z_0) + \microDriver_\varepsilon^{0, a}(t) =: X(t) + \ii Y(t) .
\end{align*}
Also, define the process 
\begin{align} \label{eq: general martingale candidate kappa = 0}
M(t) = M_{p}^a(t,z_0) := |h_t'(z_0)|^p \, ( \sin \arg Z(t) )^{-2} \, e^{-a^2t} , \qquad 
p \in \bR .
\end{align}
Note that $M(0) = y_0^{-2} |z_0|^{2}$. 
The following estimates will be used in the proof of Proposition~\ref{prop: LLE micro jumps curve with drift}.

\begin{lem} \label{lem: SDE M kappa = 0 with drift}
Fix $a \in \bR$, a L\'evy measure $\nu$, and $\varepsilon > 0$. 
Let $(h_t)_{t \geq 0}$ be the solution to~\eqref{eq: mBLE} driven by $\smash{\microDriver_\varepsilon^{0, a}}$~\eqref{eq: BM with micro and drift kappa is zero},
fix $z_0 \in \bH$, 
and consider the process $M$ defined in~\eqref{eq: general martingale candidate kappa = 0}.
Then, for each $t \geq 0$, we have
\begin{align*}
M(t) = \; & \mgle(t)
\; + \; \int_0^t M(s-) \; \bigg( D_{p}(s)
+ \frac{2 a X(s-)}{|Z(s-)|^2} - a^2
\bigg) \, \ud s  , \\[1em]
\textnormal{where} \qquad 
D_{p}(s)
= \; & \frac{(\lambda_\varepsilon - 8 + 2 p) X(s-)^2 + (\lambda_\varepsilon - 2 p) Y(s-)^2}{|Z(s-)|^4} , \qquad 
p \in \bR ,
\end{align*}
and where 
\begin{align} 
\label{eq: supermgle N kappa = 0}
\mgle(t) := \; & 
M(0) \; 
+ \; \int_0^t M(s-) \int_{|\jump| \leq \varepsilon} \bigg( \bigg| \frac{Z(s-) + \jump}{Z(s-)} \bigg|^{2} - 1 \bigg) \, \PoissonComp(\ud s, \ud \jump)  .
\end{align}
\end{lem}

\begin{proof}
As in Section~\ref{subsec: derivative local martingale}, by~\eqref{eq: mBLE} and a straightforward application of It\^o's formula, we have
\begin{align*}
|h_t'(z_0)|^p 
\; = \; \; & 1 \; + \; 2 p \int_0^t |h_{s}'(z_0)|^{p} \; \frac{X(s-)^2 - Y(s-)^2}{|Z(s-)|^4} \, \ud s ,
\end{align*}
and a tedious application of It\^o's formula (see Lemma~\ref{lem: SDE sin(arg(Z)) power} in Appendix~\ref{app: Ito calculus} with $\kappa=0$ and $r=1$) gives
\begin{align*}
( \sin \arg Z(t) )^{-2} 
\; = \; \; & ( \sin\arg z_0 )^{-2} 
\; + \; 2 a \int_0^t ( \sin \arg Z(s-) )^{-2} \; \frac{X(s-)}{|Z(s-)|^2} \, \ud s \\
\; \; & 
+ \; \int_0^t \int_{|\jump| \leq \varepsilon} ( \sin \arg Z(s-) )^{-2} \; \bigg( \bigg| \frac{Z(s-) + \jump}{Z(s-)} \bigg|^{2} - 1 \bigg) \, \PoissonComp(\ud s, \ud \jump) \\
\; \; & 
- \; 8 \, \int_0^t( \sin \arg Z(s-) )^{-2}
\; \frac{X(s-)^2}{|Z(s-)|^4} 
\, \ud s \\
\; \; & 
+ \; \int_0^t \int_{|\jump| \leq \varepsilon} ( \sin \arg Z(s-) )^{-2} \; \bigg( \bigg| \frac{Z(s-) + \jump}{Z(s-)} \bigg|^{2} - 1 - \frac{2 \jump X(s-)}{|Z(s-)|^2} \bigg) \, \nu(\ud \jump) \ud s.
\end{align*} 
Moreover, note that 
\begin{align*}
\bigg| \frac{Z(s-) + \jump}{Z(s-)} \bigg|^{2}
\; = \; \frac{(X(s-) + \jump)^2 + Y(s-)^2}{|Z(s-)|^2} 
\; = \; 1 + \frac{2 \jump X(s-)}{|Z(s-)|^2} + \frac{\jump^2}{|Z(s-)|^2}.
\end{align*}
Combining these, we obtain the asserted identity for $M$.
\end{proof}

As in Section~\ref{subsec: derivative local martingale}, 
for a suitable $p$ and $\lambda_\varepsilon$ small enough, 
$M$ can be bounded by a supermartingale $\mgle$.
We set 
\begin{align} \label{eq: lambdamax kappa = 0}
\maxjumpH{p} := 7 - 2p .
\end{align}
Note also that when $p \in (3,7/2)$, the quantity $\maxjumpH{p}$ defined in~\eqref{eq: lambdamax kappa = 0} is maximized at $p=3$: $\maxjumpH{3} = 1$. 
We will use this value in Corollary~\ref{cor: LLE micro jumps Holder with drift}.

\begin{prop} \label{prop: M supermgle, kappa = 0}
Fix $a \in \bR$, a L\'evy measure $\nu$, and $\varepsilon > 0$.  
Fix $p \in (0,7/2)$ and $\varepsilon > 0$ such that 
\begin{align} \label{eq: lambdamax kappa = 0 bound}
\lambda_\varepsilon < 2p \wedge (7 - 2p) = 2p \wedge \maxjumpH{p} . 
\end{align}
Let $(h_t)_{t \geq 0}$ be the solution to~\eqref{eq: mBLE} driven by $\smash{\microDriver_\varepsilon^{0, a}}$~\eqref{eq: BM with micro and drift kappa is zero},
fix $z_0 \in \bH$, and consider the processes 
$M$ defined in~\eqref{eq: general martingale candidate kappa = 0}
and $\mgle$ defined in~\eqref{eq: supermgle N kappa = 0}. 
Then, we have $M(0) = \mgle(0)$ and $M(t) \leq \mgle(t)$ for all $t \geq 0$. 
\end{prop}

In particular, the process $\mgle$ in Proposition~\ref{prop: M supermgle, kappa = 0} 
is a non-negative local martingale, thus a supermartingale~\cite[Lemma~5.6.8]{Cohen-Elliott:Stochastic_calculus_and_applications}.
The supermartingale property yields
\begin{align*}
\EX[M(t)] \; \leq \; \EX[\mgle(t)] \; \leq \; \mgle(0) \; = \; M(0) \quad \textnormal{ for all } t \geq 0 . 
\end{align*}

\begin{proof}
First, we note that 
\begin{align*}
\frac{2a X(s-)}{|Z(s-)|^2} \; - \; a^2 
\; = \; \frac{X(s-)^2}{|Z(s-)|^4} \; - \; \bigg(\frac{X(s-)}{|Z(s-)|^2} - a \bigg)^2
\; \leq \; \frac{X(s-)^2}{|Z(s-)|^4} ,
\end{align*}
which implies by Lemma~\ref{lem: SDE M kappa = 0 with drift} that
\begin{align*}
M(t) 
\; \leq \; \; & \mgle(t) 
\; + \; \int_0^t M(s-) \; \bigg( \frac{(\lambda_\varepsilon - 7 + 2p) X(s-)^2 + (\lambda_\varepsilon - 2p) Y(s-)^2}{|Z(s-)|^4}
\bigg) \, \ud s.
\end{align*}
As in the proof of Proposition~\ref{prop: M supermgle}, 
the assumption~\eqref{eq: lambdamax kappa = 0 bound} shows that the drift is non-positive.
\end{proof}

\subsubsection{H\"older continuity in the case of microscopic jumps with linear drift}
\label{subsubsec: boundary estimate kappa = 0}

We next derive a uniform boundary estimate analogous to Lemma~\ref{lem: uniform spatial tail estimate for h} \&~Corollary~\ref{cor: uniform spatial tail estimate for h} involving the driving function $\smash{\microDriver_\varepsilon^{0, a}}$.

\begin{lem} \label{lem: uniform spatial tail estimate for h pure jump}
Fix $T > 0$, $a \in \bR$, and a L\'evy measure $\nu$. 
Fix $p \in (3,7/2)$ and $\varepsilon > 0$ such that 
$\lambda_\varepsilon < \maxjumpH{p}$ as in~\eqref{eq: lambdamax kappa = 0}. 
Then, the following hold 
for the solution $(h_t)_{t \geq 0}$ to~\eqref{eq: mBLE} driven by 
$\smash{\microDriver_\varepsilon^{0, a}}$~\eqref{eq: BM with micro and drift kappa is zero}.

\begin{enumerate}[label=\textnormal{(\alph*):}, ref=(\alph*)]
\item \label{item: uniform spatial tail estimate for h pure jump}
For any $R > 0$, for any $\theta \in (0,1)$, and for any 
$x = \ell R 2^{-n} \in \sR_{n}^R$ with $n \in \bZnn$, 
we have
\begin{align} \label{eq: uniform spatial tail estimate for h pure jump}
\PR \Big[ \max_{t \in [0,T]}  |h_t'(\ell R 2^{-n} + \ii 2^{-n})| \geq  2^{n (1-\theta)} \Big] 
\leq \; &
e^{a^2 T} \, (\ell^2 R^2 + 1) \, 2^{-n p (1 - \theta)} .
\end{align}

\medskip

\item \label{item: uniform spatial tail estimate for max h pure jump}
For any $R > 0$ and for any 
\begin{align} \label{eq: thetamax kappa = 0}
0 < \theta < \frac{p-3}{p} =: \theta(p) ,
\end{align} 
there exist almost surely finite random constants 
$\smash{C_{\lambda_\varepsilon}^{a}}(\theta, T, R)$
such that
\begin{align*}
\max_{t \in [0,T]} |h_t'(x + \ii 2^{-n})| \leq 
C_{\lambda_\varepsilon}^{a}(\theta, T, R) \, 2^{n (1-\theta)}   \qquad 
\textnormal{for all } n \in \bZnn \textnormal{ and }  x \in \sR_{n}^R  .
\end{align*}
\end{enumerate}
\end{lem}

\begin{proof}
For each $\zeta > 0$, using the supermartingale $\mgle$ from Proposition~\ref{prop: M supermgle, kappa = 0} we find that
\begin{align*}
\PR \Big[ \max_{t \in [0,T]}  |h_t'(z_0)| \geq \zeta \Big] 
\; = \; \; & \PR \Big[ \max_{t \in [0,T]} |h_t'(z_0)|^p \geq \zeta^p \Big] 
\; = \; \; \PR \Big[ \max_{t \in [0,T]} ( \sin \arg Z(t) )^2 \, e^{a^2 t} \, M(t) \geq \zeta^p \Big] \\
\; \leq \; \; & \PR \Big[ \max_{t \in [0,T]} M(t) \geq e^{-a^2 T} \, \zeta^p \Big] 
\; \leq \; \PR \Big[ \max_{t \in [0,T]} \mgle(t) \geq e^{-a^2 T} \, \zeta^p \Big] .
\end{align*}
From~\eqref{eq: supermgle N kappa = 0}, we see that $\mgle(0) = M(0) = ( \sin \arg Z(t) )^{-2} = \big( \frac{|z_0|}{y_0} \big)^2$, and Doob's $L^1$ supermartingale inequality
(e.g.~\cite[Theorem~5.1.2(i)]{Cohen-Elliott:Stochastic_calculus_and_applications}) 
implies that
\begin{align*}
\PR \Big[ \max_{t \in [0,T]} \mgle(t) \geq e^{-a^2 T} \, \zeta^p \Big] 
\; \leq \; \; &  e^{a^2 T} \,\zeta^{-p} \, \EX[ \mgle(0) ]
\; = \; e^{a^2 T} \, \zeta^{-p} \, \Big( \frac{|z_0|}{y_0} \Big)^2 .
\end{align*}
Taking $z_0 = \ell R 2^{-n} + \ii 2^{-n}$ and $\zeta = 2^{n (1-\theta)}$, we obtain the asserted~\eqref{eq: uniform spatial tail estimate for h pure jump}. 
Next, similarly as in the proof of Corollary~\ref{cor: uniform spatial tail estimate for h}, 
item~\ref{item: uniform spatial tail estimate for h pure jump}  shows that
\begin{align*}
\sum_{n = 0}^\infty \sum_{x \in \sR_{n}^R } 
\PR \Big[ \max_{t \in [0,T]}  |h_t'(x + \ii 2^{-n})| \geq 2^{n (1-\theta)} \Big] 
\lesssim \; & \sum_{n = 0}^\infty
2^{-n ( p (1 - \theta) - 3)}
<  \infty ,
\end{align*} 
convergent since $p (1 - \theta) - 3 > 0$ by~\eqref{eq: thetamax kappa = 0}.  
The Borel-Cantelli lemma now proves item~\ref{item: uniform spatial tail estimate for max h pure jump}.
\end{proof}

Using these estimates, we conclude with the uniform H\"older continuity.

\begin{prop} \label{prop: BLE micro jumps Holder with drift kappa = 0}
Fix $T > 0$, $a \in \bR$, and a L\'evy measure $\nu$.
Fix $p \in (3,7/2)$ and $\varepsilon > 0$ such that $\lambda_\varepsilon \leq \maxjumpH{p}$ as in~\eqref{eq: lambdamax kappa = 0}.
Let $(h_t)_{t \geq 0}$ be the solution to~\eqref{eq: mBLE} driven by 
$\smash{\microDriver_\varepsilon^{0, a}}$~\eqref{eq: BM with micro and drift kappa is zero}. 
Then, for any $R > 0$ and for any  
$0 < \theta < \theta(p)$ as in~\eqref{eq: thetamax kappa = 0}, 
there exist almost surely finite random constants 
$\smash{H_{\lambda_\varepsilon}^{a}}(\theta, T, R)$ 
such~that 
\begin{align*} 
| h_t(z) - h_t(w) | \leq 
H_{\lambda_\varepsilon}^{a}(\theta, T, R) \, 
\big( \, | z-w |^{\theta} \, \vee \, | z-w | \, \big) 
\qquad \textnormal{for all } t \in [0, T]
\textnormal{ and } z,w \in (-R,R) \times \ii (0,\infty)  .
\end{align*}
In particular, almost surely, each $h_t$ extends to a continuous function on $\overline{(-R,R) \times \ii (0,\infty)}$. 
\end{prop}

\begin{proof}
This can be proven like Proposition~\ref{prop: BLE micro jumps Holder}, 
using Lemma~\ref{lem: uniform spatial tail estimate for h pure jump}~\ref{item: uniform spatial tail estimate for max h pure jump} 
instead of Corollary~\ref{cor: uniform spatial tail estimate for h}. 
\end{proof}

\bigskip{}
\section{\label{sec: forward bounds}Estimates for forward Loewner flow with martingale L\'evy drivers}
As in Section~\ref{subsec: derivative local martingale}, we fix a L\'evy measure $\nu$ 
and consider martingale driving functions of the form 
\begin{align} \label{eq: BM with micro again again} 
\microDriver_\varepsilon^{\kappa}(t) 
= \sqrt{\kappa} B(t) + \int_{|\jump| \leq \varepsilon} \jump \, \PoissonComp(t, \ud \jump) , \qquad
\kappa \geq 0 , \; 
\varepsilon > 0 ,
\end{align} 
where $B$ is a standard Brownian motion and 
$\smash{\PoissonComp}(t, \ud \jump) := \Poisson(t, \ud \jump) - t \nu(\ud \jump)$ is the compensated Poisson point process
of a Poisson point process $\Poisson$ independent of $B$ with L\'evy intensity measure $\nu$.

We also write 
\begin{align} \label{eq: BM with micro supremum} 
R_\varepsilon^{\kappa}(T) := \sup_{t \in [0,T]} | \microDriver_\varepsilon^{\kappa}(t) | .
\end{align}
Let $(g_t)_{t \geq 0}$ be a Loewner chain driven by $\smash{\microDriver_\varepsilon^{\kappa}}$ and let $(K_t)_{t \geq 0}$ be the corresponding hulls (obtained by solving the Loewner equation~\eqref{eq: LE}). 
Also, let $f_t:= g_t^{-1}$ be the inverse Loewner chain, and set
\begin{align*} 
\tilde{f}_t(w) := f_t(w + \microDriver_\varepsilon^{\kappa}(t)) , \qquad w \in \bH .
\end{align*}

In this section, we prove that the Loewner chain $(g_t)_{t \geq 0}$ 
is generated by a (c\`adl\`ag) trace (in the sense of Definition~\ref{def: Loewner trace}) 
under one of the following assumptions:
\begin{enumerate}[label=\textnormal{\bf{Ass.~\arabic*.}}, ref=Ass.~\arabic*]
\item \label{item: ass1}
either the diffusivity parameter $\kappa > 8$,

\medskip

\item \label{item: ass2}
or the diffusivity parameter $\kappa \in [0,8)$, and 
the variance measure of the L\'evy measure $\nu$ is locally (upper) Ahlfors regular near the origin in the sense of Definition~\ref{def: Ahlfors regular}.
\end{enumerate}
Our techniques are not strong enough to treat the case where the diffusivity parameter $\kappa = 8$. 
The reason for this is the same as for the case of a Brownian driver: the estimates~\cite{Rohde-Schramm:Basic_properties_of_SLE}, as do ours, fail when $\kappa = 8$. 
It is known that the modulus of continuity for the $\SLE_8$ curve (in the capacity parameterization) is logarithmic~\cite{Alvisio-Lawler:Modulus_of_continuity_of_SLE8, KMS:Regularity_of_the_SLE4_uniformizing_map_and_the_SLE8_trace}, 
and we expect that adding jumps will not improve the regularity.

\bigskip

For each L\'evy measure $\nu$, 
define a Borel measure $\mu_\nu$ by $\mu_\nu(A) := \int_A \jump^2 \, \nu(d\jump)$ for all Borel sets $A \subset \bR$. 
We call $\mu_\nu$ the \emph{variance measure} of the L\'evy measure $\nu$. 
Note that $\mu_\nu(\cdot)$ is finite for bounded Borel sets. 

\begin{defn} \label{def: Ahlfors regular}
We say that the variance measure $\mu_\nu$ of the L\'evy measure $\nu$ is \emph{locally (upper) Ahlfors regular near the origin} if the following holds.
There exists $\epsilon_\nu \in (0,1/2)$ 
only depending on $\nu$ such that  
the restriction of $\mu_\nu$ to $[-\epsilon_\nu, \epsilon_\nu]$ is upper Ahlfors regular:
there exist constants 
$\alpha_\nu, c_\nu \in (0,\infty)$ 
and $\rho_\nu \in (0,1)$ 
only depending on $\nu$ such that 
for any $x \in [-\epsilon_\nu, \epsilon_\nu]$ and 
for any $\rho < \rho_\nu$, we have
\begin{align} \label{eq: Ahlfors regularity}
\mu_\nu ((x-\rho,x+\rho) \cap [-\epsilon_\nu, \epsilon_\nu]) 
\; = \; \int_{x-\rho}^{x+\rho} \jump^2 \, \one_{[-\epsilon_\nu, \epsilon_\nu]}(\jump) \, \nu(\ud \jump)
\; \leq \; c_\nu \, \rho^{\alpha_\nu} .
\end{align}
Note that this implies that $\mu_\nu$ is dominated by the Lebesgue measure near the origin 
\textnormal{(}in particular, $\nu$ does not have atoms accumulating to the origin, but it may have atoms elsewhere\textnormal{)}.
\end{defn}

We define
\begin{align} \label{eq: lambda max}
\maxjumpT{\kappa} := 
\begin{cases}
\frac{7}{128} > 0 , & \kappa = 0 , \\[.5em]
\dfrac{\kappa}{2^{\frac{4}{\kappa} + \frac{3}{2}}} \, \dfrac{8 - \kappa}{3 \kappa +8} > 0 , & \kappa \in (0,8)  , \\[.8em]
\frac{1}{2}(\kappa-8) > 0 , & \kappa > 8 .
\end{cases}
\end{align}

\begin{restatable}{thm}{TraceExistsThm} 
\label{thm: LLE micro jumps curve}
Fix $T > 0$, $\kappa \in [0,\infty) \setminus\{8\}$, and a L\'evy measure $\nu$. Suppose that either~\ref{item: ass1} or~\ref{item: ass2} holds, and fix $\varepsilon > 0$ 
such that $\lambda_\varepsilon < \maxjumpT{\kappa}$ as in~\eqref{eq: lambda max} and $\varepsilon \in (0,\epsilon_\nu \wedge \tfrac{1}{2} \rho_\nu]$ under~\ref{item: ass2}.  
Then, the Loewner chain driven by $\microDriver_\varepsilon^{\kappa}$~\eqref{eq: BM with micro again} 
is almost surely generated by a c\`adl\`ag curve on $[0, T]$.
\end{restatable}

In order to establish the existence of the Loewner trace, we will derive 
an estimate for the derivative of the inverse map $f_t$ near the driving point $\microDriver_\varepsilon^{\kappa}(t)$ uniformly in time (Propositions~\ref{prop: needed uniform bound for f' 1}~\&~\ref{prop: needed uniform bound for f' 2}).
The existence of the Loewner trace then follows immediately from standard Loewner theory.

\begin{proof}
This is a consequence of Proposition~\ref{prop: sufficient condition curve}: the required bound~\eqref{eq: needed uniform bound for f'} is 
given by~(\ref{eq: needed uniform bound for f' holds 1},~\ref{eq: needed uniform bound for f' holds 2}). 
\end{proof}

The remainder of this section is devoted to proving 
the estimates~(\ref{eq: needed uniform bound for f' holds 1},~\ref{eq: needed uniform bound for f' holds 2}) in Propositions~\ref{prop: needed uniform bound for f' 1}~\&~\ref{prop: needed uniform bound for f' 2} under the two conditions~\ref{item: ass1} and~\ref{item: ass2}, respectively. 
While the idea is similar to the strategy used already in Section~\ref{sec: backward bounds}, in the present case deriving the estimates is significantly harder. 
First of all, \emph{pointwise in time} estimates for the inverse map $f_t:= g_t^{-1}$ 
obtained by techniques from Section~\ref{sec: backward bounds} with the mirror backward Loewner flow seem not sufficient, since there is no guarantee for a general L\'evy process to satisfy 
a c\`adl\`ag modulus of continuity analogous to that of Brownian motion, rendering the usage of estimates only at countably many time instants insufficient\footnote{In contrast to the case of Brownian drivers~\cite{Rohde-Schramm:Basic_properties_of_SLE}.}. 
We hence work directly with the forward Loewner flow and derive estimates \emph{uniformly in time}.
To this end, we shall employ a the discrete grid approximation discussed in Section~\ref{subsec: grid preli}~\cite{MSY:On_Loewner_chains_driven_by_semimartingales_and_complex_Bessel-type_SDEs, Yuan:Refined_regularity_of_SLE}.

Even with the new methodology for the forward Loewner flow, regarding a general L\'evy process, 
the estimates seem to only go through in full generality in the special case 
where the diffusivity parameter $\kappa > 8$ (\ref{item: ass1}) --- 
which for $\SLE_\kappa$ curves corresponds to the space-filling region. 
We believe that the strong diffusion smoothens the frontier of the Loewner hull in the space-filling phase 
also in the case of arbitrary jumps being present, 
which roughly speaking implies that the derivative of the Loewner chain stays in good control 
with very high probability.
In contrast, when the diffusivity parameter $\kappa < 8$, our arguments rely on the additional local upper Ahlfors regularity condition (\ref{item: ass2}), which covers most examples of L\'evy measures, but which --- we believe --- could potentially be relaxed.

\subsection{Main derivative estimate}

Our aim now is to prove the estimate~\eqref{eq: needed uniform bound for f'} appearing as the key input in Proposition~\ref{prop: sufficient condition curve} for the derivative of $f_t$.
This implies the existence of the Loewner trace.

Under~\ref{item: ass1}, we define (see also Lemma~\ref{lem: alphaprimetilde and thetaprimetilde} and Figure~\ref{fig: Alphaprime and Thetaprime 3})
\begin{align} \label{eq: Thetatilde forward case 1}
\ThetaTmax{\kappa}{\lambda_\varepsilon} := \frac{2 (\kappa - 8) (\kappa - 2 (\lambda_\varepsilon + 4))}{(\kappa + 8) (3 \kappa + 8)} 
 \; \in \; (0,2/3) , 
\qquad \textnormal{ when } \; \kappa > 8 \; \textnormal{ and } \; \lambda_\varepsilon < \maxjumpT{\kappa} .
\end{align}

\begin{prop} 
\label{prop: needed uniform bound for f' 1}
Fix $T > 0$, $\kappa \in (8,\infty)$, a L\'evy measure $\nu$, and $\varepsilon > 0$ such that 
$\lambda_\varepsilon < \maxjumpT{\kappa}$ as in~\eqref{eq: lambda max}. 
Let $(f_t)_{t \geq 0} = (g_t^{-1})_{t \geq 0}$ be the inverse Loewner chain, where $(g_t)_{t \geq 0}$ is the solution to~\eqref{eq: LE} driven by $\smash{\microDriver_\varepsilon^{\kappa}}$~\eqref{eq: BM with micro again again}. 
Then, for any $0 < \theta < \ThetaTmax{\kappa}{\lambda_\varepsilon}$ as in~\eqref{eq: Thetatilde forward case 1}, 
there exist almost surely finite random constants $\smash{C_{\lambda_\varepsilon}^{\kappa}}(\theta, T)$ 
such that 
\begin{align} \label{eq: needed uniform bound for f' holds 1}
\max_{t \in [0,T]}
|f_t'(\microDriver_\varepsilon^{\kappa}(t) + \ii 2^{-n})| \leq 
C_{\lambda_\varepsilon}^{\kappa}(\theta, T) \, 2^{n (1-\theta)}  \qquad 
\textnormal{for all } n \in \bZpos . 
\end{align}
\end{prop}

Under~\ref{item: ass2}, we define (see also Lemmas~\ref{lem: alphaprimetilde and thetaprimetilde again}~\&~\ref{lem: alphaprimetilde and thetaprimetilde again kappa=0} 
and Figures~\ref{fig: Alphaprime and Thetaprime 4}~\&~\ref{fig: Alphaprime and Thetaprime 5})
\begin{align} \label{eq: betamax forward 2}
\maxbetaT{\kappa}{\lambda_\varepsilon} := \; & 
\begin{cases}
14 , & \kappa = 0 , \\[.3em]
\dfrac{(8 - \kappa)^2}{2 \kappa^2 + 2^{\frac{4}{\kappa} +\frac{5}{2}} (3 \kappa +8) \lambda_\varepsilon} 
\Big(1 - \dfrac{\lambda_\varepsilon}{\maxjumpT{\kappa}}
\Big) \; > \; 0 , & \kappa \in (0,8) ,
\end{cases}
\qquad \textnormal{ when } \; \lambda_\varepsilon < \maxjumpT{\kappa} ,
\end{align} 
and 
\begin{align}
\label{eq: Thetatilde forward case 2}
\ThetaTmaxbeta{\alpha}{\kappa} := \; & 
\begin{cases}
\dfrac{8 \, \alpha}{15 (\alpha + 2)}
\; \in \; (0,8/15) , & \kappa = 0 , \\[.5em]
\Big( \dfrac{32 \, \alpha \,  (8 - \kappa ) }{( \kappa + 48 + 64/\kappa ) (2 \alpha \kappa + 8 - \kappa)} \Big) \wedge 1 
\; \in \; (0,1] , & \kappa \in (0,8) ,
\end{cases}
\qquad \textnormal{ when } \; \alpha < \maxbetaT{\kappa}{\lambda_\varepsilon} .
\end{align}

\begin{prop} 
\label{prop: needed uniform bound for f' 2}
Fix $T > 0$, $\kappa \in [0,8)$, 
and a L\'evy measure $\nu$ 
whose variance measure $\mu_\nu$ satisfies the local upper $\alpha_\nu$-Ahlfors regularity~\eqref{eq: Ahlfors regularity}. 
Fix $\varepsilon > 0$ such that $\lambda_\varepsilon < \maxjumpT{\kappa}$ as in~\eqref{eq: lambda max}. 
Let $(f_t)_{t \geq 0} = (g_t^{-1})_{t \geq 0}$ be the inverse Loewner chain, where $(g_t)_{t \geq 0}$ is the solution to~\eqref{eq: LE} driven by $\smash{\microDriver_\varepsilon^{\kappa}}$~\eqref{eq: BM with micro again again}. 
Then, for any $\alpha \in (0,\maxbetaT{\kappa}{\lambda_\varepsilon} \wedge \alpha_\nu)$ as in~\eqref{eq: betamax forward 2} and 
for any 
$0 < \theta < \ThetaTmaxbeta{\alpha}{\kappa}$ as in~\eqref{eq: Thetatilde forward case 2}, 
there exist almost surely finite random constants $\smash{C_{\lambda_\varepsilon}^{\kappa}}(\theta, T)$ 
such that 
\begin{align} \label{eq: needed uniform bound for f' holds 2}
\max_{t \in [0,T]}
|f_t'(\microDriver_\varepsilon^{\kappa}(t) + \ii 2^{-n})| \leq 
C_{\lambda_\varepsilon}^{\kappa}(\theta, T) \, 2^{n (1-\theta)}  \qquad 
\textnormal{for all } n \in \bZpos . 
\end{align}
\end{prop}

\begin{proof}[Proof of Propositions~\ref{prop: needed uniform bound for f' 1} and~\ref{prop: needed uniform bound for f' 2}]
First of all, Lemma~\ref{lem: grid} with $u = 2^{n (1-\theta)}$ and $\delta = 2^{-n}$ shows that 
\begin{align*}
\PR \Big[ \max_{t \in [0,T]} |f_t'(\microDriver_\varepsilon^{\kappa}(t) + \ii 2^{-n})| \geq 2^{n (1-\theta)} \Big] 
\; \leq \; \PR \bigg[ \bigcup_{z_0 \in \Grid_\varepsilon^{\kappa} } E_n^\theta(z_0) \bigg] ,
\end{align*}
where the union is taken over the grid with mesh size $\tfrac{1}{8} \, 2^{-n \theta}$ (Definition~\ref{def: grid}), 
\begin{align*}
\Grid_\varepsilon^{\kappa}
= \Grid_\varepsilon^{\kappa} (2^{-n \theta}, T, R_\varepsilon^{\kappa}(T)) 
:= \Big\{ z \in \bH \;\; | \;\; 
\re(z) = & \; \tfrac{1}{8} \, 2^{-n \theta} \, \ell \, \in \, [-R_\varepsilon^{\kappa}(T), R_\varepsilon^{\kappa}(T)] 
, \ell \in \bZ ,\textnormal{ and }
\\ 
\im(z) = & \; \tfrac{1}{8} \, 2^{-n \theta} \, (k + 8) \, \in \, [2^{-n \theta} , \sqrt{1+4T}] , \; 
k \in \bZnn \Big\}  ,
\end{align*}
and where the events of interest are
\begin{align} 
\label{eq: event of interest}
E_n^\theta(z_0)
= E_n^\theta(z_0, T) \; := \; \big\{ & \; \textnormal{there exists }  t \in [0,T] 
\textnormal{ such that } z_0 \in \bH \setminus K_t \textnormal{ and } 
 \\
 \nonumber
& \; | g_t(z_0) - \microDriver_\varepsilon^{\kappa}(t) - \ii 2^{-n} | \leq 2^{-n-1}
\textnormal{ and } 
| g_t'(z_0) | \leq \tfrac{80}{27} \, 2^{-n (1-\theta)} \big\} ,
\qquad n \in \bN .
\end{align}
Note that the grid 
$\smash{\Grid_\varepsilon^{\kappa} (2^{-n \theta}, T, R_\varepsilon^{\kappa}(T))}$ 
is random, 
and its width $\smash{2 R_\varepsilon^{\kappa}(T)}$ 
depends on the Loewner driving function 
$\smash{\microDriver_\varepsilon^{\kappa}}$ via~\eqref{eq: BM with micro supremum}.   
However, since the driving function is c\`adl\`ag, 
$\smash{R_\varepsilon^{\kappa}(T)}$ is almost surely finite. 
Hence, taking a cutoff $R \in (0,\infty)$, we may restrict to the event $\{ R_\varepsilon^{\kappa}(T) \leq R \}$. 
We write 
\begin{align*}
\PR_{R}[ \; \cdot \; ] := \PR[ \; \cdot \; \cap \{ R_\varepsilon^{\kappa}(T) \leq R \} ]
\end{align*}
for the probability measure $\PR$ restricted to the cutoff event.
On this event, we can use the Borel-Cantelli lemma to deduce that if 
\begin{align*}
\sum_{n=1}^\infty \PR_{R} \Big[ \max_{t \in [0,T]} |f_t'(\microDriver_\varepsilon^{\kappa}(t) + \ii 2^{-n})| \geq 2^{n (1-\theta)} \Big] 
\; \leq \; 
\sum_{n=1}^\infty \sum_{z_0 \in \Grid_\varepsilon^{\kappa} (2^{-n \theta}, T, R) }
\PR_{R} [ E_n^\theta(z_0) ] 
\; < \; \infty ,
\end{align*}
then on the event $\{ R_\varepsilon^{\kappa}(T) \leq R \}$, 
there exists an almost surely finite random integer $\smash{N_{\lambda_\varepsilon}^{\kappa}(\theta, T, R)}$ such that 
\begin{align*}
\max_{t \in [0,T]}
|f_t'(\microDriver_\varepsilon^{\kappa}(t) + \ii 2^{-n})| \leq 2^{n (1-\theta)} 
\qquad \textnormal{ for all } n \geq N_{\lambda_\varepsilon}^{\kappa}(\theta, T, R) 
\quad \textnormal{on the event $\{ R_\varepsilon^{\kappa}(T) \leq R \}$} ,
\end{align*}
and thus, by taking the union bound over $R \in \bN$, we see that there exists an almost surely finite random integer $\smash{N_{\lambda_\varepsilon}^{\kappa}(\theta, T)}$ such that 
\begin{align*}
\max_{t \in [0,T]}
|f_t'(\microDriver_\varepsilon^{\kappa}(t) + \ii 2^{-n})| \leq 2^{n (1-\theta)} 
\qquad \textnormal{ for all } n \geq N_{\lambda_\varepsilon}^{\kappa}(\theta, T) .
\end{align*} 
Hence, we then conclude that almost surely, we have $\smash{\underset{t \in [0,T]}{\max} |f_t'(\microDriver_\varepsilon^{\kappa}(t) + \ii 2^{-n})| \leq 2^{n (1-\theta)}}$ 
except for possibly finitely many $n \in \bZnn$, so 
\begin{align*}
C_{\lambda_\varepsilon}^{\kappa}(\theta, T) 
:= \sup_{n \in \bZnn}
2^{-n (1-\theta)}
\max_{t \in [0,T]} |f_t'(\microDriver_\varepsilon^{\kappa}(t) + \ii 2^{-n})| 
\end{align*}
gives the sought almost surely finite random constant for either Proposition~\ref{prop: needed uniform bound for f' 1} and~\ref{prop: needed uniform bound for f' 2}. 
It thus remains to verify that the summability of the probabilities of the events~\eqref{eq: event of interest} of interest holds with cutoff $R$:
\begin{align} \label{eq: summability of event of interest}
\sum_{n=1}^\infty \sum_{z_0 \in \Grid_\varepsilon^{\kappa} (2^{-n \theta}, T, R) }
\PR_{R} [ E_n^\theta(z_0) ] 
\; < \; \infty .
\end{align}
This is the content of 
Propositions~\ref{prop: summability of event of interest 1},~\ref{prop: summability of event of interest 2}, and~\ref{prop: summability of event of interest 3},  
whose proofs comprise the rest of this section. 
\end{proof}

\subsection{Forward Loewner flow and estimates}

Fix a starting point $z_0 = x_0 + \ii y_0 \in \bH$ implicitly throughout, and consider the following processes:
\begin{align*}
Z(t) = Z_\varepsilon^{\kappa}(t,z_0) := \; & g_t(z_0) - \microDriver_\varepsilon^{\kappa}(t) =: X(t) + \ii Y(t) , \\
X(t) = X_\varepsilon^{\kappa}(t,z_0) := \; & \re \big( Z_\varepsilon^{\kappa}(t,z_0) \big) , \\
Y(t) = Y_\varepsilon^{\kappa}(t,z_0) := \; & \im \big( Z_\varepsilon^{\kappa}(t,z_0) \big) , 
\end{align*}
and up to the blow-up time~\eqref{eq: blow-up time}, define
\begin{align}
\label{eq: general forward martingale candidate}
M(t) = M_{p,q,r}(t,z_0) := \; & |g_t'(z_0)|^p \,Y(t)^{q} \, ( \sin \arg Z(t) )^{-2r} , \qquad 
p, q \in \bR, \textnormal{ and } r \leq 1 . 
\end{align}
The importance of the process~\eqref{eq: general forward martingale candidate} becomes clear in Lemma~\ref{lem: forward bound}.

\bigskip

Note that for $z_0 \in \Grid_\varepsilon^{\kappa} (2^{-n \theta}, T, R)$, on the event $E_n^\theta(z_0)$, we have 
$| Z(t) - \ii 2^{-n} | \leq 2^{-n-1}$ for some $t \in [0,T]$.
The next simple observation gives a bound for the absolute value of the real part of $Z$ in terms of its imaginary part. The latter behaves well under a time-change, which we will utilize in the analysis.  

\begin{lem} \label{lem: simple observation}
Fix $\delta \in (0,1)$. 
If $| Z(t) - \ii \delta | \leq \delta/2$, then 
\begin{align*}
Y(t) \in [ \delta/2 , 3 \delta / 2]
\qquad \textnormal{and} \qquad 
| X(t) | \leq Y(t) .
\end{align*}
\end{lem}

\begin{proof}
Indeed, we have $| Y(t) - \delta | \leq | Z(t) - \ii \delta | \leq \delta/2$, which shows the bounds for $Y(t)$. From the lower bound, 
it also follows easily that $| X(t) | \leq \delta/2 \leq Y(t)$.
\end{proof}

In particular, on the event $E_n^\theta(z_0)$ we also have
\begin{align*}
| X(t) | \leq Y(t) \in [ 2^{-n-1} , 3 \cdot 2^{-n-1} ] 
\qquad \textnormal{for some } t \in [0,T] .
\end{align*}

Because the event $E_n^\theta(z_0)$ involves  
three conditions that should hold at some common time $t \in [0,T]$, it will be useful to consider a particular 
stopping time (see Equation~\eqref{eq: forward bound}). 
This will be more manageable if we consider the time-changed processes 
\begin{align*}
\hat{X}(s) := X(\sigma(s))  , \qquad 
\hat{Y}(s) := Y(\sigma(s))  , \qquad 
\hat{Z}(s) := Z(\sigma(s))  , \qquad
\hat{g}_s := g_{\sigma(s)} , \qquad
\hat{M}(s) := M(\sigma(s)) , 
\end{align*}
where $\sigma$ is given by
\begin{align} \label{eq: forward time change}
\sigma(s) = S^{-1}(s) , \qquad 
S(t) = \int_0^t \frac{\ud u}{X(u-)^2 + Y(u-)^2} 
= \int_0^t \frac{\ud u}{|Z(u-)|^2} .
\end{align}
The imaginary part of~\eqref{eq: LE} gives $Y(0) = y_0$ and 
\begin{align} \label{eq: forward Y}
Y(t) 
\; = \; \; & y_0 \, \exp \bigg( - \int_0^t \frac{2}{|Z(s-)|^2} \, \ud s \bigg) 
\; = \; y_0 \; - \; 2 \int_0^t \frac{Y(s-)}{|Z(s-)|^2} \, \ud s ,
\end{align}
which implies in particular that 
\begin{align*} 
\hat{Y}(t) = y_0 \, e^{-2t} > 0
\qquad \textnormal{for all } t \in [0, S(\tau(z_0))) .
\end{align*} 
Also,~\eqref{eq: forward time change} can be written as
$\smash{S(t) = -\frac{1}{2} \log \big( \frac{Y(t)}{y_0} \big)}$, 
so we see that $\underset{t \to \infty}{\lim} \sigma(t) = \infty$ almost surely.

Now, note that the event of interest~\eqref{eq: event of interest} is (by Lemma~\ref{lem: simple observation})
\begin{align*} 
E_n^\theta(z_0) \; = \; \big\{ & \; \textnormal{there exists } t \in [0,T] 
\textnormal{ such that } z_0 \in \bH \setminus K_t \textnormal{ and } 
\\
& \; | Z(t) - \ii 2^{-n} | \leq 2^{-n-1} 
\textnormal{ and } 
| g_t'(z_0) | \leq \tfrac{80}{27} \, 2^{-n (1-\theta)} \big\} 
\\
\; \subset \; \big\{ & \; \textnormal{there exists } t \in [0,T] 
\textnormal{ such that } 
\\
& \; | X(t) | \leq Y(t) \in [ 2^{-n-1} , 3 \cdot 2^{-n-1} ] 
\textnormal{ and } 
| g_t'(z_0) | \leq \tfrac{80}{27} \, 2^{-n (1-\theta)} \big\} .
\end{align*}
Because both conditions must hold for the same time instant $t \in [0,T]$, it will be useful to consider the derivative 
$|\hat{g}_{S_{n}}'(z_0)|$ at the following (bounded) stopping time
involving the time-changed processes: 
\begin{align} \label{eq: bounded stopping time taun}
\begin{split}
S_{n} = S_{n}(y_0) := \; & \inf \big\{ s \geq 0 
\; \big| \; |\hat{X}(s)| \leq \hat{Y}(s) = y_0 \, e^{-2s} 
\in [ 2^{-n-1} , 3 \cdot 2^{-n-1} ] \big\} 
\, \wedge \, \log \sqrt{\tfrac{y_0}{2^{-n-1}}} \\
\in \; & \big[ \log \sqrt{\tfrac{y_0}{3 \cdot 2^{-n-1}}} , \; 
\log \sqrt{\tfrac{y_0}{2^{-n-1}} } \big] 
\end{split}
\end{align}
The probabilities of the events $E_n^\theta(z_0)$ defined in~\eqref{eq: event of interest} can now be bounded using the time-changed process~(\ref{eq: general forward martingale candidate},~\ref{eq: forward time change}) at the stopping time~\eqref{eq: bounded stopping time taun}.

\begin{lem} \label{lem: forward bound}
Fix $T > 0$, $\kappa \geq 0$, a L\'evy measure $\nu$, and $\varepsilon > 0$. 
Fix also $p < 0$, $q \in \bR$, and $r \leq 1$.
Then, for any $\theta \in (0,1)$, for any $n \in \bZpos$, for any $R > 0$, and for any $z_0 = x_0 + \ii y_0 \in \Grid_\varepsilon^{\kappa} (2^{-n \theta}, T, R)$, 
\begin{align} \label{eq: forward bound}
\PR [ E_n^\theta(z_0) ] 
\leq \; & c_0 \, 2^{n \beta} \, 
\EX \big[ \hat{M}(S_{n}) \; \one{\{|\hat{X}(S_{n})| \leq y_0 \, e^{-2S_{n}}\}} \big] ,
\end{align}
where $c_0 = c_0(p, q, r) \in (0,\infty)$ is a constant, 
$\beta = (1-\theta) p + q$.
\end{lem}

\begin{proof}
Fix $z_0 = x_0 + \ii y_0 \in \Grid_\varepsilon^{\kappa} (2^{-n \theta}, T, R)$. 
Using~(\ref{eq: LE},~\ref{eq: forward Y}) and the time-change~\eqref{eq: forward time change}, we find
\begin{align*}
\partial_s^{+} \log |\hat{g}_s'(z_0)| 
= - 2 \, \frac{\hat{X}(s-)^2 - \hat{Y}(s-)^2}{\hat{X}(s-)^2 + \hat{Y}(s-)^2} 
\quad \in \quad [ -2, 2 ],
\end{align*}
which implies in particular that 
\begin{align} \label{eq: derivatives comparable}
e^{-2 u} \leq \frac{|\hat{g}_{s+u}'(z_0)| }{|\hat{g}_{s}'(z_0)| } \leq e^{2 u} .
\end{align}
At the stopping time~\eqref{eq: bounded stopping time taun}, 
we find
\begin{align*}
\frac{|\hat{g}_{S_{n}}'(z_0)| }{|\hat{g}_{s}'(z_0)| } \leq 3 \qquad 
\textnormal{for all } s \in 
\big[ \log \sqrt{\tfrac{y_0}{3 \cdot 2^{-n-1}}} , 
\log \sqrt{\tfrac{y_0}{2^{-n-1}} } \big]  .
\end{align*}
Using this, we obtain from Markov's inequality the bounds
\begin{align*} 
\PR [ E_n^\theta(z_0) ] 
\; \leq \; \; & 
\PR \big[ | \hat{g}_{S_{n}}'(z_0) | \leq \tfrac{80}{9} \, 2^{-n (1-\theta)} 
\textnormal{ and } |\hat{X}(S_{n})| \leq y_0 \, e^{-2S_{n}} \big] \\
\; \leq \; \; &
\big( \tfrac{80}{9} \big)^{-p} \, 2^{np (1-\theta)} 
\; \EX \big[ | \hat{g}_{S_{n}}'(z_0) |^{p} \; \one{ \{|\hat{X}(S_{n})| \leq y_0 \, e^{-2S_{n}}\}} \big] , \qquad p < 0 .
\end{align*} 
Using the definition~\eqref{eq: general forward martingale candidate} of the process $M$, and the fact that 
$(\sin \arg \hat{Z}(S_{n}))^{2r} \leq 2^{|r|}$ for all $r \leq 1$ on the event $\{|\hat{X}(S_{n})| \leq \hat{Y}(S_{n}) = y_0 \, e^{-2S_{n}}\}$, we obtain
\begin{align*}
\; & \EX \big[ | \hat{g}_{S_{n}}'(z_0) |^{p} \; \one{\{|\hat{X}(S_{n})| \leq y_0 \, e^{-2S_{n}}\}} \big] \\
\; = \; \; & \EX \big[ \hat{Y}(S_{n})^{-q} \, \big( \sin \arg \hat{Z}(S_{n}) \big)^{2r} \hat{M}(S_{n}) \; \one{\{|\hat{X}(S_{n})| \leq y_0 \, e^{-2S_{n}}\}} \big] \\
\; \leq \; \; & 2^{|r|} \, y_0^{-q} \; \EX \big[ e^{2 S_{n} q} \, \hat{M}(S_{n}) \; \one{\{|\hat{X}(S_{n})| \leq y_0 \, e^{-2S_{n}}\}} \big] , \qquad r \leq 1 .
\end{align*}
We can bound this as follows:
\begin{itemize}
\item If $q \leq 0$, then because $S_{n} \geq  \smash{\log \sqrt{\tfrac{y_0}{3 \cdot 2^{-n-1}}}}$, we see that
$y_0^{-q} e^{2 S_{n} q} \leq 3^{-q} \cdot 2^{q} \cdot 2^{n q}$. 

\medskip

\item If $q \geq 0$, then because $S_{n} \leq \smash{\log \sqrt{\tfrac{y_0}{2^{-n-1}}}}$, we see that
$y_0^{-q} e^{2 S_{n} q} \leq 2^{q} \cdot 2^{n q}$. 
\end{itemize}
Collecting these estimates together, we find the asserted bound~\eqref{eq: forward bound}. 
\end{proof}

Thus, our next aim will be to control the expected value 
of the time-changed process $\hat{M}(S_{n})$ 
at the stopping time $S_{n}$. 
First, we derive an SDE for $M$ 
in Lemma~\ref{lem: SDE forward M}, 
and we then estimate the drift part in Propositions~\ref{prop: M forward supermgle case 1}~\&~\ref{prop: M forward supermgle case 2}, 
which use~\ref{item: ass1} and~\ref{item: ass2}, respectively.
We then use these results in Sections~\ref{subsec: Summability 1}--\ref{subsec: Summability 3} to conclude with~sought 
summability~\eqref{eq: summability of event of interest} for the proof of Propositions~\ref{prop: needed uniform bound for f' 1} and~\ref{prop: needed uniform bound for f' 2}.

\subsection{Supermartingale bounds}
\label{subsec: Supermartingale bounds}

We begin with an analogue of Lemma~\ref{lem: SDE M}.

\begin{lem} \label{lem: SDE forward M} 
Fix $\kappa \geq 0$, a L\'evy measure $\nu$, and $\varepsilon > 0$. 
Let $(g_t)_{t \geq 0}$ be the solution to~\eqref{eq: LE} driven by 
$\smash{\microDriver_\varepsilon^{\kappa}}$~\eqref{eq: BM with micro again again},
fix $z_0 \in \bH$, and consider the process 
$M$ defined in~\eqref{eq: general forward martingale candidate}.
Then, we have
\begin{align*}
M(t) 
\; = \; \; & \mgle(t) 
\; + \; \int_0^t M(s-) \; \bigg( D_{r}(s)
- 2 p \, \frac{X(s-)^2 - Y(s-)^2}{|Z(s-)|^4} 
- \frac{2 q}{|Z(s-)|^2} \bigg) \, \ud s  , 
\qquad t \in [0,\tau(z_0)) , 
\\[1em]
\textnormal{where} \qquad 
D_{r}(s)
\; = \; \; & 
\frac{ r (\kappa (2 r - 1) + 8) X(s-)^2 + r \kappa Y(s-)^2}{|Z(s-)|^4} \\ 
\; \; & 
+ \; \int_{|\jump| \leq \varepsilon}   \bigg( \bigg| \frac{Z(s-) - \jump}{Z(s-)} \bigg|^{2r} - 1 + \frac{2r \jump X(s-)}{|Z(s-)|^2} \bigg) \nu(\ud \jump) , \qquad 
p, q \in \bR, \textnormal{ and } r \leq 1 ,
\end{align*}
and where $\mgle$ is the right-continuous local martingale 
\begin{align} \label{eq: forward supermgle N}
\begin{split}
\mgle(t) 
\; := \; \; & 
M(0) \; - \; 2 r \sqrt{\kappa} \int_0^t M(s-) \; \frac{X(s-)}{|Z(s-)|^2} \, \ud B(s) \\
\; \; & 
+ \;  \int_0^t M(s-) \int_{|\jump| \leq \varepsilon} \bigg( \bigg| \frac{Z(s-) - \jump}{Z(s-)} \bigg|^{2r} - 1 \bigg) \, \PoissonComp(\ud s, \ud \jump) , \qquad t \in [0,\tau(z_0)) .
\end{split}
\end{align}
\end{lem}

\begin{proof}
By~\eqref{eq: LE} and a straightforward application of It\^o's formula, we have
\begin{align*}
|g_t'(z_0)|^p 
\; = \; \; & 1 \; - \; 2 p \int_0^t |g_{s}'(z_0)|^{p} \; \frac{X(s-)^2 - Y(s-)^2}{|Z(s-)|^4} \, \ud s , \\
Y(t)^q
\; = \; \; & y_0 \; - \; 2 q \int_0^t  \frac{Y(s-)^q}{|Z(s-)|^2} \, \ud s .
\end{align*}
A more involved application of It\^o's formula (see Lemma~\ref{lem: SDE forward sin(arg(Z)) power} in Appendix~\ref{app: Ito calculus} with $a=0$) gives
\begin{align*}
( \sin \arg Z(t) )^{-2r} 
\; = \; \; & ( \sin \arg z_0 )^{-2r} 
\; - \; 2 r \sqrt{\kappa} \int_0^t ( \sin \arg Z(s-) )^{-2r} \; \frac{X(s-)}{|Z(s-)|^2} \, \ud B(s) \\
\; \; & 
+ \; \int_0^t \int_{|\jump| \leq \varepsilon} ( \sin \arg Z(s-)  )^{-2r} \; \bigg( \bigg| \frac{Z(s-) - \jump}{Z(s-)} \bigg|^{2r} - 1 \bigg) \, \PoissonComp(\ud s, \ud \jump) \\
\; \; & 
+ \; \int_0^t( \sin \arg Z(s-) )^{-2r}
\; D_{r}(s) \, \ud s .
\end{align*} 
Combining these, we obtain the asserted identity for $M$. 
\end{proof}

Following the same strategy as in Section~\ref{subsec: derivative local martingale}, 
we aim to bound the drift term 
\begin{align*}
M(t) - \mgle(t) = \; & 
\int_0^t \frac{M(s-)}{|Z(s-)|^4} \; 
\big( r (\kappa (2 r - 1) + 8) - 2 p - 2 q \big) X(s-)^2 \, \ud s \\
\; & \; + \; 
\int_0^t \frac{M(s-)}{|Z(s-)|^4} \; 
\big( r \kappa + 2 p - 2 q \big) Y(s-)^2 \, \ud s \\
\; & \; + \; \int_0^t M(s-) \; 
\int_{|\jump| \leq \varepsilon} F_{r}(\jump; X(s-), Y(s-)) \, \nu(\ud \jump) \, \ud s ,
\end{align*}
where for each $r \in (-\infty,1]$, the key will be to estimate the function
\begin{align} \label{eq: aux function drift}
F_{r}(\jump) = F_{r}(\jump; x, y) 
\; = \; \; & 
\bigg( \frac{(x-\jump)^2 + y^2}{x^2 + y^2} \bigg)^{r} - 1 + \frac{2 r \jump x}{x^2 + y^2} , \qquad \jump \in \bR , \; x \in \bR , \; y > 0 .
\end{align}

The following bound, Proposition~\ref{prop: M forward supermgle case 1}, 
which holds when $r \in (0, 1]$, will be sufficient to conclude the existence of the Loewner trace in the case where $\kappa > 8$ (\ref{item: ass1}).
As before, we denote the variance of the jumps of the random driving function~\eqref{eq: BM with micro again again} as in~\eqref{eq: jump variance}, 
\begin{align} \label{eq: jump variance again} 
\lambda_\varepsilon := \int_{|\jump| \leq \varepsilon} \jump^2 \, \nu(\ud \jump) \; \geq 0 .
\end{align}

\begin{prop} \label{prop: M forward supermgle case 1}
Fix $\kappa \geq 0$, a L\'evy measure $\nu$, and $\varepsilon > 0$. 
Fix $r \in (0, 1]$ and set 
\begin{align} \label{eq: p and q forward alternative}
p = p(\kappa, r ) 
:= \tfrac{1}{2} r (4 - \kappa + \kappa r ) 
\qquad \textnormal{and}  \qquad
q = q(\kappa, \lambda_\varepsilon, r ) 
:= p(\kappa, r ) + \tfrac{1}{2} r (\kappa + \lambda_\varepsilon) .
\end{align}
Let $(g_t)_{t \geq 0}$ be the solution to~\eqref{eq: LE} driven by 
$\smash{\microDriver_\varepsilon^{\kappa}}$~\eqref{eq: BM with micro again again}, 
fix $z_0 \in \bH$, and consider the processes 
$M$ defined in~\eqref{eq: general forward martingale candidate}
and $\mgle$ defined in~\eqref{eq: forward supermgle N}. 
Then, we have 
\begin{align} \label{eq: M forward supermgle case 1}
M(0) = \mgle(0) \qquad \textnormal{and} \qquad
M(t) \leq \mgle(t) \; \textnormal{ for all } t \in [0,\tau(z_0)).
\end{align}
\end{prop}

With parameters~\eqref{eq: p and q forward alternative}, 
the process $\mgle$ in~\eqref{eq: forward supermgle N} 
is a non-negative local martingale, thus a supermartingale~\cite[Lemma~5.6.8]{Cohen-Elliott:Stochastic_calculus_and_applications}.
Hence, Proposition~\ref{prop: M forward supermgle case 1} shows by the supermartingale property that
\begin{align*}
\EX[M(t)] \; \leq \; \EX[\mgle(t)] \; \leq \; \mgle(0) \; = \; M(0) \quad \textnormal{ for all } t \geq 0 . 
\end{align*}

\begin{proof}
The proof is very similar to that of Proposition~\ref{prop: M supermgle}.
Using the fact that $(1+x)^r \leq 1 + rx$ for all $x \geq -1$ and $r \in (0,1]$, we see that~\eqref{eq: aux function drift} can be bounded as
\begin{align*}
F_{r}(\jump) 
\; = \; \; & 
\bigg( 1 + \frac{-2 \jump x + \jump^2}{x^2 + y^2} \bigg)^r - 1 + \frac{2r \jump x}{x^2 + y^2} 
\; \leq \; \frac{r \jump^2 }{x^2 + y^2} .
\end{align*}
Therefore, using the definition~\eqref{eq: jump variance again} of the variance $\lambda_\varepsilon$, with $x = X(s-)$ and $y = Y(s-)$ we obtain
\begin{align*}
\int_{|\jump| \leq \varepsilon} F_{r}(\jump; X(s-), Y(s-)) \, \nu(\ud \jump) 
\leq \; & \frac{r \lambda_\varepsilon }{|Z(s-)|^2} .
\end{align*}
By Lemma~\ref{lem: SDE forward M}, we thus obtain
\begin{align*}
M(t) 
\; \leq \; \; & \mgle(t) \; + \; 
\int_0^t \frac{M(s-)}{|Z(s-)|^4} \; 
\big( r (\kappa (2 r - 1) + 8) - 2 p - 2 q + r \lambda_\varepsilon \big) X(s-)^2 \, \ud s \\
\; & \; + \; 
\int_0^t \frac{M(s-)}{|Z(s-)|^4} \; 
\big( r \kappa + 2 p - 2 q + r \lambda_\varepsilon \big) Y(s-)^2 \, \ud s , \qquad t \in [0,\tau(z_0)) . 
\end{align*}
Plugging in identities~\eqref{eq: p and q forward alternative}, we see that the drift equals zero, so we have $M(t) \leq \mgle(t)$.
\end{proof}

The following bound, which holds when $r \in (-\infty,0)$, 
will be needed to conclude the existence of the Loewner trace in the case where $\kappa < 8$. This is the only estimate that requires the additional condition,~\ref{item: ass2}, 
that the variance measure of the L\'evy measure $\nu$ is locally upper Ahlfors regular near the origin (in the sense of Definition~\ref{def: Ahlfors regular}).
(Compare also with~\cite[Theorem~27.7]{Sato:Levy_processes_and_infinitely_divisible_distributions}.) 
It might be possible to perform a more careful analysis and lift this assumption --- however, most common L\'evy measures, including those of $\alpha$-stable processes, already satisfy this assumption~\ref{item: ass2}.

\begin{prop} \label{prop: M forward supermgle case 2}
Fix $\kappa \in [0,8)$ and a L\'evy measure $\nu$ 
whose variance measure $\mu_\nu$ satisfies the local upper Ahlfors regularity~\eqref{eq: Ahlfors regularity} 
with constants 
$\epsilon_\nu \in (0,1/2)$, and 
$\alpha_\nu, c_\nu \in (0,\infty)$, and $\rho_\nu \in (0,1)$. 
Fix also 
$\varepsilon \in (0,\epsilon_\nu \wedge \tfrac{1}{2} \rho_\nu]$ and parameters $r \in (-\infty, 0)$ and $p < 0$ such that the following inequalities\footnote{We will see examples of such parameters in Sections~\ref{subsec: Summability 2} and~\ref{subsec: Summability 3}.} 
hold:
\begin{align} 
\label{eq: p and q forward inequality 1}
- 2 p - 2 q + r (\kappa (2 r - 1) + 8) + 2^{3-2r} r (r-1) \, \lambda_\varepsilon
\; & \leq 0 
\\
\label{eq: p and q forward inequality 2}
2 p - 2 q + r \kappa + 2^{3-2r} r (r-1) \, \lambda_\varepsilon 
\; & \leq 0 . 
\end{align}

Let $(g_t)_{t \geq 0}$ be the solution to~\eqref{eq: LE} driven by 
$\smash{\microDriver_\varepsilon^{\kappa}}$~\eqref{eq: BM with micro again again}, 
fix $z_0 \in \bH$, and consider the processes 
$M$ defined in~\eqref{eq: general forward martingale candidate}
and $\mgle$ defined in~\eqref{eq: forward supermgle N}. 
Fix $\alpha \in (0,\alpha_\nu]$. 
Then, we have $M(0) = \mgle(0)$ and 
\begin{align} \label{eq: M forward supermgle case 2}
M(t) \leq \; & \mgle(t) \; + \; L(t) \quad \textnormal{ for all } t \in [0, \tau(z_0)), 
\\ 
\nonumber
\textnormal{where} \qquad 
L(t) = L_{p,q,r,\alpha}(t,z_0) 
:= \; & 8  \, \frac{(\lambda_\varepsilon +  c_\nu)}{y_0^p} 
\int_0^t \frac{Y(s-)^{p + q + \alpha \, \varsigma_r(\alpha)}}{|Z(s-)|^2} 
 \, \ud s , 
 \qquad \textnormal{and} \qquad \varsigma_r(\alpha) = \frac{-2r}{\alpha - 2r}. 
\end{align}
\end{prop}

Note that the process $L$ behaves well under the time-change~\eqref{eq: forward time change}, so that $\hat{Y}(u) = y_0 \, e^{-2u}$: 
\begin{align*}
\hat{L}(t) = \hat{L}_{p,q,r,\alpha}(t,z_0) 
= \; & 8  \, ( \lambda_\varepsilon +  c_\nu )
\, y_0^{q + \alpha \, \varsigma_r(\alpha)} 
\int_0^{S(t)} e^{-2u(p + q + \alpha \, \varsigma_r(\alpha))} \, \ud u .
\end{align*}

\begin{proof}
We Taylor expand the function~\eqref{eq: aux function drift} 
around $\jump = 0$. 
We have $F_{r}(0) = 0$ and $F_{r}'(0) = 0$, and
\begin{align*}
F_{r}''(\jump) 
= \; & 4 r (r-1) \frac{(x-\jump)^2}{(x^2 + y^2)^2} \bigg( \frac{x^2 + y^2}{(x-\jump)^2 + y^2} \bigg)^{2-r} 
\; + \; \underbrace{\frac{2 r }{x^2 + y^2}
\bigg( \frac{x^2 + y^2}{(x-\jump)^2 + y^2} \bigg)^{1-r}}_{\leq 0} \\
\leq \; & \frac{4 r (r-1) }{x^2 + y^2} 
\bigg( \frac{x^2 + y^2}{(x-\jump)^2 + y^2} \bigg)^{1-r} ,
\end{align*}
because $r < 0$. Note also that the exponent is $1-r > 1$ (it is in particular non-negative).

{\bf Case 1.}
If $|x| \leq y$, then for all $\jump \in \bR$, we have
\begin{align*}
\frac{x^2 + y^2}{(x-\jump)^2 + y^2} \leq 2 
\qquad \Longrightarrow \qquad
F_{r}''(\jump) 
\leq \; & 2^{3-r} \, \frac{r (r-1) }{x^2 + y^2} .
\end{align*}
Using Taylor's theorem, this shows that 
\begin{align*}
F_{r}(\jump) \leq \; & 2^{2-r} \, \frac{r (r-1) \, \jump^2 }{x^2 + y^2} , \qquad \jump \in \bR , \; |x| \leq y .
\end{align*}
Therefore, using the definition~\eqref{eq: jump variance again} of the variance $\lambda_\varepsilon$, we obtain
\begin{align} \label{eq: F integral bound x < y}
\int_{|\jump| \leq \varepsilon} F_{r}(\jump; x,y) \, \nu(\ud \jump) 
\leq \; & 
2^{2-r} \, \frac{r (r-1) \, \lambda_\varepsilon }{x^2 + y^2} , \qquad |x| \leq y .
\end{align}

{\bf Case 2.}
Assume then that $y \leq |x|$. We split the integration region $\{|\jump| \leq \varepsilon\}$ into three parts:
\begin{align*}
\{|\jump| \leq \varepsilon\} 
= \; & 
\underbrace{\{|\jump| \leq \tfrac{|x|}{2} \wedge \varepsilon\}}_{=: A_1(x)} \cup 
\underbrace{\big\{ \tfrac{|x|}{2} \leq |\jump| \leq \varepsilon \textnormal{ and } \mathrm{sgn}(v) \neq \mathrm{sgn}(x) \big\}}_{=: A_2(x)} 
\cup 
\underbrace{\big\{ \tfrac{|x|}{2} \leq |\jump| \leq \varepsilon \textnormal{ and } \mathrm{sgn}(v) = \mathrm{sgn}(x) \big\}}_{=: A_3(x)} ,
\end{align*}
so that
\begin{align*}
\int_{|\jump| \leq \varepsilon} F_{r}(\jump; x,y) \, \nu(\ud \jump) 
= \; & \underbrace{\int_{A_1(x)} F_{r}(\jump; x,y) \, \nu(\ud \jump)}_{=: I_1(x,y)} 
\; + \;
\underbrace{\int_{A_2(x)} F_{r}(\jump; x,y) \, \nu(\ud \jump)}_{=: I_2(x,y)}  
\; + \;
\underbrace{\int_{A_3(x)} F_{r}(\jump; x,y) \, \nu(\ud \jump)}_{=: I_3(x,y)}  .
\end{align*}
\begin{enumerate}[leftmargin=*, label=\textnormal{\arabic*.}, ref=\arabic*.]
\item 
The first integral $I_1(x,y)$ can be bounded similarly as before: 
with $|x-\jump| \geq \tfrac{1}{2} |x|$, we have 
\begin{align*}
\frac{x^2 + y^2}{(x-\jump)^2 + y^2} \leq 4 
\qquad \Longrightarrow \qquad 
F_{r}(\jump) 
\leq \; & 2^{4-2r} \, \frac{r (r-1) \, \jump^2 }{x^2 + y^2} .
\end{align*}
Therefore, using the definition~\eqref{eq: jump variance again} of the variance $\lambda_\varepsilon$, we obtain
\begin{align} \label{eq: F integral bound first x > y}
I_1(x,y)
= \int_{A_1(x)} F_{r}(\jump; x,y) \, \nu(\ud \jump) 
\leq \; & 
2^{3-2r} \, \frac{r (r-1) \, \lambda_\varepsilon }{x^2 + y^2} , \qquad y \leq |x| .
\end{align}

\item 
The second integral $I_2(x,y)$ can also be bounded similarly: 
with $|x-\jump|^2 \geq 2 x^2$, we have 
\begin{align*}
\frac{x^2 + y^2}{(x-\jump)^2 + y^2} \leq 1 
\qquad \Longrightarrow \qquad 
F_{r}(\jump) 
\leq \; & \frac{4 r (r-1) \, \jump^2 }{x^2 + y^2} .
\end{align*}
Therefore, using the definition~\eqref{eq: jump variance again} of the variance $\lambda_\varepsilon$, we obtain
\begin{align} \label{eq: F integral bound second x > y}
I_2(x,y)
= \int_{A_2(x)} F_{r}(\jump; x,y) \, \nu(\ud \jump) 
\leq \; & 
\frac{2 r (r-1) \, \lambda_\varepsilon }{x^2 + y^2} , \qquad y \leq |x| .
\end{align}

\item Lastly, to bound the third integral $I_3(x,y)$, we will use the local $\alpha_\nu$-Ahlfors regularity assumption~\eqref{eq: Ahlfors regularity}. 
Since $r < 0$ and $y \leq |x| \leq 2 \varepsilon \leq 2 \epsilon_\nu \wedge \rho_\nu \leq \rho_\nu < 1$, we have
\begin{align*}
0 < y^{\varsigma_r(\alpha)} \leq y^{\varsigma_r(\alpha_\nu)} \leq \rho_\nu^{\varsigma_r(\alpha_\nu)}  < 1 
\qquad \textnormal{ for all } \alpha \leq \alpha_\nu ,
\qquad \textnormal{ where } 
\varsigma_r(\alpha) = \frac{-2r}{\alpha - 2r} . 
\end{align*}
We split the interval $A_3(x)$ into regions inside and outside of $[x - y^{\varsigma_r(\alpha)}, x + y^{\varsigma_r(\alpha)}]$ thus:
\begin{align} \label{eq: integral split I3}
\begin{split} 
I_3(x,y) 
= \; & \int_{-\varepsilon}^{x-y^{\varsigma_r(\alpha)}}
F_{r}(\jump; x,y) \, \one_{A_3(x)}(\jump) \, \nu(\ud \jump) 
\; + \; 
\int_{x+y^{\varsigma_r(\alpha)}}^{\varepsilon} 
F_{r}(\jump; x,y) \, \one_{A_3(x)}(\jump) \, \nu(\ud \jump) 
\\
\; & \; + \; \int_{x-y^{\varsigma_r(\alpha)}}^{x+y^{\varsigma_r(\alpha)}} 
F_{r}(\jump; x,y) \, \one_{A_3(x)}(\jump) \, \nu(\ud \jump) ,
\end{split} 
\end{align}
Note that when $\mathrm{sgn}(v) = \mathrm{sgn}(x)$, we have 
\begin{align*}
\frac{2 r \jump x}{x^2 + y^2} < 0 , \qquad r < 0 ,
\end{align*}
and with $\tfrac{|x|}{2} \leq |\jump|$ and $y \leq |x|$, we have
\begin{align*}
\frac{x^2 + y^2}{ \jump^2 }  \leq 8 .
\end{align*}
Hence, we see that
\begin{align*}
F_{r}(\jump) 
= \; &
\bigg( \frac{x^2 + y^2}{(x-\jump)^2 + y^2} \bigg)^{-r} - 1 + \frac{2 r \jump x}{x^2 + y^2} 
\; \leq \; \frac{8  \jump^2 }{x^2 + y^2} 
\bigg( \frac{x^2 + y^2}{y^2} \bigg)^{-r} 
\bigg( \frac{y^2}{(x-\jump)^2 + y^2} \bigg)^{-r} .
\end{align*}
We can now bound the integrals in~\eqref{eq: integral split I3} as follows. 

\smallskip

\begin{itemize}[leftmargin=1em]
\item When $\jump \in A_3(x) \setminus [x - y^{\varsigma_r(\alpha)}, x + y^{\varsigma_r(\alpha)}]$, 
we have $|x-\jump| \geq y^{\varsigma_r(\alpha)}$, 
which implies that
\begin{align*}
\bigg( \frac{y^2}{(x-\jump)^2 + y^2} \bigg)^{-r} 
\leq y^{2 r ( \varsigma_r(\alpha) - 1)}
\qquad \Longrightarrow \qquad 
F_{r}(\jump) 
\leq \; & 8 y^{\alpha \, \varsigma_r(\alpha)} \, 
\frac{\jump^2 }{x^2 + y^2} 
\bigg( \frac{x^2 + y^2}{y^2} \bigg)^{-r} ,
\end{align*}
since $2 r ( \varsigma_r(\alpha) - 1) = \alpha \, \varsigma_r(\alpha)$.
Hence, 
the first two integrals in~\eqref{eq: integral split I3} 
can be bounded as 
\begin{align} 
\nonumber
\; & \int_{-\varepsilon}^{x-y^{\varsigma_r(\alpha)}}
F_{r}(\jump; x,y) \, \one_{A_3(x)}(\jump) \, \nu(\ud \jump) 
\; + \; 
\int_{x+y^{\varsigma_r(\alpha)}}^{\varepsilon} 
F_{r}(\jump; x,y) \, \one_{A_3(x)}(\jump) \, \nu(\ud \jump) 
\\ 
\label{eq: F integral bound third x > y}
\leq \; & \frac{8 \, \lambda_\varepsilon }{x^2 + y^2} 
\bigg( \frac{x^2 + y^2}{y^2} \bigg)^{-r} \, y^{\alpha \, \varsigma_r(\alpha)} ,
\end{align}
using the definition~\eqref{eq: jump variance again} of the variance $\lambda_\varepsilon$. 

\medskip

\item When $\jump \in A_3(x) \cap [x - y^{\varsigma_r(\alpha)}, x + y^{\varsigma_r(\alpha)}]$, 
we have $|x-\jump| \leq y^{\varsigma_r(\alpha)}$.
Hence, with $\alpha \leq \alpha_\nu$, 
the last integral in~\eqref{eq: integral split I3} can be bounded as 
\begin{align} 
\nonumber
\; & \int_{x-y^{\varsigma_r(\alpha)}}^{x+y^{\varsigma_r(\alpha)}} 
F_{r}(\jump; x,y) \, \one_{A_3(x)}(\jump) \, \nu(\ud \jump) \\
\nonumber
\leq \; & \frac{8}{x^2 + y^2} \bigg( \frac{x^2 + y^2}{y^2} \bigg)^{-r} 
\int_{x-y^{\varsigma_r(\alpha)}}^{x+y^{\varsigma_r(\alpha)}} \jump^2 \, \one_{[-\epsilon_\nu, \epsilon_\nu]}(\jump) \, \nu(\ud \jump) \\
\label{eq: F integral bound third x > y bad case}
\leq \; & \frac{8 \, c_\nu}{x^2 + y^2} \bigg( \frac{x^2 + y^2}{y^2} \bigg)^{-r} \, y^{\alpha_\nu \varsigma_r(\alpha)} 
\; \leq \; \frac{8 \, c_\nu}{x^2 + y^2} \bigg( \frac{x^2 + y^2}{y^2} \bigg)^{-r} \, y^{\alpha \, \varsigma_r(\alpha)} ,
\end{align}
using the local $\alpha_\nu$-Ahlfors regularity assumption~\eqref{eq: Ahlfors regularity}.
\end{itemize}
\end{enumerate}
After collecting all of the above estimates~(\ref{eq: F integral bound x < y},~\ref{eq: F integral bound first x > y},~\ref{eq: F integral bound second x > y},~\ref{eq: F integral bound third x > y},~\ref{eq: F integral bound third x > y bad case}) together, we conclude that
\begin{align*} 
\int_{|\jump| \leq \varepsilon} F_{r}(\jump; x,y) \, \nu(\ud \jump) 
\leq \; & 2^{3-2r} \, \frac{r (r-1) \, \lambda_\varepsilon }{x^2 + y^2}
\; + \; \frac{8 \, ( \lambda_\varepsilon +  c_\nu ) }{x^2 + y^2} \bigg( \frac{x^2 + y^2}{y^2} \bigg)^{-r} \, y^{\alpha \, \varsigma_r(\alpha)} .
\end{align*}
Taking $x = X(s-)$ and $y = Y(s-)$, and 
plugging in inequalities
(\ref{eq: p and q forward inequality 1},~\ref{eq: p and q forward inequality 2}), 
we see by Lemma~\ref{lem: SDE forward M} that 
\begin{align*}
M(t) 
\; \leq \; \; & \mgle(t) \; + \; 
\int_0^t \frac{M(s-)}{|Z(s-)|^4} \; 
\big( r (\kappa (2 r - 1) + 8) - 2 p - 2 q + 2^{3-2r} r (r-1) \, \lambda_\varepsilon  \big) X(s-)^2 \, \ud s \\
\; & \; + \; 
\int_0^t \frac{M(s-)}{|Z(s-)|^4} \; 
\big( r \kappa + 2 p - 2 q + 2^{3-2r} r (r-1) \, \lambda_\varepsilon  \big) Y(s-)^2 \, \ud s  \\
\; & \; + \; \int_0^t M(s-) \; 
\frac{8 \, ( \lambda_\varepsilon +  c_\nu ) }{|Z(s-)|^2} \bigg( \frac{|Z(s-)|^2}{Y(s-)^2} \bigg)^{-r} 
Y(s-)^{\alpha \, \varsigma_r(\alpha)} 
\, \ud s \\
\; \leq \; \; & \mgle(t) \; + \; 
\int_0^t M(s-) \; 
\frac{8 \, ( \lambda_\varepsilon +  c_\nu ) }{|Z(s-)|^2} \bigg( \frac{|Z(s-)|^2}{Y(s-)^2} \bigg)^{-r} 
Y(s-)^{\alpha \, \varsigma_r(\alpha)}
\, \ud s , \qquad t \in [0,\tau(z_0)) . 
\end{align*}
Using the identity 
$\sin \arg (z) = \frac{y}{|z|}$ for $\bH \ni z = x + \ii y$, 
we can write the last term in the form
\begin{align} \label{eq: drift last term}
\int_0^t |g_{s}'(z_0)|^p \; 
\frac{8 \, ( \lambda_\varepsilon +  c_\nu ) \, Y(s-)^{q}}{|Z(s-)|^2} 
Y(s-)^{\alpha \, \varsigma_r(\alpha)} \, \ud s .
\end{align}
It now remains to bound~\eqref{eq: drift last term}. 
From~\eqref{eq: CRg}, we obtain 
\begin{align*}
|g_s'(z_0)|^p \leq \frac{Y(s)^p}{y_0^p}
, \qquad p < 0 , \; s \in [0,\tau(z_0)) .
\end{align*}
This gives~\eqref{eq: M forward supermgle case 2}.
\end{proof}

\subsection{Summability: $\kappa > 8$}
\label{subsec: Summability 1}

In this section, we finish the proof of Proposition~\ref{prop: needed uniform bound for f' 1}. 
Making use of Remark~\ref{rem: tune epsilon}, we may choose $\lambda_\varepsilon$ arbitrarily small by picking a small enough cutoff $\varepsilon > 0$.  
For deriving the derivative estimate~\eqref{eq: needed uniform bound for f' holds 1}
in Proposition~\ref{prop: needed uniform bound for f' 1}, 
we used a Borel-Cantelli argument relying on the summability of the probabilities~\eqref{eq: summability of event of interest}, 
and the purpose of this section is to verify the remaining~\eqref{eq: summability of event of interest} in Proposition~\ref{prop: summability of event of interest 1}.
For this, it is necessary for our argument that 
\begin{align*}
p(\kappa, r) + q(\kappa, \lambda_\varepsilon, r) 
= \tfrac{1}{2} r ( 8 - \kappa + \lambda_\varepsilon) + \kappa r^2 
< 0 ,
\end{align*} 
with suitably chosen parameter $r \in (0,1]$ and cutoff $\varepsilon > 0$, 
where $\kappa > 8$ and $p = p(\kappa, r )$ and $q = q(\kappa, \lambda_\varepsilon, r )$ are given by~\eqref{eq: p and q forward alternative}.  
Note that the function $r \mapsto p(\kappa, r) + q(\kappa, 0, r)$ 
has a unique minimum at $r = \frac{1}{4} - \frac{2}{\kappa}$.
This motivates the following choices (which are not optimal, but sufficiently convenient).

\begin{lem} \label{lem: alphaprimetilde and thetaprimetilde}
Fix $\kappa \in (8,\infty)$, a L\'evy measure $\nu$, and $\varepsilon > 0$.
Assume that identities~\eqref{eq: p and q forward alternative} hold with 
\begin{align} \label{eq: choice of r forward}
r = \rparamT(\kappa) := \tfrac{1}{4} - \tfrac{2}{\kappa} \; \in \; (0,1/4) , \qquad \kappa > 8 ,
\end{align}
and define
\begin{align} \label{eq: lambdatildemax forward}
\maxjumpT{\kappa} := \tfrac{1}{2}(\kappa-8) \; > \; 0 , \qquad \kappa > 8 .
\end{align}
Define also
\begin{align} \label{eq: Thetatilde forward case 1 again}
\ThetaTmax{\kappa}{\lambda_\varepsilon} := \frac{p(\kappa, \rparamT(\kappa) ) + q(\kappa, \lambda_\varepsilon, \rparamT(\kappa) ) }{p(\kappa, \rparamT(\kappa) ) - 2} 
= \frac{2 (\kappa - 8) (\kappa - 2 (\lambda_\varepsilon + 4))}{(\kappa + 8) (3 \kappa + 8)} .
\end{align}
Then, we have
\begin{align*}
0 \leq \lambda_\varepsilon < \maxjumpT{\kappa}
\qquad \Longrightarrow \qquad
\ThetaTmax{\kappa}{\lambda_\varepsilon} \in (0,2/3) .
\end{align*}
In particular, $0 \leq \lambda_\varepsilon < \maxjumpT{\kappa}$ implies that, 
for any $0 < \theta < \ThetaTmax{\kappa}{\lambda_\varepsilon}$, 
we have
\begin{align*}
(1-\theta) \, p(\kappa, \rparamT(\kappa)) + q(\kappa, \lambda_\varepsilon, \rparamT(\kappa)) < - 2 \theta .
\end{align*} 
\end{lem}

\begin{proof}
The map $\lambda_\varepsilon \mapsto \ThetaTmax{\kappa}{\lambda_\varepsilon}$~\eqref{eq: Thetatilde forward case 1 again} 
is decreasing when $\lambda_\varepsilon \geq 0$ and $\kappa > 8$. Moreover, we have 
\begin{align*}
\ThetaTmax{\kappa}{\lambda_\varepsilon} = 0 \qquad \Longleftrightarrow \qquad  
\lambda_\varepsilon = \maxjumpT{\kappa} .
\end{align*}
With $\lambda_\varepsilon = 0$, the map 
$\kappa \mapsto \ThetaTmax{\kappa}{0}$
is increasing when $\kappa > 8$, and
\begin{align*}
\lim_{\kappa \to 8} \ThetaTmax{\kappa}{0} = 0
\qquad \textnormal{and} \qquad 
\lim_{\kappa \to \infty} \ThetaTmax{\kappa}{0}
= 2/3 .
\end{align*}
Hence, for fixed $\kappa > 8$, 
the parameter~\eqref{eq: Thetatilde forward case 1 again} satisfies $\ThetaTmax{\kappa}{\lambda_\varepsilon} \in (0,2/3)$ when $0 \leq \lambda_\varepsilon < \maxjumpT{\kappa}$.
See also Figure~\ref{fig: Alphaprime and Thetaprime 3}. 
The last claim follows directly from the definition~\eqref{eq: Thetatilde forward case 1 again} of $\ThetaTmax{\kappa}{\lambda_\varepsilon}$.
\end{proof}

\noindent 
\begin{figure}[ht!]
\includegraphics[width=.4\textwidth]{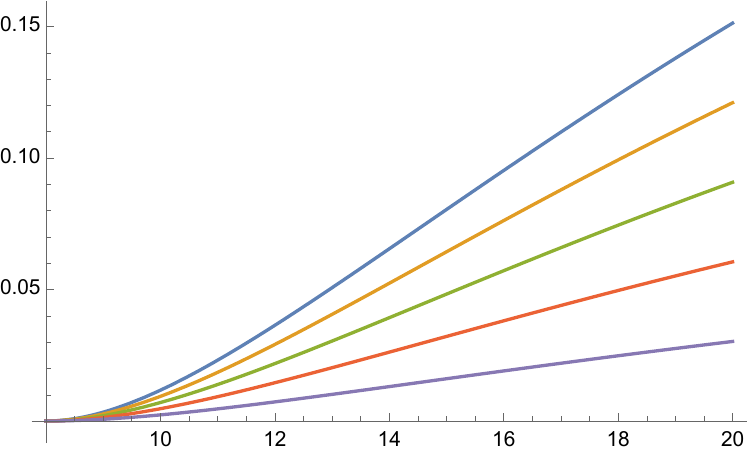}
\qquad
\includegraphics[width=.4\textwidth]{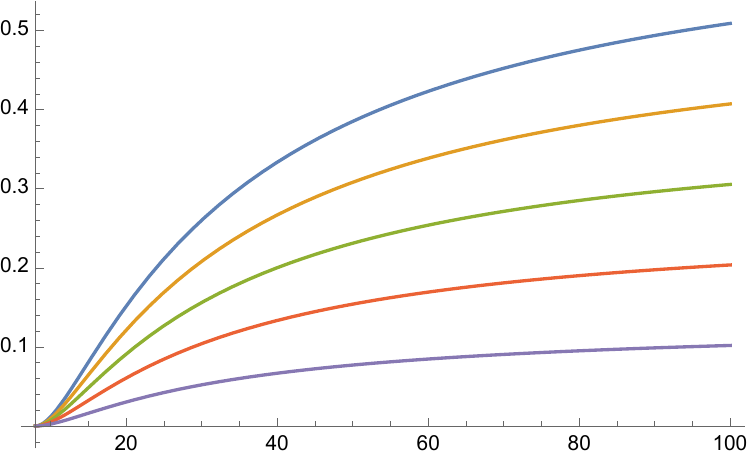}
\caption{\label{fig: Alphaprime and Thetaprime 3}
Illustrating quantities in Lemma~\ref{lem: alphaprimetilde and thetaprimetilde}: 
Plots of $\kappa \mapsto \ThetaTmax{\kappa}{\lambda_\varepsilon}$
with discrete values $\lambda_\varepsilon = c \, \maxjumpT{\kappa}$ for $c \in \{ 0, \frac{1}{5}, \frac{2}{5} ,\frac{3}{5} , \frac{4}{5} \}$. The largest plot (blue) has $\lambda_\varepsilon = 0$.
}
\end{figure}

\begin{prop} \label{prop: summability of event of interest 1}
Fix $T > 0$, $\kappa \in (8,\infty)$, a L\'evy measure $\nu$, and $\varepsilon > 0$ such that 
$\lambda_\varepsilon < \maxjumpT{\kappa}$ as in~\eqref{eq: lambdatildemax forward}. 
Then, for any  
$0 < \theta < \ThetaTmax{\kappa}{\lambda_\varepsilon}$ as in~\eqref{eq: Thetatilde forward case 1 again} and for any $R > 0$, 
on the event $\{ R_\varepsilon^{\kappa}(T) \leq R \}$, 
the probabilities of events~\eqref{eq: event of interest} are summable\textnormal{:}
\begin{align} \label{eq: summability of event of interest 1}
\sum_{n=1}^\infty \sum_{z_0 \in \Grid_\varepsilon^{\kappa} (2^{-n \theta}, T, R) }
\PR_{R} [ E_n^\theta(z_0) ] 
\; < \; \infty .
\end{align}
\end{prop}

\begin{proof}
Set $r = \rparamT(\kappa)$ as in~\eqref{eq: choice of r forward} and $p = p(\kappa, \rparamT(\kappa) )$ 
and $q = q(\kappa, \lambda_\varepsilon, \rparamT(\kappa) )$ 
as in~\eqref{eq: p and q forward alternative}:
\begin{align*}
p = \; & p(\kappa, \rparamT(\kappa) ) 
= - \frac{(\kappa - 8)(3 \kappa - 8)}{32 \kappa} < 0 , \\
q = \; & q(\kappa, \lambda_\varepsilon, \rparamT(\kappa) )
= \frac{(\kappa - 8)(\kappa + 8 + 4 \lambda_\varepsilon)}{32 \kappa} > 0 .
\end{align*}
Fix $z_0 = x_0 + \ii y_0 \in \Grid_\varepsilon^{\kappa}(2^{-n \theta}, T, R)$. 
Then, by Lemma~\ref{lem: forward bound}, Proposition~\ref{prop: M forward supermgle case 1} (with $\hat{M}(0) = \mgle(0)$), and 
the Optional stopping theorem (OST) (e.g.~\cite[Theorem~3.25]{LeGall:BM_book}), we obtain 
\begin{align*} 
\PR [ E_n^\theta(z_0) ] 
\leq \; & c_0(p, q, r) \, 2^{n \beta} \, 
\EX \big[ \hat{M}(S_{n}) \; \one{\{|\hat{X}(S_{n})| \leq y_0 \, e^{-2S_{n}}\}} \big] 
&& \textnormal{[by~\eqref{eq: forward bound}]} \\
\leq \; & c_0(p, q, r) \, 2^{n \beta} \, 
y_0^{q} \, \Big(\frac{|z_0|}{y_0}\Big)^{2r} ,
&& \textnormal{[by~\eqref{eq: M forward supermgle case 1} and OST]}
\end{align*}
where $\beta = (1-\theta) p + q$. 
Using Lemma~\ref{lem: sum over grid} with $a=2^{-n \theta}$, 
and $r = \rparamT(\kappa) > 0 > -1/2$, 
and
\begin{align*}
q - 2r = q(\kappa, \lambda_\varepsilon, \rparamT(\kappa) )  - 2\rparamT(\kappa) 
= \frac{(\kappa - 8) (\kappa + 4 \lambda_\varepsilon - 8)}{32 \kappa } \geq 0 > -1 ,
\end{align*}
we obtain 
\begin{align*} 
\sum_{n=1}^\infty \sum_{z_0 \in \Grid_\varepsilon^{\kappa}(2^{-n \theta}, T, R) }
\PR_{R} [ E_n^\theta(z_0) ] 
\; \leq \; \; &
c_0(p, q, r) \, 
\sum_{n=1}^\infty \hspace*{2mm} 2^{n \beta} \hspace*{-2mm}
\sum_{z_0 \in \Grid_\varepsilon^{\kappa}(2^{-n \theta}, T, R) }
y_0^{q} \, \Big(\frac{|z_0|}{y_0}\Big)^{2r} \\
\; \leq \; \; &
c_0(p, q, r) \, c_{\rm grid}(q, r, T, R) \, 
\sum_{n=1}^\infty 2^{n ((1-\theta) \, p + q + 2 \theta)} 
\; < \; \infty ,
\end{align*}
where by choices of $\theta < \ThetaTmax{\kappa}{\lambda_\varepsilon}$ and the other parameters, 
Lemma~\ref{lem: alphaprimetilde and thetaprimetilde} shows that
$(1-\theta) \, p + q + 2 \theta < 0$.
\end{proof}

\subsection{Summability: $\kappa \in (0,8)$}
\label{subsec: Summability 2}

In this section and the next one, we finish the proof of Proposition~\ref{prop: needed uniform bound for f' 2}. 
Making use of Remark~\ref{rem: tune epsilon}, we may choose $\lambda_\varepsilon$ arbitrarily small by picking a small enough cutoff $\varepsilon > 0$.  
For deriving the derivative estimate~\eqref{eq: needed uniform bound for f' holds 2}
in Proposition~\ref{prop: needed uniform bound for f' 2}, 
we used a Borel-Cantelli argument relying on the summability of the probabilities~\eqref{eq: summability of event of interest}, 
and the purpose of this section is to verify the remaining~\eqref{eq: summability of event of interest} in Propositions~\ref{prop: summability of event of interest 2} and~\ref{prop: summability of event of interest 3}. 
It is useful to choose suitable parameters $r \in (-\infty, 0)$ and 
\begin{align} \label{eq: p and q forward nonzero kappa}
p = p(\kappa, r ) 
:= \tfrac{1}{4} r ( 8 -\kappa + 2 \kappa r )
\qquad \textnormal{and}\qquad
q = q(\kappa, \lambda_\varepsilon, r ) 
:= p(\kappa, r ) + 4^{1-r} r (r-1) \,  \lambda_\varepsilon.
\end{align}
Note that we have $p = p(\kappa, r ) < 0$ when $r \in (\frac{\kappa - 8}{2 \kappa}, 0)$ 
and the inequalities
(\ref{eq: p and q forward inequality 1},~\ref{eq: p and q forward inequality 2}) appearing in the drift in the proof of Proposition~\ref{prop: M forward supermgle case 2} 
hold with the choices~\eqref{eq: p and q forward nonzero kappa}, equaling zero and $\kappa r < 0$, respectively.

It is necessary for our argument that, with suitably chosen parameter $r \in (\frac{\kappa - 8}{2 \kappa}, 0)$ and 
cutoff $\varepsilon > 0$, 
\begin{align*}
p(\kappa, r) + q(\kappa, \lambda_\varepsilon, r) 
= \tfrac{1}{2} r ( 8 - \kappa + 2^{3-2r} (r-1) \, \lambda_\varepsilon ) + \kappa r^2 
< 0 .
\end{align*} 
Note that the function $r \mapsto p(\kappa, r) + q(\kappa, 0, r)$ has a unique minimum at $r = \frac{1}{4} - \frac{2}{\kappa}$.
This motivates the following choices (which are not optimal, but sufficiently convenient).

\noindent 
\begin{figure}[ht!]
\includegraphics[width=.4\textwidth]{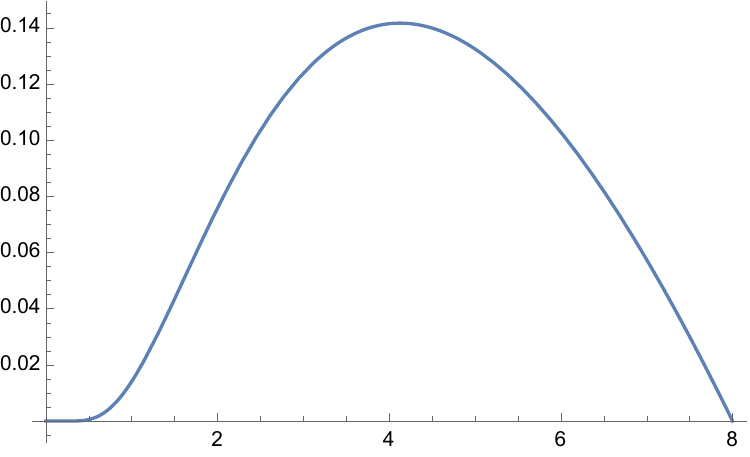}
\caption{\label{fig: Alphaprime and Thetaprime 6}
Illustrating quantities in Lemma~\ref{lem: alphaprimetilde and thetaprimetilde again}: 
Plot of $\kappa \mapsto \maxjumpH{\kappa}$.
}
\end{figure}

\noindent 
\begin{figure}[ht!]
\includegraphics[width=.4\textwidth]{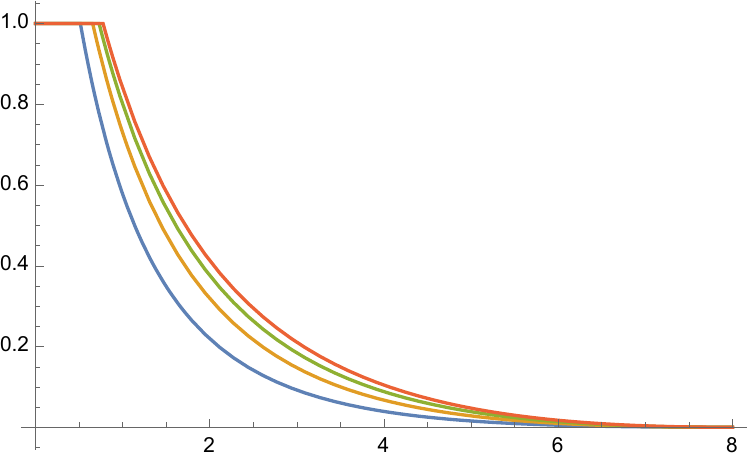}
\qquad
\includegraphics[width=.4\textwidth]{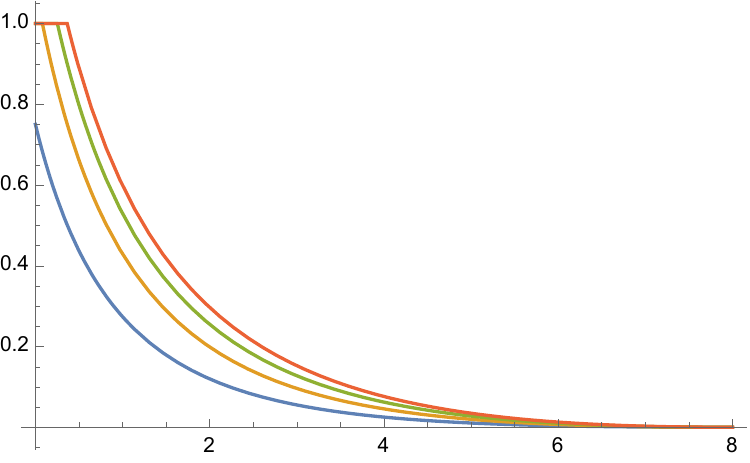}
\caption{\label{fig: Alphaprime and Thetaprime 4}
Illustrating quantities in Lemma~\ref{lem: alphaprimetilde and thetaprimetilde again}: 
Plots of $\kappa \mapsto \ThetaTmaxbeta{\alpha}{\kappa}$
with $\lambda_\varepsilon = 0$ (left) and
$\lambda_\varepsilon = \frac{1}{4} \, \maxjumpH{\kappa}$ (right), 
and discrete values $\alpha = c \, \maxbetaT{\kappa}{\lambda_\varepsilon}$ for $c \in \{ \frac{1}{5}, \frac{2}{5} ,\frac{3}{5} , \frac{4}{5} \}$.
}
\end{figure}

\begin{lem} \label{lem: alphaprimetilde and thetaprimetilde again}
Fix $\kappa \in (0,8)$ and a L\'evy measure $\nu$ 
whose variance measure $\mu_\nu$ satisfies the local upper Ahlfors regularity~\eqref{eq: Ahlfors regularity}. 
Assume that identities~\eqref{eq: p and q forward nonzero kappa} hold with
\begin{align} \label{eq: choice of r forward small kappa}
r = \rparamT(\kappa) := \tfrac{1}{4} - \tfrac{2}{\kappa}  \; \in \; (-\infty,0) , \qquad \kappa \in (0,8) , 
\end{align}  
and define
\begin{align} \label{eq: lambda max ass 2}
\maxjumpT{\kappa} := \frac{\kappa}{2^{\frac{4}{\kappa }+\frac{3}{2}}} \, \frac{8 - \kappa}{3 \kappa +8} 
\; > \; 0 , \qquad \kappa \in (0,8) .
\end{align}
See also Figure~\ref{fig: Alphaprime and Thetaprime 6}.
Define also\footnote{Note that $\ThetaTmaxbeta{\alpha}{\kappa}$ is independent of the cutoff parameter $\varepsilon$, but it may depend on the Ahlfors regularity parameter $\alpha_\nu$.} 
\begin{align} \label{eq: Thetatilde forward again small kappa}
\ThetaTmaxbeta{\alpha}{\kappa} 
:= \; & \frac{\alpha \, \varsigma_{\rparamT(\kappa)}(\alpha)}{2 - p(\kappa, \rparamT(\kappa) )} \wedge 1 
= \Big( \frac{32 \, \alpha \,  (8 - \kappa ) }{( \kappa + 48 + 64/\kappa ) (2 \alpha \kappa + 8 - \kappa)} \Big) \wedge 1 , 
\end{align}
where $\varsigma_{\rparamT(\kappa)}(\alpha) = \frac{-2 \rparamT(\kappa)}{\alpha - 2 \rparamT(\kappa)}$, 
and 
\begin{align} \label{eq: betamax forward}
\maxbetaT{\kappa}{\lambda_\varepsilon} 
:= \; & \frac{2 \rparamT(\kappa) ( p(\kappa, \rparamT(\kappa) ) + q(\kappa, \lambda_\varepsilon, \rparamT(\kappa) ) )}{p(\kappa, \rparamT(\kappa) ) + q(\kappa, \lambda_\varepsilon, \rparamT(\kappa) ) - 2 \rparamT(\kappa)} 
= \frac{(8 - \kappa)^2}{2 \kappa^2 + 2^{\frac{4}{\kappa} +\frac{5}{2}} (3 \kappa +8) \lambda_\varepsilon} 
\Big(1 - \frac{\lambda_\varepsilon}{\maxjumpT{\kappa}}
\Big) .
\end{align}
See also Figures~\ref{fig: Alphaprime and Thetaprime 4} and~\ref{fig: Alphaprime and Thetaprime 5}.
Then, we have
\begin{align*}
0 \leq \lambda_\varepsilon < \maxjumpT{\kappa}
\qquad \Longrightarrow \qquad
\maxbetaT{\kappa}{\lambda_\varepsilon} > 0 ,
\end{align*}
and 
\begin{align} \label{eq: alpha bounds for vartheta}
\begin{cases}
0 \leq \lambda_\varepsilon < \maxjumpT{\kappa} , \\[.3em]
0 < \alpha < \maxbetaT{\kappa}{\lambda_\varepsilon}
\end{cases}
\qquad \Longrightarrow \qquad
\begin{cases}
\ThetaTmaxbeta{\alpha}{\kappa} \in (0,1) , \\[.3em]
p(\kappa, \rparamT(\kappa) ) + q(\kappa, \lambda_\varepsilon, \rparamT(\kappa) ) + \alpha \, \varsigma_{\rparamT(\kappa)}(\alpha) < 0 .
\end{cases}
\end{align}
In particular, $0 \leq \lambda_\varepsilon < \maxjumpT{\kappa}$ and $0 < \alpha < \maxbetaT{\kappa}{\lambda_\varepsilon}$ together imply that, 
for any $0 < \theta < \ThetaTmaxbeta{\alpha}{\kappa}$, 
we have
\begin{align*}
\begin{cases}
\theta \, p(\kappa, \rparamT(\kappa) ) + \alpha \, \varsigma_{\rparamT(\kappa)}(\alpha) > 2 \theta , \\[.3em]
(1-\theta) \, ( p(\kappa, \rparamT(\kappa)) + q(\kappa, \lambda_\varepsilon, \rparamT(\kappa)) ) < 0 , \\[.3em]
(1-\theta) \, p(\kappa, \rparamT(\kappa)) + q(\kappa, \lambda_\varepsilon, \rparamT(\kappa)) < - 2 \theta .
\end{cases}
\end{align*} 
Moreover, these choices satisfy the assumptions $r \in (-\infty,0)$, 
$p < 0$, 
and \textnormal{(}\ref{eq: p and q forward inequality 1},~\ref{eq: p and q forward inequality 2}\textnormal{)} 
in Proposition~\ref{prop: M forward supermgle case 2}.
\end{lem}

\begin{proof}
The map $\lambda_\varepsilon \mapsto \maxbetaT{\kappa}{\lambda_\varepsilon}$~\eqref{eq: betamax forward} 
is decreasing when $\lambda_\varepsilon \geq 0$ and $\kappa \in (0,8)$. Moreover, we have
\begin{align*}
\maxbetaT{\kappa}{\lambda_\varepsilon} = 0 \qquad \Longleftrightarrow \qquad  
\lambda_\varepsilon = \maxjumpT{\kappa} .
\end{align*} 
With $\lambda_\varepsilon = 0$, the map 
$\kappa \mapsto \maxbetaT{\kappa}{0}$
\begin{align*}
\kappa \; \longmapsto \; 
\maxbetaT{\kappa}{0}
= \; & \frac{(8-\kappa )^2}{2 \kappa ^2} 
\end{align*}
is decreasing when $\kappa \in (0,8)$, and 
\begin{align*}
\lim_{\kappa \to 8} \maxbetaT{\kappa}{0}
= 0  . 
\end{align*}
Hence, for fixed $\kappa \in (0,8)$, 
the parameter~\eqref{eq: betamax forward} satisfies $\maxbetaT{\kappa}{\lambda_\varepsilon} > 0$ when $0 \leq \lambda_\varepsilon < \maxjumpT{\kappa}$.

Next, note that, for fixed $\kappa \in (0,8)$, the map 
\begin{align*}
\lambda_\varepsilon \; \longmapsto \; p(\kappa, \rparamT(\kappa) ) + q(\kappa, \lambda_\varepsilon, \rparamT(\kappa) ) 
\end{align*}
is increasing, and 
\begin{align*}
p(\kappa, \rparamT(\kappa) ) + q(\kappa, \lambda_\varepsilon, \rparamT(\kappa) ) = 0 
\qquad \Longleftrightarrow \qquad  
\lambda_\varepsilon = \maxjumpT{\kappa}.
\end{align*}
Also, for fixed $\kappa \in (0,8)$ and $\lambda_\varepsilon \in [0,\maxjumpT{\kappa})$, the map 
\begin{align*}
\alpha \; \longmapsto \; p(\kappa, \rparamT(\kappa) ) + q(\kappa, \lambda_\varepsilon, \rparamT(\kappa) ) + \alpha \, \varsigma_{\rparamT(\kappa)}(\alpha)
\end{align*}
is increasing, and 
\begin{align*}
p(\kappa, \rparamT(\kappa) ) + q(\kappa, \lambda_\varepsilon, \rparamT(\kappa) ) + \alpha \, \varsigma_{\rparamT(\kappa)}(\alpha) = 0 
\qquad \Longleftrightarrow \qquad  
\alpha = \maxbetaT{\kappa}{\lambda_\varepsilon}.
\end{align*}
This shows~\eqref{eq: alpha bounds for vartheta}.
Note also that 
if $\alpha < \maxbetaT{\kappa}{\lambda_\varepsilon}$, then by~\eqref{eq: alpha bounds for vartheta}, we have
\begin{align*}
\ThetaTmaxbeta{\alpha}{\kappa} 
\leq \; & 
\underbrace{ \Big( \frac{-\alpha \, \varsigma_{\rparamT(\kappa)}(\alpha)}{p(\kappa, \rparamT(\kappa) ) + q(\kappa, \lambda_\varepsilon, \rparamT(\kappa) )} \Big)}_{< 1} \, 
\frac{p(\kappa, \rparamT(\kappa) ) + q(\kappa, \lambda_\varepsilon, \rparamT(\kappa) )}{p(\kappa, \rparamT(\kappa) ) - 2} 
\; < \; \frac{p(\kappa, \rparamT(\kappa) ) + q(\kappa, \lambda_\varepsilon, \rparamT(\kappa) )}{p(\kappa, \rparamT(\kappa) ) - 2} .
\end{align*}
The other claims then follow using the definition~\eqref{eq: Thetatilde forward again small kappa} of $\ThetaTmax{\kappa}{\lambda_\varepsilon}$.
\end{proof}

\noindent 
\begin{figure}[ht!]
\includegraphics[width=.4\textwidth]{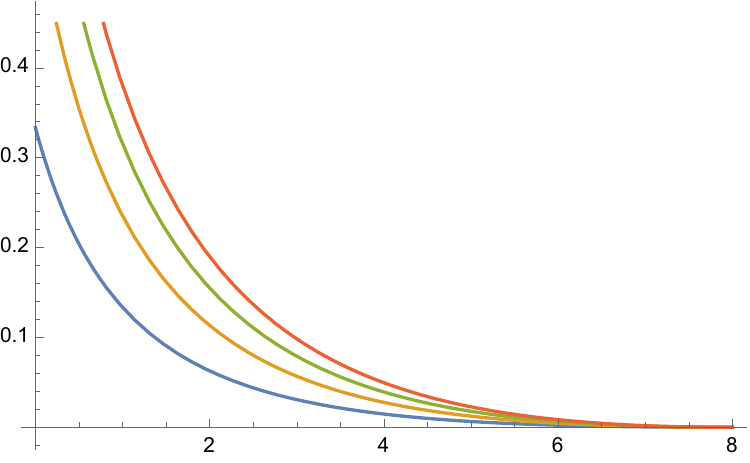}
\qquad
\includegraphics[width=.4\textwidth]{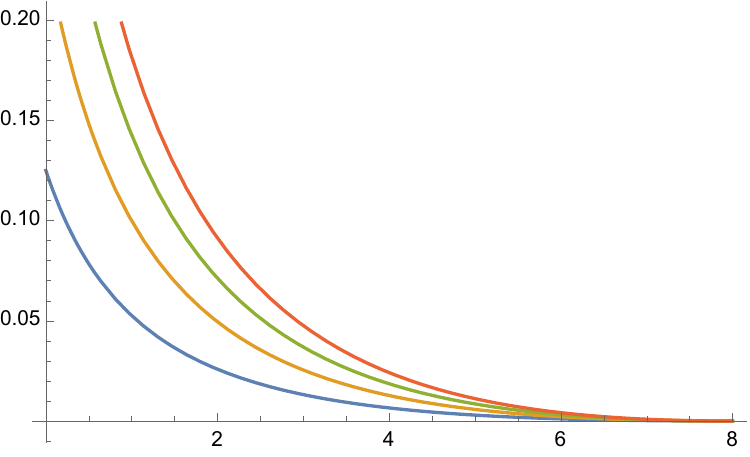}
\caption{\label{fig: Alphaprime and Thetaprime 5}
Illustrating quantities in Lemma~\ref{lem: alphaprimetilde and thetaprimetilde again}: 
Plots of $\kappa \mapsto \ThetaTmaxbeta{\alpha}{\kappa}$
with  $\lambda_\varepsilon = \frac{1}{2} \, \maxjumpH{\kappa}$ (left) and
$\lambda_\varepsilon = \frac{3}{4} \, \maxjumpH{\kappa}$ (right), 
and discrete values $\alpha = c \, \maxbetaT{\kappa}{\lambda_\varepsilon}$ for $c \in \{ \frac{1}{5}, \frac{2}{5} ,\frac{3}{5} , \frac{4}{5} \}$.
}
\end{figure}

\begin{prop} \label{prop: summability of event of interest 2}
Fix $T > 0$, $\kappa \in (0,8)$, and a L\'evy measure $\nu$ 
whose variance measure $\mu_\nu$ satisfies the local upper $\alpha_\nu$-Ahlfors regularity~\eqref{eq: Ahlfors regularity}. 
Fix $\varepsilon > 0$ such that $\lambda_\varepsilon < \maxjumpT{\kappa}$ as in~\eqref{eq: lambda max ass 2}. 
Then, for any $\alpha \in (0,\maxbetaT{\kappa}{\lambda_\varepsilon} \wedge \alpha_\nu)$ as in~\eqref{eq: betamax forward}, 
for any  
$0 < \theta < \ThetaTmaxbeta{\alpha}{\kappa}$ as in~\eqref{eq: Thetatilde forward again small kappa}, and for any $R > 0$, 
on the event $\{ R_\varepsilon^{\kappa}(T) \leq R \}$, 
the probabilities of events~\eqref{eq: event of interest} are summable\textnormal{:}
\begin{align} \label{eq: summability of event of interest 2}
\sum_{n=1}^\infty \sum_{z_0 \in \Grid_\varepsilon^{\kappa}(2^{-n \theta}, T, R) }
\PR_{R} [ E_n^\theta(z_0) ] 
\; < \; \infty .
\end{align}
\end{prop}

\begin{proof}
Set $r = \rparamT(\kappa)$ as in~\eqref{eq: choice of r forward small kappa} and $p = p(\kappa, \rparamT(\kappa) )$ 
and $q = q(\kappa, \lambda_\varepsilon, \rparamT(\kappa) )$ 
as in~\eqref{eq: p and q forward nonzero kappa}:
\begin{align*}
p = \; & p(\kappa, \rparamT(\kappa) ) 
= -\frac{(8 - \kappa)^2}{32 \kappa } < 0 , \\
q = \; & q(\kappa, \lambda_\varepsilon, \rparamT(\kappa) ) 
= -\frac{(8 - \kappa)^2}{32 \kappa } 
+ 2^{\frac{4}{\kappa }-\frac{5}{2}}  \, \frac{(8 - \kappa) (3 \kappa +8) }{\kappa^2} \, \lambda_\varepsilon .
\end{align*}
Fix $z_0 = x_0 + \ii y_0 \in \Grid_\varepsilon^{\kappa}(2^{-n \theta}, T, R)$. 
Then, by Lemma~\ref{lem: forward bound}, Proposition~\ref{prop: M forward supermgle case 2}
(with $\hat{M}(0) = \mgle(0)$), and 
the Optional stopping theorem (OST) (e.g.~\cite[Theorem~3.21]{LeGall:BM_book}), we obtain 
\begin{align*} 
\PR [ E_n^\theta(z_0) ] 
\leq \; & c_0(p, q, r) \, 2^{n \beta} \, 
\EX \big[ \hat{M}(S_{n}) \; \one{\{|\hat{X}(S_{n})| \leq y_0 \, e^{-2S_{n}}\}} \big] 
&& \textnormal{[by~\eqref{eq: forward bound}]} \\
\leq \; & c_0(p, q, r) \, 2^{n \beta} \, 
\big( \hat{M}(0) + \EX [ \hat{L}(S_{n}) ] \big) ,
&& \textnormal{[by~\eqref{eq: M forward supermgle case 2} and OST]}
\end{align*}
where $\beta = (1-\theta) p + q$, 
and where 
$\hat{M}(0) = y_0^{q - 2r} |z_0|^{2r}$ and 
\begin{align*}
\hat{L}(S_{n})
= \; & 8  \, ( \lambda_\varepsilon +  c_\nu ) 
\, y_0^{q + \alpha \, \varsigma_r(\alpha)}
\int_0^{S_{n}} e^{-2u(p + q + \alpha \, \varsigma_r(\alpha))} \, \ud u .
\end{align*}

{\bf Term 1.}
We will first consider the first term $2^{n \beta} \, \hat{M}(0) = 2^{n \beta} \, y_0^{q - 2r} |z_0|^{2r}$. 
Using Lemma~\ref{lem: sum over grid} with $a=2^{-n \theta}$, 
and $r = \rparamT(\kappa) = \tfrac{1}{4} - \tfrac{2}{\kappa}$, 
and $q = q(\kappa, \lambda_\varepsilon, \rparamT(\kappa) )$, 
we obtain 
\begin{align*} 
\sum_{n=1}^\infty \sum_{z_0 \in \Grid_\varepsilon^{\kappa}(2^{-n \theta}, T, R) }
2^{n \beta} \, \hat{M}(0)
\; \leq \; \; & c_{\rm grid}(q, r, T, R) \, 
\sum_{n=1}^\infty 2^{n ((1-\theta) \, p + q)} \, \chi_{q, r}(2^{-n \theta}) \; < \; \infty ,
\end{align*}
where $\chi_{q, r}(2^{-n \theta})$ is defined in~\eqref{eq: grid chi} 
and where by the choice of $\theta < \ThetaTmaxbeta{\alpha}{\kappa}$ and the other parameters, 
Lemma~\ref{lem: alphaprimetilde and thetaprimetilde again} shows that
$(1-\theta) \, p + q < - 2 \theta$ and 
$(1-\theta) \, (p + q) < 0$, guaranteeing summability.

{\bf Term 2.} We then consider the second term. 
By~\eqref{eq: alpha bounds for vartheta} in Lemma~\ref{lem: alphaprimetilde and thetaprimetilde again}, we have
\begin{align*}
p(\kappa, \rparamT(\kappa) ) + q(\kappa, \lambda_\varepsilon, \rparamT(\kappa) ) + \alpha \, \varsigma_{\rparamT(\kappa)}(\alpha) < 0 .
\end{align*} 
Thus, computing the expected value of $\hat{L}(S_{n})$ gives
\begin{align*}
\EX [ \hat{L}(S_{n}) ]
= \; & - 8  \, ( \lambda_\varepsilon +  c_\nu ) 
\, \frac{y_0^{q + \alpha \, \varsigma_r(\alpha)}}{2(p + q + \alpha \, \varsigma_r(\alpha))} 
\, \EX \big[ e^{-2 S_{n}(p + q + \alpha \, \varsigma_r(\alpha))}  - 1 \big] \\
\leq \; & c(p,q,r,\alpha,T) \, 2^{-n(p + q + \alpha \, \varsigma_r(\alpha))} ,
\end{align*}
since $S_{n} \leq \smash{\log \sqrt{\tfrac{y_0}{2^{-n-1}}}}$,
where the constant $c(p,q,r,\alpha,T) \in (0,\infty)$ also depends on the L\'evy measure~$\nu$. 
Using Lemma~\ref{lem: sum over grid} with $a=2^{-n \theta}$, and with exponents zero ($\chi_{0,0}(2^{-n \theta}) = 2^{2n \theta}$ in~\eqref{eq: grid chi}), we obtain 
\begin{align*} 
\; & \sum_{n=1}^\infty \sum_{z_0 \in \Grid_\varepsilon^{\kappa}(2^{-n \theta}, T, R) } 2^{n \beta} \, \EX [ \hat{L}(S_{n}) ]
\\
\; \leq \; \; & c(p,q,r,\alpha,T) \,
\sum_{n=1}^\infty \sum_{z_0 \in \Grid_\varepsilon^{\kappa}(2^{-n \theta}, T, R) } 2^{n ((1-\theta) p + q)} \, 
2^{-n(p + q + \alpha \, \varsigma_{\rparamT(\kappa)}(\alpha))} \\
\; \leq \; \; & c(p,q,r,\alpha,T) \, c'(T, R)
\sum_{n=1}^\infty 2^{-n ( \theta p + \alpha \, \varsigma_{\rparamT(\kappa)}(\alpha) - 2 \theta )} 
\; < \; \infty ,
\end{align*}
where by the choice of $\theta < \ThetaTmaxbeta{\alpha}{\kappa}$ and the other parameters, 
Lemma~\ref{lem: alphaprimetilde and thetaprimetilde again} gives $\theta p + \alpha \, \varsigma_{\rparamT(\kappa)}(\alpha) - 2 \theta > 0$. 
\end{proof}

\subsection{Summability: $\kappa = 0$ and linear drift}
\label{subsec: Summability 3}

The above choices~(\ref{eq: choice of r forward small kappa},~\ref{eq: lambda max ass 2}) 
do not apply in the case where $\kappa = 0$. 
One could instead use a different choice. 
For later purposes, however, in this section we also modify the process~\eqref{eq: general forward martingale candidate} slightly and include the possibility of a linear drift to the driving function. 
Hence, in the rest of this section, we consider driving functions with no diffusion part $(\kappa=0)$ but allowing microscopic jumps and a linear drift:
\begin{align} \label{eq: BM with micro and drift kappa is zero again again}
\microDriver_\varepsilon^{0, a}(t) = a t + \int_{|\jump| \leq \varepsilon} \jump \, \PoissonComp(t, \ud \jump) , \qquad
a \in \bR, \, \varepsilon > 0 ,
\end{align}
where $\smash{\PoissonComp}(t, \ud \jump) := \Poisson(t, \ud \jump) - t \nu(\ud \jump)$ is the compensated Poisson point process
of a Poisson point process $\Poisson$ with L\'evy intensity measure $\nu$. As in~\eqref{eq: BM with micro supremum}, we write 
\begin{align} 
R_\varepsilon^{0,a}(T) := \sup_{t \in [0,T]} | \microDriver_\varepsilon^{0,a}(t) | ,
\end{align}
we let $(g_t)_{t \geq 0}$ be a Loewner chain driven by $\smash{\microDriver_\varepsilon^{0,a}}$, 
let $(K_t)_{t \geq 0}$ be the corresponding hulls (obtained by solving the Loewner equation~\eqref{eq: LE}), 
and we write $f_t:= g_t^{-1}$ and set $\tilde{f}_t(w) := f_t(w + \smash{\microDriver_\varepsilon^{0,a}}(t))$. 
We work under~\ref{item: ass2}: we suppose that 
the variance measure of the L\'evy measure $\nu$ is locally (upper) Ahlfors regular near the origin in the sense of Definition~\ref{def: Ahlfors regular}.
We shall consider the events 
\begin{align} 
\begin{split} \label{eq: event of interest kappa = 0}
E_n^\theta(z_0)
= E_n^\theta(z_0, T) \; := \; \big\{ & \; \textnormal{there exists }  t \in [0,T] 
\textnormal{ such that } z_0 \in \bH \setminus K_t \textnormal{ and } 
 \\
& \; | g_t(z_0) - \microDriver_\varepsilon^{0,a}(t) - \ii 2^{-n} | \leq 2^{-n-1}
\textnormal{ and } 
| g_t'(z_0) | \leq \tfrac{80}{27} \, 2^{-n (1-\theta)} \big\} ,
\qquad n \in \bN ,
\end{split}
\end{align}
where $z_0$ ranges over the grid
\begin{align*}
\Grid_\varepsilon^{0,a}
= \Grid_\varepsilon^{0,a} (2^{-n \theta}, T, R_\varepsilon^{0,a}(T)) 
:= \Big\{ z \in \bH \;\; | \;\; 
\re(z) = & \; \tfrac{1}{8} \, 2^{-n \theta} \, \ell \, \in \, [-R_\varepsilon^{0,a}(T), R_\varepsilon^{0,a}(T)] \textnormal{ and }
\\ 
\im(z) = & \; \tfrac{1}{8} \, 2^{-n \theta} \, (k + 8) \, \in \, (2^{-n \theta} , \sqrt{1+4T}] , \; 
\ell, k \in \bZ \Big\}  .
\end{align*}

Similarly as before, we fix a starting point $z_0 = x_0 + \ii y_0 \in \bH$ implicitly throughout, and consider 
\begin{align}
\nonumber
Z(t) = Z_\varepsilon^{0,a}(t,z_0) := \; & g_t(z_0) - \microDriver_\varepsilon^{0,a}(t) =: X(t) + \ii Y(t) , 
\\
\label{eq: general forward martingale candidate kappa=0}
M(t) = M_{p}^{a}(t,z_0) := \; & |g_t'(z_0)|^p \, ( \sin \arg Z(t) )^{2} \, e^{-a^2t} , \qquad 
p \in \bR , \;  t \in [0, \tau(z_0)) .
\end{align}
Note that $M(0) = y_0^{2} \, |z_0|^{-2}$.

\begin{lem} \label{lem: SDE forward M with drift} 
Fix $a \in \bR$, a L\'evy measure $\nu$, and $\varepsilon > 0$. 
Let $(g_t)_{t \geq 0}$ be the solution to~\eqref{eq: LE} driven by 
$\smash{\microDriver_\varepsilon^{0,a}}$~\eqref{eq: BM with micro and drift kappa is zero again again},
fix $z_0 \in \bH$, and consider the process 
$M$ defined in~\eqref{eq: general forward martingale candidate kappa=0}.
Then, we have
\begin{align*}
M(t) 
\; = \; \; & \mgle(t) 
\; + \; \int_0^t M(s-) \; \bigg( D_{p}(s) + \frac{2 a X(s-)}{|Z(s-)|^2} - a^2 \bigg) \, \ud s , \qquad t \in [0, \tau(z_0)) ,  \\[1em]
\textnormal{where} \qquad 
D_{p}(s)
= \; & \frac{- (8 + 2 p) X(s-)^2 + 2 p Y(s-)^2}{|Z(s-)|^4} 
\\
\; \; & 
+ \; \int_{|\jump| \leq \varepsilon} \bigg( \bigg| \frac{Z(s-) - \jump}{Z(s-)} \bigg|^{-2} - 1 - \frac{2 \jump X(s-)}{|Z(s-)|^2} \bigg) \nu(\ud \jump) , \qquad 
p \in \bR ,
\end{align*}
and where $\mgle$ is the right-continuous local martingale
\begin{align} \label{eq: forward supermgle N kappa = 0}
\begin{split}
\mgle(t) 
\; := \; \; & 
M(0) \; + \;  \int_0^t M(s-) \int_{|\jump| \leq \varepsilon}  \bigg( \bigg| \frac{Z(s-) - \jump}{Z(s-)} \bigg|^{-2} - 1 \bigg) \, \PoissonComp(\ud s, \ud \jump) , \qquad t \in [0, \tau(z_0)) .
\end{split}
\end{align}
\end{lem}

\begin{proof}
As in Section~\ref{subsec: Supermartingale bounds}, by~\eqref{eq: LE} and a straightforward application of It\^o's formula, we have
\begin{align*}
|g_t'(z_0)|^p 
\; = \; \; & 1 \; - \; 2 p \int_0^t |g_{s}'(z_0)|^{p} \; \frac{X(s-)^2 - Y(s-)^2}{|Z(s-)|^4} \, \ud s , 
\end{align*}
and a tedious application of It\^o's formula (see Lemma~\ref{lem: SDE forward sin(arg(Z)) power} in Appendix~\ref{app: Ito calculus} with $\kappa=0$ and $r=-1$) gives
\begin{align*}
( \sin \arg Z(t) )^{2} 
\; = \; \; & ( \sin \arg z_0 )^{2} 
\; + \; 2 a \int_0^t ( \sin \arg Z(s-) )^{2} \; \frac{X(s-)}{|Z(s-)|^2} \, \ud s \\
\; \; & 
+ \; \int_0^t \int_{|\jump| \leq \varepsilon} ( \sin \arg Z(s-)  )^{2} \; \bigg( \bigg| \frac{Z(s-) - \jump}{Z(s-)} \bigg|^{-2} - 1 \bigg) \, \PoissonComp(\ud s, \ud \jump) \\
\; \; & 
- \; 8 \, \int_0^t( \sin \arg Z(s-) )^{2}
\; \frac{X(s-)^2}{|Z(s-)|^4} 
\, \ud s \\
\; \; & 
+ \; \int_0^t( \sin \arg Z(s-) )^{2} \int_{|\jump| \leq \varepsilon}
\bigg( \bigg| \frac{Z(s-) - \jump}{Z(s-)} \bigg|^{-2} - 1 - \frac{2 \jump X(s-)}{|Z(s-)|^2} \bigg) \nu(\ud \jump) \, \ud s .
\end{align*} 
Combining these, we obtain the asserted identity for $M$.
\end{proof}

As in Section~\ref{subsec: Supermartingale bounds}, 
for a suitable $p$ and $\lambda_\varepsilon$ small enough, 
$M$ can be bounded in the following manner.

\begin{prop} \label{prop: M forward supermgle case 2 kappa = 0}
Fix $a \in \bR$ and a L\'evy measure $\nu$ 
whose variance measure $\mu_\nu$ satisfies the local upper Ahlfors regularity~\eqref{eq: Ahlfors regularity} 
with constants 
$\epsilon_\nu \in (0,1/2)$, and 
$\alpha_\nu, c_\nu \in (0,\infty)$, and $\rho_\nu \in (0,1)$. 
Fix also 
$\varepsilon \in (0,\epsilon_\nu \wedge \tfrac{1}{2} \rho_\nu]$ such that 
\begin{align} \label{eq: lambdamax forward kappa = 0}
\lambda_\varepsilon < \tfrac{7}{128} =: \maxjumpT{0} . 
\end{align}
Fix also parameter $p \in (-2,0)$ such that\footnote{Here, we may allow $p$ to depend on $\varepsilon$, but this is not necessary: the inequalities~\eqref{eq: p and q forward inequality kappa = 0} hold for $p = -7/4$ for all $\lambda_\varepsilon < \tfrac{7}{128}$.} the following inequalities hold:
\begin{align} 
\label{eq: p and q forward inequality kappa = 0}
- 2 p - 7 + 64 \, \lambda_\varepsilon \leq 0 
\qquad \textnormal{and}  \qquad
2 p + 64 \, \lambda_\varepsilon \leq 0 .
\end{align}

Let $(g_t)_{t \geq 0}$ be the solution to~\eqref{eq: LE} driven by 
$\smash{\microDriver_\varepsilon^{0,a}}$~\eqref{eq: BM with micro and drift kappa is zero again again}, 
fix $z_0 \in \bH$, and consider the processes 
$M$ defined in~\eqref{eq: general forward martingale candidate kappa=0}
and $\mgle$ defined in~\eqref{eq: forward supermgle N kappa = 0}. 
Fix $\alpha \in (0,\alpha_\nu]$. 
Then, we have $M(0) = \mgle(0)$ and 
\begin{align} \label{eq: M forward supermgle case 2 kappa = 0}
M(t) \leq \; & \mgle(t) \; + \; L(t) \quad \textnormal{ for all } t \in [0,\tau(z_0)) 
, \\ 
\nonumber
\textnormal{where} \qquad 
L(t) = L_{p,\alpha}(t,z_0) 
:= \; & 8  \, \frac{(\lambda_\varepsilon +  c_\nu)}{y_0^p} 
\int_0^t \frac{Y(s-)^{p + \alpha \, \varsigma_r(\alpha)}}{|Z(s-)|^2} 
 \, \ud s , 
 \qquad\textnormal{and} \qquad \varsigma_{-1}(\alpha) = \frac{2}{\alpha + 2}.
\end{align}
\end{prop}

Note that the process $L$ behaves well under the time-change~\eqref{eq: forward time change}, so that $\hat{Y}(u) = y_0 \, e^{-2u}$: 
\begin{align*}
\hat{L}(t) = \hat{L}_{p,\alpha}(t,z_0) 
= \; & 8  \, ( \lambda_\varepsilon +  c_\nu ) 
\, y_0^{\alpha \, \varsigma_{-1}(\alpha)} 
\int_0^{S(t)} e^{-2u(p + \alpha \, \varsigma_{-1}(\alpha))} \, \ud u .
\end{align*}

\begin{proof}
The proof is very similar to that of Proposition~\ref{prop: M forward supermgle case 2}.
First, we note that 
\begin{align*}
\frac{2a X(s-)}{|Z(s-)|^2} \; - \; a^2 
\; = \; \frac{X(s-)^2}{|Z(s-)|^4} \; - \; \bigg(\frac{X(s-)}{|Z(s-)|^2} - a \bigg)^2
\; \leq \; \frac{X(s-)^2}{|Z(s-)|^4} ,
\end{align*}
which implies by Lemma~\ref{lem: SDE forward M with drift} that
\begin{align*}
M(t) 
\; \leq \; \; & \mgle(t) 
\; + \; \int_0^t M(s-) \; \bigg( \frac{- (7 + 2 p) X(s-)^2 + 2 p Y(s-)^2}{|Z(s-)|^4}\bigg) \, \ud s \\
\; \; & 
+ \; \int_0^t M(s-) \; \int_{|\jump| \leq \varepsilon}\bigg( \bigg| \frac{Z(s-) - \jump}{Z(s-)} \bigg|^{-2} - 1 - \frac{2 \jump X(s-)}{|Z(s-)|^2} \bigg) \nu(\ud \jump) \, \ud s , \qquad t \in [0,\tau(z_0)) . 
\end{align*}
As in the proof of Proposition~\ref{prop: M forward supermgle case 2}, with $p < 0$, $q = 0$, and $r = -1 < 0$, 
the inequalities~\eqref{eq: p and q forward inequality kappa = 0} and 
the local $\alpha_\nu$-Ahlfors regularity assumption~\eqref{eq: Ahlfors regularity} yield the asserted bound~\eqref{eq: M forward supermgle case 2 kappa = 0}. 
\end{proof}

\begin{lem} \label{lem: alphaprimetilde and thetaprimetilde again kappa=0}
Fix $a \in \bR$ and a L\'evy measure $\nu$ 
whose variance measure $\mu_\nu$ satisfies the local upper $\alpha_\nu$-Ahlfors regularity~\eqref{eq: Ahlfors regularity}. 
Fix $\varepsilon > 0$ such that $\lambda_\varepsilon < \maxjumpT{0}$ as in~\eqref{eq: lambdamax forward kappa = 0}. 
Fix also parameter $p \in (-2,0)$ such that the inequalities~\eqref{eq: p and q forward inequality kappa = 0} hold. 
Define
\begin{align} \label{eq: Thetatilde forward again kappa=0}
\theta(p,\alpha) 
:= \; & \frac{\alpha \, \varsigma_{-1}(\alpha)}{2 - p} 
= \Big( \frac{2 \alpha}{(\alpha + 2)(2 - p)} \Big) \; \in \; (0,1) , 
\end{align}
where $\varsigma_{-1}(\alpha) = \frac{2}{\alpha + 2}$, 
and 
\begin{align} \label{eq: betamax forward kappa=0}
\alpha(p) := \; & - \frac{2 p}{p + 2} \; > \; 0 .
\end{align}
Then, we have 
\begin{align} \label{eq: alpha bounds for vartheta kappa=0}
0 < \alpha < \alpha(p)
\qquad \Longrightarrow \qquad
p + \alpha \, \varsigma_{-1}(\alpha) < 0 .
\end{align}
In particular, $0 < \alpha < \alpha(p)$ implies that, 
for any $0 < \theta < \theta(p,\alpha)$, 
we have 
\begin{align*}
\theta \, p + \alpha \, \varsigma_{-1}(\alpha) > 2 \theta , \\[.3em]
(1-\theta) \, p < - 2 \theta .
\end{align*} 
Moreover, $p = -7/4$ satisfies the inequalities~\eqref{eq: p and q forward inequality kappa = 0} for all $\lambda_\varepsilon < \maxjumpT{0}$, and in this case, 
\begin{align*} 
\theta(-7/4,\alpha) = \frac{8 \, \alpha}{15 (\alpha + 2)} =: 
\smash{\vartheta_{\alpha}^{0}} 
\; \in \; (0,8/15) ,
\qquad \textnormal{and} \qquad 
\alpha(-7/4) = 14 =: \smash{\alpha_{\lambda_\varepsilon}^{0}} .
\end{align*}
\end{lem}

\begin{proof}
The map $p \mapsto \alpha(p)$~\eqref{eq: betamax forward kappa=0} 
is decreasing when $p \in (-2,0)$. 
Moreover, we have
\begin{align*}
\lim_{p \to -2} \alpha(p) = + \infty
\qquad \textnormal{and} \qquad 
\lim_{p \to 0} \alpha(p) = 0 .
\end{align*} 
This shows that $\alpha(p) > 0$ when $p \in (-2,0)$. 
Next, note that, for fixed $p \in (-2,0)$, the map 
\begin{align*}
\alpha \; \longmapsto \; p + \alpha \, \varsigma_{-1}(\alpha) 
\end{align*}
is increasing, and 
\begin{align*}
p + \alpha \, \varsigma_{-1}(\alpha) = 0 
\qquad \Longleftrightarrow \qquad  
\alpha = \alpha(p) .
\end{align*}
This shows~\eqref{eq: alpha bounds for vartheta kappa=0}. 
Note also that 
if $\alpha < \alpha(p)$, then by~\eqref{eq: alpha bounds for vartheta kappa=0}, we have
\begin{align*}
\theta(p,\alpha) 
\leq \; & 
\underbrace{ \Big( \frac{-2 \alpha}{(\alpha + 2) \, p} \Big)}_{< 1} \, \frac{p}{p-2} 
\; < \; \frac{p}{p-2}  .
\end{align*}
The other claims then follow using the definition~\eqref{eq: Thetatilde forward again kappa=0} of $\theta(p,\alpha)$, and noting that the inequalities~\eqref{eq: p and q forward inequality kappa = 0} 
for all $\lambda_\varepsilon < \maxjumpT{0}$ if and only if $p = - 2^{9} \, \maxjumpT{0} =  2^{9} \, \maxjumpT{0} - 7/2 = -7/4$.
\end{proof}

\begin{prop} \label{prop: summability of event of interest 3}
Fix $T > 0$, $a \in \bR$, and a L\'evy measure $\nu$ 
whose variance measure $\mu_\nu$ satisfies the local upper $\alpha_\nu$-Ahlfors regularity~\eqref{eq: Ahlfors regularity}. 
Fix $\varepsilon > 0$ such that $\lambda_\varepsilon < \maxjumpT{0}$ as in~\eqref{eq: lambdamax forward kappa = 0}. 
Fix also parameter $p \in (-2,0)$ such that the inequalities~\eqref{eq: p and q forward inequality kappa = 0} hold. 
Then, for any $\alpha \in (0,\alpha(p) \wedge \alpha_\nu)$ as in~\eqref{eq: betamax forward kappa=0} and 
for any  
$0 < \theta < \theta(p,\alpha)$ as in~\eqref{eq: Thetatilde forward again kappa=0} and for any $R > 0$, 
on the event $\{ R_\varepsilon^{0,a}(T) \leq R \}$, 
the probabilities of events~\eqref{eq: event of interest kappa = 0} are summable\textnormal{:}
\begin{align} \label{eq: summability of event of interest 3}
\sum_{n=1}^\infty \sum_{z_0 \in \Grid_\varepsilon^{0,a}(2^{-n \theta}, T, R) }
\PR_{R} [ E_n^\theta(z_0) ] 
\; < \; \infty .
\end{align}
\end{prop}

\begin{proof}
Fix $p \in (-2,0)$ as in~\eqref{eq: p and q forward inequality kappa = 0}. 
Fix $z_0 = x_0 + \ii y_0 \in \Grid_\varepsilon^{0,a}(2^{-n \theta}, T, R)$. 
Then, similarly as in the proof of Lemma~\ref{lem: forward bound}, using~(\ref{eq: LE},~\ref{eq: forward Y}), the time-change~\eqref{eq: forward time change}, 
the stopping time~\eqref{eq: bounded stopping time taun}, 
Markov's inequality, 
and the process $\hat{M}(t) = M(\sigma(t))$ defined in~\eqref{eq: general forward martingale candidate kappa=0}, 
we obtain
\begin{align}
\nonumber 
\PR [ E_n^\theta(z_0) ] 
\; \leq \; \; & 
\PR \big[ | \hat{g}_{S_{n}}'(z_0) | \leq \tfrac{80}{9} \, 2^{-n (1-\theta)} 
\textnormal{ and } |\hat{X}(S_{n})| \leq y_0 \, e^{-2S_{n}} \big] \\
\nonumber 
\; \leq \; \; &
\big( \tfrac{80}{9} \big)^{-p} \, 2^{np (1-\theta)} 
\; \EX \big[ | \hat{g}_{S_{n}}'(z_0) |^{p} \; \one{ \{|\hat{X}(S_{n})| \leq y_0 \, e^{-2S_{n}}\}} \big] \\
\nonumber 
\; \leq \; \; & 2 \, \big( \tfrac{80}{9} \big)^{-p} \, 2^{np (1-\theta)} \, e^{a^2 T}
\; \EX \big[ \hat{M}(S_{n}) \; \one{\{|\hat{X}(S_{n})| \leq y_0 \, e^{-2S_{n}}\}} \big] \\
\label{eq: forward bound kappa = 0}
\; \leq \; \; & c_0(p, a, T) \, 2^{np (1-\theta)} \, \EX \big[ \hat{M}(S_{n}) \; \one{\{|\hat{X}(S_{n})| \leq y_0 \, e^{-2S_{n}}\}} \big] ,
\end{align} 
where we also used the fact that 
$(\sin \arg \hat{Z}(S_{n}))^{-2} \leq 2$ on the event $\{|\hat{X}(S_{n})| \leq \hat{Y}(S_{n}) = y_0 \, e^{-2S_{n}}\}$. 
Combining this with 
Proposition~\ref{prop: M forward supermgle case 2 kappa = 0} (with $\hat{M}(0) = \mgle(0)$), and 
the Optional stopping theorem (OST) (e.g.~\cite[Theorem~3.21]{LeGall:BM_book}), we obtain 
\begin{align*} 
\PR [ E_n^\theta(z_0) ] 
\leq \; & c_0(p, a, T) \, 2^{np (1-\theta)} \, 
\EX \big[ \hat{M}(S_{n}) \; \one{\{|\hat{X}(S_{n})| \leq y_0 \, e^{-2S_{n}}\}} \big] 
&& \textnormal{[by~\eqref{eq: forward bound kappa = 0}]} \\
\leq \; & c_0(p, a, T) \, 2^{np (1-\theta)} \, 
\big( \hat{M}(0) + \EX [ \hat{L}(S_{n}) ] \big) .
&& \textnormal{[by~\eqref{eq: M forward supermgle case 2 kappa = 0} and OST]}
\end{align*}
where 
$\hat{M}(0) = y_0^{2} \, |z_0|^{-2}$ and 
\begin{align*}
\hat{L}(S_{n})
= \; & 8  \, ( \lambda_\varepsilon +  c_\nu )  \, 
y_0^{\alpha \, \varsigma_{-1}(\alpha)} 
\int_0^{S_{n}} e^{-2u(p + \alpha \, \varsigma_{-1}(\alpha))} \, \ud u .
\end{align*}

{\bf Term 1.} Using Lemma~\ref{lem: sum over grid} we obtain for the first term the bound 
\begin{align*} 
\sum_{n=1}^\infty \sum_{z_0 \in \Grid_\varepsilon^{0,a}(2^{-n \theta}, T, R) }
2^{np (1-\theta)} \, \hat{M}(0)
\; \leq \; \; & c_1(T, R) \, 
\sum_{n=1}^\infty 2^{np (1-\theta)} \, 2^{2n \theta} \; < \; \infty ,
\end{align*}
where by the choice of $\theta < \theta(p,\alpha)$ and the other parameters, 
Lemma~\ref{lem: alphaprimetilde and thetaprimetilde again kappa=0} shows that $(1-\theta) \, p < - 2 \theta$.

{\bf Term 2.} We then consider the second term. 
By~\eqref{eq: alpha bounds for vartheta kappa=0} in Lemma~\ref{lem: alphaprimetilde and thetaprimetilde again kappa=0}, we have
\begin{align*}
p + \alpha \, \varsigma_{-1}(\alpha) < 0 .
\end{align*} 
Thus, computing the expected value of $\hat{L}(S_{n})$ gives
\begin{align*}
\EX [ \hat{L}(S_{n}) ]
= \; & - 8  \, ( \lambda_\varepsilon +  c_\nu )  \, 
\frac{y_0^{\alpha \, \varsigma_{-1}(\alpha)}}{2(p + \alpha \, \varsigma_{-1}(\alpha))} 
\, \EX \big[ e^{-2 S_{n}(p + \alpha \, \varsigma_{-1}(\alpha))}  - 1 \big] \\
\leq \; & c(p,\alpha,T) \, 2^{-n(p + \alpha \, \varsigma_{-1}(\alpha))} ,
\end{align*}
since $S_{n} \leq \smash{\log \sqrt{\tfrac{y_0}{2^{-n-1}}}}$,
where the constant $c(p,\alpha,T) \in (0,\infty)$ also depends on the L\'evy measure~$\nu$.
Using Lemma~\ref{lem: sum over grid} 
we obtain 
\begin{align*} 
\; & \sum_{n=1}^\infty \sum_{z_0 \in \Grid_\varepsilon^{0,a}(2^{-n \theta}, T, R) } 2^{np (1-\theta)} \, \EX [ \hat{L}(S_{n}) ]
\\
\; \leq \; \; & c(p,\alpha,T) \,
\sum_{n=1}^\infty \sum_{z_0 \in \Grid_\varepsilon^{0,a}(2^{-n \theta}, T, R) } 2^{np (1-\theta)} \, 
2^{-n(p + \alpha \, \varsigma_{-1}(\alpha))} \\
\; \leq \; \; & c(p,\alpha,T) \, c'(T, R)
\sum_{n=1}^\infty 
2^{np (1-\theta)} \, 
2^{-n(p + \alpha \, \varsigma_{-1}(\alpha))}
\, 2^{2 n \theta}
\\
\; \leq \; \; & c(p,\alpha,T) \, c'(T, R)
\sum_{n=1}^\infty 2^{-n ( \theta p + \alpha \, \varsigma_{-1}(\alpha) - 2 \theta )} 
\; < \; \infty ,
\end{align*}
where by the choice of $\theta < \theta(p,\alpha)$ and the other parameters, 
Lemma~\ref{lem: alphaprimetilde and thetaprimetilde again kappa=0}  gives $\theta p + \alpha \, \varsigma_{-1}(\alpha) - 2 \theta > 0$. 
\end{proof}

\bigskip{}
\section{\label{sec: main results}Loewner traces driven by L\'evy processes}
In this section, we arrive at the main results of this article: both
Theorems~\ref{thm: LLE curve if kappa not 8} and~\ref{thm: LLE Holder if kappa not 4} shall be proven in Section~\ref{subsec: conclusion}. 
To this end, we first gather the needed estimates in Sections~\ref{subsec: main results micro}--\ref{subsec: main results macro}.

\subsection{Loewner traces with martingale L\'evy drivers}
\label{subsec: main results micro}

To begin with, as in Sections~\ref{sec: backward bounds}~\&~\ref{sec: forward bounds}, 
we fix a L\'evy measure $\nu$ 
and consider martingale driving functions of the form
\begin{align} \label{eq: BM with micro again} 
\microDriver_\varepsilon^{\kappa}(t) 
= \sqrt{\kappa} B(t) + \int_{|\jump| \leq \varepsilon} \jump \, \PoissonComp(t, \ud \jump) , \qquad
\kappa \geq 0 , \; \varepsilon > 0 . 
\end{align} 
Let $(g_t)_{t \geq 0}$ be a Loewner chain driven by $\microDriver_\varepsilon^{\kappa}$, let $f_t:= g_t^{-1}$ be the inverse Loewner chain, 
let $h_t$ be the solution to the mirror backward Loewner equation~\eqref{eq: mBLE} driven by $\microDriver_\varepsilon^{\kappa}$, and set
\begin{align*}  
\tilde{f}_t(w) := f_t(w + \microDriver_\varepsilon^{\kappa}(t)) , \qquad w \in \bH .
\end{align*}

\subsubsection{Boundary regularity}

H\"older continuity of the mirror backward Loewner chain (Proposition~\ref{prop: BLE micro jumps Holder}) implies via 
Lemma~\ref{lem: f' = h' in distribution} H\"older continuity of the inverse Loewner chain $f_t$. 
However, since Lemma~\ref{lem: f' = h' in distribution} only applies for a fixed time instant $t$, the result in Proposition~\ref{prop: LLE micro jumps Holder} only holds pointwise in time.

\begin{prop} \label{prop: LLE micro jumps Holder}
Fix $t \geq 0$, $\kappa \in [0,\infty) \setminus\{4\}$, a L\'evy measure $\nu$, and $\varepsilon > 0$ 
such that\footnote{Note that by Lemma~\ref{lem: alphaprime and thetaprime} and the assumptions, we have $\maxjumpH{\kappa} > 0$ and $\ThetaHmax{\kappa}{\lambda_\varepsilon} \in (0, 2/5)$ for $\kappa \neq 4$. 
Hence, it is possible to choose $\varepsilon > 0$ 
such that $\lambda_\varepsilon < \maxjumpH{\kappa}$ and $\theta = \theta(\kappa, \lambda_\varepsilon) \in (0, 2/5)$ as claimed.} 
$\lambda_\varepsilon < \maxjumpH{\kappa}$ as in~\eqref{eq: lambdatildemax}.
Then, 
the following hold almost surely for the Loewner chain driven by $\microDriver_\varepsilon^{\kappa}$~\eqref{eq: BM with micro again}.
\begin{enumerate}[label=\textnormal{(\alph*):}, ref=(\alph*)]
\item \label{item: LLE micro jumps locally connected}
The map $z \mapsto f_t(z)$ extends to a continuous function $f_t \colon \overline{\bH} \to \overline{\bH \setminus K_t}$. 

\medskip

\item \label{item: LLE micro jumps Holder}
$\bH \setminus K_t$ is a H\"older domain: 
there exists a constant $\theta = \theta(\kappa, \lambda_\varepsilon) \in (0, 2/5)$ and a random constant $H(\theta,t) \in (0,\infty)$ such that 
\begin{align} \label{eq: Holder property for f}
| f_t(z) - f_t(w) | \leq H(\theta,t) \, 
\big( \, | z-w |^{\theta} \, \vee \, | z-w | \, \big) 
\qquad \textnormal{for all } z,w \in \bH .
\end{align}

\medskip

\item \label{item: LLE micro jumps Hdim}
The Hausdorff dimension of $\bdry K_t$ satisfies $\mathrm{dim} (\bdry K_t) < 2$, and we have $\mathrm{area} (\bdry K_t) = 0$. 
\end{enumerate}
\end{prop}

\begin{proof}
Item~\ref{item: LLE micro jumps Holder}
implies both item~\ref{item: LLE micro jumps locally connected}, 
giving the boundary of $\bH \setminus K_t$ a $\theta$-H\"older continuous parameterization, and 
item~\ref{item: LLE micro jumps Hdim} 
by the result \cite[Theorem~C.2]{Jones-Makarov:Density_properties_of_harmonic_measure} in the unit disc $\bD$, after conjugating with a conformal map between $\bD$~and~$\bH$. 
So, it suffices to prove~\ref{item: LLE micro jumps Holder}. 
Note that the conformal map $z \mapsto f_t(z)$ is asymptotically the identity near $z=\infty$ (by~\eqref{eq: inverse mof Laurent exp}). 
Hence,~\eqref{eq: Holder property for f} follows from Lemma~\ref{lem: f' = h' in distribution} together with 
Proposition~\ref{prop: BLE micro jumps Holder} 
with large $R > 1$: 
the map $z \mapsto \smash{\tilde{f}_t(z)} - \smash{\microDriver_\varepsilon^{\kappa}(t)}$ has the same distribution as $z \mapsto h_t(z)$, and since the translation by $\smash{\microDriver_\varepsilon^{\kappa}(t)}$ 
makes no difference to~\eqref{eq: Holder property for f},
we establish item~\ref{item: LLE micro jumps Holder}.
\end{proof}

\begin{rem}
When $\kappa = 4$, the 
arguments in Section~\ref{subsec: Holder continuity} 
are not strong enough to conclude the estimate in Lemma~\ref{lem: uniform spatial tail estimate for h}, 
and hence, not strong enough to obtain H\"older continuity of $h_t$ and $f_t$.
In fact, it is known that for $\SLEk$ with $\kappa = 4$,
the complements of the Loewner hulls are not H\"older domains~\cite{GMS-Almost_sure_multifractal_spectrum_of_SLE}, 
and we do not expect the Loewner chain driven by $\microDriver_\varepsilon^{\kappa}$ with $\kappa = 4$ to have this property either\footnote{However, under~\ref{item: ass2}, item~\ref{item: LLE micro jumps locally connected}
of Proposition~\ref{prop: LLE micro jumps Holder} also holds for $\kappa=4$, see Theorem~\ref{thm: LLE curve if kappa not 8}.}.
\end{rem}

\noindent 
\begin{figure}[ht!]
\includegraphics[width=.4\textwidth]{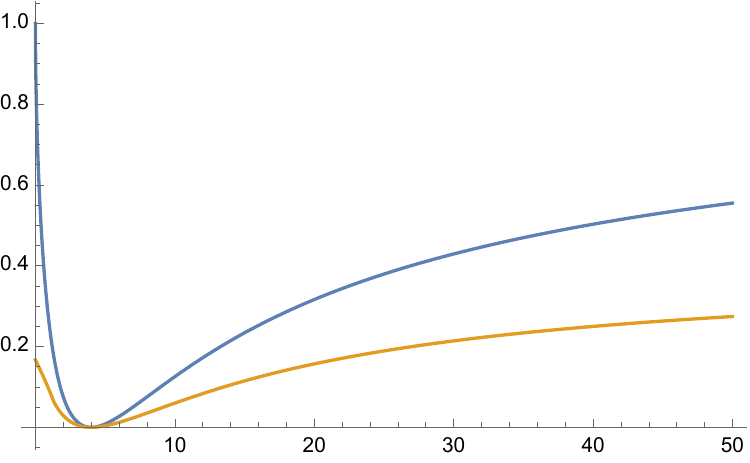}
\caption{\label{fig: Holder exp}
Plots of the quantities~\eqref{eq: inverse map param holder} (orange) and~\eqref{eq: inverse map param holder optimal} (blue), 
which is the known H\"older exponent for $f_t$ 
with driving function $W = \sqrt{\kappa} B$. }
\end{figure}

\begin{rem} \label{rem: inverse map param holder optimal}
In the case of no jumps \textnormal{(}$\lambda_\varepsilon = 0$\textnormal{)}, 
Proposition~\ref{prop: LLE micro jumps Holder} together with 
Lemma~\ref{lem: alphaprime and thetaprime} 
implies that the inverse Loewner map $f_t$ sending $\bH$ to the complement of the 
$\SLE_\kappa$ hull is H\"older continuous with any exponent strictly smaller than~\eqref{eq: Thetatilde}\textnormal{:} 
\begin{align} \label{eq: inverse map param holder}
\ThetaHmax{\kappa}{0}
= \; &
\begin{cases}
1 + \frac{10}{\kappa - 12} , &
0 \leq \kappa \leq 4/3 ,  \\[.3em]
\frac{2 (\kappa - 4 )^2 }{(\kappa + 4) (5 \kappa + 36)}  , &
4/3 < \kappa . 
\end{cases}
\end{align}
Comparing with the optimal exponent 
\begin{align} \label{eq: inverse map param holder optimal}
1 - \frac{4 \kappa + 2\sqrt{2}\sqrt{\kappa (\kappa+2) (\kappa+8)}}{(4+\kappa)^2} 
\end{align}
found by Lind~\cite{Lind:Holder_regularity_of_the_SLE_trace} and later 
proven by Gwynne, Miller, and Sun~\cite{GMS-Almost_sure_multifractal_spectrum_of_SLE}, 
we see that~\eqref{eq: inverse map param holder} is always less than~\eqref{eq: inverse map param holder optimal}, with equality only at $\kappa = 4$, where both exponents vanish.
See also Figure~\ref{fig: Holder exp}. 
Note that 
\begin{align*}
\underset{\kappa \to 0}{\lim} \ThetaHmax{\kappa}{0} = 1/6
\qquad \textnormal{and} \qquad
\underset{\kappa \to \infty}{\lim} \ThetaHmax{\kappa}{0} = 2/5 ,
\end{align*}
while 
\begin{align*}
\underset{\kappa \to 0}{\lim} \Big(\smash{1 - \frac{4 \kappa + 2\sqrt{2}\sqrt{\kappa (\kappa+2) (\kappa+8)}}{(4+\kappa)^2}}\Big) = 1 
\qquad \textnormal{and} \qquad
\underset{\kappa \to \infty}{\lim} \Big(\smash{1 - \frac{4 \kappa + 2\sqrt{2}\sqrt{\kappa (\kappa+2) (\kappa+8)}}{(4+\kappa)^2}}\Big) = 1 .
\end{align*}
\end{rem}

\begin{rem}
The boundary of any H\"older domain is conformally removable~\cite[Corollary~2]{Jones-Smirnov:Removability_theorems_for_Sobolev_functions_and_quasiconformal_maps},
but this is not at all clear for other kinds of fractals. 
For $\SLEk$ with $\kappa = 4$, it was proven only very recently 
that the $\SLE_4$ curve is indeed conformally removable, using couplings of $\SLE$ with the Gaussian free field~\cite{KMS:Conformal_removability_of_SLE4}.
We do not foresee that those techniques could be adapted as such to the present setup. 
\end{rem}

\subsubsection{C\`adl\`ag Loewner trace}

Recall that the existence of the  (c\`adl\`ag) Loewner trace (in the sense of Definition~\ref{def: Loewner trace}) 
is subject to one of the following assumptions:
\begin{enumerate}[label=\textnormal{\bf{Ass.~\arabic*.}}, ref=Ass.~\arabic*.]
\item 
either the diffusivity parameter $\kappa > 8$,

\medskip

\item 
or the diffusivity parameter $\kappa \in [0,8)$, and 
the variance measure of the L\'evy measure $\nu$ is locally (upper) Ahlfors regular near the origin in the sense of Definition~\ref{def: Ahlfors regular}.
\end{enumerate}
\noindent 
Under these assumptions, the main result of Section~\ref{sec: forward bounds} gives the existence of the Loewner trace:

\TraceExistsThm*

\noindent 
\begin{figure}[ht!]
\includegraphics[width=.4\textwidth]{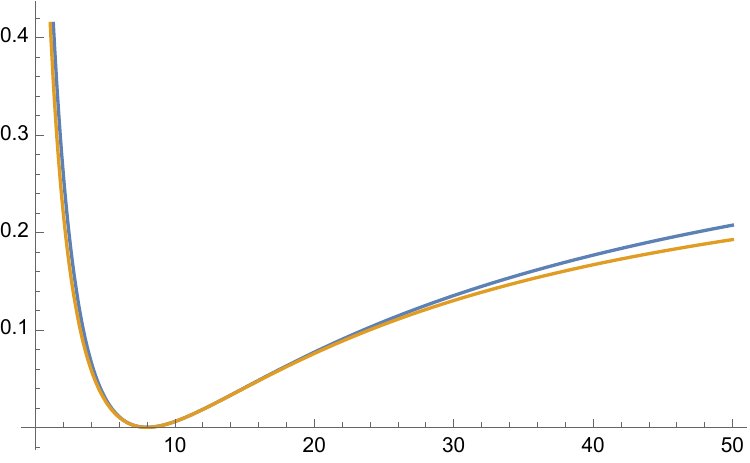}
\caption{\label{fig: Holder exp curve}
Plots of the quantities~\eqref{eq: capacity param holder} (orange) and~\eqref{eq: capacity param holder optimal} (blue), 
which is the known H\"older exponent for the (capacity parameterized) $\SLEk$ curve. 
}
\end{figure}

\begin{rem}
In the case of no jumps \textnormal{(}$\lambda_\varepsilon = 0$\textnormal{)}, 
the estimates derived in Propositions~\ref{prop: needed uniform bound for f' 1}~\&~\ref{prop: needed uniform bound for f' 2} while proving Theorem~\ref{thm: LLE micro jumps curve} 
hint that the $\SLEk$ curve parameterized by capacity would be H\"older continuous with any exponent strictly smaller than 
\begin{align} \label{eq: capacity param holder}
\tfrac{1}{2} \, \ThetaTmax{\kappa}{0}
= \begin{cases}
\frac{(8 - \kappa )^2}{\kappa  (\kappa + 48) + 64} , & 0 < \kappa < 8, \\[.5em]
\frac{(\kappa - 8)^2}{(\kappa + 8) (3 \kappa + 8)}  , & \kappa > 8 .
\end{cases}
\end{align}
Comparing with the optimal exponent 
\begin{align} \label{eq: capacity param holder optimal}
1 - \frac{\kappa}{2 \kappa + 24 - 8 \sqrt{\kappa+8}} 
\end{align}
found by Lind~\cite{Lind:Holder_regularity_of_the_SLE_trace} and proven by Johansson-Viklund \& Lawler~\cite{Lawler-Viklund:Optimal_Holder_exponent_for_the_SLE_path},
we see that~\eqref{eq: capacity param holder} is always less than~\eqref{eq: capacity param holder optimal}, with equality only at $\kappa = 8$, where both exponents vanish.
See also Figure~\ref{fig: Holder exp curve}.
Note that 
\begin{align*}
\underset{\kappa \to 0}{\lim} \tfrac{1}{2} \, \ThetaTmax{\kappa}{\lambda_\varepsilon} = 1
\qquad \textnormal{and} \qquad
\underset{\kappa \to \infty}{\lim} \tfrac{1}{2} \, \ThetaTmax{\kappa}{\lambda_\varepsilon} = 1/3 ,
\end{align*}
while 
\begin{align*}
\underset{\kappa \to 0}{\lim} \Big(1 - \frac{\kappa}{2 \kappa + 24 - 8 \sqrt{\kappa+8}}\Big) = 1
\qquad \textnormal{and} \qquad
\underset{\kappa \to \infty}{\lim} \Big(1 - \frac{\kappa}{2 \kappa + 24 - 8 \sqrt{\kappa+8}}\Big) = 1/2 .
\end{align*}
The exponent~\eqref{eq: capacity param holder optimal} 
was also derived recently by Yuan~\cite{Yuan:Refined_regularity_of_SLE} 
by arguments similar to those in Section~\ref{sec: forward bounds}.
\end{rem}

\subsection{Adding linear drift to the driver}
\label{subsec: main results micro with drift}

In this section, we consider driving functions with microscopic jumps and a linear drift:
\begin{align} \label{eq: BM with micro and drift}
\microDriver_\varepsilon^{\kappa, a}(t) = a t + \sqrt{\kappa} B(t) + \int_{|\jump| \leq \varepsilon} \jump \, \PoissonComp(t, \ud \jump) , \qquad
a \in \bR, \; \kappa \geq 0 , \;  \varepsilon > 0 . 
\end{align}
We will show that Proposition~\ref{prop: LLE micro jumps Holder} and Theorem~\ref{thm: LLE micro jumps curve} also hold for the driving function~$\microDriver_\varepsilon^{\kappa, a}$.

We first extend Theorem~\ref{thm: LLE micro jumps curve} to the case of~$\microDriver_\varepsilon^{\kappa, a}$.  
The most naive idea is to introduce the drift to~\eqref{eq: BM with micro and drift} in terms of a suitable absolutely continuous change of measure from the driftless case (Girsanov's theorem).
In this way, we can easily treat the case where $\kappa>0$, including a Brownian component. 
However, for a general pure jump process with $\kappa=0$, one cannot expect to recover a linear drift using a change of measure --- for increasing L\'evy processes, adding a negative drift would give a singular change of measure. We can use the arguments of Sections~\ref{subsec: kappa is zero} 
and~\ref{subsec: Summability 3} to deal with this case. 

\begin{prop} \label{prop: LLE micro jumps curve with drift}
Fix $T > 0$, $\kappa \in [0,\infty) \setminus\{8\}$, $a \in \bR$, and a L\'evy measure $\nu$. Suppose that either~\ref{item: ass1} or~\ref{item: ass2} holds, and fix $\varepsilon > 0$ 
such that $\lambda_\varepsilon < \maxjumpT{\kappa}$ as in~\textnormal{(\ref{eq: lambda max},~\ref{eq: lambdamax forward kappa = 0})} and $\varepsilon \in (0,\epsilon_\nu \wedge \tfrac{1}{2} \rho_\nu]$ under~\ref{item: ass2}. 
Then, the Loewner chain driven by $\microDriver_\varepsilon^{\kappa, a}$~\eqref{eq: BM with micro and drift} 
is almost surely generated by a c\`adl\`ag curve on $[0, T]$.
\end{prop}

\begin{proof}
We consider separately the cases where $\kappa > 0$ and $\kappa = 0$.
\begin{itemize}
\item[$\kappa>0$:]
We use Girsanov's theorem (cf.~\cite[Corollary~15.3.4]{Cohen-Elliott:Stochastic_calculus_and_applications})
to compare with the case of no drift ($a=0$). 
Specifically, if 
$\smash{\PR_\varepsilon^{a, \kappa}}$ is the probability measure of $\smash{\microDriver_\varepsilon^{\kappa, a}}$,
by Girsanov's theorem we find a new probability measure $\smash{\GirsPR_\varepsilon^{a, \kappa}}$, mutually absolutely continuous with $\smash{\PR_\varepsilon^{a, \kappa}}$, such that the process $\smash{\GirsB(t) := B(t) + \frac{a}{\sqrt{\kappa}} t}$
is a $\smash{\GirsPR_\varepsilon^{a, \kappa}}$-Brownian motion.
Such a measure change leaves Poisson point processes invariant, since they are independent of the Brownian motion in the L\'evy-It\^o decomposition. 
In particular, we see that under $\smash{\GirsPR_\varepsilon^{a, \kappa}}$, the process $\smash{\microDriver_\varepsilon^{\kappa, a}}$ has the same distribution as
$\smash{\microDriver_\varepsilon^{\kappa}}$. 
Thus, the claim follows from Theorem~\ref{thm: LLE micro jumps curve}.

\medskip

\item[$\kappa=0$:]
The claim follows similarly as 
Theorem~\ref{thm: LLE micro jumps curve} by arguments presented in Section~\ref{subsec: Summability 3}. 
\qedhere
\end{itemize}
\end{proof}

Next, we extend Proposition~\ref{prop: LLE micro jumps Holder} to the case of $\microDriver_\varepsilon^{\kappa, a}$.

\begin{prop} \label{prop: BLE micro jumps Holder with drift kappa > 0}
Fix $T, R > 0$, 
$\kappa \in (0,\infty) \setminus\{4\}$, $a \in \bR$, a L\'evy measure $\nu$, and $\varepsilon > 0$ 
such that 
$\lambda_\varepsilon < \maxjumpH{\kappa}$ as in~\eqref{eq: lambdatildemax}.
Let $(h_t)_{t \geq 0}$ be the solution to~\eqref{eq: mBLE} driven by 
$\microDriver_\varepsilon^{\kappa, a}$~\eqref{eq: BM with micro and drift}. 
Then, there exists a constant $\theta = \theta(\kappa, \lambda_\varepsilon) \in (0, 1)$ 
and an almost almost surely finite random constant $H(\theta,T,R)$ such that 
\begin{align*} 
| h_t(z) - h_t(w) | \leq H(\theta,T,R) \, 
\big( \, | z-w |^{\theta} \, \vee \, | z-w | \, \big) 
\qquad \textnormal{for all } t \in [0, T]
\textnormal{ and } (-R,R) \times \ii (0,\infty)  .
\end{align*}
In particular, almost surely, each $h_t$ extends to a continuous function on $\smash{\overline{(-R,R) \times \ii (0,\infty)}}$. 
\end{prop}

\begin{proof}
This follows from Girsanov's theorem \& Proposition~\ref{prop: BLE micro jumps Holder} (as in the proof of Proposition~\ref{prop: LLE micro jumps curve with drift}).
\end{proof}

\begin{cor} \label{cor: LLE micro jumps Holder with drift}
Fix $t > 0$, $\kappa \in [0,\infty) \setminus\{4\}$, $a \in \bR$, a L\'evy measure $\nu$, and  $\varepsilon > 0$ 
such that 
$\lambda_\varepsilon < \maxjumpH{\kappa} \wedge 1$ as in~\eqref{eq: lambdatildemax}.
Then, the following hold almost surely for the Loewner chain driven by $\microDriver_\varepsilon^{\kappa, a}$~\eqref{eq: BM with micro and drift}.

\begin{enumerate}[label=\textnormal{(\alph*):}, ref=(\alph*)]
\item \label{item: locally connected with drift}
$z \mapsto f_t(z)$ extends to a continuous function $f_t \colon \overline{\bH} \to \overline{\bH \setminus K_t}$, 

\medskip

\item \label{item: Holder with drift}
$\bH \setminus K_t$ is a H\"older domain in the sense of~\eqref{eq: Holder property for f}, and 

\medskip

\item \label{item: Haus dim with drift}
$\mathrm{dim} (\bdry K_t) < 2$ and $\mathrm{area} (\bdry K_t) = 0$.
\end{enumerate}
\end{cor}

\begin{proof}
This follows similarly as Proposition~\ref{prop: LLE micro jumps Holder}, by 
using Lemma~\ref{lem: f' = h' in distribution} and Propositions~\ref{prop: BLE micro jumps Holder with drift kappa > 0} \&~\ref{prop: BLE micro jumps Holder with drift kappa = 0}. 
\end{proof}

\subsection{Adding macroscopic jumps to the driver}
\label{subsec: main results macro}

We now consider general driving functions\footnote{Note that with $\varepsilon \to 1$, we obtain the general driving function~\eqref{eq: general DF}.} 
\begin{align} \label{eq: BM with micro macro and drift}
\macroDriver_\varepsilon^{\kappa, a}(t) = at + \sqrt{\kappa} B(t) + \int_{|\jump| \leq \varepsilon} \jump \, \PoissonComp(t, \ud \jump) + \int_{|\jump| > \varepsilon} \jump \, \Poisson(t, \ud \jump) , \qquad
a \in \bR, \; \kappa \geq 0 , \;  \varepsilon > 0 . 
\end{align}

The next result extends Proposition~\ref{prop: LLE micro jumps curve with drift} to the case where the driving process may have macroscopic jumps.
This extension is essentially a direct consequence of the domain Markov property,  that follows from 
Lemma~\ref{lem: some markov property} and the strong Markov property of the driving function $\macroDriver_\varepsilon^{\kappa, a}$.
However, there is an important subtlety: 
we need in addition to make sure that the (inverse) mapping-out functions $f_t$ extend continuously to the boundary \emph{at all times simultaneously}, which is guaranteed by Proposition~\ref{prop: locally connected}.

\begin{prop} \label{prop: generated by a curve still containing epsilon}
Fix $T > 0$, $\kappa \in [0,\infty) \setminus\{8\}$, $a \in \bR$, and a L\'evy measure $\nu$. Suppose that either~\ref{item: ass1} or~\ref{item: ass2} holds, and fix $\varepsilon > 0$ 
such that $\lambda_\varepsilon < \maxjumpT{\kappa}$ as in~\eqref{eq: lambda max} and $\varepsilon \in (0,\epsilon_\nu \wedge \tfrac{1}{2} \rho_\nu]$ under~\ref{item: ass2}. 
Then, the following hold almost surely for the Loewner chain driven by $\macroDriver_\varepsilon^{\kappa, a}$~\eqref{eq: BM with micro macro and drift}.
\begin{enumerate}[label=\textnormal{(\alph*):}, ref=(\alph*)]
\item \label{item: LLE curve if kappa not 8 still containing epsilon}
The Loewner chain is generated by a c\`adl\`ag curve on $[0, T]$.

\medskip

\item \label{item: LLE locally conn if kappa not 8 still containing epsilon}
For each $t \in [0, T]$, the map $z \mapsto f_t(z)$ extends to a continuous function $f_t \colon \overline{\bH} \to \overline{\bH \setminus K_t}$. 
\end{enumerate}
\end{prop}

\begin{proof}
Note that
\begin{align} \label{eq: macroDriver in terms of microDriver}
\macroDriver_\varepsilon^{\kappa, a}(t)
= \microDriver_\varepsilon^{\kappa, a}(t)
+ \int_{|\jump| > \varepsilon} \jump \, \Poisson(t, \ud \jump)  , \qquad t \geq 0 ,
\end{align}
in distribution. 
Since $\nu$ is a L\'evy measure, $\smash{\macroDriver_\varepsilon^{\kappa, a}}$ has 
almost surely finitely many jumps of size larger than $\varepsilon$ on $[0, T]$. 
Fix $n \in \bN$.
Then, on the event $E_n$ that $\smash{\macroDriver_\varepsilon^{\kappa, a}}$ has exactly $n$ jumps of size larger than $\varepsilon$ on $[0, T]$ 
occurring at stopping times $0 \leq \tau_1 < \cdots < \tau_n < T$, the strong Markov property gives
\begin{align} \label{eq: macroDriver Markov property}
\begin{split}
\macroDriver_\varepsilon^{\kappa, a}(t) = 
\begin{cases}
\microDriver_\varepsilon^{\kappa, a}(t)
, & t \in [0, \tau_1) , \\
\macroDriver_\varepsilon^{\kappa, a}(\tau_1) + \microDriver_\varepsilon^{\kappa, a}(t - \tau_1)
, & t \in [\tau_1, \tau_2) , \\
\qquad\qquad \;\, \vdots & \quad \vdots \\
\macroDriver_\varepsilon^{\kappa, a}(\tau_n) + \microDriver_\varepsilon^{\kappa, a}(t - \tau_n)
, & t \in [\tau_n, T] ,
\end{cases}
\end{split}
\end{align}
in distribution, where the pieces are independent. 
For definiteness, we also write $\tau_0 := 0$ and $\tau_{n+1} := T$.
We then define for each $\ell \in \{0,1,\ldots,n\}$ the 
hulls 
\begin{align} \label{eq: macroDriver hulls}
\mathring{K}^{\tau_\ell}_s
:= & \; \overline{g_{\tau_\ell}(K_{\tau_\ell + s} \setminus K_{\tau_\ell}) - \macroDriver_\varepsilon^{\kappa, a}(\tau_\ell)} , \qquad s \geq 0 .
\end{align}
By Lemma~\ref{lem: some markov property}, each mapping-out function 
\begin{align} \label{eq: macroDriver Loewner map}
\mathring{g}^{\tau_\ell}_s(z) 
:= g_{\mathring{K}^{\tau_\ell}_s} (z)
= ( g_{\tau_\ell + s} \circ f_{\tau_\ell} )(z + \macroDriver_\varepsilon^{\kappa, a}(\tau_\ell)) - \macroDriver_\varepsilon^{\kappa, a}(\tau_\ell)
\end{align} 
solves~\eqref{eq: LE} with driving function  
$\smash{(\mathring{\macroDriver}_{\varepsilon}^{\kappa, a} )^{\tau_\ell}}(s) := \macroDriver_\varepsilon^{\kappa, a}(\tau_\ell + s) - \macroDriver_\varepsilon^{\kappa, a}(\tau_\ell)$.
Note that $\smash{(\mathring{\macroDriver}_{\varepsilon}^{\kappa, a} )^{\tau_\ell}}(s)$ has the same distribution as 
$\microDriver_\varepsilon^{\kappa, a}(s)$
when $s \in [0, \tau_{\ell+1} - \tau_\ell)$, 
and also their left limits at $\tau_{\ell+1} - \tau_\ell$
have the same distribution. 
Hence, Proposition~\ref{prop: LLE micro jumps curve with drift} implies that 
almost surely (on the event $E_n$), for all $\ell \in \{0,1,\ldots,n\}$, 
the Loewner chain driven by $\smash{(\mathring{\macroDriver}_{\varepsilon}^{\kappa, a} )^{\tau_\ell}}$ 
is  generated by a c\`adl\`ag curve $\smash{\newmathringcadlag_\ell}$ on $[0, \tau_{\ell+1} - \tau_\ell]$. 
We now define 
\begin{align*}
\cadlag(t) := \; & 
\begin{cases}
\cadlag_0(t)
, & t \in [0, \tau_1) , \\
\cadlag_1(t)
, & t \in [\tau_1, \tau_2) , \\
\quad \vdots & \quad \vdots  \\
\cadlag_{n}(t)
, & t \in [\tau_n, T] ,
\end{cases}
\end{align*}
which is a concatenation of the c\`adl\`ag curves (we shall verify below that these are well-defined)
\begin{align*}
\cadlag_\ell \colon [\tau_\ell, \tau_{\ell+1}] \to \bH , \qquad
\cadlag_\ell (t) := f_{\tau_\ell}(\smash{\newmathringcadlag_\ell}(t - \tau_\ell) + \macroDriver_\varepsilon^{\kappa, a}(\tau_\ell)) .
\end{align*}
By the strong Markov property and Lemma~\ref{lem: some markov property}, each $\cadlag_\ell$ generates the Loewner chain driven by $\macroDriver_\varepsilon^{\kappa, a}$ on $[\tau_\ell, \tau_{\ell+1}] \ni t$, and $\cadlag$ is a c\`adl\`ag curve that generates 
the Loewner chain driven by $\macroDriver_\varepsilon^{\kappa, a}$ on $[0, T]$.

To see that $\cadlag$ is well-defined, 
we argue recursively that almost surely (on the event $E_n$), 
the inverse Loewner map $f_\sigma = g_\sigma^{-1}$ extends continuously to $\overline{\bH}$ at each time $\sigma \in \{ \tau
_1, \tau_2, \ldots, \tau_n \}$.
Indeed, Proposition~\ref{prop: locally connected} implies that almost surely (on the event $E_n$), 
at each time $t \in [0, \tau_1]$, 
the boundary $\bdry (\bH \setminus K_t)$ 
for the hull $K_t$ generated by the c\`adl\`ag curve $\cadlag_0[0, t]$
is locally connected, which implies by Carath\'eodory's theorem 
(see~\cite[Chapter~2]{Pommerenke:Boundary_behaviour_of_conformal_maps})
that
$f_{\tau_1}$ extends continuously to the real line.
Knowing this, we see that $\cadlag$ is a well-defined c\`adl\`ag curve
on the time interval $[0, \tau_2]$.
Similarly, for $\ell = 2,3,\ldots,n$, applying Proposition~\ref{prop: locally connected} 
to the boundary $\bdry (\bH \setminus K_t)$ 
for the hull $K_t$ generated by the c\`adl\`ag curve 
$\cadlag[0, \tau_\ell]$, we see that $f_{\tau_\ell}$ extends continuously to the real line for each $\ell = 2,3,\ldots,n$.

This proves~\ref{item: LLE curve if kappa not 8 still containing epsilon} on the event $E_n$, and Proposition~\ref{prop: locally connected} then implies~\ref{item: LLE locally conn if kappa not 8 still containing epsilon} by Carath\'eodory's theorem.
Taking the disjoint union of the countably many events $E_n$ over $n \in \bZnn$ concludes the proof.
\end{proof}

Our next aim is to extend Corollary~\ref{cor: LLE micro jumps Holder with drift} 
to the case where the driving process may have macroscopic jumps.
This extension is again a consequence of the domain Markov property.
In addition, we need uniform H\"older continuity of the mirror backward Loewner chain $(h_t)_{t \geq 0}$
driven by $\macroDriver_\varepsilon^{\kappa, a}$ 
for $t$ in compact time intervals, given in 
Proposition~\ref{prop: BLE micro jumps Holder still containing epsilon} (see also~\cite[Corollary~5.4]{Chen-Rohde:SLE_driven_by_symmetric_stable_processes}). 
Note that we do not have uniformly-in-time H\"older continuity of the inverse Loewner chain $(f_t)_{t \geq 0}$ at our disposal.

\begin{prop} \label{prop: BLE micro jumps Holder still containing epsilon}
Fix $T, R > 0$, 
$\kappa \in [0,\infty) \setminus\{4\}$, $a \in \bR$, a L\'evy measure $\nu$, and $\varepsilon > 0$ 
such that $\lambda_\varepsilon < \maxjumpH{\kappa}  \wedge 1$ as in~\eqref{eq: lambdatildemax}.
Let $(h_t)_{t \geq 0}$ be the solution to~\eqref{eq: mBLE} driven by 
$\macroDriver_\varepsilon^{\kappa, a}$~\eqref{eq: BM with micro macro and drift}.  
Then, almost surely, there exist random 
constants $\theta = \theta(\kappa, \lambda_\varepsilon, T) \in (0, 1)$ 
and $H(\theta,T,R) \in (0,\infty)$ such that
\begin{align*} 
| h_t(z) - h_t(w) | \leq H(\theta,T,R) \, 
\big( \, | z-w |^{\theta} \, \vee \, | z-w | \, \big) 
\qquad \textnormal{for all } t \in [0, T]
\textnormal{ and } z,w \in (-R,R) \times \ii (0,\infty) . 
\end{align*}
In particular, almost surely, each $h_t$ extends to a continuous function on $\smash{\overline{(-R,R) \times \ii (0,\infty)}}$. 
\end{prop}

\begin{proof}
We proceed similarly 
as in the proof of Proposition~\ref{prop: generated by a curve still containing epsilon}, with the same notation.
Fix $n \in \bN$.
On the event $E_n$, using identities~(\ref{eq: macroDriver in terms of microDriver},~\ref{eq: macroDriver Markov property},~\ref{eq: macroDriver hulls}), 
we see that since $\smash{(\mathring{\macroDriver}_{\varepsilon}^{\kappa, a} )^{\tau_\ell}}(s)$ has the same distribution as 
$\microDriver_\varepsilon^{\kappa, a}(s)$
when $s = t - \tau_\ell \in [0, \tau_{\ell+1} - \tau_\ell)$, 
and also their left limits at $\tau_{\ell+1} - \tau_\ell$ 
have the same distribution, 
Lemma~\ref{lem: some markov property for h} shows that each 
\begin{align*}
\mathring{h}^{\tau_\ell}_{s}(z) := ( h_{\tau_\ell + s} \circ h_{\tau_\ell}^{-1} )(z - \macroDriver_\varepsilon^{\kappa, a}(\tau_\ell)) + \macroDriver_\varepsilon^{\kappa, a}(\tau_\ell) 
\end{align*}
solves~\eqref{eq: mBLE} with driving function
$\smash{(\mathring{\macroDriver}_{\varepsilon}^{\kappa, a} )^{\tau_\ell}}(s) := \macroDriver_\varepsilon^{\kappa, a}(\tau_\ell + s) - \macroDriver_\varepsilon^{\kappa, a}(\tau_\ell)$.
In other words, we have
$h_t (z) = \smash{\mathring{h}^{\tau_\ell}_{t - \tau_\ell}} ( h_{\tau_\ell} (z) + \macroDriver_\varepsilon^{\kappa, a}(\tau_\ell) )- \macroDriver_\varepsilon^{\kappa, a}(\tau_\ell)$ when
$t \in [\tau_\ell, \tau_{\ell+1})$. 
Iterating this observation, we obtain
\begin{align*}
h_t (z) 
= \; &  \Big( \smash{\mathring{h}^{\tau_\ell}_{t - \tau_\ell}} ( \cdot + \macroDriver_\varepsilon^{\kappa, a}(\tau_\ell) )- \macroDriver_\varepsilon^{\kappa, a}(\tau_\ell) \Big) 
\circ 
\Big( \smash{\mathring{h}^{\tau_{\ell-1}}_{\tau_\ell - \tau_{\ell-1}}} ( \cdot + \macroDriver_\varepsilon^{\kappa, a}(\tau_{\ell-1}) )- \macroDriver_\varepsilon^{\kappa, a}(\tau_{\ell-1}) \Big) 
\circ \cdots  \\
\; & \cdots \circ 
\Big( \smash{\mathring{h}^{\tau_{1}}_{\tau_2 - \tau_{1}}} ( \cdot + \macroDriver_\varepsilon^{\kappa, a}(\tau_{1}) )- \macroDriver_\varepsilon^{\kappa, a}(\tau_{1}) \Big) 
\circ 
h_{\tau_1} ( \cdot ), \qquad t \in [\tau_\ell, \tau_{\ell+1}) .
\end{align*}
Now, each map in this composition is H\"older continuous
almost surely (on the event $E_n$) by 
Propositions~\ref{prop: BLE micro jumps Holder with drift kappa > 0}
and~\ref{prop: BLE micro jumps Holder with drift kappa = 0}. 
Thus, almost surely on the event $E_n$,
the composed map is H\"older continuous as well. 
Taking the disjoint union of the countably many events $E_n$ over $n \in \bZnn$ concludes the proof.
\end{proof}

\begin{cor} \label{cor: locally connected still containing epsilon}
Fix $t > 0$, $\kappa \in [0,\infty) \setminus\{4\}$, $a \in \bR$, a L\'evy measure $\nu$, and  $\varepsilon > 0$ 
such that 
$\lambda_\varepsilon < \maxjumpH{\kappa} \wedge 1$ as in~\eqref{eq: lambdatildemax}.
Then, the following hold almost surely for the Loewner chain driven by $\macroDriver_\varepsilon^{\kappa, a}$~\eqref{eq: BM with micro macro and drift}.
\begin{enumerate}[label=\textnormal{(\alph*):}, ref=(\alph*)]
\item \label{item: LLE locally conn if kappa not 4 still containing epsilon}
$z \mapsto f_t(z)$ extends to a continuous function $f_t \colon \overline{\bH} \to \overline{\bH \setminus K_t}$, 

\medskip

\item \label{item: LLE Holder if kappa not 4 still containing epsilon}
$\bH \setminus K_t$ is a H\"older domain in the sense of~\eqref{eq: Holder property for f} 
\textnormal{(}with random $\theta(\kappa, \lambda_\varepsilon,t)$ and $H(\theta,t)$\textnormal{)}, and 

\medskip

\item \label{item: LLE Haus dim if kappa not 4 still containing epsilon}
$\mathrm{dim} (\bdry K_t) < 2$ and $\mathrm{area} (\bdry K_t) = 0$.
\end{enumerate}
\end{cor}

\begin{proof}
This can be proven similarly as Proposition~\ref{prop: LLE micro jumps Holder} by using
Proposition~\ref{prop: BLE micro jumps Holder still containing epsilon} instead of~\ref{prop: BLE micro jumps Holder}.
\end{proof}

As Lemma~\ref{lem: f' = h' in distribution} only applies for a fixed time instant, the result in Corollary~\ref{cor: locally connected still containing epsilon} also only holds pointwise in time. This subtlety also explains the necessity to use the mirror backward Loewner chain in Proposition~\ref{prop: BLE micro jumps Holder still containing epsilon}. Also, the proof of Proposition~\ref{prop: BLE micro jumps Holder still containing epsilon} only gives a random and almost surely positive H\"older constant $\theta = \theta(\kappa, \lambda_\varepsilon, T)$ in~\eqref{eq: Holder property for f}, 
depending on the number $n$ of macroscopic jumps of $\smash{\macroDriver_\varepsilon^{\kappa, a}}$ on $[0, T]$.

\subsection{General case}
\label{subsec: conclusion}

Now we gather the results obtained above for the driving function~\eqref{eq: general DF}:  
\begin{align*}
\macroDriver(t) = a t + \sqrt{\kappa} B(t) + \int_{|\jump| \leq 1} \jump \, \PoissonComp(t, \ud \jump) + \int_{|\jump| > 1} \jump \, \Poisson(t, \ud \jump) , \qquad
a \in \bR, \; \kappa \geq 0 .
\end{align*}
This will prove our main Theorems~\ref{thm: LLE curve if kappa not 8} \&~\ref{thm: LLE Holder if kappa not 4}
in full generality. Recall that we denote $f_t := g_t^{-1}$. 

\MainCurveThm*

\MainRegThm*

\begin{proof}[Proof of Theorems~\ref{thm: LLE curve if kappa not 8} and~\ref{thm: LLE Holder if kappa not 4}]
As in Remark~\ref{rem: tune epsilon}, we find $\varepsilon_{\textnormal{tr}}, \varepsilon_{\textnormal{h\"ol}} \in (0,1)$ 
such that 
$\lambda_{\varepsilon_{\textnormal{h\"ol}}} < \maxjumpH{\kappa} \wedge 1$ and 
$\lambda_{\varepsilon_{\textnormal{tr}}} < \maxjumpT{\kappa}$, 
and also $\varepsilon_{\textnormal{tr}} \in (0,\epsilon_\nu \wedge \tfrac{1}{2} \rho_\nu]$ under~\ref{item: ass2}.
Then, by Proposition~\ref{prop: generated by a curve still containing epsilon} and
Corollary~\ref{cor: locally connected still containing epsilon}, for the Loewner chain driven by
$\smash{\macroDriver_\delta^{\kappa, b}}$ as in~\eqref{eq: BM with micro macro and drift} with any $b \in \bR$,
the claim of Theorem~\ref{thm: LLE curve if kappa not 8} 
(resp.~Theorem~\ref{thm: LLE Holder if kappa not 4})
holds with $\delta = \varepsilon_{\textnormal{tr}}$ 
(resp.~$\delta = \varepsilon_{\textnormal{h\"ol}}$).
We can then change the cutoff of the jumps to be one: choosing 
\begin{align*}
b = a \, - \, \int_{\delta < |\jump| \leq 1} \jump \, \nu(\ud \jump) 
\end{align*}
and noticing that
\begin{align*}
\macroDriver(s) 
= \; & a s + \sqrt{\kappa} B(s) + \int_{|\jump| \leq 1} \jump \, \PoissonComp(s, \ud \jump) + \int_{|\jump| > 1} \jump \, \Poisson(s, \ud \jump) \\
= \; & b s + \sqrt{\kappa} B(s) + \int_{|\jump| \leq \delta} \jump \, \PoissonComp(s, \ud \jump) 
+ \int_{|\jump| > \delta} \jump \, \Poisson(s, \ud \jump) 
\; = \; \macroDriver_\delta^{\kappa, b}(s) , \qquad s \geq 0 ,
\end{align*}
we see that both asserted 
Theorems~\ref{thm: LLE curve if kappa not 8} and~\ref{thm: LLE Holder if kappa not 4}
hold for the Loewner chain driven by $\macroDriver$ as well.
\end{proof}


\appendix

\bigskip{}
\section{\label{app: inverse Loewner chain}Properties of Loewner chains}
In this appendix, we focus on basic properties of forward and inverse Loewner chains. We assume the notation and terminology from Section~\ref{sec: preli}. 
Recall that $t \mapsto g_t(z)$ is the unique absolutely continuous solution to~\eqref{eq: LE}, and
the map $(t, z) \mapsto g_t(z)$ is jointly continuous on
$\{ (t, z) \in  [0, \infty) \times \bH \; | \; t < \tau(z) \}$
(see Lemma~\ref{lem: LE properties} for a short proof). 
\eqref{eq: LE} also implies that
\begin{align} \label{eq: derivative of mof}
g_t'(z) = \exp \bigg( - \int_0^t \frac{2 \, \ud s}{(g_{s}(z) - W(s-))^2} \bigg) , \qquad t \geq 0 \; \textnormal{ and } \; z \in \bH \setminus K_t ,
\end{align}
and the map $(t, z) \mapsto g_t'(z)$ as well as  
the higher complex derivatives of $g_t(z)$ 
are also jointly continuous on
$\{ (t, z) \in  [0, \infty) \times \bH \; | \; t < \tau(z) \}$. 
We write $(f_t)_{t \geq 0}$ for the inverse Loewner chain 
\begin{align*}
f_t := g_t^{-1} \colon \bH \to \bH \setminus K_t .
\end{align*}
The main purpose of this appendix is to gather the following properties for the inverse Loewner chain:
\begin{itemize}[leftmargin=*]
\item (E.g.,~\cite[Lemma~4.3]{Kemppainen:SLE_book}):
Each inverse Loewner map $f_t$ has the Laurent expansion
\begin{align} \label{eq: inverse mof Laurent exp}
f_t \colon \bH \to \bH \setminus K_t , 
\qquad \qquad
f_t(z) = z + \sum_{n = 1}^{\infty} b_n(K_t) \, z^{-n} , \qquad |z| \to \infty ,
\end{align}
with real coefficients $b_n(K_t)$, where the first coefficient is $b_1(K_t) = -a_1(K_t) = -\mathrm{hcap}(K_t) = -2t$.

\medskip

\item (Lemma~\ref{lem: inverse LE}): 
The following \emph{inverse Loewner differential equation} holds for $(f_t)_{t \geq 0}$: for each $z \in \bH$,
\begin{align} \label{eq: inverse LE}
\tag{inv-LE}
\partial_t^{+} f_t(z) = \frac{-2 \, f_t'(z)}{z - W(t)} \qquad \textnormal{with initial condition} \qquad f_0(z) = z ,
\end{align}
where $\partial_t^{+}$ denotes the right derivative. 
The proof of~\eqref{eq: inverse LE} is roughly similar to the derivation of Loewner's equation~\eqref{eq: LE} for the mapping-out functions $(g_t)_{t \geq 0}$ associated to a locally growing family $(K_t)_{t \geq 0}$ of hulls parameterized by capacity. 
We present the proof in Section~\ref{subsec: inverse Loewner equation}.

\medskip

\item The following distortion estimate, well-known in the case of continuous drivers, holds:

\begin{restatable}{lem}{InverseFEstimates} 
\label{lem: derivative distortion}
We have $|f_t'(x + \ii y)| \asymp |f_{t+s}'(x + \ii y)|$
for all $x + \ii y \in \bH$, $t \geq 0$, and $s \in [0, y^2]$.
\end{restatable}

The proof of Lemma~\ref{lem: derivative distortion} follows the lines of~\cite[Proof of Lemma~6.7]{Kemppainen:SLE_book}, allowing however a possibly discontinuous driving function. 
We present the proof in Section~\ref{subsec: derivative distortion}.
\end{itemize} 

\subsection{Basic continuity properties of Loewner flows}
\label{subsec: g is continuous}

To begin, we record basic properties of the Loewner chain $(g_t)_{t \geq 0}$, which follow immediately from the local growth of the hulls. 

\begin{lem} \label{lem: LE properties}
The following hold for the Loewner chain $(g_t)_{t \geq 0}$. 
\begin{enumerate}[label=\textnormal{(\alph*):}, ref=(\alph*)]
\item \label{item: FTC for g}
For each $z \in \bH$, the map $t \mapsto g_t(z)$ is absolutely continuous on compact sub-intervals of~$[0,\tau(z))$ and 
\begin{align*}
g_t(z) \; =\;  z + \int_0^t \partial_s^{+} g_s(z) \, \ud s 
\; = \; z + \int_0^t \frac{2 \, \ud s }{g_{s}(z) - W(s-)} , \qquad t \in [0,\tau(z)) .
\end{align*}

\medskip

\item \label{item: g jointly continuous}
The Loewner chain $(t, z) \mapsto g_t(z)$ is jointly continuous on $\smash{\{ (t, z) \in  [0, \infty) \times \overline{\bH} \; | \; t < \tau(z) \}}$. 
\end{enumerate}
\end{lem}

\begin{proof} 
(Recall that the existence and uniqueness of an absolutely continuous solution $t \mapsto g_t(z)$ to~\eqref{eq: LE} follows from general ODE theory~\cite[Chapter~I.5.,~Theorems~5.1--5.3]{Hale:Ordinary_differential_equations} 
(cf.~Remark~\ref{rem: existence and uniqueness ODE}).)
\begin{enumerate}[label=\textnormal{(\alph*):}, ref=(\alph*), leftmargin=*]
\item The right-hand side of~\eqref{eq: LE} as a function of $t$ is Lebesgue-integrable on any compact sub-interval of $[0,\tau(z))$. Therefore, by Lebesgue's differentiation theorem (for Dini derivatives, cf.~\cite{Hagood-Thomson:Recovering_function_from_Dini_derivative}), it suffices to prove that $t \mapsto g_t(z)$ is continuous. 
Using Lemma~\ref{lem: some markov property} together with~\eqref{eq: shrinking hulls out}, we find that
\begin{align} \label{eq: g right-continuous}
| g_{t+\delta}(z) - g_t(z) |
= | \smash{\mathring{g}^t_{\delta}}(g_t(z) - W(t)) - ( g_t(z) - W(t)) |
\lesssim \diam(\mathring{K}^t_\delta) 
\quad \overset{\delta \to 0+}{\longrightarrow} \quad 0 ,
\end{align}
so $t \mapsto g_t(z)$ is right-continuous. 
The left-continuity follows by a similar argument using~\eqref{eq: shrinking hulls in}: 
\begin{align} \label{eq: g left-continuous}
| g_t(z) - g_{t-\delta}(z) | 
\lesssim \diam(\mathring{K}^{t-\delta}_\delta) 
\quad \overset{\delta \to 0+}{\longrightarrow} \quad 0 .
\end{align}
This proves item~\ref{item: FTC for g}.
Note also that the limits~(\ref{eq: g right-continuous},~\ref{eq: g left-continuous}) are uniform in $z$. 

\medskip

\item 
Fix $(t, z) \in [0, \infty) \times \overline{\bH}$ such that $t < \tau(z)$. Then, as $g_t$ is a conformal map around $z$ and~(\ref{eq: g right-continuous},~\ref{eq: g left-continuous}) hold for $\delta \to 0+$, we have
\begin{align*}
|g_{t+\delta}(w) - g_t(z)| 
\leq |g_{t+\delta}(w) - g_t(w)| + |g_t(w)- g_t(z)| 
\lesssim \diam(\mathring{K}^t_\delta) + |g_t(w)- g_t(z)| 
\; & \quad \overset{w \to z}{\underset{\delta \to 0+}{\longrightarrow}} \quad 0 , \\
|g_{t-\delta}(w) - g_t(z)| 
\leq |g_{t-\delta}(w) - g_t(w)| + |g_t(w)- g_t(z)| 
\lesssim \diam(\mathring{K}^{t-\delta}_\delta) + |g_t(w)- g_t(z)| 
\; & \quad \overset{w \to z}{\underset{\delta \to 0+}{\longrightarrow}} \quad 0 .
\end{align*}
This proves item~\ref{item: g jointly continuous}. \qedhere
\end{enumerate}
\end{proof}

\subsection{Estimates for inverse Loewner flow}

Recall from Lemma~\ref{lem: f' = h' in distribution} and its proof the following facts:
\begin{itemize}[leftmargin=*]
\item (Item~\ref{item: f = h in distribution} of Lemma~\ref{lem: f' = h' in distribution}): The map $z \mapsto f_t(z + W(t)) - W(t)$ 
has the same distribution as $z \mapsto h_t(z)$, where $t \mapsto h_t(z)$ is the unique absolutely continuous solution to 
\begin{align} \label{eq: mBLE again} 
\tag{mBLE}
\partial_t^{+} h_t(z) = \frac{-2}{h_t(z) + W(t)} \qquad \textnormal{with initial condition} \qquad h_0(z) = z  .
\end{align}

\item (Proof of Lemma~\ref{lem: f' = h' in distribution}): 
The map $z \mapsto f_t(z)$ has the same distribution as $z \mapsto k_t(z)$, 
where $t \mapsto k_t(z)$ is the unique absolutely continuous solution to
\begin{align} \label{eq: BLE again} 
\tag{BLE}
\partial_s^{+} k_s(z) = \frac{-2}{k_s(z) - W(t-s)} \qquad \textnormal{with initial condition} \qquad k_0(z) = z  .
\end{align}
\end{itemize}

From these relations, we can derive the following estimates for $f_t$ pointwise in time: 
\begin{itemize}[leftmargin=*]
\item The imaginary part of~\eqref{eq: mBLE again} gives
\begin{align*}
\im(h_t(z))
\; = \; \im(z) \; + \; 2 \int_0^t \frac{\im(h_{s}(z))}{|h_{s}(z) + W(s-)|^2} \, \ud s ,
\end{align*}
which implies in particular that 
\begin{align*}
0 < \im(z) \leq \im(h_t(z)) \leq \sqrt{(\im(z))^2 + 4t} 
\qquad \textnormal{for all } t \geq 0 .
\end{align*} 
This shows that for each \emph{fixed} time $t \geq 0$, we almost surely have
\begin{align} \label{eq: imf sqrt bound} 
\im(f_t(z + W(t))) \leq \sqrt{(\im(z))^2 + 4t} .
\end{align} 

\item The real part of~\eqref{eq: BLE again} gives
\begin{align*}
\partial_s^{+} \re(k_s(z))
\; = \; - 2 \, \frac{\re(k_s(z)) - W(t-s)}{|k_s(z) - W(t-s)|^2} ,
\end{align*}
which implies that the map $s \mapsto \re(k_s(z))$ on $[0,t]$ is 
\begin{align*}
\textnormal{decreasing if } \; \re(z) > \underset{s \in [0,t]}{\sup} | W(s) | 
\qquad \textnormal{and} \qquad
\textnormal{increasing if } \; \re(z) < - \underset{s \in [0,t]}{\sup} | W(s) | ,
\end{align*}
so in particular, if $|\re(z)| \leq \smash{\underset{s \in [0,t]}{\sup} | W(s) |}$, 
then we have $|\re(k_u(z))| \leq \smash{\underset{s \in [0,t]}{\sup} | W(s) |}$ for all $u \in [0,t]$. 

\bigskip

\noindent 
This shows that for each \emph{fixed} time $t \geq 0$, we almost surely have
\begin{align} \label{eq: ref bound} 
|\re(z)| \leq \smash{\underset{s \in [0,t]}{\sup} | W(s) |}
\qquad \Longrightarrow \qquad 
|\re(f_t(z))| \leq \smash{\underset{s \in [0,t]}{\sup} | W(s) |} .
\end{align} 
\end{itemize}

\subsection{Inverse Loewner equation}
\label{subsec: inverse Loewner equation}

\begin{lem} \label{lem: inverse LE}
For each $z \in \bH$, the map $t \mapsto f_t(z)$ is right-differentiable and satisfies~\eqref{eq: inverse LE}.
\end{lem}

\begin{proof}
From Lemma~\ref{lem: some markov property}, we have
$f_{t+\delta}(z) = f_t \big( \smash{\mathring{f}^t_{\delta}}(z - W(t)) + W(t) \big)$, where $\smash{\mathring{f}^t_{\delta}}$ is the inverse of $\smash{\mathring{g}^t_{\delta}} = \smash{g_{\mathring{K}^t_{\delta}}}$. Hence, expanding the holomorphic function $f_t(\cdot)$ at $z \in \bH$ gives
\begin{align*}
\bigg| \frac{f_{t + \delta}(z) - f_t(z)}{\delta} + \frac{2 \, f_t'(z)}{z - W(t)} \bigg| 
\; \leq \;\; & |f_t'(z)| \, \bigg|  \frac{\mathring{f}^t_{\delta}(z - W(t)) - (z - W(t))}{\delta} + \frac{2}{z - W(t)}  \bigg| \\
\; & + \frac{1}{\delta} \, \OO \big( | \mathring{f}^t_{\delta}(z - W(t)) - (z - W(t)) |^2 \big) .
\end{align*}
We will show that the right-hand side tends to zero as $\delta \to 0+$.

First, by Lemma~\ref{lem: some markov property} the hulls $\smash{\mathring{K}^t_{\delta}}$ are parameterized by capacity and 
$\smash{\mathring{g}^t_{\delta}}$ are their mapping-out functions,
so applying the standard estimate~\cite[Lemma~4.7]{Kemppainen:SLE_book} to these hulls yields
\begin{align} \label{eq: diam K mathring}
\bigg| \frac{\mathring{f}^t_{\delta}(z - W(t)) - (z - W(t))}{\delta} + \frac{2}{z - W(t)}  \bigg| 
\lesssim \frac{\diam(\mathring{K}^t_{\delta})}{|z - W(t)|^2} 
\end{align} 
when 
$\smash{\diam(\mathring{K}^t_{\delta})} \lesssim |z - W(t)|$. 
Now, by the local growth and Wolff's lemma (e.g.,~\cite[Lemma~4.6]{Kemppainen:SLE_book}), 
we have $\smash{\diam(\mathring{K}^t_{\delta})} \to 0$ as $\delta \to 0+$. 
Hence, the right-hand side of~\eqref{eq: diam K mathring} tends to zero as $\delta \to 0+$.

Second, for the error term $\frac{1}{\delta} \, \OO \big( | \smash{\mathring{f}^t_{\delta}}(z - W(t)) - (z - W(t)) |^2 \big)$, 
we note that~\eqref{eq: diam K mathring} gives
\begin{align*}
\bigg| \frac{\mathring{f}^t_{\delta}(z - W(t)) - (z - W(t)) }{\delta} \bigg| 
\leq \; & \; 
\bigg| \frac{\mathring{f}^t_{\delta}(z - W(t)) - (z - W(t)) }{\delta} + \frac{2 \, f_t'(z)}{z - W(t)} \bigg| 
+ \bigg| \frac{2 \, f_t'(z)}{z - W(t)} \bigg|
\\
\overset{\delta \to 0+}{\longrightarrow} \; & \; 
\bigg| \frac{2 \, f_t'(z)}{z - W(t)} \bigg| ,
\end{align*}
while 
$| \smash{\mathring{f}^t_{\delta}}(\cdot) - (\cdot) | \to 0$ as $\delta \to 0+$ by~\eqref{eq: shrinking hulls out}, which implies that 
$\frac{1}{\delta} \, | \smash{\mathring{f}^t_{\delta}}(z - W(t)) - (z - W(t)) |^2 \to 0$ as $\delta \to 0+$. 
Together with~\eqref{eq: diam K mathring}, this shows that $t \mapsto f_t(z)$ is right-differentiable and satisfies~\eqref{eq: inverse LE}.
\end{proof}

The local growth of the hulls can also be used to easily verify continuity of the map $t \mapsto f_t(z)$.

\begin{lem} \label{lem: FTC for f}
For each $z \in \bH$, the map $t \mapsto f_t(z)$ is absolutely continuous on compact time intervals~and 
\begin{align*}
f_t(z) \; = \; z \; + \; \int_0^t \partial_s^{+} f_s(z) \, \ud s 
\; = \; z \; - \; 2 \, \int_0^t \frac{f_{s}'(z) \, \ud s }{z - W(s-)} , \qquad t \geq 0 .
\end{align*}
\end{lem}

\begin{proof}
The right-hand side of~\eqref{eq: inverse LE} as a function of $t$ is Lebesgue-integrable on any compact sub-interval of $[0,\infty)$. Therefore, by Lebesgue's differentiation theorem (for Dini derivatives, cf.~\cite{Hagood-Thomson:Recovering_function_from_Dini_derivative}), it suffices to prove that $t \mapsto f_t(z)$ is continuous.
As in the proof of Lemma~\ref{lem: inverse LE}, we have
\begin{align*}
| f_{t+\delta}(z) - f_t(z) | 
= \OO( | \mathring{f}^t_{\delta}(z - W(t)) - (z - W(t)) | ) 
\quad \overset{\delta \to 0+}{\longrightarrow} \quad
0 ,
\end{align*}
since $z \mapsto f_t(z)$ is continuous (holomorphic) and 
$| \smash{\mathring{f}^t_{\delta}}(\cdot) - (\cdot) | \to 0$ as $\delta \to 0+$ by~\eqref{eq: shrinking hulls out}.
This shows that $t \mapsto f_t(z)$ is right-continuous. To show 
the left-continuity, we can use a similar argument with~\eqref{eq: shrinking hulls in}:
\begin{align*}
| f_t(z) - f_{t-\delta}(z) |
= \OO( | \mathring{f}^{t-\delta}_{\delta}(z - W(t-\delta)) - (z - W(t-\delta)) | ) 
\quad \overset{\delta \to 0+}{\longrightarrow} \quad
0 .
\end{align*}
This concludes the proof.
\end{proof}

One can also show that $(t, z) \mapsto f_t(z)$ is jointly continuous on $[0, \infty) \times \bH$ (see Lemma~\ref{lem: f is continuous}).

\begin{lem} \label{lem: derivatives commute}
For each $z \in \bH$ and $t \geq 0$, we have
\begin{align} \label{eq: inverse derivative LE}
\partial_t^{+} f_t'(z) \; = \; (\partial_t^+ f_t)'(z)
\; = \; \frac{-2 \, f_t''(z)}{z - W(t)} + \frac{2 \, f_t'(z)}{(z - W(t))^2} .
\end{align}
\end{lem}

\begin{proof}
On the one hand, by differentiating~\eqref{eq: inverse LE} with respect to $z$, we obtain
\begin{align*}
(\partial_t^+ f_t)'(z) = \frac{-2 \, f_t''(z)}{z - W(t)} + \frac{2 \, f_t'(z)}{(z - W(t))^2} .
\end{align*}
On the other hand, by differentiating the identity 
$z = g_t(f_t(z))$ with respect to $z$, we obtain
$1 = g_t'(f_t(z)) \, f_t'(z)$, 
and taking the right derivative $\partial_t^+$ of this, 
we obtain (using\footnote{$g_t'$ is differentiable with $\smash{(\partial_t g_t')(w) = \frac{-2 \, g_t'(w)}{(g_t(w) - W(t))^2} }$ by~\eqref{eq: derivative of mof}, and $f_t$ is right-differentiable by Lemma~\ref{lem: inverse LE}.} also~\eqref{eq: derivative of mof} and~\eqref{eq: inverse LE}) 
\begin{align*}
0 = \; & f_t'(z) \Big( (\partial_t g_t')(f_t(z)) 
+ (\partial_t^+ f_t(z)) \, g_t''(f_t(z)) \Big) 
+ g_t'(f_t(z)) \, (\partial_t^+ f_t')(z) \\
= \; & f_t'(z) \bigg( \frac{-2}{(z - W(t))^2} \, g_t'(f_t(z)) 
+ \frac{-2 \, f_t'(z)}{z - W(t)} \, g_t''(f_t(z)) \bigg) 
+ g_t'(f_t(z)) \, (\partial_t^+ f_t')(z)
\end{align*}
which implies that 
\begin{align*}
\partial_t^+ f_t' (z)
= \; & \frac{ f_t'(z) }{- g_t'(f_t(z))}
\bigg( \frac{-2}{(z - W(t))^2} \, g_t'(f_t(z)) 
+ \frac{-2 \, f_t'(z)}{z - W(t)} \, g_t''(f_t(z)) \bigg) \\
= \; & \frac{2 \, f_t'(z) }{(z - W(t))^2}
+ \frac{2 \, ( f_t'(z) )^2}{z - W(t)} \frac{g_t''(f_t(z))}{g_t'(f_t(z))}  \\
= \; & \frac{2 \, f_t'(z) }{(z - W(t))^2}
+ \frac{2}{z - W(t)} \, \frac{g_t''(f_t(z))}{g_t'(f_t(z))} \, ( f_t'(z) )^2 .
\end{align*}
We also see that $t \mapsto f_t'(z) = 1 / g_t'(f_t(z))$ is a continuous function.
By differentiating the identity 
$z = g_t(f_t(z))$ twice with respect to $z$, we obtain 
$g_t''(f_t(z)) \, ( f_t'(z) )^2 + g_t'(f_t(z)) \, f_t''(z) = 0$, which gives 
\begin{align*}
\partial_t^+ f_t' (z)
= \; &  \frac{2 \, f_t'(z) }{(z - W(t))^2}
- \frac{2 \, f_t''(z)}{z - W(t)} .
\qedhere
\end{align*}
\end{proof}

\subsection{Distortion estimate in time --- proof of Lemma~\ref{lem: derivative distortion}}
\label{subsec: derivative distortion}

The following distortion property is well-known for continuous driving functions (see, e.g.,~\cite[Lemma 6.7]{Kemppainen:SLE_book}). The proof in the case of c\`adl\`ag  driving functions is very similar. For  readers' convenience, we give an outline of the proof.

\InverseFEstimates*

\begin{proof}
Fix $t \geq 0$, $z = x + \ii y \in \bH$, and $s \in [0, y^2]$.
From Lemma~\ref{lem: derivatives commute}, the triangle inequality, and the inequality $|x + \ii y - W(u)| \geq y$,
we have
\begin{align} \label{eq: estimate partial f}
| \partial_u^+ f_u' (x + \ii y) | \leq 
\frac{2}{y} \, | f_u'' (x + \ii y) | + \frac{2}{y^2} \, | f_u' (x + \ii y) | .
\end{align}
We can estimate $|f_u''|$ in terms of $|f_u'|$ using Bieberbach's theorem (a consequence of the area theorem) \cite[Theorem~2.2]{Duren:Univalent_functions} 
in the following manner. 
Consider the M\"obius map $\phi \colon \bD \to \bH$ given by
$\smash{\phi (w) = x +  \ii y \frac{1-w}{1+w}}$.
The function
\begin{align*}
\psi (w) 
:= \frac{ (f_u \circ \phi )(w) - f_u(x + \ii y)}{f_u'(x + \ii y) \, \phi'(0)} 
= w + \sum_{n = 2}^{\infty} a_n \, w^{n} , \qquad |w| < 1 ,
\end{align*}
is univalent (holomorphic and injective on $\bD$).
Hence, Bieberbach's theorem gives $|a_2| \leq 2$, that is, 
\begin{align*}
\bigg| \frac{f_u''(x + \ii y) \, ( \phi'(0) )^2 + f_u'(x + \ii y) \, \phi''(0)}{2 \, f_u'(x + \ii y) \, \phi'(0)} \bigg|
= |a_2| \leq 2 .
\end{align*}
Therefore, using the identities $\phi'(0) = - 2 \, \ii y$ and $\phi''(0) = 4 \, \ii y$, we have
\begin{align*}
4 y^2 \, | f_u''(x + \ii y) |
= | f_u''(x + \ii y) |  \, | \phi'(0) |^2 
\; \leq \; \; & | f_u''(x + \ii y) \, ( \phi'(0) )^2 + f_u'(x + \ii y) \, \phi''(0) | + 
| f_u'(x + \ii y) \, \phi''(0) | \\
\; \leq \; \; &
| f_u'(x + \ii y) | \big( 4 \, |\phi'(0)| + |\phi''(0)| \big)
\; = \; 12 |y| \, | f_u'(x + \ii y) |  \\
\qquad \Longrightarrow \qquad
| f_u''(x + \ii y) | \; \leq \; \; & \frac{3}{|y|} \, | f_u'(x + \ii y) |  .
\end{align*}
Plugging this into~\eqref{eq: estimate partial f}, we conclude that
$- 8 y^{-2} \leq \partial_u^+ \log | f_u' (x + \ii y) | \leq 8 y^{2}$.
After integrating this inequality with respect to $u \in [t, t+s]$ and using the fundamental theorem of calculus (e.g.,~\cite{Hagood-Thomson:Recovering_function_from_Dini_derivative})
to the continuous function $u \mapsto \log | f_u' (x + \ii y) |$ 
(which has a right derivative $\partial_u^+ \log | f_u' (x + \ii y) |$ at every point $u \in [t, t+s]$ by Lemma~\ref{lem: derivatives commute}), we find that  
\begin{align*}
- \frac{8s}{y^2} 
\leq \log \frac{| f_{t+s}' (x + \ii y) |}{| f_t' (x + \ii y) |} 
\leq \frac{8s}{y^2} .
\end{align*}
This implies that 
$C^{-1} \, |f_t'(x + \ii y)| \leq |f_{t+s}'(x + \ii y)| \leq C \, |f_t'(x + \ii y)|$, where $C = e^8$, since $s \in [0, y^2]$. 
\end{proof}

\subsection{Continuity of inverse Loewner flows jointly in space and time}
\label{subsec: f is continuous}

\begin{lem} \label{lem: f is continuous}
The inverse Loewner chain $(t, z) \mapsto f_t(z)$ is jointly continuous on $[0, \infty) \times \bH$.	
\end{lem}

We will make use of domain Markov properties from Lemma~\ref{lem: some markov property} and~\eqref{eq: inverse markov property} (with $\sigma=t$ or $\sigma=t-\delta$):
\begin{align*} 
f_{\sigma+\delta}(w) = f_\sigma \big( \mathring{f}^\sigma_{\delta}(w - W(\sigma)) + W(\sigma) \big) , \qquad \sigma, \delta \geq 0 ,
\end{align*}
where $\smash{\mathring{f}^\sigma_{\delta}}$ is the inverse map of $\smash{g_{\mathring{K}^\sigma_{\delta}}}$, 
together with the consequences~(\ref{eq: shrinking hulls out},~\ref{eq: shrinking hulls in}) of the bilateral local growth that follow from conformal distortion estimates and the property $\smash{\diam(\mathring{K}^\sigma_\delta)} \to 0$ as $\delta \to 0+$.

\begin{proof}
Fix $t \geq 0$ and $z \in \bH$. 
Consider $w \in B(z,\epsilon)$ with $\smash{\overline{B(z,8\epsilon)}} \subset \bH$ and $\epsilon = \epsilon(t) > 0$ such that $\smash{ \diam(\mathring{K}^t_{\delta}), \, \diam(\mathring{K}^{t-\delta}_{\delta}) } < \epsilon$ when $\delta \in (0,\epsilon^2)$ is small enough. 
As the holomorphic functions $f_t(\cdot)$ and $f_{t-\delta}(\cdot)$ are locally Lipschitz, we see that, on the one hand,  
\begin{align*}
|f_{t+\delta}(w) - f_t(z)| 
\leq & \; \Big( \sup_{u \in \overline{B(z,7\epsilon)}} |f_t'(u)| \Big) \, \big|\mathring{f}^t_{\delta}(w - W(t)) - (z - W(t)) \big| 
\qquad && \textnormal{[by~\eqref{eq: inverse markov property}]} \\
\leq & \; \Big( \sup_{u \in \overline{B(z,7\epsilon)}} |f_t'(u)| \Big) 
\Big( 5 \, \diam(\mathring{K}^t_{\delta}) + |z - w| \Big) 
\quad \overset{w \to z}{\underset{\delta \to 0+}{\longrightarrow}} \quad 0 ,
\qquad && \textnormal{[by~\eqref{eq: shrinking hulls out}]} 
\end{align*}
where
$\smash{\mathring{f}^t_{\delta}}(w - W(t)) + W(t) \in B(z,7\epsilon)$ by~\eqref{eq: shrinking hulls out},
and on the other hand, 
\begin{align*}
|f_{t-\delta}(w) - f_t(z)| 
\leq & \; \Big( \sup_{u \in \overline{B(w,7\epsilon)}} |f_{t-\delta}'(u)| \Big) \, \big| \mathring{f}^{t-\delta}_{\delta}(z - W(t-\delta)) - ( w - W(t-\delta) ) \big| 
\qquad && \textnormal{[by~\eqref{eq: inverse markov property}]} \\
\leq & \; \Big( \sup_{u \in \overline{B(z,7\epsilon)}} |f_{t-\delta}'(u)| \Big) 
\Big( 5 \, \diam(\mathring{K}^{t-\delta}_{\delta}) + |w - z| \Big) ,
\qquad && \textnormal{[by~\eqref{eq: shrinking hulls in}]} 
\end{align*}
where 
$\smash{\mathring{f}^{t-\delta}_{\delta}}(z - W(t-\delta)) + W(t-\delta) \in B(w,6\epsilon) \subset B(z,7\epsilon)$ by~\eqref{eq: shrinking hulls in}.
To evaluate this limit as $\delta \to 0+$ and $w \to z$, we 
note that Lemma~\ref{lem: derivative distortion} implies that 
\begin{align*} 
\sup_{u \in \overline{B(z,7\epsilon)}} |f_{t-\delta}'(u)|
\lesssim \sup_{u \in \overline{B(z,7\epsilon)}} |f_{t}'(u)| , 
\qquad \textnormal{since} \; 0 < \delta < \epsilon^2 \leq (\im(u))^2 \textnormal{ for all } u \in \overline{B(z,7\epsilon)} .
\end{align*}
Therefore, we conclude that
\begin{align*}
|f_{t-\delta}(w) - f_t(z)| \lesssim & \; \Big( \sup_{u \in \overline{B(z,7\epsilon)}} |f_{t}'(u)| \Big) 
\Big( 5 \, \diam(\mathring{K}^{t-\delta}_{\delta}) + |w - z| \Big) 
\quad \overset{w \to z}{\underset{\delta \to 0+}{\longrightarrow}} \quad 0 .
\end{align*}
This proves the joint continuity of the map $(t, z) \mapsto f_t(z)$ at an arbitrary point $(t,z) \in [0, \infty) \times \bH$. 
\end{proof}

It now also follows immediately from the joint continuity of the various maps $f_t(z)$, $g_t'(z)$, and $g_t''(z)$ together 
with the identities $1 = g_t'(f_t(z)) \, f_t'(z)$ and 
$g_t''(f_t(z)) \, ( f_t'(z) )^2 + g_t'(f_t(z)) \, f_t''(z) = 0$ 
that the maps
\begin{align*}
(t, z) \mapsto f_t'(z)
\qquad \textnormal{and} \qquad 
(t, z) \mapsto f_t''(z)
\end{align*}
are jointly continuous on $[0, \infty) \times \bH$.

\bigskip{}
\section{\label{app: Ito calculus}It\^o-D\"oblin formula and applications}
This appendix contains additional computations for 
Lemmas~\ref{lem: SDE M}
and~\ref{lem: SDE M kappa = 0 with drift}. 
We use the following version of It\^o's formula, applicable to L\'evy type stochastic integrals involving processes 
\begin{align*}
F, G_1, G_2 \colon [0, \infty) \to \bR , \qquad
H, K \colon [0, \infty) \times \bR \setminus \{0\} \to \bR 
\end{align*}
on a filtered probability space satisfying the usual conditions 
(i.e., the filtration is right-continuous and the probability space is completed~\cite[Chapter~4]{Applebaum:Levy_processes_and_stochastic_calculus},~\cite[Chapter~14]{Cohen-Elliott:Stochastic_calculus_and_applications}).
Let $B$ be a one-dimensional Brownian motion,
$\Poisson$ an independent Poisson point process with L\'evy intensity measure $\nu$ on $\bR$ and 
$\smash{\PoissonComp}(t) := \Poisson(t) - t \nu$ the related compensated Poisson point process. 
We consider a complex-valued process $Z = Z_1 + \ii Z_2 \colon [0, \infty) \to \domain$ taking values in an open set $\domain \subset \bC$, defined by $Z(0) := x_0 + \ii y_0 \in \domain$ and 
\begin{align*}
Z_1(t) \; 
= \; \; & x_0 \; 
+ \; \int_0^t G_1(s) \, \ud s \; 
+ \; \int_0^t F(s) \, \ud B(s) \\
\; \; & 
+ \; \int_0^t \int_{|\jump| \leq \varepsilon} H(s, \jump) \, \PoissonComp(\ud s, \ud \jump)
\; + \; \int_0^t \int_{|\jump| > \varepsilon} K(s, \jump) \, \Poisson(\ud s, \ud \jump) , \\
Z_2(t) \; 
= \; \; & y_0 \; 
+ \; \int_0^t G_2(s) \, \ud s .
\end{align*}

\begin{defn}
Fix $T, \varepsilon > 0$. 
We say that the process $Z$ satisfies the \emph{It\^o-D\"oblin assumptions} if the processes $F, G_1, G_2, H, K$ satisfy the following almost sure properties:
\begin{enumerate}[label=\textnormal{(\alph*):}, ref=(\alph*)]
\item $F$, $\sqrt{|G_1|}$, $\sqrt{|G_2|}$ are predictable processes whose squares are integrable on $[0, T]$, 

\item $K$ and $H$ are predictable processes and $H$ is square-integrable on $[0, T] \times \bR$, and

\item for all $t > 0$, we have 
\begin{align*}
\sup_{0 \leq s \leq t}\; \sup_{0 < |\jump| \leq \varepsilon} | H(s, \jump) | < \infty .
\end{align*}
\end{enumerate}
\end{defn}

\begin{thm} \label{thm: Ito}
\textnormal{(It\^o-D\"oblin formula)}
Suppose that the process $Z$ satisfies the It\^o-D\"oblin assumptions.
Then, for any function $f \colon \domain \to \bR$ such that the derivatives 
\begin{align*}
(\partial_1 f)(z) = (\partial_x f)(z) , \qquad
(\partial_2 f)(z) = (\partial_y f)(z) , \qquad
(\partial_1^2 f)(z) = (\partial_x^2 f)(z) , \quad z = x + \ii y \in \domain ,
\end{align*}
exist and are continuous on $\domain$, 
and for all $t \in [0, T]$, we have almost surely
\begin{align*}
f(Z(t)) \; 
= \; \; & f(Z(0)) \; 
+ \; \int_0^t G_1(s) \, (\partial_1 f)(Z(s-)) \, \ud s 
+ \; \int_0^t G_2(s) \, (\partial_2 f)(Z(s-)) \, \ud s \\
\; \; & 
+ \; \int_0^t F(s) \, (\partial_1 f)(Z(s-)) \, \ud B(s) 
+ \; \frac{1}{2} \int_0^t (F(s))^2 \, (\partial_1^2 f)(Z(s-)) \, \ud s 
\\
\; \; & 
+ \; \int_0^t \int_{|\jump| \leq \varepsilon} \Big( f \big( Z(s-) + H(s, \jump) \big) - f(Z(s-)) \Big) \, \PoissonComp(\ud s, \ud \jump) \\
\; \; & 
+ \; \int_0^t \int_{|\jump| \leq \varepsilon} \Big( f \big( Z(s-) + H(s, \jump) \big) - f(Z(s-)) 
- (\partial_1 f) \, (Z(s-)) H(s, \jump) \Big) \, \nu(\ud \jump) \, \ud s \\
\; \; & 
+ \; \int_0^t \int_{|\jump| > \varepsilon} \Big( f \big( Z(s-) + K(s, \jump) \big) - f(Z(s-)) \Big) \, \Poisson(\ud  \jump, \ud s) .
\end{align*}
In particular, the real-valued process $f \circ Z$ also satisfies the It\^o-D\"oblin assumptions.
\end{thm} 

\begin{proof}
See, for instance,~\cite[Theorems~14.2.3~\&~14.2.4]{Cohen-Elliott:Stochastic_calculus_and_applications}.
\end{proof}

\subsection{Applications involving mirror backward Loewner flow}

Fix $\kappa \geq 0$, $a \in \bR$, $\varepsilon > 0$, and
a L\'evy measure $\nu$.
Let $(h_t)_{t \geq 0}$ be the (absolutely continuous) solution to~\eqref{eq: mBLE} with driving function $W = \microDriver_\varepsilon^{\kappa, a}(t)$ of the form~\eqref{eq: BM with micro and drift}:
\begin{align*} \label{eq: BM with micro and drift app}
\microDriver_\varepsilon^{\kappa, a}(t) = a t + \sqrt{\kappa} B(t) + \int_{|\jump| \leq \varepsilon} \jump \, \PoissonComp(t, \ud \jump) , \qquad
a \in \bR, \; \kappa \geq 0 , \;  \varepsilon > 0 .
\end{align*}
Fix $z_0 = x_0 + \ii y_0 \in \bH$ and set 
\begin{align*}
Z(t) = Z_\varepsilon^{a, \kappa}(t,z_0) := h_t(z_0) + \microDriver_\varepsilon^{\kappa, a}(t) =: X(t) + \ii Y(t) .
\end{align*}
Then,~\eqref{eq: mBLE} shows that
\begin{align*}
X(t) \; = \; \; & x_0 \; - \; 2 \int_0^t \frac{X(s-)}{|Z(s-)|^2} \, \ud s \; + \; a t 
\; + \; \sqrt{\kappa} B(t) + \int_0^t \int_{|\jump| \leq \varepsilon} \jump \, \PoissonComp(\ud s, \ud \jump) , \\
Y(t) \; = \; \; & y_0 \; + \; 2 \int_0^t \frac{Y(s-)}{|Z(s-)|^2} \, \ud s ,
\end{align*}
where almost surely, we have 
$|Z(s)| = \smash{\sqrt{X(s)^2 + Y(s)^2}} \geq X(s) \vee Y(s) \geq y_0 > 0$ for all $s \geq 0$.

We apply Theorem~\ref{thm: Ito} to derive formulas for functions of $Z$, needed in Lemmas~\ref{lem: SDE M} and~\ref{lem: SDE M kappa = 0 with drift}.

\begin{lem} \label{lem: SDE sin(arg(Z))}
For each $t \geq 0$, we have almost surely
\begin{align*}
\sin \arg Z(t) \; = \; \; & \sin \arg Z(0)
\; - \; \sqrt{\kappa} \int_0^t \sin \arg Z(s-) \, \frac{X(s-)}{|Z(s-)|^2} \, \ud B(s) \\
\; \; &
- \; a \int_0^t \sin \arg Z(s-) \, \frac{X(s-)}{|Z(s-)|^2} \, \ud s \\
\; \; &
+ \; \int_0^t \sin \arg Z(s-) \, \frac{(\kappa + 4) X(s-)^2 - \frac{\kappa}{2} Y(s-)^2}{|Z(s-)|^4} \, \ud s  \\
\; \; &
+ \; \int_0^t \int_{|\jump| \leq \varepsilon} \Big( \sin \arg(Z(s-) + \jump)) - \sin \arg Z(s-) \Big) \, \PoissonComp(\ud s, \ud \jump) \\
\; \; &
+ \; \int_0^t \int_{|\jump| \leq \varepsilon} 
\bigg( \sin \arg(Z(s-) + \jump) - \sin \arg Z(s-) 
+ \sin \arg Z(s-) \, 
\frac{\jump X(s-)}{|Z(s-)|^2}  \bigg) \, \nu(\ud \jump) \, \ud s ,
\end{align*}
and the process $\sin \arg Z$ satisfies the It\^o-D\"oblin assumptions. 
\end{lem}

\begin{proof}
Note that $\sin \arg (z) =\frac{y}{|z|}$ 
is smooth on $\bH \ni z = x + \ii y$. Moreover, the process $Z$ with  
\begin{align*}
F(s) := \sqrt{\kappa} , \qquad
G_1(s) := \frac{-2 X(s-)}{|Z(s-)|^2}  + a , \qquad
G_2(s) := \frac{2 Y(s-)}{|Z(s-)|^2} , \qquad
H(s, \jump) := \jump , 
\end{align*}
and $K(s, \jump) := 0$ 
satisfies the It\^o-D\"oblin assumptions 
(for Theorem~\ref{thm: Ito}), thanks to~\eqref{eq: mBLE}. 
Hence, the claim follows from a direct computation.
\end{proof}

\begin{lem} \label{lem: SDE sin(arg(Z)) power}
For each $t \geq 0$ and $r \in \bR$, we have almost surely
\begin{align*}
(\sin \arg Z(t))^{-2r} 
\; = \; \; & (\sin \arg Z(0))^{-2r} \; 
+ \; 2r \sqrt{\kappa} \int_0^t (\sin \arg Z(s-))^{-2r} \frac{X(s-)}{|Z(s-)|^2} \, \ud B(s) \\
\; \; &
+ \; 2 r a \int_0^t (\sin \arg Z(s-))^{-2r} \frac{X(s-)}{|Z(s-)|^2} \, \ud s \\
\; \; &
- \; 2 r \int_0^t (\sin \arg Z(s-))^{-2r} \frac{(\kappa + 4) X(s-)^2 - \frac{\kappa}{2} Y(s-)^2}{|Z(s-)|^4} \, \ud s \\
\; \; &
+ \; r (2r + 1) \kappa \int_0^t (\sin \arg Z(s-))^{-2r}  \frac{X(s-)^2}{|Z(s-)|^4} \, \ud s \\
\; \; & 
+ \; \int_0^t \int_{|\jump| \leq \varepsilon} (\sin \arg Z(s-))^{-2r} \Bigg( \bigg| \frac{Z(s-) + \jump}{Z(s-)} \bigg|^{2r} - 1 \Bigg) \, \PoissonComp(\ud s, \ud \jump) \\
\; \; &
+ \; \int_0^t \int_{|\jump| \leq \varepsilon} (\sin \arg Z(s-))^{-2r}  \Bigg( \bigg| \frac{Z(s-) + \jump}{Z(s-)} \bigg|^{2r} - 1  
- \frac{2r \jump  X(s-)}{|Z(s-)|^2} \Bigg) \, \nu(\ud \jump) \, \ud s ,
\end{align*}
and the process $(\sin \arg Z)^{-2r} $ satisfies the It\^o-D\"oblin assumptions. 
\end{lem}

\begin{proof}
We shall apply the It\^o-D\"oblin formula 
to the real-valued process $t \mapsto \sin \arg Z(t)$ 
(i.e., we take $Z_1 = \sin \arg Z(t)$ and $Z_2 = 0$ in Theorem~\ref{thm: Ito}), with
\begin{align*}
F(s) := \; & - \sqrt{\kappa} \sin \arg Z(s-) \, \frac{X(s-)}{|Z(s-)|^2} , \\
G_1(s) := \; & \sin \arg Z(s-) \bigg( \frac{(\kappa + 4) X(s-)^2 - \frac{\kappa}{2} Y(s-)^2}{|Z(s-)|^4} - a \, \frac{X(s-)}{|Z(s-)|^2} \bigg) \\
\; \; & + \; \int_{|\jump| \leq \varepsilon} 
\bigg( \sin \arg(Z(s-) + \jump) - \sin \arg Z(s-) 
+ \sin \arg Z(s-) \, 
\frac{\jump X(s-)}{|Z(s-)|^2}  \bigg) \, \nu(\ud \jump) , \\
H(s, \jump) := \; & \sin \arg(Z(s-) + \jump)) - \sin \arg Z(s-) ,
\end{align*}
and $K(s, \jump) := 0$.
From Lemma~\ref{lem: SDE sin(arg(Z))}, we know that the It\^o-D\"oblin assumptions hold, and a the asserted identity follows after a direct computation, using the identity 
$\sin \arg (z) =\frac{y}{|z|}$ for $\bH \ni z = x + \ii y$ 
to write 
\begin{align*}
(\sin \arg(Z(s-) + \jump))^{-2r}
= \; & (\sin \arg Z(s-))^{-2r} \bigg( \frac{\sin \arg Z(s-)}{\sin \arg(Z(s-) + \jump)} \bigg)^{2r} \\
= \; & (\sin \arg Z(s-))^{-2r} \bigg| \frac{Z(s-) + \jump}{Z(s-)} \bigg|^{2r} .
\qedhere
\end{align*}
\end{proof}

\subsection{Applications involving forward Loewner flow}

Fix $\kappa \geq 0$, $a \in \bR$, $\varepsilon > 0$, and
a L\'evy measure~$\nu$.
Let $(g_t)_{t \geq 0}$ be the (absolutely continuous) solution to~\eqref{eq: LE} with driving function $W = \smash{\microDriver_\varepsilon^{\kappa, a}}$ of the form
\begin{align*} 
\microDriver_\varepsilon^{\kappa, a}(t) 
= a t + \sqrt{\kappa} B(t) + \int_{|\jump| \leq \varepsilon} \jump \, \PoissonComp(t, \ud \jump) , \qquad
a \in \bR, \, \kappa \geq 0, \, \varepsilon > 0 .
\end{align*}
Fix $z_0 = x_0 + \ii y_0 \in \bH$ and set 
\begin{align*}
Z(t) = Z_\varepsilon^{a, \kappa}(t,z_0) := g_t(z_0) - \microDriver_\varepsilon^{\kappa, a}(t) =: X(t) + \ii Y(t) .
\end{align*}
Then,~\eqref{eq: LE} shows that
\begin{align*}
X(t) \; = \; \; & x_0 \; + \; 2 \int_0^t \frac{X(s-)}{|Z(s-)|^2} \, \ud s \; - \; a t 
\; - \; \sqrt{\kappa} B(t) - \int_0^t \int_{|\jump| \leq \varepsilon} \jump \, \PoissonComp(\ud s, \ud \jump) , \\
Y(t) \; = \; \; & y_0 \; - \; 2 \int_0^t \frac{Y(s-)}{|Z(s-)|^2} \, \ud s ,
\end{align*}
for all $t < \tau(z_0)$.

We apply Theorem~\ref{thm: Ito} to derive formulas for functions of $Z$, needed in Lemmas~\ref{lem: SDE forward M} and~\ref{lem: SDE forward M with drift}.

\begin{lem} \label{lem: SDE forward sin(arg(Z))}
For each $t \in [0,\tau(z_0))$, we have almost surely
\begin{align*}
\sin \arg Z(t) \; = \; \; & \sin \arg Z(0)
\; + \; \sqrt{\kappa} \int_0^t \sin \arg Z(s-) \, \frac{X(s-)}{|Z(s-)|^2} \, \ud B(s) \\
\; \; &
+ \; a \int_0^t \sin \arg Z(s-) \, \frac{X(s-)}{|Z(s-)|^2} \, \ud s \\
\; \; &
+ \; \int_0^t \sin \arg Z(s-) \, \frac{(\kappa - 4) X(s-)^2 - \frac{\kappa}{2} Y(s-)^2}{|Z(s-)|^4} \, \ud s  \\
\; \; &
+ \; \int_0^t \int_{|\jump| \leq \varepsilon} \Big( \sin \arg(Z(s-) - \jump)) - \sin \arg Z(s-) \Big) \, \PoissonComp(\ud s, \ud \jump) \\
\; \; &
+ \; \int_0^t \int_{|\jump| \leq \varepsilon} 
\bigg( \sin \arg(Z(s-) - \jump) - \sin \arg Z(s-) 
- \sin \arg Z(s-) \, 
\frac{\jump X(s-)}{|Z(s-)|^2}  \bigg) \, \nu(\ud \jump) \, \ud s ,
\end{align*}
and the process $\sin \arg Z$ satisfies the It\^o-D\"oblin assumptions. 
\end{lem}

\begin{proof}
Note that $\sin \arg (z) =\frac{y}{|z|}$ 
is smooth on $\bH \ni z = x + \ii y$. Moreover, the process $Z$ with  
\begin{align*}
F(s) :=  -\sqrt{\kappa} , \qquad
G_1(s) := \frac{2 X(s-)}{|Z(s-)|^2} - a , \qquad
G_2(s) := - \frac{2 Y(s-)}{|Z(s-)|^2} , \qquad
H(s, \jump) := - \jump , 
\end{align*}
and $K(s, \jump) := 0$ 
satisfies the It\^o-D\"oblin assumptions 
(for Theorem~\ref{thm: Ito}, with $t \in [0,\tau(z_0))$), thanks to~\eqref{eq: LE}. 
Hence, the claim follows from a direct computation.
\end{proof}

\begin{lem} \label{lem: SDE forward sin(arg(Z)) power}
For each $t \in [0,\tau(z_0))$ and $r \in \bR$, we have almost surely
\begin{align*}
(\sin \arg Z(t))^{-2r} 
\; = \; \; & (\sin \arg Z(0))^{-2r} \; 
- \; 2r \sqrt{\kappa} \int_0^t (\sin \arg Z(s-))^{-2r} \frac{X(s-)}{|Z(s-)|^2} \, \ud B(s) \\
\; \; &
- \; 2 r a \int_0^t (\sin \arg Z(s-))^{-2r} \frac{X(s-)}{|Z(s-)|^2} \, \ud s \\
\; \; &
- \; 2 r \int_0^t (\sin \arg Z(s-))^{-2r} \frac{(\kappa - 4) X(s-)^2 - \frac{\kappa}{2} Y(s-)^2}{|Z(s-)|^4} \, \ud s \\
\; \; &
+ \; r (2r + 1) \kappa \int_0^t (\sin \arg Z(s-))^{-2r}  \frac{X(s-)^2}{|Z(s-)|^4} \, \ud s \\
\; \; & 
+ \; \int_0^t \int_{|\jump| \leq \varepsilon} (\sin \arg Z(s-))^{-2r} \Bigg( \bigg| \frac{Z(s-) - \jump}{Z(s-)} \bigg|^{2r} - 1 \Bigg) \, \PoissonComp(\ud s, \ud \jump) \\
\; \; &
+ \; \int_0^t \int_{|\jump| \leq \varepsilon} (\sin \arg Z(s-))^{-2r}  \Bigg( \bigg| \frac{Z(s-) - \jump}{Z(s-)} \bigg|^{2r} - 1  
+ \frac{2r \jump  X(s-)}{|Z(s-)|^2} \Bigg) \, \nu(\ud \jump) \, \ud s ,
\end{align*}
and the process $(\sin \arg Z)^{-2r} $ satisfies the It\^o-D\"oblin assumptions. 
\end{lem}

\begin{proof}
We shall apply the It\^o-D\"oblin formula 
to the real-valued process $t \mapsto \sin \arg Z(t)$ 
(i.e., we take $Z_1 = \sin \arg Z(t)$ and $Z_2 = 0$ in Theorem~\ref{thm: Ito}), with
\begin{align*}
F(s) := \; & \sqrt{\kappa} \sin \arg Z(s-) \, \frac{X(s-)}{|Z(s-)|^2} , \\
G_1(s) := \; & \sin \arg Z(s-) \bigg( \frac{(\kappa - 4) X(s-)^2 - \frac{\kappa}{2} Y(s-)^2}{|Z(s-)|^4} + a \, \frac{X(s-)}{|Z(s-)|^2} \bigg) \\
\; \; & + \; \int_{|\jump| \leq \varepsilon} 
\bigg( \sin \arg(Z(s-) - \jump) - \sin \arg Z(s-) 
- \sin \arg Z(s-) \, 
\frac{\jump X(s-)}{|Z(s-)|^2}  \bigg) \, \nu(\ud \jump) , \\
H(s, \jump) := \; & \sin \arg(Z(s-) - \jump)) - \sin \arg Z(s-) , 
\end{align*}
and $K(s, \jump) := 0$.
From Lemma~\ref{lem: SDE forward sin(arg(Z))}, we know that the It\^o-D\"oblin assumptions hold, and a the asserted identity follows after a direct computation, using the identity 
$\sin \arg (z) =\frac{y}{|z|}$ for $\bH \ni z = x + \ii y$ 
to write 
\begin{align*}
(\sin \arg(Z(s-) - \jump))^{-2r}
= \; & (\sin \arg Z(s-))^{-2r} \bigg( \frac{\sin \arg Z(s-)}{\sin \arg(Z(s-) - \jump)} \bigg)^{2r} \\
= \; & (\sin \arg Z(s-))^{-2r} \bigg| \frac{Z(s-) - \jump}{Z(s-)} \bigg|^{2r} .
\qedhere
\end{align*}
\end{proof}


\bigskip{}
\bibliographystyle{annotate}
\newcommand{\etalchar}[1]{$^{#1}$}

\end{document}